\documentclass[11pt, reqno]{amsart}

\usepackage[top=3.75cm, bottom=3cm, left=3cm, right=3cm]{geometry}
\usepackage{amsmath}
\usepackage{amsthm}
\usepackage{amsfonts}
\usepackage{amssymb}
\usepackage{bbm}
\usepackage{comment}
\usepackage{color}
\usepackage[shortlabels]{enumitem}
\usepackage[mathscr]{eucal}
\usepackage{graphicx}
\usepackage[colorlinks=true, citecolor=darkblue, urlcolor=darkblue, linkcolor=darkblue, linktocpage = true]{hyperref}
\usepackage{mathtools}
\usepackage{tikz}
\usepackage{tikz-cd}
\usepackage[new]{old-arrows}
\usepackage{xpatch}
\usepackage[all,cmtip]{xy}

\setlist[itemize]{topsep=5pt,itemsep=3pt}
\setlist[enumerate]{topsep=5pt,itemsep=3pt}

\definecolor{darkblue}{rgb}{0,0,.75}

\numberwithin{equation}{section}

\theoremstyle{plain}
\newtheorem{theorem}{Theorem}[section]
\newtheorem{lemma}[theorem]{Lemma}
\newtheorem{proposition}[theorem]{Proposition}
\newtheorem{corollary}[theorem]{Corollary}
\theoremstyle{remark}
\newtheorem*{claim*}{Claim}
\theoremstyle{definition}
\newtheorem{setup}[theorem]{Setup}
\newtheorem{definition}[theorem]{Definition}
\newtheorem{example}[theorem]{Example}
\newtheorem{remark}[theorem]{Remark}
\newtheorem{notation}[theorem]{Notation}
\newtheorem{warning}[theorem]{Warning}

\makeatletter
\newtheoremstyle{italicsname}
{3pt}
{3pt}
{\itshape}
{}
{\bfseries}
{.}
{.5em}
{\thmname{#1}\thmnumber{\@ifnotempty{#1}{ }#2}%
	\thmnote{ {\the\thm@notefont(#3)}}}
\makeatother
\theoremstyle{italicsname}
\newtheorem{innerstep}{Step}
\newenvironment{step}[1]
{\renewcommand\theinnerstep{#1}\innerstep}
{\endinnerstep}

\newcommand{\st}{\mid}
\newcommand{\set}[1]{\left\{ \, #1 \, \right\}}
\newcommand{\id}{\mathrm{id}}
\newcommand{\wtilde}{\widetilde}
\def\Re{\mathop{\rm Re}\nolimits}
\def\Im{\mathop{\rm Im}\nolimits}
\newcommand{\bv}{\mathbf{v}}
\newcommand{\bw}{\mathbf{w}}

\def\Ctor{{\ensuremath{C\text{-}\mathrm{tor}}}}
\def\Ctf{{\ensuremath{C\text{-}\mathrm{tf}}}}
\newcommand{\Db}{\mathrm{D^b}}
\newcommand{\Dperf}{\mathrm{D}_{\mathrm{perf}}}
\newcommand{\Dqc}{\mathrm{D}_{\mathrm{qc}}}
\newcommand{\Dpug}{\mathrm{D_{pug}}}
\newcommand{\Coh}{\mathrm{Coh}}
\newcommand{\num}{\mathrm{num}}
\DeclareMathOperator{\Knum}{\rK_{\num}}
\DeclareMathOperator{\Stab}{Stab}
\newcommand{\Stabdagger}{\mathrm{Stab}^\dagger}

\DeclareMathOperator{\Coprod}{Coprod}
\DeclareMathOperator{\Pic}{Pic}
\DeclareMathOperator{\Spec}{Spec}
\DeclareMathOperator{\RHom}{RHom}
\DeclareMathOperator{\Hom}{Hom}
\DeclareMathOperator{\Ext}{Ext}
\DeclareMathOperator{\Aut}{Aut}
\DeclareMathOperator{\GL}{GL}
\DeclareMathOperator{\Quot}{Quot}
\DeclareMathOperator{\Tor}{Tor}
\DeclareMathOperator{\cone}{cone}
\DeclareMathOperator{\ch}{{ch}}
\DeclareMathOperator{\rk}{{rk}}
\newcommand{\td}{\mathrm{td}}

\newcommand{\cO}{\mathcal{O}}
\newcommand{\cE}{\mathcal{E}}
\newcommand{\cF}{\mathcal{F}}
\newcommand{\cG}{\mathcal{G}}
\newcommand{\cH}{\mathcal{H}}
\newcommand{\cL}{\mathcal{L}}
\newcommand{\cM}{\mathcal{M}}
\newcommand{\cN}{\mathcal{N}}
\newcommand{\cT}{\mathcal{T}}
\newcommand{\cZ}{\mathcal{Z}}

\newcommand{\cA}{\mathscr{A}}
\newcommand{\cC}{\mathscr{C}}
\newcommand{\cD}{\mathscr{D}}
\newcommand{\cP}{\mathscr{P}}

\newcommand{\rH}{\mathrm{H}}
\newcommand{\rK}{\mathrm{K}}
\newcommand{\fS}{\mathfrak{S}}
\newcommand{\bZ}{\mathbf{Z}}

\newcommand{\KK}{\mathbbm{k}}
\def\CC{\ensuremath{\mathbb{C}}}
\def\QQ{\ensuremath{\mathbb{Q}}}
\def\RR{\ensuremath{\mathbb{R}}}
\def\ZZ{\ensuremath{\mathbb{Z}}}
\def\PP{\ensuremath{\mathbb{P}}}

\makeatletter   
\xpatchcmd{\@tocline}
{\hfil\hbox to\@pnumwidth{\@tocpagenum{#7}}\par}
{\ifnum#1<0\hfill\else\dotfill\fi\hbox to\@pnumwidth{\@tocpagenum{#7}}\par}
{}{}
\makeatother 

\makeatletter
\renewcommand\part{%
	\if@noskipsec \leavevmode \fi
	\par
	\addvspace{4ex}%
	\@afterindentfalse
	\secdef\@part\@spart}

\def\@part[#1]#2{%
	\ifnum \c@secnumdepth >\m@ne
	\refstepcounter{part}%
	\addcontentsline{toc}{part}{Part \thepart.\hspace{1em}#1}%
	\else
	\addcontentsline{toc}{part}{#1}%
	\fi
	{\parindent \z@ \raggedright
		\interlinepenalty \@M
		\normalfont
		\ifnum \c@secnumdepth >\m@ne
		\centering 
		\Large\bfseries \partname\nobreakspace\thepart     
		\nobreak. 
		\fi
		\Large \bfseries { #2}%
	\par}%
\nobreak
\vskip 3ex
\@afterheading}
\def\@spart#1{%
{\parindent \z@ \raggedright
	\interlinepenalty \@M
	\normalfont
	\huge \bfseries #1\par}%
\nobreak
\vskip 3ex
\@afterheading}
\makeatother

\renewcommand{\thepart}{\Roman{part}}

\makeatletter 
\def\l@subsection{\@tocline{2}{0pt}{3pc}{6pc}{}} 
\makeatother

\def\citestacks#1{\cite[\href{https://stacks.math.columbia.edu/tag/#1}{#1}]{stacks-project}}

\DeclareRobustCommand\longtwoheadrightarrow
{\relbar\joinrel\twoheadrightarrow}

\title[Stability conditions and moduli spaces on projective families]{Stability conditions and moduli spaces on \\ projective families} 

\begin{document}

\author[C.~Li]{Chunyi Li}
\address{\parbox{0.9\textwidth}{Mathematics Institute (WMI), University of Warwick\\[1pt]
Coventry, CV4 7AL, United Kingdom
\vspace{1mm}}}
\email{C.Li.25@warwick.ac.uk}
\urladdr{\url{https://sites.google.com/site/chunyili0401/}}

\author[Z.~Liu]{Zhiyu Liu}
\address{\parbox{0.9\textwidth}{Department of Mathematics, Princeton University\\[1pt] Princeton, NJ 08544, USA
\vspace{1mm}}}
\email{zl3301@princeton.edu}
\urladdr{\url{https://sites.google.com/view/zhiyuliu}}

\author[Z.~Liu]{Ziqi Liu}
\address{\parbox{0.9\textwidth}{Dipartimento di Matematica ``F.~Enriques'', Universit\`a degli Studi di Milano\\[1pt]
Via Cesare Saldini 50, 20133 Milano, Italy
\vspace{1mm}}}
\email{ziqi.liu@unimi.it}
\urladdr{\url{https://sites.google.com/view/ziqiliu}}

\author[E.~Macr\`i]{Emanuele Macr\`i}
\address{\parbox{0.9\textwidth}{Universit\'e Paris-Saclay, CNRS, Laboratoire de Math\'ematiques d'Orsay\\[1pt]
Rue Michel Magat, B\^at. 307, 91405 Orsay, France
\vspace{1mm}}}
\email{{emanuele.macri@universite-paris-saclay.fr}}
\urladdr{\url{https://www.imo.universite-paris-saclay.fr/~macri/}}

\author[A.~Perry]{Alexander Perry}
\address{\parbox{0.9\textwidth}{Department of Mathematics, University of Michigan\\[1pt]
530 Church Street, Ann Arbor, MI 48109, USA
\vspace{1mm}}}
\email{arper@umich.edu}
\urladdr{\url{http://www-personal.umich.edu/~arper/}}

\author[P.~Stellari]{Paolo Stellari}
\address{\parbox{0.9\textwidth}{Dipartimento di Matematica ``F.~Enriques'', Universit\`a degli Studi di Milano\\[1pt]
Via Cesare Saldini 50, 20133 Milano, Italy
\vspace{1mm}}}
\email{paolo.stellari@unimi.it}
\urladdr{\url{https://sites.unimi.it/stellari}}

\author[X.~Zhao]{Xiaolei Zhao}
\address{\parbox{0.9\textwidth}{Department of Mathematics, University of California Santa Barbara\\[1pt]
552 University Rd, Santa Barbara, CA 93117, USA
\vspace{1mm}}}
\email{xlzhao@ucsb.edu}
\urladdr{\url{https://sites.google.com/site/xiaoleizhaoswebsite/}}


\begin{abstract}
We extend the construction of stability conditions on projective schemes over a field to projective families over an arbitrary base, and prove that they admit proper relative moduli spaces of semistable objects. 
We also prove a number of complementary results: the existence of mass-Hom bounds for these stability conditions, as conjectured by Halpern-Leistner and Robotis; a comparison with tilt-stability on surfaces and threefolds; a construction of stability conditions on the supported derived category of total spaces of certain vector bundles, including all local Calabi--Yau varieties; and a simple new proof of Bondal and Orlov's reconstruction theorem. 
\end{abstract}

\maketitle

\setcounter{tocdepth}{1}
\tableofcontents


\section{Introduction}\label{sec:intro}

Despite the profound influence of Bridgeland stability conditions in algebraic and symplectic geometry, representation theory, and mathematical physics, 
their construction on arbitrary projective varieties remained elusive for many years. 
The paper~\cite{chunyi-stability} solved this problem for smooth projective varieties over the complex numbers. 
The purpose of this article is to extend the construction to projective families over an arbitrary base, and to establish the existence of proper moduli spaces of semistable objects. 

\subsection{Main result}
\label{subsec:BridgelandIntro}

Let $X$ be a projective scheme over a field $\KK$. 
An ample line bundle on $X$ can be used to define the classical notion of (Gieseker or slope) stability for coherent sheaves on $X$. 
A stability condition $\sigma$ on $\Db(X)$, as introduced by Bridgeland~\cite{Bridgeland:Stab} and reviewed in Section~\ref{subsec:PreliminaryStability}, similarly provides a notion of stability for objects in the bounded derived category of coherent sheaves $\Db(X)$. 
Unfortunately, unlike the case of Gieseker stability, it is not known in general whether the moduli functor parameterizing $\sigma$-semistable objects (with fixed class and phase) is necessarily representable by an algebraic stack, and if so whether the algebraic stack admits a proper good moduli space. 
This is a particularly important problem, as virtually all applications of stability conditions in algebraic geometry go through the use of such moduli spaces. 

The work \cite{stability-families} gave a partial solution to this problem, 
while simultaneously generalizing the theory to the relative setting. 
Namely, let $\pi \colon X \to S$ be a flat projective morphism of schemes. 
We assume that $X$ and $S$ satisfy some very mild technical hypotheses, recorded in Setup~\ref{assum-stab}; in particular, $X$ and $S$ could be arbitrary quasi-projective varieties, or quite general schemes of mixed characteristic. 
Then \cite{stability-families} introduced the notion of a \emph{stability condition $\sigma$ on $\Db(X)$ over $S$}, 
which is a collection $\sigma_s$ of stability conditions on the fibers $\Db(X_s)$ for all (possibly non-closed) points $s \in S$, subject to several compatibility and tameness conditions reviewed in Section~\ref{sec:StabilityFamily}. 
Parallel to the classical case, $\sigma$ is defined with respect to an auxiliary homomorphism $v \colon \Knum(X/S) \to \Lambda$, where $\Knum(X/S)$ is the relative numerical K-group (see Section \ref{subsec:RelativeKnum}) and $\Lambda$ is a finite rank free abelian group. 
We emphasize that even when $S = \Spec(\KK)$ is the spectrum of a field, an arbitrary stability condition $\sigma$ on $\Db(X)$ is not a priori a stability condition \emph{over $\KK$} (i.e., over $\Spec(\KK)$) in this sense.

The theory of relative stability conditions satisfies two key theorems. 
First, the moduli functor parameterizing relatively $\sigma$-semistable objects of a fixed class is an algebraic stack, which is of finite type and quasi-proper over $S$, and in characteristic $0$ admits a proper good moduli space~\cite{AHLH}. 
Second, extending Bridgeland's fundamental deformation theorem, 
the collection $\Stab_{(\Lambda, v)}(\Db(X)/S)$ of all stability conditions over $S$ with respect to the data $(\Lambda, v)$ forms a complex manifold of dimension $\rk \Lambda$.  

Our main result can now be summarized as follows. 

\begin{theorem}
\label{main-theorem} 
Let $\pi \colon X \to S$ be a flat projective morphism of schemes, where $X$ and $S$ satisfy the assumptions of Setup~\ref{assum-stab}. 
Let $H_{X/S}$ be the relative numerical class of a relatively ample line bundle on $X/S$. 
Then there is a homomorphism $v_{H_{X/S}} \colon \Knum(X/S) \to \Lambda_{H_{X/S}}$ to a finite rank free abelian group and a nonempty distinguished connected component  
\begin{equation*}
\varnothing \neq \Stabdagger_{{H_{X/S}}}(\Db(X)/S) \subset \Stab_{\left(\Lambda_{H_{X/S}},v_{H_{X/S}}\right)}(\Db(X)/S)
\end{equation*}
of the stability manifold which contains geometric stability conditions. 
\end{theorem}

Here, a stability condition $\sigma$ on $\Db(X)$ over $S$ is called \emph{geometric} if for every geometric point $\overline{s}$ over $S$ and every closed point $\overline{x} \in X_{\overline{s}}$ of the fiber, the skyscraper sheaf $\cO_{\overline{x}} \in \Db(X_{\overline{s}})$ is $\sigma_{\overline{s}}$-stable of a fixed phase (see Definition~\ref{def:GeometricSatbility}).  
As the name suggests, such stability conditions are important because of their tight connection to the geometry of $X$ (see e.g. Section~\ref{section-intro-BO} below). 

In view of the discussion above, the stability conditions on $\Db(X)$ over $S$ given by Theorem~\ref{main-theorem} admit well-behaved moduli spaces of semistable objects. 
This result is new even in the absolute case where the base $S = \Spec(\CC)$ is the spectrum of the complex numbers. 

Theorem~\ref{main-theorem} is a combination of the more precise Theorems~\ref{thm:ProjectiveScheme},~\ref{thm:canonical-component}, and~\ref{thm:MainSecondVersion} in the body of the paper. 
In particular, Theorem~\ref{thm:MainSecondVersion} provides an explicit $2$-parameter family of stability conditions $\wtilde{\sigma}_{X/S}^{a,b}$ on $\Db(X)$ over $S$, parameterized by $(a,b) \in \RR_{>0} \times \RR$ with $a \gg 0$, generalizing the construction from \cite{chunyi-stability} for smooth projective varieties over the complex numbers. 
Moreover, we note that $v_{H_{X/S}} \colon \Knum(X/S) \to \Lambda_{H_{X/S}}$ is defined explicitly in terms of the choice of a closed embedding $\iota \colon X \hookrightarrow \PP^n_S$ such that the class of a multiple $mH_{X/S}$ is equal to the class of $\iota^*\cO_{\PP^n_S}(1)$ (see~\eqref{definition-vHXS}), but the resulting connected component $\Stabdagger_{{H_{X/S}}}(\Db(X)/S)$ does not depend on this choice. 

\begin{remark}
    When $S = \Spec(\KK)$ is the spectrum of a field (in which case we write simply $H_X$ for $H_{X/S}$), there is an incarnation of Theorem~\ref{main-theorem} that makes no reference to relative stability conditions. Namely, Theorems~\ref{thm:ProjectiveScheme} and ~\ref{thm:canonical-component} give a nonempty distinguished connected component 
    \begin{equation*}
        \varnothing \neq \Stabdagger_{H_X}(\Db(X)) \subset \Stab_{(\Lambda_{H_X}, v_{H_X})}(\Db(X)) 
    \end{equation*}
    of the classical stability manifold which contains geometric stability conditions. As we observe in Corollary~\ref{cor:Stabdagger=Stabdagger}, this agrees with the distinguished component of the relative stability manifold, i.e., we have  
    \begin{equation*}
        \Stabdagger_{H_X}(\Db(X)) = \Stabdagger_{H_X}(\Db(X)/\Spec(\KK)). 
    \end{equation*}
Concretely, this implies that all of the stability conditions in the left-hand side admit well-behaved moduli spaces of semistable objects. 
\end{remark}

Before sketching the main ideas of proof of Theorem~\ref{main-theorem}, we highlight some complementary results of independent interest that we obtain along the way. 

\subsection{Mass-Hom bounds} 
Halpern-Leistner and Robotis~\cite{DHL-robotis} introduced an interesting perspective on the problem of the existence of moduli spaces of $\sigma$-semistable objects.
This depends on the notion of a \emph{mass-Hom bound} for $\sigma$, which says that on a projective scheme $X$ over a field $\KK$, for any objects $A\in\Dperf(X)$ and $F\in\Db(X)$, the quantity $\dim_\KK \Hom(A,F)$ is bounded by the mass of $F$ with respect to $\sigma$ up to a constant which depends only on $A$ (see Definition~\ref{definition-mass-hom}). 
In these terms, if $X$ is smooth and $\mathrm{char}(\KK)=0$, they proved that the existence of a mass-Hom bound implies the existence of proper good moduli spaces of semistable objects; see Theorem~\ref{theorem-DHL-robotis} for the precise statement. 
Moreover, they conjectured that any stability condition on a smooth proper variety (or more generally, a smooth proper category) admits a mass-Hom bound. 
We provide significant evidence for their conjecture. 

\begin{theorem}\label{thm:MassHomBoundMain}
Let $X$ be a projective scheme over a field $\KK$ equipped with an ample numerical class $H_X$.  
Then any stability condition $\sigma \in \Stabdagger_{H_X}(\Db(X))$ admits a mass-Hom bound. 
\end{theorem}

In particular, by Theorem~\ref{theorem-DHL-robotis} mentioned above, for smooth $X$ this yields a different argument for 
the existence of well-behaved moduli spaces of $\sigma$-semistable objects, established more generally in Theorem~\ref{main-theorem}. 

\begin{remark}
    The existence of a mass-Hom bound is invariant under derived equivalences. Hence,  Theorem~\ref{thm:MassHomBoundMain} holds more generally for any stability condition contained in the union 
    \[
\Stab^\ast_{H_X}(\Db(X))\coloneqq\bigcup\limits_{\scriptstyle\Phi\in \Aut_{(\Lambda_{H_X}, v_{H_X})}(\Db(X))}\Phi\big(\Stabdagger_{H_X}(\Db(X)\big), 
\]
where $\Aut_{(\Lambda_{H_X}, v_{H_X})}(\Db(X))$ is the group of autoequivalences which preserve $(\Lambda_{H_X}, v_{H_X})$ (see~\eqref{AutLambdav} for the precise definition). 
\end{remark}

\subsection{Tilt-stability} 
Prior to the work~\cite{chunyi-stability}, the main approach to constructing stability conditions proceeded via tilting from coherent sheaves~\cite{Bridgeland:K3,TodaK3,BMT:BG,BMS:StabCY3s}. 
This method led to the existence of stability conditions on all smooth projective surfaces \cite{ArcaraBertram}, as well as on many special classes of threefolds \cite{Li:FanoThreefolds,Li:CY,FKLR:CY,LiuMao:Tilt,PT:Moduli}. 
In the threefold case, this method reduces the problem to proving the Bayer--Macr\`{i}--Toda (BMT) Conjecture~\cite[Conjecture 4.1]{BMS:StabCY3s}, which predicts an inequality involving the Chern characters of certain tilt-semistable objects.

When $(X, H_X)$ is a polarized complex surface, the stability conditions on $\Db(X)$ constructed by tilting are contained in the distinguished connected component $\Stabdagger_{H_X}(\Db(X))$; in fact, as we observe in Proposition~\ref{prop:MassHomBoundSurfaces}, this holds more generally for any geometric stability condition on a surface. 
Similarly, when $(X, H_X)$ is a polarized complex threefold satisfying the BMT conjecture and a technical condition on the lattice $(\Lambda_{H_X}, v_{H_X})$, we prove in Proposition~\ref{prop:BMTmassbound} that the stability conditions constructed by tilting are contained in $\Stabdagger_{H_X}(\Db(X))$. 
In particular, by Theorem~\ref{thm:MassHomBoundMain}, all of these stability conditions admit mass-Hom bounds. 
  
\begin{remark}
    In the recent paper \cite{FeyYiran:Threefolds}, Cheng and Feyzbakhsh prove that for any polarized complex threefold $(X, H_X)$, certain stability conditions in $\Stabdagger(\Db(X))$ (the ones in Theorem~\ref{thm:ProjectiveScheme}) are in fact obtained from the tilting construction, and they use this to deduce a weaker form of the BMT conjecture.
\end{remark}

\subsection{Supported derived categories of total spaces of vector bundles} 

Beyond projective varieties, if we are given a quasi-projective variety $Y$ containing a projective variety $X \subset Y$, then we may consider the bounded derived category $\mathrm{D}_X^{\mathrm{b}}(Y)$ of coherent sheaves on $Y$ with set-theoretic support on $X$. 
An important case is when $Y = \mathrm{Tot}(\cE)$ is the total space of a vector bundle on $X$ and $X \hookrightarrow Y$ is embedded as the $0$-section, in which case stability conditions on $\mathrm{D}_X^{\mathrm{b}}(Y)$ have connections to  Springer theory (see, e.g.,~\cite{Bridgeland:NonCompactCY,ABM:Stability}),  counting invariants (see, e.g.,~\cite{Bou:Takahashi,BDLP:Dentroscopy}), and quadratic differentials (see, e.g.,~\cite{BridgelandSmith:IHES,HKK}).  

Bootstrapping from the projective case, we obtain stability conditions on many interesting categories of this type. 
A coarse version of the result is as follows; see Theorems~\ref{thm:stabonlocal} and \ref{thm:local-cy} for more precise statements. 

\begin{theorem}\label{thm:LocalMain}
Let $X$ be a projective scheme over a field $\KK$, and let $\cE$ be a vector bundle on $X$ such that $\cE^{\vee}$ is globally generated or is an ample line bundle. 
Then the category $\mathrm{D}_X^{\mathrm{b}}(\mathrm{Tot}(\cE))$ admits stability conditions.
\end{theorem}

As a special case, when $X$ is a Fano variety, we obtain stability conditions on the so-called \emph{local Calabi--Yau} $\mathrm{Tot}(\omega_X)$. 
In fact, for any line bundle $\cL$ on $X$, we also find a more precise criterion for when a stability condition on $\Db(X)$ induces one on $\mathrm{D}_X^{\mathrm{b}}(\mathrm{Tot}(\cL))$, as explained in Remark~\ref{remark-inducing-TotL}. 

\subsection{Stability of skyscraper sheaves and the Bondal--Orlov theorem} 
\label{section-intro-BO} 
Bondal and Orlov famously proved that for a smooth projective variety $X$ with ample or antiample canonical bundle, the derived category $\Db(X)$ determines $X$. 
Using the theory of moduli spaces of stable objects, we obtain a new proof of this result (see Theorem~\ref{thm:BOreconstruction}). 

The key point is that for certain stability conditions $\sigma \in \Stabdagger_{H_X}(\Db(X))$ (the ones in Theorem~\ref{thm:ProjectiveScheme}), where $H_X = \pm K_X$ according to whether $K_X$ or $-K_X$ is ample, we can show that $X$ is recovered as a moduli space of $\sigma$-stable objects, with the isomorphism given by sending a closed point $x \in X$ to the skyscraper sheaf $\cO_x$; see Proposition~\ref{prop:Mpt=X} for the precise statement. 

\subsection{Outline of the proof of the main theorem} 

Let us explain the strategy of the proof of Theorem~\ref{main-theorem}. 
We first consider the case where $S = \Spec(\KK)$ is the spectrum of a field, and explain the construction of stability conditions on $X$ \emph{over $\KK$} (in the technical sense recalled in Section~\ref{subsec:BridgelandIntro} above). 

Fix an elliptic curve $E$ over $\KK$. 
By \cite{Liu:StabProd} (see also \cite{curves} and Theorem~\ref{thm:Yucheng}), any power $E^n$ admits stability conditions $\sigma^{a,b}$ parameterized by $(a,b)\in \QQ_{>0}\times\QQ$; in fact, as we prove in Proposition~\ref{prop:sigmaabRonEn}, there is a continuous extension to a family parameterized by $(a,b) \in \RR_{>0} \times \RR$. 
There is a natural finite map $g \colon E^n\to\PP^n$, obtained by quotienting by $-1$ on each factor and the symmetric group permuting the factors. 
Then the arguments of~\cite{chunyi-stability} (adapted to an arbitrary base field $\KK$) show that:  
\begin{enumerate}[{\rm (i)}]
    \item\label{enum:IntroSketch1} $\sigma^{a,b}$ can be ``pushed forward'' along $g$ to a stability condition $\sigma_{\PP^n}^{a,b}$ on~$\PP^n$ (Theorem~\ref{theorem-chunyi-Pn}). 
    \item\label{enum:IntroSketch2} 
    If $\iota \colon X \hookrightarrow \PP^n$ is a closed immersion, then 
    $\sigma^{a,b}_{\PP^n}$ can be ``pulled back'' along $\iota$ to a stability condition $\wtilde{\sigma}_X^{a,b}$ on $X$  for $a \gg 0$ (Theorem~\ref{thm:ProjectiveScheme}). 
\end{enumerate}

In Section~\ref{sec:PullbackPushforward} we develop the theory of pullback and pushforward of stability conditions in the absolute setting, and in Section~\ref{sec:PullPushRelativeStability} we treat the relative setting; in particular, in Theorem~\ref{thm:main-criterion} we prove that these operations preserve the property of being a stability condition over a base. 
Thus, the problem reduces to proving that the stability condition $\sigma^{a,b}$ on $E^n$ is a stability condition over $\KK$.  
We do not know how to prove this statement for $\sigma^{a,b}$, so we work on an intermediate space instead. 
Namely, the morphism $g \colon E^n \to \PP^n$ factors as 
\begin{equation*}
    g \colon E^n \xrightarrow{\, f \,} (\PP^1)^n \xrightarrow{\, h \,} \PP^n, 
\end{equation*}
where $f$ is the quotient by multiplication by $-1$ on each factor and $h$ is the quotient by the symmetric group $\fS_n$ permuting the factors. 
The proof of~\ref{enum:IntroSketch1} above proceeds by showing that $\sigma^{a,b}$ pushes forward along $f$ to a stability condition $\sigma^{a,b}_{(\PP^1)^n}$ on $(\PP^1)^n$, 
and then that $\sigma^{a,b}_{(\PP^1)^n}$ pushes forward along $h$ to a stability condition $\sigma^{a,b}_{\PP^n}$ on $\PP^n$. 
Thus, by the above logic, it suffices to prove that $\sigma^{a,b}_{(\PP^1)^n}$ is a stability condition over $\KK$. 

The key point is that the stability conditions $\sigma_{(\PP^1)^n}^{a,b}$ on $(\PP^1)^n$ are actually \emph{algebraic} for $b$ suitably close to $1$ and $a>0$ sufficiently small (Theorem~\ref{thm:AlgebraicP1n}).
Concretely, this means that the moduli of stable objects for $\sigma_{(\PP^1)^n}^{a,b}$ identifies with the moduli of  stable modules over an associative finite-dimensional algebra over $\KK$ (in the sense of King~\cite{King}), and from this it follows that $\sigma_{(\PP^1)^n}^{a,b}$ is a stability condition over $\KK$ (Proposition~\ref{prop:AlgebraicP1nFamily}). 
We note that, although our argument sidesteps proving that the stability condition $\sigma^{a,b}$ on $E^n$ is one over $\KK$, the proof of the algebraicity of $\sigma_{(\PP^1)^n}^{a,b}$ crucially relies on an analysis of $\sigma^{a,b}$-stable line bundles on~$E^n$ (see Proposition~\ref{prop:LineBundleStability}). 
By a deformation argument, it follows that $\sigma_{(\PP^1)^n}^{a,b}$ is a stability condition over $\KK$ for all parameters $(a,b) \in \RR_{>0} \times \RR$.

Altogether, this explains the construction of a continuous family of stability conditions $\wtilde{\sigma}^{a,b}_X$ on $X$ over $\KK$ for $(a, b) \in \RR_{>0} \times \RR$ with $a \gg 0$. 
These stability conditions are geometric because the stability conditions $\sigma^{a,b}$ are geometric  (Theorem~\ref{thm:Eninj}\ref{enum:Eninj1}) and this property is preserved under pushforward and pullback. 
The distinguished connected component appearing in Theorem~\ref{main-theorem} is then the connected component containing the family $\wtilde{\sigma}_X^{a,b}$; 
in Section~\ref{subsec:ConnectedComponent}, we show that the distinguished component only depends on the underlying ample numerical class, not the embedding $\iota \colon X \hookrightarrow \PP^n$. 

Finally, let us consider the case of a general base scheme $S$. 
Again, we choose a closed immersion $\iota \colon X \hookrightarrow \PP^n_S$. 
Then for each $s \in S$, we obtain by the absolute case discussed above a stability condition $\wtilde{\sigma}_{X_s}^{a,b}$ on $X_s$ for $a \gg 0$; as we show in Theorem~\ref{thm:ProjectiveScheme}, there is a bound for how large $a$ must be chosen which depends only on the Hilbert polynomial of $X_s \subset \PP^n_s$, so $a \gg 0$ may be chosen uniformly across all fibers. 
The collection $\wtilde{\sigma}^{a,b}_X = (\wtilde{\sigma}_{X_s}^{a,b})_{s \in S}$ is then proved to be a stability condition on $X$ over $S$ along similar lines to the absolute case, by 
reducing to the assertion that the collection $\sigma^{a,b}_{(\PP^1_S)^n} = (\sigma^{a,b}_{(\PP^1_s)^n})_{s \in S}$ forms a stability condition on $(\PP^1_S)^n$ over $S$.

\subsection{Open questions and future directions} 
This paper establishes foundational results about stability conditions and moduli spaces on projective families, putting the theory on nearly equal footing with the general theory of Gieseker stability for coherent sheaves.  
It also suggests various avenues for future research, some of which we highlight below.

\subsubsection*{Projectivity of moduli spaces} 
An important outstanding question is whether the moduli spaces of semistable objects for the stability conditions in Theorem~\ref{main-theorem} are projective, instead of merely proper. 
Projectivity holds for moduli of Gieseker semistable sheaves, but the proof relies on a GIT interpretation of the moduli space, which is unavailable in our setting. 
Due to \cite{BM-nef-divisor} (see Remark~\ref{rmk:LineBundle}), there is always at least a natural strictly nef divisor on the moduli space of semistable objects; it would be very interesting to understand if, or under which assumptions, this divisor gives projectivity.

\subsubsection*{The full support property} 
In this paper, we show that stability conditions exist on a polarized scheme $(X, H_X)$ over a field $\KK$ with respect to a certain quotient $\Knum(X) \twoheadrightarrow \Lambda_{H_X}$ of the numerical Grothendieck group. 
An important open question is whether the support property is in fact satisfied with respect to the entire numerical Grothendieck group. 
One can deduce a stronger support property by embedding $X$ into a product of projective spaces, but this is not enough; for instance, in the case of cubic fourfolds which are not very general, it is not clear how to obtain the full support property with our methods. 

\subsubsection*{Stability conditions on other categories} 
Beyond the derived categories of projective varieties, 
there are a number of other interesting categories arising in algebraic geometry on which stability conditions are expected, but not yet known, to exist. 
One natural example is the derived category of a smooth proper Deligne--Mumford stack.
Relatedly, it would be interesting to construct stability conditions on twisted derived categories $\Db(X, \alpha)$, where $\alpha$ is a Brauer class, even in the case where $X$ is a smooth projective variety (\cite{Yos:Twisted,HS:Twisted,Lieblich:Twisted,HP:PeriodIndex}). 
Our techniques do not seem to directly yield stability conditions in these examples. 

More generally, it is an interesting problem to construct stability conditions on semiorthogonal components $\cC \subset \Db(X)$ of the derived categories of smooth projective varieties. 
While there are some partial results in this direction \cite{BLMS, PPZ:GM, Peize-cubic-fivefolds, KLP}, 
in general it is not even clear when to expect stability conditions to exist on such a category $\cC$.

\subsection{Organization of the paper} 
The article is divided into three parts. 

Part~\ref{part:Absolute}, consisting of Sections~\ref{sec:StabilityConditionsBaseChange}--\ref{sec:Complements}, treats the absolute setting, i.e., the case of stability conditions in the classical sense (as opposed to relative \emph{\`{a} la}~\cite{stability-families}) on projective schemes over a field. 
The main results are the identification of algebraic stability conditions on powers of the projective line (Theorem~\ref{thm:AlgebraicP1n}), 
the extension of the construction in~\cite{chunyi-stability} to projective schemes over an arbitrary base field (Theorem~\ref{thm:ProjectiveScheme}), and our complementary results on mass-Hom bounds (Theorem~\ref{thm:MassHomBoundProjective}), tilt-stability (Propositions~\ref{prop:MassHomBoundSurfaces} and~\ref{prop:BMTmassbound}), and stability conditions on supported derived categories (Theorem~\ref{thm:stabonlocal} and \ref{thm:local-cy}). We refer to the beginning of Part~\ref{part:Absolute} for a more detailed overview of these results. 

In Part~\ref{part:Relative}, consisting of Sections~\ref{sec:StabilityFamily}--\ref{sec:BondalOrlov}, we turn to the relative setting, i.e., the case of stability conditions on projective families over a base. 
The main results are criteria for pullback and pushforward to preserve stability conditions over a base (Theorem~\ref{thm:main-criterion}), the main theorem of the paper (Theorem~\ref{main-theorem}), and a characterization of skyscraper sheaves and its application to the Bondal--Orlov reconstruction theorem (Section~\ref{sec:BondalOrlov}). 

Part~\ref{part:Appendices} consists of two appendices. 
In Appendix~\ref{app:Inducing}, we review and extend some fundamental results of Polishchuk~\cite{P-t-structures} on induced t-structures. 
In Appendix~\ref{app:Filtration}, we give a different proof for the geometric part of the argument to pushforward stability conditions from $(\PP^1)^n$ to $\PP^n$, which works over any field.

\subsection{Conventions}
For a scheme $X$, $\Dqc(X)$ denotes the unbounded derived category of $\cO_X$-modules with quasi-coherent cohomology, $\Db(X)\subset\Dqc(X)$ the full subcategory of pseudo-coherent complexes with bounded cohomology (which in the noetherian case coincides with the bounded derived category of coherent sheaves), and $\Dperf(X)\subset\Dqc(X)$ the subcategory of perfect complexes. 
All functors are derived. 
Our conventions on stability conditions are reviewed in Section~\ref{sec:StabilityConditionsBaseChange}. 

\subsection{Acknowledgements}
The paper benefited from many useful discussions with the following people, whom we gratefully acknowledge: Arend Bayer, Tom Bridgeland, Soheyla Feyzbakhsh, Naoki Koseki, Alexander Kuznetsov, Shengxuan Liu, and Laura Pertusi.

During the preparation of this paper, C.L.~was supported by the Royal Society University Research Fellowship URF\textbackslash R1\textbackslash 251025 ``Stability condition and application in algebraic geometry''; Zh.L.~was supported by NSFC Grant 123B2002; E.M.~was partially supported by the ERC Synergy grant ERC-2020-SyG-854361-HyperK; A.P.~was partially supported by the NSF CAREER grant DMS-2143271; P.S.~was partially supported by the ERC Consolidator grant ERC-2017-CoG-771507-StabCondEn, by the research project PRIN 2017 ``Moduli and Lie Theory'', and by the research project FARE 2018 HighCaSt (grant number R18YA3ESPJ); and X.Z.~was partially supported by the NSF FRG grant DMS-2052665. 
Zi.L.~ and P.S.~ are members of the INdAM research group GNSAGA.


\newpage 

\part{Absolute case}\label{part:Absolute}


We fix a field $\KK$.
The goal for the first part of the paper is to construct and study Bridgeland stability conditions on projective schemes over $\KK$, without yet considering moduli spaces of semistable objects. 
The main results, Theorem~\ref{theorem-chunyi-Pn} and Theorem~\ref{thm:ProjectiveScheme}, are extensions of~\cite{chunyi-stability} to the singular case over an arbitrary field, which follow from similar arguments. 
In particular, for any projective scheme $X$ over $\KK$ equipped with an embedding $\iota \colon X \hookrightarrow \PP^n$, we obtain a distinguished connected component $\Stabdagger_{\iota}(\Db(X))$ in the space of stability conditions (with respect to a lattice defined in terms of $\iota$). 
Our main new result is Theorem~\ref{thm:AlgebraicP1n}, which shows that certain stability conditions in $\Stabdagger_{\iota}(\Db(X))$ come from algebraic stability conditions on $\Db((\PP^1)^n)$. 
We use this result to show that all stability conditions in $\Stabdagger_{\iota}(\Db(X))$ have a mass-Hom bound (Theorem~\ref{thm:MassHomBoundProjective}), and later in Part~\ref{part:Relative} to show the existence of proper moduli spaces. 
In Section~\ref{subsec:ConnectedComponent}, we also prove that $\Stabdagger_{\iota}(\Db(X))$ only depends on the underlying ample numerical class, not the embedding $\iota \colon X \hookrightarrow \PP^n$. 

In Section~\ref{sec:StabilityConditionsBaseChange}, we review the basics of stability conditions, as well as their behavior under base change, which plays an important role in our arguments.
In Section~\ref{sec:PullbackPushforward}, we study the operations of pullback and pushforward of stability conditions and introduce the fundamental Bayer properties.
In Section~\ref{sec:ProductEllipticCurves}, we extend the construction of stability conditions on the product of elliptic curves from~\cite{Liu:StabProd} to an arbitrary base field, and study phases of line bundles.
All of this is applied in Section~\ref{sec:ProjectiveSpace} to show Theorem~\ref{thm:AlgebraicP1n} and Theorem~\ref{theorem-chunyi-Pn}.
The case of general projective schemes is treated in Section~\ref{sec:ProjectiveSchemes}, where we prove Theorem~\ref{thm:ProjectiveScheme}. 
In Section~\ref{sec:Complements}, we discuss several complements to these results: the existence of mass-Hom bounds mentioned above; the relation between our construction and tilt-stability in small dimensions; and the construction of stability conditions on the supported derived category of total spaces of vector bundles.


\section{Stability conditions and base change}\label{sec:StabilityConditionsBaseChange} 

In this section, we review Bridgeland's notion of a stability condition.
The basic definitions are discussed in Section~\ref{subsec:PreliminaryStability}, while in Section~\ref{subsec:BaseChange} we discuss the behavior under change of the base field.

\subsection{Basic definitions}\label{subsec:PreliminaryStability}

Stability conditions were introduced in~\cite{Bridgeland:Stab}.
We follow here the presentation in~\cite[Section 2]{curves}. 
Let $\cD$ be a triangulated category.

\subsubsection{Stability conditions in terms of slicings} 
\label{section-stability-conditions-slicings}
\begin{definition}\label{def:slicing}
A \emph{slicing} $\cP$ of $\cD$ is a collection of full additive subcategories $\cP(\phi)\subset\cD$, indexed by $\phi\in\RR$, satisfying:
\begin{enumerate}[{\rm (1)}]
\item\label{enum:slicing1} For all $\phi\in\RR$, $\cP(\phi+1)=\cP(\phi)[1]$. 
\item\label{enum:slicing2} If $\phi_1>\phi_2$ and $F_1\in\cP(\phi_1)$, $F_2\in\cP(\phi_2)$, then $\Hom(F_1,F_2)=0$. 
\item\label{enum:slicing3}
For every nonzero object $F \in \cD$, there exists a finite sequence of morphisms
\[ 
0 = F_0 \xlongrightarrow{s_1} F_1 \xlongrightarrow{s_2} \cdots \xlongrightarrow{s_m} F_m = F 
\]
such that $A_i\coloneqq\mathrm{cone}(s_i)\in\cP(\phi_i)$ for a sequence of real numbers 
$\phi_1 > \phi_2 > \dots > \phi_m$. 
\end{enumerate}
\end{definition}

As a consequence of the definition, the categories $\cP(\phi)$ are abelian.
The nonzero objects of $\cP(\phi)$ are called \emph{semistable} of \emph{phase} $\phi$; the simple objects in $\cP(\phi)$ are called \emph{stable}. 
For a nonzero object $F\in\cD$, the sequence of morphisms appearing in condition~\ref{enum:slicing3} is unique and called the \emph{Harder--Narasimhan (HN) filtration} of $F$, and the objects $A_i$ are called the \emph{Harder--Narasimhan factors} of $F$. 
We denote by $\phi_\cP^+(F)\coloneqq \phi_1$ and $\phi_\cP^-(F)\coloneqq\phi_m$ the maximal and minimal phases of the Harder--Narasimhan factors; when $F$ is semistable, these numbers coincide and are denoted simply by $\phi_{\cP}(F)$.

A slicing gives a collection of bounded t-structures, parameterized by $\phi\in\RR$. 
More precisely, for an interval $I\subset\RR$, we write 
\begin{equation*}
\cP(I)\coloneqq\left\{F\in\cD \st F\neq 0 \text{ and } \phi_\cP^+(F),\phi_\cP^-(F)\in I\right\}\cup \set{ 0};
\end{equation*}
equivalently, $\cP(I)= \langle\cP(\phi)\rangle_{\phi \in I}\subset\cD$ is the extension closure of the subcategories $\cP(\phi)$ for $\phi \in I$. 
Then for any $\phi \in \RR$, there is a bounded t-structure on $\cD$ given by the pair 
\begin{equation}
\label{tau-phi-in-text}
    (\cP( > \phi), \cP( \leq \phi + 1)) ,
\end{equation}
with heart $\cP((\phi,\phi+1])$. 
We refer to Appendix~\ref{app-t-structures-preliminaries} for background on t-structures.

Finally, we define a generalized distance between slicings as follows:
\[
\mathrm{dist}(\cP_1,\cP_2)\coloneqq \sup_{0\neq F\in \cD}\left\{\left|\phi^+_{\cP_1}(F)-\phi^+_{\cP_2}(F)\right|,\left|\phi^-_{\cP_1}(F)-\phi^-_{\cP_2}(F)\right|\right\}\in [0,+\infty].
\]
For later use, note that the above definition gives in particular that, if $\mathrm{dist}(\cP_1,\cP_2)<1$, then for all $\phi \in \RR$ we have 
\begin{equation}\label{eq:Distance}
\cP_1(\phi)\subset\cP_2((\phi-1,\phi+1))\subset\cP_2(<\phi+1).
\end{equation}

\begin{definition}
\label{definition-stability-condition}
Let $\Lambda$ be an abelian group and fix a homomorphism $v\colon\rK_0(\cD) \to\Lambda$.
Consider a pair $\sigma=(Z,\cP)$, where $Z\colon\Lambda\rightarrow\CC$ is a homomorphism and $\cP$ is a slicing of $\cD$. 
We say that: 
\begin{enumerate}[{\rm (1)}]
    \item $\sigma$ is a \emph{pre-stability condition} on $\cD$ with respect to $(\Lambda,v)$ if for all phases $\phi \in \RR$ and objects $0 \neq F \in \cP(\phi)$, we have $Z(v(F)) \in\RR_{>0}\cdot e^{i\pi\phi}$. 
    \item $\sigma$ is a \emph{stability condition} on $\cD$ with respect to $(\Lambda,v)$ if it is a pre-stability condition, $\Lambda$ is free of finite rank, and the \emph{support property} holds, i.e., there exists a quadratic form $Q$ on the vector space $\Lambda_\RR\coloneqq\Lambda\otimes\RR$ such that
\begin{itemize}
\item $Q$ is negative definite on the kernel of the $\RR$-linear extension $Z_{\RR}\colon\Lambda_\RR\to\CC$, and
\item for any $\sigma$-semistable object $F \in \cD$, we have $Q(v(F)) \geqslant 0$. 
\end{itemize}
\end{enumerate}
\end{definition}

We often suppress $v$ from the notation, for instance writing $Z(F)$ for $Z(v(F))$.
Given a pre-stability condition $\sigma$ and $F\in\cD$ a nonzero object, we denote by 
\begin{equation}\label{eq:DefMass}
m_{\sigma}(F)\coloneqq\sum |Z(A_i)|    
\end{equation}
the \emph{mass} of $F$, where the sum runs over all HN factors $A_i$ of $F$ with respect to $\sigma$.

\subsubsection{The space of stability conditions} 

\begin{definition}
Let $\Lambda$ be a free abelian group of finite rank and fix a homomorphism $v\colon\rK_0(\cD) \to\Lambda$. 
The set of all stability conditions on $\cD$ with respect to $(\Lambda,v)$ is denoted by $\Stab_{(\Lambda, v)}(\cD)$.
\end{definition}

The set $\Stab_{(\Lambda, v)}(\cD)$ carries a generalized metric $\mathrm{dist}(-,-)$, defined as follows: for any $\sigma_1=(Z_1,\cP_1)$ and $\sigma_2=(Z_2,\cP_2)$ in $\Stab_{(\Lambda, v)}(\cD)$, we set
\[
\mathrm{dist}(\sigma_1,\sigma_2)\coloneqq\sup_{0\neq F\in \cD}\left\{\left|\phi^+_{\cP_1}(F)-\phi^+_{\cP_2}(F)\right|,\left|\phi^-_{\cP_1}(F)-\phi^-_{\cP_2}(F)\right|,\left|\log \frac{m_{\sigma_1}(F)}{m_{\sigma_2}(F)}\right|\right\}\in [0,+\infty].
\]
This is locally equivalent to the generalized metric given by
\[
\max\left\{\mathrm{dist}(\cP_1,\cP_2), \|Z_1-Z_2\|\right\},
\]
where $\|-\|$ denotes any fixed norm on the finite-dimensional vector space $\Hom_{\ZZ}(\Lambda, \CC)$. 
Hence we may regard $\Stab_{(\Lambda, v)}(\cD)$ as a topological space with the induced topology by either metric. 

\begin{theorem}[Bridgeland]\label{thm:bridgeland-deformation} 
Let $\Lambda$ be a finite rank free abelian group, and let 
$v \colon \rK_0(\cD) \to \Lambda$ be a homomorphism. 
Then the forgetful map
$$ \Stab_{(\Lambda, v)}(\cD)\longrightarrow \Hom_{\ZZ}(\Lambda, \CC),\quad \sigma=(Z,\cP)\longmapsto Z,$$
is a local homeomorphism.
\end{theorem}

To conclude, we recall that $\Stab_{(\Lambda, v)}(\cD)$ carries two group actions.
The first is by the universal cover $\widetilde{\mathrm{GL}}_2^+(\RR)$ of the $2\times2$ real matrices with positive determinant, which acts on $Z$ by multiplication.
On the other hand, the group of autoequivalences of $\cD$ naturally acts on the set of slicings. 
This upgrades to an action of the group 
\begin{equation}
\label{AutLambdav}
    \Aut_{(\Lambda, v)}(\cD) \coloneqq \set{ (\Phi, a) \in \Aut(\cD) \times \Aut(\Lambda) \st v \circ \Phi_* = a \circ v}
\end{equation}
on $\Stab_{(\Lambda, v)}(\cD)$, 
where $\Phi_* \colon \rK_0(\cD) \to \rK_0(\cD)$ denotes the induced automorphism of the Grothendieck group; explicitly, the action on $\sigma = (Z, \cP) \in \Stab_{(\Lambda, v)}(\cD)$ is given by 
\begin{equation*}
    (\Phi, a) \cdot \sigma \coloneqq (Z \circ a^{-1}, \Phi \cdot \cP), 
\end{equation*}
where $\Phi \cdot \cP$ is the slicing with $(\Phi \cdot \cP)(\phi) = \Phi(\cP(\phi))$ for $\phi \in \RR$. 
Note that if $v \colon \rK_0(\cD) \to \Lambda$ is surjective, which is often the case in practice, then $\Aut_{(\Lambda, v)}(\cD)$ is in fact a subgroup of $\Aut(\cD)$.

\subsubsection{Stability conditions in terms of hearts}
There is an equivalent description of stability conditions in terms of hearts of bounded t-structures, which is sometimes useful. 

\begin{definition}
    Let $\cA$ be an abelian category. 
    Let $\Lambda$ be an abelian group and fix a homomorphism $v \colon \rK_0(\cA) \to \Lambda$. 
    Consider a homomorphism $Z \colon \Lambda \to \CC$. 
    \begin{enumerate}[{\rm (1)}]
        \item We say that $Z$ is a \emph{stability function} on $\cA$ with respect to $(\Lambda, v)$ if for all objects $0 \neq F \in \cA$, we have 
        \begin{equation*}
            Z(v(F)) \in \set{ z \in \CC \st \Im z > 0, \text{ or } \Im z = 0 \text{ and } \Re z < 0}. 
        \end{equation*}
        In this case, we define the \emph{phase} of a nonzero object $0 \neq F \in \cA$ by 
        \begin{equation*} 
        \phi(F) \coloneqq \frac{1}{\pi} \arg Z(F) \in (0,1],
        \end{equation*}
        and we say $F$ is \emph{$Z$-semistable} if for all subobjects $0 \neq A \hookrightarrow F$ we have $\phi(A) \leq \phi(F)$. 
        \item If $Z$ is a stability function, we say that it satisfies the \emph{Harder--Narasimhan (HN) property} if for every nonzero object $F \in \cA$ there exists a sequence
        \begin{equation*}
            0 = F_0 \hookrightarrow F_1 \hookrightarrow \dots \hookrightarrow F_m = F 
        \end{equation*}
        of subobjects in $\cA$ such that each factor $F_i/F_{i-1}$ is $Z$-semistable and 
        \begin{equation*}
        \phi(F_1/F_0) > \phi(F_2/F_1) > \cdots > \phi(F_m/F_{m-1}). 
        \end{equation*} 
    \end{enumerate}
\end{definition}

\begin{lemma}[{\cite[Proposition~5.3]{Bridgeland:Stab}}]
    Let $\Lambda$ be a finite rank free abelian group and fix a homomorphism
    $v \colon \rK_0(\cD) \to \Lambda$. 
    Then the following data are equivalent: 
    \begin{enumerate}[{\rm (1)}]
        \item \label{correspondence-P} A pre-stability condition on $\cD$ with respect to $(\Lambda, v)$. 
        \item \label{correspondence-A} A pair $(Z, \cA)$ where $\cA \subset \cD$ is the heart of a bounded t-structure and $Z$ is a stability function on $\cA$ with respect to $(\Lambda, v)$ which satisfies the Harder--Narasimhan property.
    \end{enumerate}
    Explicitly, given a pre-stability condition $\sigma = (Z, \cP)$ as in~\ref{correspondence-P}, we get a pair $(Z, \cA)$ as in~\ref{correspondence-A} by taking $\cA = \cP((0,1])$. 
    Conversely, given a pair $(Z, \cA)$ as in~\ref{correspondence-A}, we get a pre-stability condition $(Z, \cP)$ as in~\ref{correspondence-P} by taking $\cP(\phi)$ to consist of $0$ and $Z$-semistable objects of phase $\phi$ in $\cA$ for $\phi \in (0,1]$ and extending to all $\phi \in \RR$ via the rule $\cP(\phi+1) = \cP(\phi)[1]$. 
\end{lemma}

When we want to emphasize the heart $\cA$ instead of the slicing $\cP$, we use the notation $\sigma = (Z, \cA)$ for a (pre-)stability condition on $\cD$.

\subsection{Base change}\label{subsec:BaseChange}

Let $X$ be a projective scheme over the field $\KK$.
In this section we take $\cD=\Db(X)$, with its $\KK$-linear structure.
We set $\rK_0(X)\coloneqq \rK_0(\Db(X))$. 
Recall that we have the Euler pairing
\[
\chi_{\KK}(-,-)\colon \rK_0(\Dperf(X))\times \rK_0(X)\longrightarrow \ZZ,\quad \chi_\KK(F_1,F_2) = \sum_i (-1)^i \dim_\KK \Ext^i(F_1,F_2). 
\]
When the base field is clear from the context, we will omit it from the notation. 
We define the \emph{numerical Grothendieck group} $\Knum(X)$ of $X$ to be the quotient of $\rK_0(X)$ by the null space of $\chi$ on the right. 
By~\cite[Lemma 12.7]{stability-families}, it is a finite rank free abelian group. 

\begin{definition}
A homomorphism $v\colon\rK_0(X)\to\Lambda$ is called \emph{numerical} if $\Lambda$ is free of finite rank and $v$ factors via the canonical quotient map $\rK_0(X)\to\Knum(X)$.
Similarly, a stability condition on $\Db(X)$ with respect to $(\Lambda, v)$ is called \emph{numerical} if $v$ is numerical.
If $\Lambda=\Knum(X)$ and $v$ is the quotient map, we will use the notation $\Stab_\mathrm{num}(\Db(X))$ for the space of stability conditions with respect to $(\Lambda, v)$.
\end{definition}

Now fix a field extension $\KK \subset \ell$. 
We will recall how to base change numerical stability conditions on $X$ to those on $X_{\ell}$. 

First, by \cite[Proposition and Definition 12.15]{stability-families}, we have a homomorphism
\[
\eta^{\vee}_{\ell/\KK}\colon\Knum(X_{\ell})\longrightarrow\Knum(X)\otimes\QQ,
\]
such that the image 
\begin{equation}
\label{Knumell}
\Knum(X)_{\ell}\coloneqq\mathrm{im}(\eta^{\vee}_{\ell/\KK})
\end{equation}
contains $\Knum(X)$ as a subgroup of finite index, and is contained in $\Knum(X)_{\overline{\KK}}$, where $\overline{\KK}$ denotes the algebraic closure of $\KK$.

Let $v \colon \rK_0(X) \to \Lambda$ be a numerical homomorphism, which factors via a homomorphism $v_{\num} \colon \Knum(X) \to \Lambda$. 
We define $\Lambda_{\ell} \subset \Lambda \otimes {\QQ}$ as the subgroup generated by $\Lambda$ and the image of the map 
\begin{equation}
\label{KXell}
    \Knum(X_{\ell}) \xrightarrow{ \eta^{\vee}_{\ell/\KK}} \Knum(X) \otimes \QQ \xrightarrow{v_{\num}} \Lambda \otimes {\QQ}. 
\end{equation}

Then, by the previous paragraph, $\Lambda_{\ell}$ contains $\Lambda$ as a subgroup of finite index. 
Moreover, by construction there is an induced numerical homomorphism  
\begin{equation*}
    v_{\ell} \colon \rK_0(X_{\ell}) \to \Lambda_{\ell}
\end{equation*}
given by projection to $\rK_{\num}(X_{\ell})$ followed by the map~\eqref{KXell}. 
Finally, given a homomorphism $Z \colon \Lambda \to \CC$, we write
\begin{equation*}
    Z_{\ell} \colon \Lambda_{\ell} \to \CC 
\end{equation*}
for the canonical extension along the inclusion of the finite index $\Lambda \subset \Lambda_{\ell}$. 

The following result is a combination of~\cite[Theorem 5.3, Proposition 5.9, Theorem 12.17]{stability-families}. 

\begin{theorem}\label{thm:base-change-stab}
Let $X$ be a projective scheme over the field $\KK$. 
Let $\sigma=(Z,\cP)$ be a numerical stability condition on $\Db(X)$ with respect to $(\Lambda,v)$. Let $\KK \subset \ell$ be a field extension and let $\pi\colon X_{\ell}\to X$ denote the base change morphism.
\begin{enumerate}[{\rm (1)}]
\item\label{enum:base-change-stab1} 
There is a slicing $\cP_{\ell}$ on $\Db(X_{\ell})$ such that for any $\phi \in \RR$ we have 
\begin{align*}
    \cP_{\ell}(> \phi) & = \set{ F \in \Db(X_{\ell}) \st  \pi_*F \in \widehat{\cP}(> \phi)}, \\ 
    \cP_{\ell}(\leq \phi+1) & = \set{ F \in \Db(X_{\ell}) \st \pi_*F \in \widehat{\cP}(\leq \phi + 1)}, 
\end{align*}
where $\widehat{\cP}(> \phi)$ and $\widehat{\cP}(\leq \phi + 1)$ are as given in Definition~\ref{definition-Ind-slicing}. 
\item\label{enum:base-change-stab2} The pair $\sigma_{\ell}\coloneqq (Z_{\ell}, \cP_{\ell}) $ is a numerical stability condition on $\Db(X_{\ell})$ with respect to $(\Lambda_{\ell}, v_{\ell})$.
\item\label{enum:base-change-stab3} An object $F\in \Db(X)$ is $\sigma$-semistable of phase $\phi$ if and only if its pullback $F_{\ell}\coloneqq\pi^*F$ is $\sigma_{\ell}$-semistable of phase $\phi$. 
\item\label{enum:base-change-stab4} Let $F\in \Db(X)$. Let $\KK \subset \overline{\KK}$ and $\ell \subset \overline{\ell}$ be algebraic closures. 
Then $F_{\overline{\KK}}$ is $\sigma_{\overline{\KK}}$-stable if and only if $F_{\overline{\ell}}$ is $\sigma_{\overline{\ell}}$-stable; in particular, in this case $F_{\ell}$ is $\sigma_{\ell}$-stable. 
\end{enumerate}
\end{theorem}

\begin{definition}
\label{definition-geometrically-stable}
In the situation of Theorem~\ref{thm:base-change-stab}, we say $F\in \Db(X)$ is \emph{geometrically $\sigma$-stable} if $F_{\overline{\KK}}$ is $\sigma_{\overline{\KK}}$-stable.
\end{definition} 

 The following result was stated without proof  after the statement of~\cite[Theorem~12.17]{stability-families}.

\begin{lemma}\label{lem:base-change-stab-manifold}
In the above notation, we have a commutative diagram
\[\begin{tikzcd}
	{\Stab_{(\Lambda, v)}(\Db(X))} & {\Stab_{(\Lambda_{\ell}, v_{\ell})}(\Db(X_{\ell}))} \\
	{\Hom_{\ZZ}(\Lambda, \CC)} & {\Hom_{\ZZ}(\Lambda_{\ell}, \CC)}
    \arrow["\Pi", from=1-1, to=1-2]
	\arrow[from=1-1, to=2-1]
	\arrow[from=1-2, to=2-2]
	\arrow[from=2-1, to=2-2]
\end{tikzcd}\]
where $\Pi$ is an open embedding given by $\Pi(Z,\cP)=(Z_{\ell},\cP_{\ell})$, 
the vertical maps are the local homeomorphisms given by the forgetful maps, 
and the bottom map is the homeomorphism given by $Z \mapsto Z_{\ell}$. 

Moreover, if the cokernel of $v \colon \rK_0(X) \to \Lambda$ is finite, then $\Pi$ is a homeomorphism onto a union of connected components of $\Stab_{(\Lambda_{\ell}, v_{\ell})}(\Db(X_{\ell}))$.
\end{lemma}

\begin{proof}
The commutativity of the diagram holds by the construction of the maps involved. 
The fact that forgetful maps are local homeomorphisms is Theorem~\ref{thm:bridgeland-deformation}. 
Since $\Lambda$ is a finite index subgroup of $\Lambda_{\ell}$, the bottom map is a linear isomorphism, and we can fix compatible norms on both sides; in particular, it is a homeomorphism. 

By our choice of norms, the definition of $\mathrm{dist}(-,-)$, and Theorem \ref{thm:base-change-stab}\ref{enum:base-change-stab3}, we have $$\mathrm{dist}(\sigma_1,\sigma_2)\leq \mathrm{dist}(\Pi(\sigma_1), \Pi(\sigma_2))$$ for any $\sigma_1,\sigma_2\in \Stab_{(\Lambda, v)}(\Db(X))$. In particular, $\Pi$ is injective. 
As vertical maps are local homeomorphisms, we conclude that 
$\Pi$ is an open embedding.

Let $S \subset \Stab_{(\Lambda, v)}(\Db(X))$ and $S'\subset\Stab_{(\Lambda_{\ell}, v_{\ell})}(\Db(X_{\ell}))$ be connected components such that $\Pi(S)\subset S'$. If the cokernel of $v$ is finite, then by \cite[Theorem 3.7]{woolf}, $S$ and $S'$ are complete. It follows that $\Pi(S)$ is closed in $S'$. By the previous paragraph it is also open, so we conclude that $\Pi(S)=S'$.
\end{proof}

The following definition will be used throughout this paper.

\begin{definition}\label{def:GeometricSatbility}
A numerical stability condition $\sigma\in\Stab_{(\Lambda,v)}(\Db(X))$ is \emph{geometric} if $\sigma_{\overline{\KK}}$ has the property that all skyscraper sheaves $\cO_p$ of closed points $p\in X_{\overline{\KK}}$ are $\sigma_{\overline{\KK}}$-stable with the same phase. 
\end{definition}


\section{Pullback and pushforward of stability conditions}\label{sec:PullbackPushforward} 

In this section, we discuss the theory of pullback and pushforward of stability conditions. 
First, in Sections~\ref{subsection-pullback-stability} and~\ref{subsection-pushforward-stability}, we introduce the basic definitions, properties, and criteria for the existence of pullback and pushforward stability.  
In Section~\ref{section-base-change-pull-push-stability}, we discuss the compatibility of these operations with base field extension.
Finally, in Section~\ref{subsec:BayerConditions}, we introduce the Bayer properties, and use them to formulate particularly convenient criteria for the existence of pushforward and pullback stability conditions (Propositions~\ref{prop:PushforwardStability} and~\ref{prop:PullbackStability}). 

\subsection{Pullback of stability conditions} \label{subsection-pullback-stability}

We begin with the following:

\begin{definition}\label{def:Pullback}
Let $f\colon X\to Y$ be a proper morphism between noetherian schemes.
\begin{enumerate}[{\rm (1)}]
\item If $\cP$ is a slicing of $\Db(Y)$, then its \emph{pullback} $f^{\sharp}\cP$ is the $\RR$-indexed collection of subcategories of $\Db(X)$ given by 
\begin{equation*}
f^{\sharp}\cP(\phi)\coloneqq \set{F \in \Db(X) \st f_*F \in \cP(\phi)} \quad \text{for } \phi \in \RR. 
\end{equation*}
(Note that since $f$ is proper, $f_*F$ lies in $\Db(Y)$.)
\item If $v \colon \rK_0(Y) \to \Lambda$ is a homomorphism, then its \emph{pullback} $f^{\sharp}v \colon \rK_0(X) \to \Lambda$ is the homomorphism   given by 
\begin{equation*}
    f^{\sharp} v \coloneqq v \circ f_*.
\end{equation*}
\item If $\sigma = (Z,\cP)$ is a stability condition on $\Db(Y)$ with respect to $v \colon \rK_0(Y) \to \Lambda$, then its \emph{pullback} is the pair
\begin{equation*}
f^{\sharp} \sigma \coloneqq (Z,f^{\sharp}\cP). 
\end{equation*}
\end{enumerate}
\end{definition}

\begin{remark}
\label{remark-pullback-slicing-stability}
In the generality of Definition~\ref{def:Pullback}, we cannot always expect $f^{\sharp}\cP$ to be a slicing or $f^{\sharp}\sigma$ to be a stability condition. 
For instance, in order for $f^{\sharp}\cP$ to be a slicing, it is necessary that $f_* \colon \Db(X) \to \Db(Y)$ is conservative; for applications, we will be interested in the case where $f$ is finite, which implies this conservativity. 
On the other hand, if $f_*$ is conservative, then $f^\sharp\sigma$ is a stability condition with respect to $(\Lambda, f^{\sharp}v)$ as soon as $f^{\sharp}\cP$ is a slicing. 
Indeed, then the conditions in Definition~\ref{definition-stability-condition} for $f^{\sharp}\sigma$ follow directly from those for $\sigma$, with $f^{\sharp}\sigma$ satisfying the support property with respect to the same quadratic form on $\Lambda_{\RR}$ as $\sigma$. 
\end{remark}

The next result gives a useful criterion for the pullback of a slicing or stability condition to remain such. 
This result depends on the material discussed in Appendix~\ref{app:Inducing} on inducing t-structures and slicings. 
In particular, given a slicing $\cP$ of $\Db(Y)$ and $\phi \in \RR$, we denote by $\widehat{\cP}(\geq \phi)$ the subcategory of $\Dqc(Y)$ generated by $\cP(\geq \phi)$ under coproducts and extensions (see Definition~\ref{definition-Ind-slicing}). 

\begin{proposition}\label{proposition-pullback-stability}
Let $f \colon X \to Y$ be a finite morphism between noetherian schemes.
Let $\cP$ be a slicing of $\Db(Y)$ such that  
\begin{equation}\label{f_*OX-right-exact}
f_*\cO_X \otimes\cP(\phi)\subset\widehat{\cP}  (\geq\phi) \quad \text{for all } \phi \in \RR.
\end{equation}
Then the pullback $f^{\sharp}\cP$ is a slicing of $\Db(X)$. 
If $X$ and $Y$ are defined over a scheme $S$ which is quasi-projective over a noetherian affine scheme, then if $\cP$ is an $S$-local slicing, so is $f^{\sharp}\cP$. 

Moreover, if $\sigma=(Z,\cP)$ is a stability condition on $\Db(Y)$ with respect to $v \colon \rK_0(Y) \to \Lambda$ such that $\cP$ satisfies~\eqref{f_*OX-right-exact}, then $f^{\sharp}\sigma$ is a stability condition on $\Db(X)$ with respect to $(\Lambda,f^{\sharp} v)$.
\end{proposition}

\begin{proof}
This is an extension of \cite[Corollary 2.2.2]{P-t-structures}. 
Namely, the result for $f^{\sharp}\cP$ follows by applying Corollary~\ref{cor:Polishchuk} 
to the functor $\Phi=f_*$, which satisfies the hypotheses of the corollary due to Example~\ref{ex:PullPushShriek}\ref{enum:PullPushShriek1} and our assumption~\eqref{f_*OX-right-exact}. 
Then the result for $f^{\sharp}\sigma$ follows from Remark~\ref{remark-pullback-slicing-stability}.
\end{proof}

\begin{remark}
\label{remark-finite-tor-dimension-f_*OX-right-exact}
If $f$ has finite $\Tor$-dimension, then the condition~\eqref{f_*OX-right-exact} may be reformulated without the use of the Ind-completed category $\widehat{\cP}(\geq \phi)$; indeed, in this case $f_* \cO_X \otimes \cP(\phi)$ is contained in $\Db(Y)$, so the condition is equivalent to 
    \begin{equation*} 
f_*\cO_X \otimes\cP(\phi)\subset {\cP}  (\geq\phi) \quad \text{for all } \phi \in \RR.
\end{equation*}
In turn, it is easy to see that this condition is equivalent to 
\begin{equation*}
    f_*\cO_X \otimes\cP(>\phi) \subset \cP(> \phi) \quad \text{for all } \phi \in \RR. 
\end{equation*}
\end{remark}

We collect some easy compatibility properties of pullbacks of slicings. 
\begin{lemma}
\label{lemma-pullback-slicing-compatibility}
    Let $f \colon X \to Y$ be a proper morphism between noetherian schemes. 
    Let $\cP$ be a slicing of $\Db(Y)$ such that $f^{\sharp}\cP$ is a slicing of $\Db(X)$. 
    \begin{enumerate}[{\rm (1)}]
    \item \label{pullback-slicing-compatibility-tensorL}
    For all invertible sheaves $\cL$ on $Y$, $f^{\sharp}(\cP \otimes \cL)$ is a slicing of $\Db(X)$ which satisfies 
    \[
    f^\sharp\left(\cP\otimes\cL\right)=f^\sharp\cP\otimes f^*\cL.
    \]

    \item \label{pullback-slicing-compatibility-phipm}
    For any $0 \neq F \in \Db(X)$ we have 
    \begin{equation*} 
    \phi^{\pm}_{f^{\sharp}\cP}(F) = \phi^{\pm}_{\cP}(f_*F).
    \end{equation*} 

    \item \label{pullback-slicing-compatibility-dist}
If $\cP_1$ and $\cP_2$ are slicings of $\Db(Y)$ such that $f^\sharp\cP_1$ and $f^\sharp\cP_2$ are slicings of $\Db(X)$, then
\[
\mathrm{dist}(\cP_1,\cP_2)\geq\mathrm{dist}(f^\sharp\cP_1,f^\sharp\cP_2).
\]
    \end{enumerate}
\end{lemma}

\begin{proof}
    Part~\ref{pullback-slicing-compatibility-tensorL} follows immediately from the definition of $f^{\sharp}$ and the projection formula. 
    Moreover, from the definition of $f^{\sharp}\cP$, the HN filtration of any $0 \neq F \in \Db(X)$ with respect to $f^{\sharp}\cP$ pushes forward to the HN filtration of $f_*F$ with respect to $\cP$; in particular,~\ref{pullback-slicing-compatibility-phipm} holds. 
    Finally,~\ref{pullback-slicing-compatibility-dist} follows from~\ref{pullback-slicing-compatibility-phipm}. 
\end{proof}

\subsection{Pushforward of stability conditions} 
\label{subsection-pushforward-stability}

The important construction in this section is the following.

\begin{definition}\label{def:Pushforward}
Let $f\colon X\to Y$ be a morphism of finite $\Tor$-dimension between noetherian schemes. 
\begin{enumerate}[{\rm (1)}]
\item If $\cP$ is a slicing of $\Db(X)$, then its \emph{pushforward} $f_{\sharp} \cP$ is the $\RR$-indexed collection of subcategories of $\Db(Y)$ given by 
\begin{equation*}
f_\sharp\cP(\phi)\coloneqq \set{F\in\Db(Y)\st f^*F\in\cP(\phi)} \quad\text{for }\phi \in\RR.
\end{equation*}
(Note that since $f$ has finite $\Tor$-dimension, $f^*F$ lies in $\Db(X)$.) 
\item If $v \colon \rK_0(X) \to \Lambda$ is a homomorphism, then its \emph{pushforward} $f_{\sharp}v \colon \rK_0(Y) \to \Lambda$ is the homomorphism given by 
\begin{equation*}
f_{\sharp} v \coloneqq v \circ f^*.
\end{equation*} 
\item If $\sigma = (Z,\cP)$ is a stability condition on $\Db(X)$ with respect to $v \colon \rK_0(X) \to \Lambda$, then its \emph{pushforward} is the pair 
\begin{equation*}
f_{\sharp} \sigma \coloneqq (Z,f_{\sharp}\cP). 
\end{equation*} 
\end{enumerate}    
\end{definition}

\begin{remark}
\label{remark-pushforward-slicing-stability}
As in the case of pullbacks, in the generality of Definition~\ref{def:Pushforward} we cannot always expect $f_{\sharp}\cP$ to be a slicing or $f_{\sharp}\sigma$ to be a stability condition. 
In particular, in order for $f_{\sharp}\cP$ to be a slicing, it is necessary that $f^* \colon \Db(Y) \to \Db(X)$ is conservative; for applications, we will be interested in the case where $f$ is faithfully flat, which implies this conservativity. 
On the other hand, if $f^*$ is conservative and $f_{\sharp}\cP$ is a slicing, then $f_{\sharp}\sigma$ is a stability condition with respect to $(\Lambda, f_{\sharp} v)$. 
\end{remark}

Similarly to the pullback case, we have the following criterion for $f_{\sharp} \cP$ to be a slicing or $f_{\sharp}\sigma$ to be a stability condition. 

\begin{proposition}\label{proposition-pushforward-stability}
Let $f \colon X \to Y$ be a proper faithfully flat morphism between noetherian schemes such that the relative dualizing complex $\omega_f^\bullet$ is contained in $\Dperf(X)$.
Let $\cP$ be a slicing on $\Db(X)$ such that 
\begin{equation}\label{eq:PushforwardStability}
    f^*f_*\cP(\phi) \subset \cP(\leq\phi) \quad\text{for all }\phi\in\RR. 
\end{equation}
Then the pushforward $f_{\sharp}\cP$ is a slicing on $\Db(Y)$. 
If $X$ and $Y$ are defined over a scheme $S$ which is quasi-projective over a noetherian affine scheme, then if $\cP$ is an $S$-local slicing, so is~$f_{\sharp}\cP$. 

Moreover, if $\sigma=(Z,\cP)$ is a stability condition on $\Db(X)$ with respect to $v\colon \rK_0(X)\to\Lambda$ such that $\cP$ satisfies~\eqref{eq:PushforwardStability}, then $f_{\sharp}\sigma$ is a stability condition on $\Db(Y)$ with respect to $(\Lambda,f_{\sharp}v)$.  
\end{proposition}

\begin{proof}
This is an extension of~\cite[Proposition 3.4]{chunyi-stability}.
For the claim about $f_{\sharp} \cP$, we want to apply Corollary~\ref{cor:Polishchuk} (with $S = \Spec(\ZZ)$) to the functor $\Phi=f^*$. 
By Example~\ref{ex:PullPushShriek}\ref{enum:PullPushShriek2}, in order to check that $f^*$ satisfies the assumptions of the corollary, it suffices to show that 
\[
f^*f_!\cP(\phi)\subset \widehat{\cP}(\geq\phi), 
\]
where $f_! \colon \Dqc(X) \to \Dqc(Y)$ is the left adjoint of $f^*$. 
Moreover, as noted in Example~\ref{ex:PullPushShriek}\ref{enum:PullPushShriek2}, the assumption that $\omega^{\bullet}_f$ is perfect implies that $f_!$ takes $\Db(X)$ to $\Db(Y)$. 
In particular, we see that $f^*f_!\cP(\phi) \subset \Db(X)$, and the above inclusion is equivalent to 
\begin{equation*}
    f^*f_! \cP(\phi) \subset \cP(\geq \phi). 
\end{equation*}

To show this, let $F\in\cP(\phi)$ and $G\in\cP(<\phi)$.
Then, since $\cP$ is a slicing, we have $G\in\cP(\leq\phi')$ for some $\phi'<\phi$.
By adjunction, we have
\[
\Hom(f^*f_!F,G)=\Hom(f_!F,f_*G)=\Hom(F,f^*f_*G)=0
\]
because $f^*f_*G\in\cP(\leq\phi')\subset\cP(<\phi)$ by the assumption~\eqref{eq:PushforwardStability}. 
Hence, $f^*f_!F\in\cP(\geq\phi)$, as we wanted.

Finally, the statement about $f_{\sharp}\sigma$ follows from the slicing case and Remark~\ref{remark-pushforward-slicing-stability}. 
\end{proof}

Now we record the analog of Lemma~\ref{lemma-pullback-slicing-compatibility} for pushforwards of slicings, which holds by the same argument. 

\begin{lemma}
    \label{lemma-pushforward-slicing-compatibility}
    Let $f \colon X \to Y$ be a morphism of finite Tor-dimension between noetherian schemes. 
    Let $\cP$ be a slicing of $\Db(X)$ such that $f_\sharp \cP$ is a slicing of $\Db(Y)$. 
    \begin{enumerate}[{\rm (1)}]
    \item \label{pushforward-slicing-compatibility-tensorL}
    For all invertible sheaves $\cL$ on $Y$, $f_{\sharp}(\cP \otimes f^*\cL)$ is a slicing of $\Db(Y)$ which satisfies
    \[
    f_\sharp\left(\cP\otimes f^*\cL\right)=f_\sharp\cP\otimes \cL.
    \]

    \item \label{pushforward-slicing-compatibility-phipm}
    For any $0 \neq F \in \Db(Y)$ we have 
    \begin{equation*} 
    \phi^{\pm}_{f_{\sharp}\cP}(F) = \phi^{\pm}_{\cP}(f^*F).
    \end{equation*} 

    \item \label{pushforward-slicing-compatibility-dist}
If $\cP_1$ and $\cP_2$ are slicings of $\Db(X)$ such that $f_\sharp\cP_1$ and $f_\sharp\cP_2$ are slicings of $\Db(Y)$, then
\[
\mathrm{dist}(\cP_1,\cP_2)\geq\mathrm{dist}(f_\sharp\cP_1,f_\sharp\cP_2).
\]
    \end{enumerate}
\end{lemma}

\subsection{Base change} 
\label{section-base-change-pull-push-stability}
For later use, we discuss the compatibility of pullback and pushforward of stability conditions with field extensions.

\begin{lemma}\label{lem:base-change}
Let $f \colon X \to Y$ be a morphism between projective schemes over $\KK$, and let $\KK \subset \ell$ be a field extension.  
\begin{enumerate}[{\rm (1)}]
    \item\label{enum:BaseChange1} Assume that $f$ is finite of finite Tor-dimension. 
    Let $\sigma=(Z,\cP)$ be a numerical stability condition on $\Db(Y)$ with respect to $(\Lambda,v)$. 
    Let $\sigma_{\ell} = (Z_{\ell}, \cP_{\ell})$ be the base changed stability condition on $\Db(Y_{\ell})$ with respect to $(\Lambda_{\ell}, v_{\ell})$, as in Theorem~\ref{thm:base-change-stab}. 
    \begin{enumerate}[label={\rm(1.\alph*)}]
    \item \label{base-change-pullback-criterion} 
    We have 
    \begin{equation}
    \label{eq:push-f-k}
        f_* \cO_X \otimes \cP(\phi) \subset\cP(\geq\phi) \quad \text{for all } \phi \in \RR
    \end{equation}
    if and only if
    \begin{equation}\label{eq:push-l}
    f_{\ell*}\cO_{X_{\ell}} \otimes \cP_{\ell}(\phi) \subset 
    \cP_{\ell}(\geq \phi) \quad \text{for all } \phi \in \RR.
    \end{equation}
    \item \label{pullback-commute-w-bc}
    When the above equivalent conditions hold, the formation of pullback stability commutes with base change to $\ell$. 
    More precisely, by Proposition~\ref{proposition-pullback-stability} we can form the pullback stability condition $f^{\sharp}\sigma$ on $\Db(X)$ with respect to $(\Lambda, f^{\sharp}v)$, as well as $f_{\ell}^{\sharp} \sigma_{\ell}$ on $\Db(X_{\ell})$ with respect to $(\Lambda_{\ell}, f_{\ell}^{\sharp} v_{\ell})$. 
    Then $(\Lambda_{\ell}, (f^{\sharp}v)_{\ell}) = (\Lambda_{\ell}, f_{\ell}^{\sharp} v_{\ell})$ and the base change $(f^{\sharp}\sigma)_{\ell}$ of $f^{\sharp}\sigma$ coincides with  $f_{\ell}^{\sharp} \sigma_{\ell}$. 
    \end{enumerate} 
    \item\label{enum:BaseChange2} Assume that $f$ is faithfully flat and that the relative dualizing complex $\omega_f^{\bullet}$ is perfect. Let $\sigma=(Z,\cP)$ be a numerical stability condition on $\Db(X)$ with respect to $(\Lambda,v)$. 
    Let $\sigma_{\ell} = (Z_{\ell}, \cP_{\ell})$ be the base changed stability condition on $\Db(X_{\ell})$ with respect to $(\Lambda_{\ell}, v_{\ell})$, as in Theorem~\ref{thm:base-change-stab}. 
    \begin{enumerate}[label={\rm(2.\alph*)}]
        \item \label{base-change-pushforward-criterion} We have 
        \begin{equation}
        \label{eq:pull-f-k}
                f^*f_{*}(\cP(\phi))\subset \cP(\leq \phi) \quad \text{for all } \phi \in \RR
        \end{equation}
        if and only if 
        \begin{equation}\label{eq:pull-l}
                    f^*_{\ell}f_{\ell*}(\cP_{\ell}(\phi))\subset \cP_{\ell}(\leq \phi) \quad \text{for all } \phi \in \RR. 
        \end{equation}
        \item When the above equivalent conditions hold, the formation of pushforward stability commutes with base change to $\ell$. 
        More precisely, by Proposition~\ref{proposition-pushforward-stability} we can form the pushforward stability condition $f_{\sharp}\sigma$ on $\Db(Y)$ with respect to $(\Lambda, f_{\sharp} v)$, as well as $f_{\ell \sharp}\sigma_{\ell}$ on $\Db(Y_{\ell})$ with respect to $(\Lambda_{\ell}, f_{\ell \sharp} v_{\ell})$. 
        Then $(\Lambda_{\ell}, (f_{\sharp} v)_{\ell}) = (\Lambda_{\ell}, f_{\ell \sharp} v_{\ell})$ and the base change $(f_{\sharp} \sigma)_{\ell}$ of $f_{\sharp} \sigma$ coincides with $f_{\ell \sharp} \sigma_{\ell}$. 
    \end{enumerate}
\end{enumerate}
\end{lemma}

\begin{proof}
We prove~\ref{enum:BaseChange1}; the proof of~\ref{enum:BaseChange2} is analogous and is left to the reader. 

First, we prove~\ref{base-change-pullback-criterion}. 
By base change and Theorem~\ref{thm:base-change-stab}\ref{enum:base-change-stab3}, it is clear that~\eqref{eq:push-l} implies~\eqref{eq:push-f-k}.
For the converse, we first note (cf. Remark~\ref{remark-finite-tor-dimension-f_*OX-right-exact}) that it is easy to see that the statement~\eqref{eq:push-f-k} is equivalent to 
\begin{equation}\label{eq:PullAlt-k}
f_*\cO_{X} \otimes \cP(>\phi) \subset \cP(> \phi) \quad \text{for all } \phi \in \RR. 
\end{equation}
Similarly, the statement \eqref{eq:push-l} is equivalent to
\begin{equation}\label{eq:PullAlt}
f_{\ell*}\cO_{X_{\ell}} \otimes \cP_{\ell}(>\phi) \subset \cP_{\ell}(> \phi) \quad \text{for all } \phi \in \RR. 
\end{equation}
Now, let $\pi_X\colon X_{\ell}\to X$ and $\pi_Y\colon Y_{\ell}\to Y$ denote the base change morphisms.
Then we have 
\begin{equation*}
    \pi_{Y*}(f_{\ell*}\cO_{X_{\ell}} \otimes \cP_{\ell}(>\phi)) 
    = f_*\cO_X \otimes \pi_{Y*}(\cP_{\ell}(> \phi)) \subset 
    f_*\cO_X \otimes \widehat{\cP}(> \phi), 
\end{equation*}
where the equality holds by the isomorphism $f_{\ell*}\cO_{X_{\ell}} \cong \pi_Y^*f_*\cO_X$ and the projection formula, and the inclusion holds by Theorem~\ref{thm:base-change-stab}\ref{enum:base-change-stab1}. 
On the other hand, if~\eqref{eq:push-f-k} holds, then by \eqref{eq:PullAlt-k}, the fact that $f_*\cO_X \otimes -$ preserves extensions and coproducts, and the definition of $\widehat{\cP}(> \phi)$, we find that 
$f_*\cO_X \otimes \widehat{\cP}(> \phi) \subset \widehat{\cP}(> \phi)$, and hence
\begin{equation*}
 \pi_{Y*}(f_{\ell*}\cO_{X_{\ell}} \otimes \cP_{\ell}(>\phi)) \subset \widehat{\cP}(> \phi). 
\end{equation*} 
Again by Theorem~\ref{thm:base-change-stab}\ref{enum:base-change-stab1}, this proves that~\eqref{eq:PullAlt} holds, which as noted above is equivalent to the desired statement~\eqref{eq:push-l}. 

Now we prove~\ref{pullback-commute-w-bc}. 
The claim  $(\Lambda_{\ell}, (f^{\sharp}v)_{\ell}) = (\Lambda_{\ell}, f_{\ell}^{\sharp} v_{\ell})$ follows easily from the definitions, so we focus on proving the equality of the slicings $(f^{\sharp}\cP)_{\ell}$ and $f^{\sharp}_{\ell}\cP_{\ell}$. 
For $\phi \in \RR$, by Theorem~\ref{thm:base-change-stab}\ref{enum:base-change-stab1} we have 
\begin{equation*}
    (f^{\sharp}\cP)_{\ell}(> \phi) = \set{F \in \Db(X_{\ell}) \st \pi_{X*}(F) \in \widehat{f^{\sharp}\cP}(> \phi)}.
\end{equation*}
On the other hand, by the definition of $f^{\sharp}\cP$, the functor $f_* \colon \Db(X) \to \Db(Y)$ is t-exact with respect to the t-structures $(f^{\sharp}\cP(> \phi), f^{\sharp}\cP( \leq \phi+1))$ and $(\cP(> \phi), \cP(\leq \phi+1))$. 
Thus, by Lemma~\ref{lemma-ind-extension-exact-functor}, the functor $f_* \colon \Dqc(X) \to \Dqc(Y)$ is t-exact with respect to the t-structures 
$(\widehat{f^{\sharp}\cP}(> \phi), \widehat{f^{\sharp}\cP}( \leq \phi+1))$ and $(\widehat{\cP}(> \phi), \widehat{\cP}(\leq \phi+1))$; alternatively, this follows from unwinding the proof that  $f^{\sharp}\cP$ is a slicing --- see in particular Step~\ref{polishchuk-step-1} in the proof of Theorem~\ref{thm:Polishchuk}. 
Since the functor $f_* \colon \Dqc(X) \to \Dqc(Y)$ is also conservative, it follows that 
\begin{equation*}
        (f^{\sharp}\cP)_{\ell}(> \phi) = \set{F \in \Db(X_{\ell}) \st f_*\pi_{X*}(F) \in \widehat{\cP}(> \phi)}. 
\end{equation*}
Since $f_*\pi_{X*} = \pi_{Y*}f_{\ell*}$, 
by Theorem~\ref{thm:base-change-stab}\ref{enum:base-change-stab1} this coincides with $f^{\sharp}_{\ell}\cP_{\ell}(>\phi)$. 
This completes the proof of the desired equality $(f^{\sharp}\cP)_{\ell} = f^{\sharp}_{\ell}\cP_{\ell}$. 
\end{proof}

\begin{remark}\label{rmk:GeometricPushPull}
In the context of Lemma~\ref{lem:base-change}, it is straightforward to see that geometric stability conditions are preserved under pullback and pushforward when $f$ is finite. 
\end{remark}

\subsection{The Bayer property}\label{subsec:BayerConditions}

Conditions~\eqref{f_*OX-right-exact} and~\eqref{eq:PushforwardStability} are not easy to verify directly.
We introduce two simplified conditions to address this problem.
To do this, we first recall a partial ordering between slicings, introduced in~\cite{Li:sb}.

\begin{definition}\label{def:Orderslicings}
Let $\cP_1$ and $\cP_2$ be slicings on a triangulated category $\cD$. We define
\begin{align*}
&\cP_1\prec\cP_2\quad\text{ if }\cP_1(\phi)\subset\cP_2(<\phi)\text{ for all }\phi\in\RR\\
&\cP_1\preceq\cP_2\quad\text{ if }\cP_1(\phi)\subset\cP_2(\leq\phi) \text{ for all }\phi\in\RR.
\end{align*}
For stability conditions $\sigma_1=(Z_1,\cP_1)$ and $\sigma_2=(Z_2,\cP_2)$ on $\cD$, we write
\[
\sigma_1\prec \sigma_2 \quad\text{ if } \quad 
\cP_1\prec \cP_2, 
\]
and similarly for $\preceq$.
\end{definition}

Concretely, the above definition can be rephrased as follows. 

\begin{lemma}[{\cite[Lemma~4.11]{Li:sb}}]
\label{lemma-preceq-concrete} 
    Let $\cP_1$ and $\cP_2$ be slicings on a triangulated category $\cD$. 
    Then the following are equivalent: 
    \begin{enumerate}[{\rm (1)}]
        \item $\cP_1 \prec \cP_2$. 
        \item For every $\phi \in \RR$ and $0 \neq F \in \cP_1(\phi)$, we have 
        $\phi^+_{\cP_2}(F) < \phi$. 
        \item For every $\phi \in \RR$ and $0 \neq F \in \cP_2(\phi)$, we have $\phi < \phi^-_{\cP_1}(F)$. 
        \item For every $0 \neq F \in \cD$, we have $\phi^+_{\cP_2}(F) < \phi^+_{\cP_1}(F)$ and $\phi^-_{\cP_2}(F) < \phi^-_{\cP_1}(F)$. 
    \end{enumerate}
    Similarly, the following are equivalent: 
    \begin{enumerate}[{\rm (1)}]
        \item $\cP_1 \preceq \cP_2$. 
        \item For every $\phi \in \RR$ and $0 \neq F \in \cP_1(\phi)$, we have 
        $\phi^+_{\cP_2}(F) \leq \phi$. 
        \item For every $\phi \in \RR$ and $0 \neq F \in \cP_2(\phi)$, we have $\phi \leq \phi^-_{\cP_1}(F)$. 
        \item For every $0 \neq F \in \cD$, we have $\phi^+_{\cP_2}(F) \leq  \phi^+_{\cP_1}(F)$ and $\phi^-_{\cP_2}(F) \leq \phi^-_{\cP_1}(F)$. 
    \end{enumerate}
\end{lemma}

\begin{remark} 
\label{rmk:properties-preceq}
Let us collect some easy observations about the relations $\prec$ and $\preceq$ on slicings. 
\begin{enumerate}[{\rm (1)}]
\item The relations $\prec$ and $\preceq$ are transitive. 
\item The relations $\prec$ and $\preceq$ are invariant under autoequivalences of $\cD$. 
\item 
    \label{enum:CompatibilitySlicings3} 
    Let $f\colon X\to Y$ be a morphism of finite $\Tor$-dimension between noetherian schemes, and let $\cP_1,\cP_2$ be slicings of $\Db(X)$ such that $f_\sharp\cP_1,f_\sharp\cP_2$ are slicings of $\Db(Y)$.
    If $\cP_1\preceq\cP_2$, then $f_\sharp\cP_1\preceq f_\sharp\cP_2$.
    The analogous statement holds for $\prec$, as well as for pullback $f^{\sharp}$ of slicings. 
    In fact, this is a general statement which holds in the context of Corollary~\ref{cor:Polishchuk}.
\end{enumerate} 
\end{remark}

Moreover, the relation $\preceq$ behaves well  with respect to base change. 

\begin{lemma}\label{lem:base-change-preserve-order-stab}
Let $X$ be a projective scheme over $\KK$. 
Let $\sigma_1 = (Z_1, \cP_1)$ and $\sigma_2 = (Z_2, \cP_2)$ be two numerical stability conditions on $\Db(X)$.  
Fix a field extension $\KK\subset\ell$. 
Then $\sigma_1 \preceq \sigma_2$ if and only if $(\sigma_1)_{\ell} \preceq (\sigma_2)_{\ell}$.
\end{lemma}

\begin{proof}
Note that the condition $\sigma_1 \preceq \sigma_2$ is equivalent to $\cP_2(> \phi) \subset \cP_{1}(> \phi)$ for all $\phi \in \RR$, and similarly for $(\sigma_1)_{\ell} \preceq (\sigma_2)_{\ell}$. 
Hence the result follows from Theorem~\ref{thm:base-change-stab}\ref{enum:base-change-stab1}.
\end{proof}

\begin{definition}\label{def:BayerConditions}
Let $X$ be a noetherian scheme, let $\cL$ be an invertible sheaf on $X$, and let $l\in \ZZ$ be an integer.
We say that a slicing $\cP$ on $\Db(X)$ satisfies the \emph{Bayer property} with respect to $\cL$ and $l$  if $\cP\preceq\cP\otimes\cL[l]$.
\end{definition}

\begin{lemma}
    \label{lemma-BL-concrete}
Let $X$ be a noetherian scheme, $\cL$ an invertible sheaf on $X$, and $l \in \ZZ$ an integer.  
Let $\cP$ be a slicing on $\Db(X)$. 
Then the following are equivalent: 
\begin{enumerate}[{\rm (1)}]
    \item $\cP$ satisfies the Bayer property with respect to $\cL$ and $l$. 
    \item For every $\phi \in \RR$ and $0 \neq F \in \cP(\phi)$, we have $\phi^+_\cP(F\otimes\cL^{-1}[-l])\leq\phi$. 
    \item For every $\phi \in \RR$ and $0 \neq F \in \cP(\phi)$, we have $\phi^-_\cP(F\otimes\cL[l]) \geq \phi$. 
\end{enumerate}
\end{lemma}

\begin{proof}
    This is an easy consequence of Lemma~\ref{lemma-preceq-concrete}. 
\end{proof}

\begin{remark}
    A particular case of the Bayer property that will be of interest to us is when $l=1$; in fact, in this case the condition can be checked in terms of distance between slicings. 
More precisely, if $\mathrm{dist}(\cP,\cP\otimes\cL)<1$, then $\cP\prec\cP\otimes\cL[1]$, as observed in~\eqref{eq:Distance}.
\end{remark}

We can finally prove the main results in this section, which are a simple extension of the techniques developed in~\cite{Li:sb,chunyi-stability}.
We start by introducing a technical condition on finite morphisms, which will be useful in the case of pushforward.

\begin{definition}\label{def:FiltrationProperty}
Let $f \colon X \to Y$ be a separated morphism of schemes.
Let $X\times_Y X$ be the fiber product and $p_1,p_2\colon X\times_Y X\to X$ its two projections.
Consider automorphisms $g_1,\dots,g_m\in\Aut_Y(X)$ of $X$ as a $Y$-scheme and $\cL_1,\dots,\cL_m$ invertible sheaves on $X$.
We say that $f$ has the \emph{filtration property} with respect to $g_1,\dots,g_m$ and $\cL_1,\dots,\cL_m$ if there exists a finite filtration
\[
0=F_{0}\subset F_1\subset\dots\subset F_{m-1}\subset F_m=\cO_{X\times_Y X},
\]
such that, for all $i=1,\dots,m$,
\[
F_i/F_{i-1} \cong p_2^*\cL_i^{-1}\otimes\cO_{\Gamma(g_i)},
\]
where $\Gamma(g_i)\subset X\times_Y X$ denotes the graph of $g_i$.    
\end{definition}

\begin{example}\label{ex:PushForward}
Three interesting situations are the following.
\begin{enumerate}
\item\label{enum:etale} Any finite \'etale Galois cover $f \colon X \to Y$ has the filtration property. 
More precisely, if $G = \Aut_Y(X)$ denotes the Galois group, then we have 
\begin{equation*}
    X \times_Y X = \bigsqcup_{g \in G} \Gamma(g). 
\end{equation*}
Hence, $f$ satisfies the filtration property with respect to $g_1, \dots,  g_m$ and $\cL_1, \dots, \cL_m$, where the $g_i$ are the elements of $G$ and $\cL_i = \cO_X$ for all $i$. 

\item\label{enum:ProductElliptic} Let $E$ be an elliptic curve over $\KK$. Consider the action of the multiplication by $-1$ and the quotient map $E\to\PP^1$. Then for any $\KK$-scheme $M$, the induced morphism $E\times M\to\PP^1\times M$ has the filtration property (see the proof of~\cite[Lemma 6.1]{chunyi-stability} in the case $\KK=\CC$; it holds over any field). In this case, $m=2$, $g_2=\id$, $g_1$ is the multiplication by $-1$ on $E$, $\cL_2=\cO_{E\times M}$, and $\cL_1=p_E^*\cO_E(R)$, where $R$ is the subscheme of points of order~2 in $E$ and $p_E\colon E\times M\to E$ is the projection.

\item\label{enum:ProductP1} The natural morphism $(\PP^1)^n\to\PP^n$ given by the quotient by the symmetric group has the filtration property (see~\cite[Corollary 4.2]{chunyi-stability} for the proof in the case $\KK=\CC$; it holds over any field). The filtration is quite complicated in this example, and we refer to~\cite[Section 4]{chunyi-stability} for all the details; we also provide a different argument  in Appendix~\ref{app:Filtration}.
For our purposes, it is enough to keep in mind that the automorphisms $g_1,\dots,g_m$ are given by the group action of $\fS_n$ on $(\PP^1)^n$ permuting the factors (here $m=n!$), and the line bundles $\cL_1,\dots,\cL_m$ are suitable effective line bundles on $(\PP^1)^n$.
\end{enumerate}
\end{example}

\begin{proposition}\label{prop:PushforwardStability}
Let $f\colon X\to Y$ be a finite faithfully flat morphism between noetherian schemes  such that the relative dualizing complex $\omega_f^\bullet$ is in $\Dperf(X)$. 
Assume further that $f$ satisfies the filtration property with respect to $g_1,\dots,g_m$ and $\cL_1,\dots,\cL_m$. 
Let $\cP$ be a slicing on $\Db(X)$ such that, for all $i=1,\dots,m$, 
\begin{enumerate}[{\rm (i)}]
\item\label{enum:BayerP} $\cP$ satisfies the Bayer property with respect to $\cL_i$ and $l=0$, i.e., $\cP\preceq\cP\otimes\cL_i$; and 
\item\label{enum:Push} for all $\phi \in \RR$, we have $g_{i*}\cP(\phi) = \cP(\phi)$.
\end{enumerate}
Then for all $\phi\in\RR$ we have 
\[
f^*f_*\cP(\phi) \subset \cP(\leq\phi).
\]
In particular, $f_\sharp\cP$ is a slicing on $\Db(Y)$.

Moreover, if $\sigma=(Z,\cP)$ is a stability condition on $\Db(X)$ with respect to $(\Lambda,v)$ such that $\cP$ satisfies~\ref{enum:BayerP} and~\ref{enum:Push}, then $f_\sharp\sigma$ is a stability condition on $\Db(Y)$ with respect to $(\Lambda,f_\sharp v)$.
\end{proposition}

\begin{proof}
Once we establish $f^*f_*\cP(\phi) \subset \cP(\leq \phi)$ for $\phi \in \RR$, the claims about $f_{\sharp}\cP$ and $f_{\sharp} \sigma$ follow from Proposition~\ref{proposition-pushforward-stability}. 

Let $F\in\cP(\phi)$. 
By base change, $f^*f_*=p_{2*}p_1^*$, where $p_1,p_2\colon X\times_Y X\to X$ are the projections, so it suffices to prove that $p_{2*}p_1^*F \in \cP(\leq \phi)$. 

Since $f$ has the filtration property, we can write $p_1^*F$ as an iterated extension of objects of the form
\[
p_2^*\cL_i^{-1}\otimes\cO_{\Gamma(g_i)}\otimes p_1^*F.
\]
Hence, we can write $p_{2*}p_1^*F$ as an iterated extension of objects of the form
\[
\cL_i^{-1}\otimes p_{2*}\left(\cO_{\Gamma(g_i)}\otimes p_1^*F\right)\cong\cL_i^{-1}\otimes g_{i*}F.
\]
By assumption~\ref{enum:Push} we have $g_{i*}F\in\cP(\phi)$, so by assumption~\ref{enum:BayerP} we find $\cL_i^{-1}\otimes g_{i*}F\in\cP(\leq\phi)$. 
This shows that $p_{2*}p_1^*F \in \cP(\leq \phi)$, as required. 
\end{proof}

Similarly to the pushforward, we first introduce a technical condition on morphisms before discussing the pullback.

\begin{definition}\label{def:CofiltrationProperty}
Let $f\colon X\to Y$ be a proper morphism of noetherian schemes.
Consider invertible sheaves $\cL_1,\dots,\cL_m$ on $Y$, and integers $l_1,\dots,l_m,N\in\ZZ$.
We say that $f$ has the \emph{cofiltration property} with respect to $\cL_1,\dots,\cL_m$ and $l_1,\dots,l_m,N$ if there exists a finite sequence of morphisms in $\Db(Y)$
\[
f_*\cO_X=G_0 \xlongrightarrow{s_1} G_1 \xlongrightarrow{s_2}\dots\xlongrightarrow{s_{m+1}} G_{m+1}=0,
\]
and integers $a_1, \dots, a_m \geq 1$ such that
\begin{equation*}
\begin{split}
&\cone(s_i)[-1]\cong\cL^{\oplus a_i}_i[l_i] \quad \text{ for } i=1,\dots,m,\\
&\cone(s_{m+1})[-1]\in \Db(Y)^{\leqslant -N}, 
\end{split}
\end{equation*}
where the last condition is with respect to the standard t-structure on $\Db(Y)$.

Moreover, when the above sequence can be chosen so that $\cone(s_{m+1}) = 0$, we will say that $f$ has the cofiltration property with respect to $\cL_1, \dots, \cL_m$, $l_1, \dots, l_m$, and $N = \infty$. 
\end{definition}

\begin{example}\label{ex:Pullback}
We will be interested in the following  examples.
\begin{enumerate}
\item\label{enum:etale2} Let $f\colon X\to Y$ be a finite morphism such that $f_*\cO_X\cong\cL_1\oplus\dots\oplus\cL_m$ for line bundles $\cL_1,\dots,\cL_m\in\Pic(Y)$. Then $f$ has the cofiltration property with respect to $\cL_1,\dots,\cL_m$, $l_1=\dots=l_m=0$, and $N = \infty$. 

\item \label{enum:regular-section}
Let $Y$ be a noetherian scheme, let $\cL$ be an invertible sheaf on $Y$, and let $s\in \rH^0(Y,\cL)$ be a regular section, i.e., a section such that the corresponding map $\cO_Y \to \cL$ is injective. 
Let $X \coloneqq V(s) \subset Y$ be the zero locus of $s$ and denote by $\iota\colon X\hookrightarrow Y$ the closed embedding. 
Then $\iota$ has the cofiltration property with respect to $\cO_Y, \cL^{\vee}$, $l_1 = 0, l_2=1$, and $N = \infty$. 
Indeed, this follows from the resolution 
\[
0\longrightarrow\cL^\vee\longrightarrow\cO_Y\longrightarrow\iota_*\cO_X\longrightarrow0. 
\]

\item\label{enum:immersions} Let $\iota\colon X\hookrightarrow Y$ be a closed immersion of projective schemes over $\KK$, and let $\cO_Y(1)$ be an ample line bundle on $Y$. 
Then we can look at a truncated resolution of $\iota_*\cO_X$: 
\[
0\longrightarrow \cF\longrightarrow\cO_Y(-b_m)^{\oplus a_m}\longrightarrow\dots\longrightarrow\cO_Y(-b_2)^{\oplus a_2}\longrightarrow\cO_Y\longrightarrow\iota_*\cO_X\longrightarrow 0,
\]
with $a_2,\dots,a_m,b_2,\dots,b_m\geq1$ and $\cF\in\Coh(Y)$.
Then $\iota$ has the cofiltration property with respect to $\cO_Y,\cO_Y(-b_2),\dots,\cO_Y(-b_m)$, $l_i=i-1$ (for $i=1,\dots,m$), and $N=m$. 

\item \label{enum:minimal-resolution} As a slight variant of the previous example, let $X$ be a projective scheme over $\KK$ with an embedding $\iota \colon X \hookrightarrow \PP^n$. Then we can consider the minimal resolution of $\iota_*\cO_X$: 
\[
0\longrightarrow \bigoplus_{j \geq n} \cO_{\PP^n}(-j)^{\oplus \beta_{n,j}}\longrightarrow\dots\longrightarrow\bigoplus_{j\geq 1} \cO_{\PP^n}(-j)^{\oplus \beta_{1,j}}\longrightarrow\cO_{\PP^n}\longrightarrow\iota_*\cO_X\longrightarrow 0. 
\]
Here, the integers $\beta_{i,j} \geq 0$ are canonically determined by the embedding $\iota \colon X \hookrightarrow \PP^n$, and all but finitely many vanish. 
The morphism $\iota$ has the cofiltration property with respect to the line bundles $\cL_{i,j} = \cO_{\PP^n}(-j)$ for $\beta_{i,j} \neq 0$ (where we define $\beta_{0,0} = 1$ and $\beta_{0,j} = 0$ for $j > 0$), the integers $l_{i,j} = i$, and $N = \infty$.
\end{enumerate}   
\end{example}

\begin{proposition}\label{prop:PullbackStability}
Let $f\colon X\to Y$ be a finite morphism between noetherian schemes.
Assume that $f$ has the cofiltration property with respect to $\cL_1,\dots,\cL_m$ and $l_1,\dots,l_m,N$.
Let $\cP$ be a slicing of $\Db(Y)$ such that 
\begin{enumerate}[{\rm (i)}]
\item\label{enum:LiP} $\cP$ satisfies the Bayer property with respect to $\cL_i$ and $l_i$, i.e., $\cP\preceq\cP\otimes\cL_i[l_i]$, for all $i=1,\dots,m$; and 
\item\label{enum:Pull} either $N = \infty$, or there exist integers $n_1 \leq n_2$ such that $N\geq n_2-n_1+1$ and 
\[
\Db(Y)^{\leqslant n_1} \subset\cP(>0)\subset
\Db(Y)^{\leqslant n_2}.
\]
\end{enumerate}
Then for all $\phi\in\RR$ we have 
\[
f_*\cO_X\otimes\cP(\phi) \subset \widehat{\cP}(\geq\phi).
\]
In particular, $f^\sharp\cP$ is a slicing of $\Db(X)$.

Moreover, if $\sigma=(Z,\cP)$ is a stability condition on $\Db(Y)$ with respect to $(\Lambda,v)$ such that $\cP$ satisfies~\ref{enum:LiP} and~\ref{enum:Pull}, then $f^\sharp\sigma$ is a stability condition on $\Db(X)$ with respect to $(\Lambda,f^\sharp v)$.
\end{proposition}

\begin{proof}
Once we establish $f_*\cO_X \otimes \cP(\phi) \subset \widehat{\cP}(\geq \phi)$ for all $\phi \in \RR$, the claims about $f^{\sharp}\cP$ and $f^{\sharp}\sigma$ follow from Proposition~\ref{proposition-pullback-stability}. 

Let $0\neq F\in\cP(\phi)$. After shifting, we can assume $0<\phi\leq 1$. Since $f$ has the cofiltration property we can write $f_*\cO_X\otimes F$ as an iterated extension of objects of the form
\[
\cL_i^{\oplus a_i} [l_i]\otimes F\qquad\text{and}\qquad G\otimes F,
\]
where $G \in \Db(Y)^{\leqslant -N}$.
By assumption~\ref{enum:LiP} and Lemma~\ref{lemma-BL-concrete}, for all $i = 1, \dots, m$ we have $\phi^-_\cP(\cL_i[l_i]\otimes F)\geq\phi$, and hence 
\begin{equation*}
   \cL_i^{\oplus a_i}[l_i]\otimes F \in \cP(\geq \phi).  
\end{equation*}
To study $G\otimes F$, we use assumption~\ref{enum:Pull}. 
If $N = \infty$, then $G = 0$ and hence $G \otimes F = 0$. 
Otherwise, since $F\in\cP(>0)$, we have $F\in\mathrm{D}^{\leqslant n_2}(Y)$. 
Thus, since $N\geq n_2-n_1+1$, we get
\[
G\otimes F\in\mathrm{D}^{\leqslant n_2-N}_\mathrm{qc}(Y)
\subset\mathrm{D}^{\leqslant n_1-1}_\mathrm{qc}(Y).
\]
By Example~\ref{example-ind-completion-tauX}, $\Dqc(Y)^{\leq n_1-1}$ can be described as the subcategory of $\Dqc(Y)$ generated by $\Db(Y)^{\leq n_1 - 1}$ under coproducts and extensions. 
Since $\Db(Y)^{\leq n_1-1} \subset \cP(> 1) \subset \cP( \geq \phi)$, we conclude that 
\begin{equation*}
    G \otimes F \in \widehat{\cP}(\geq \phi). 
\end{equation*}
Altogether, this proves that $f_*\cO_X \otimes F$ is an iterated extension of objects contained in $\widehat{\cP}(\geq \phi)$, and hence is contained in $\widehat{\cP}(\geq \phi)$ itself. 
\end{proof}

A first case where Proposition~\ref{prop:PullbackStability} applies is suitable divisors: 

\begin{lemma}
\label{lem:RestCartierDiv}
Let $X$ be a noetherian scheme with an invertible sheaf $\cL$. 
Let $\cP$ be a slicing of $\Db(X)$ satisfying the Bayer property with respect to $\cL^{\vee}$ and $l = 1$. 
Let $\iota \colon Y \hookrightarrow X$ be the inclusion of the zero locus of a regular section of $\cL$.
Then $\iota^\sharp\cP$ is a slicing of $\Db(Y)$ and it satisfies the Bayer property with respect to $\iota^*\cL^\vee$ and $l=1$. 

Moreover, if $\cP$ underlies a stability condition $\sigma = (Z, \cP)$ on $\Db(X)$ with respect to $(\Lambda, v)$, then $\iota^{\sharp}\sigma$ is a stability condition on $\Db(Y)$ with respect to $(\Lambda, \iota^\sharp v)$. 
\end{lemma}

\begin{proof}
    By Example~\ref{ex:Pullback}\eqref{enum:regular-section}, $\iota \colon Y \to X$ has the cofiltration property with respect to $\cO_X, \cL^{\vee}$, $l_1 = 0, l_2 = 1$, and $N = \infty$. 
    Hence we may apply Proposition~\ref{prop:PullbackStability} to conclude that $\iota^{\sharp} \cP$ is a slicing and that $\iota^{\sharp} \sigma$ is a stability condition. 
    It follows from Lemma~\ref{lemma-pullback-slicing-compatibility}\ref{pullback-slicing-compatibility-tensorL} and Remark~\ref{rmk:properties-preceq}\ref{enum:CompatibilitySlicings3} that $\iota^{\sharp}\cP$ satisfies the Bayer property with respect to $\iota^*\cL^\vee$ and $l=1$.
\end{proof}

We can use Lemma~\ref{lem:RestCartierDiv} inductively to obtain a relation between the Bayer properties.

\begin{proposition}\label{prop:lvlimpliesbayer}
Let $X$ be a projective scheme over $\KK$, and let $\cL$ be a very ample invertible sheaf on $X$.
Let $\sigma=(Z, \cP)$ be a numerical stability condition on $\Db(X)$ such that $\cP$ satisfies the Bayer property with respect to $\cL^\vee$ and $l=1$, i.e.,
\[
\cP\preceq\cP\otimes\cL^\vee[1].
\]
Then $\cP$ satisfies the Bayer property with respect to $\cL$ and $l=0$, i.e., 
\[
\cP\preceq\cP\otimes\cL.
\]
\end{proposition}

\begin{proof}
By Theorem~\ref{thm:base-change-stab} and Lemma~\ref{lem:base-change-preserve-order-stab}, we can assume $\KK=\overline{\KK}$.
We proceed by induction on the dimension $n$ of $X$, the statement being trivial for $n=0$.
Assume therefore that the statement is true for all projective schemes of dimensions smaller than $n$.

Let us choose a general section of $\cL$ such that its zero locus $\iota\colon Y\hookrightarrow X$ has dimension $n-1$.
By Lemma~\ref{lem:RestCartierDiv}, $\iota^\sharp\cP$ is a slicing of $\Db(Y)$ satisfying the Bayer property
\begin{equation*}\label{eq:lvlimpliesbayer3}
\iota^\sharp\cP\preceq\iota^\sharp\cP\otimes\iota^*\cL^\vee[1].    
\end{equation*}
By induction, we have
\begin{equation}\label{eq:lvlimpliesbayer4}
\iota^\sharp\cP\preceq\iota^\sharp\cP\otimes\iota^*\cL.    
\end{equation}

Let $0 \neq F\in\cP(\phi)$.
By Lemma~\ref{lemma-BL-concrete}, we must check that 
\[
\phi_{\cP}^-(F\otimes\cL)\geq\phi.
\]
In view of the distinguished triangle
\begin{equation*}
    F\longrightarrow F\otimes\cL \longrightarrow \iota_*\iota^*(F\otimes\cL), 
\end{equation*}
it suffices to prove that 
\begin{equation}
\label{phi-P-iotapushpull-FL} \phi^-_{\cP}(\iota_*\iota^*(F\otimes\cL))\geq\phi.
\end{equation}
We have 
\begin{equation*}
    \phi^{-}_{\cP}(\iota_*\iota^*(F\otimes\cL)) 
    = \phi^-_{\iota^{\sharp}\cP}(\iota^*(F\otimes\cL)) \geq \phi^-_{\iota^{\sharp} \cP \otimes \iota^*\cL}(\iota^*(F) \otimes \iota^*\cL) 
    = \phi^-_{\iota^\sharp\cP}(\iota^*F)  
    = \phi^-_{\cP}(\iota_*\iota^*F),
\end{equation*}
where the first and last equalities follow from Lemma~\ref{lemma-pullback-slicing-compatibility}\ref{pullback-slicing-compatibility-phipm}, the inequality follows from Lemma~\ref{lemma-preceq-concrete} and \eqref{eq:lvlimpliesbayer4}, and the second equality is clear as $-\otimes \iota^*\cL$ is an autoequivalence. 
On the other hand, using the exact triangle 
\begin{equation*}
    F \longrightarrow \iota_*\iota^* F\longrightarrow F\otimes\cL^\vee[1], 
\end{equation*}
Lemma~\ref{lemma-BL-concrete}
again, and the fact that $\cP$ satisfies the Bayer property with respect to $\cL^{\vee}$ and $l = 1$, we find that $\phi^-_{\cP}(\iota_*\iota^*F)\geq\phi$. 
Hence, the required inequality~\eqref{phi-P-iotapushpull-FL} holds.
\end{proof}

The same argument also gives semistability of skyscraper sheaves in the smooth case.

\begin{proposition}\label{prop:BayerGeometric}
Let $X$ be a smooth projective scheme over the field $\KK$, and let $\cL$ be a very ample invertible sheaf on $X$.
Let $\sigma=(Z,\cP)$ be a numerical stability condition on $\Db(X)$ such that $\cP$ satisfies the Bayer property with respect to $\cL^\vee$ and $l=1$, i.e.,
\[
\cP\preceq\cP\otimes\cL^\vee[1].
\]
Then for all closed points $x\in X$, the skyscraper sheaf  $\cO_x$ is $\sigma$-semistable. Moreover, for any closed points $x,y\in X$ that lie in the same connected component, we have $\phi_{\sigma}(\cO_x)=\phi_{\sigma}(\cO_y)$. 
\end{proposition}

\begin{proof}
As above, by Theorem~\ref{thm:base-change-stab} and Lemma~\ref{lem:base-change-preserve-order-stab}, we can assume $\KK=\overline{\KK}$.
Then the statement is~\cite[Corollary 6.10]{Li:sb}, but we write the simple argument here.
We proceed again by induction on $n=\dim{X}$. Also, for the semistability of $\cO_x$, one can assume without loss of generality that $X$ is connected. If $n=0$, there is nothing to prove.
Therefore, let us assume $n>0$ and let $x\in X$ be a closed point.
In this case, we can choose a general section of $\cL$ such that its zero locus $\iota\colon Y\hookrightarrow X$ is smooth, has dimension $n-1$, and contains $x$. 

By Lemma~\ref{lem:RestCartierDiv}, $\iota^\sharp\cP$ is a slicing of $\Db(Y)$ satisfying the Bayer property with respect to $\iota^*\cL^\vee$ and $l=1$.
By induction, $\cO_x$ is semistable in $\iota^\sharp\cP$, and thus in $\cP$, as we wanted. 
The last statement follows from~\cite[Proposition 2.9]{FLZ:Albanese}.
\end{proof}

We end this section with the following uniqueness result.

\begin{lemma}\label{lem:unique-general}
Let $X$ be a projective scheme and $\cL$ be a very ample line bundle. 
Let $\sigma_1=(Z_1, \cP_1)$ and $\sigma_2=(Z_2, \cP_2)$ be two numerical stability conditions on $\Db(X)$ satisfying $Z_{1}=Z_{2}$. 
If 
\[
\cP_i\prec\cP_i\otimes\cL^{\vee}[1] \quad \text{and} \quad \cP_i\preceq\cP_i\otimes\cL \quad \text{for } i=1,2,
\]
and $\cO_x$ is $\sigma_i$-semistable with $\phi_{\sigma_1}(\cO_x)=\phi_{\sigma_2}(\cO_x)$ for any closed point $x\in X$, then $\sigma_1=\sigma_2$.
\end{lemma}

\begin{proof}
We may assume that $X$ is connected and $\KK=\overline{\KK}$. Therefore, the proof of \cite[Corollary 6.13]{Li:sb} works verbatim once we know that two stability conditions $\tau_1$ and $\tau_2$ on a connected one-dimensional projective scheme $C$ with $Z_{\tau_1}=Z_{\tau_2}$, $\cO_p$ being $\tau_i$-semistable, and $\phi_{\tau_i}(\cO_p)=1$ for each $p\in C$ and $i=1,2$ satisfy $\tau_1=\tau_2$.

Using \cite[Lemma 8.11]{BMS:StabCY3s} and the assumption $Z_{\tau_1}=Z_{\tau_2}$, it suffices to show that
\[
\cP_{\tau_i}((0,1])\subset \langle \Coh(C)[1], \Coh(C)\rangle
\]
for each $i=1,2$. 
For any $E\in \cP_{\tau_i}((0,1])$, we set
\[
s\coloneqq \min \{i\mid \cH^i(E)\neq 0\}\quad \text{and}\quad t\coloneqq \max \{i\mid \cH^i(E)\neq 0\}.
\]
Since there always exists a non-zero map $E\to \cO_p[-t]$ for a closed point $p\in C$, from $E\in \cP_{\tau_i}((0,1])$, we see that $t\leq 0$. 
We may take a closed point $p\in C$ and $k\in \{0,1\}$ such that $\Ext^k(\cO_p, \cH^s(E))\neq 0$. 
Therefore, we get a non-zero map $\cO_p[-s-k]\to E$. By $E\in \cP_{\tau_i}((0,1])$, the only possibility is either $s=0$ or $s=-1$ and $k=1$. 
In both cases, we get $E\in \langle \Coh(C)[1], \Coh(C)\rangle$ as desired.
\end{proof}


\section{Stability conditions on powers of an elliptic curve}\label{sec:ProductEllipticCurves}

Let $\KK$ be a field, with algebraic closure $\overline{\KK}$.
Let $E$ be an elliptic curve over $\KK$; in particular, $E$ admits a $\KK$-rational point.
Let $n$ be a positive integer. 

In this section, we study stability conditions on the $n$-fold power of $E$. 
First, in Section~\ref{subsec:ContinuousFamilyElliptic}, we review the construction of a special family of stability conditions $\{\sigma^{a,b}\}_{(a,b)\in\RR_{>0}\times\RR}$ on $E^n$.
Then in Section~\ref{subsec:StableLineBundles}, we deal with the key property that suitable shifts of line bundles belong to the heart associated to the stability condition $\sigma^{a,b}$, for $b\in(1-\frac{1}{n},1)$ and $a>0$ sufficiently small.

\subsection{A continuous family of stability conditions}\label{subsec:ContinuousFamilyElliptic}

For notational convenience, we introduce an index for each of the factors of $E^n$, i.e., we write 
\begin{equation*}
    E^n = E_1 \times \dots \times E_n
\end{equation*}
where $E_i = E$ for all $i$. 
For each $1\le i\le n$, fix a $\KK$-rational point $q_i\in E_i$ and set
\[
H_i \coloneqq
E_1\times \dots \times E_{i-1}\times \{q_i\}\times E_{i+1}\times \dots
\times E_n
\subset E^n .
\]
By abuse of notation, we also denote by $H_i$ its numerical class and
\begin{equation}
\label{equation-H}
    H \coloneqq H_1+\dots+H_n .
\end{equation}
Moreover, for $a_1,\dots,a_n\in\ZZ$, we will use the notation
\begin{equation}\label{eq:LineBundleH}
\cO_{E^n}(a_1H_1+\dots+a_nH_n)
\end{equation}
for any line bundle whose numerical first Chern class is $a_1H_1+\dots+a_nH_n$.
We consider the free abelian group $\Lambda_n\cong\ZZ^{2^n}$ and define a numerical homomorphism $v_n\colon\rK_0(E^n)\to\Lambda_n$ by 
\begin{equation*}
\label{v-elliptic-curve}
v_n(F)\coloneqq\bigl(H_{j_1}\dots H_{j_s} \ch_{n-s}(F)\bigr)_{1\le j_1<\dots<j_s\le n}, 
\end{equation*}
as in~\cite[Section 4.2]{curves}. 

\begin{theorem}\label{thm:Eninj}
Let $\sigma$ be a stability condition on $\Db(E^n)$ with respect to $(\Lambda_n,v_n)$. 
Then:
\begin{enumerate}[{\rm (1)}]
\item\textup{({\cite[Theorem~1.1 and Proposition~2.9]{FLZ:Albanese}})}\label{enum:Eninj1}
All skyscraper sheaves $\{\cO_p\}_{p\in E^n}$ of closed points are $\sigma$-semistable with the same phase. Moreover, $\cO_p$ is geometrically $\sigma$-stable for any $\KK$-rational point $p\in E^n$.
\item\textup{({\cite[Theorem 1.1]{curves}})}\label{enum:Eninj2}
The map
\begin{equation*}\label{eq35}
\Stab_{(\Lambda_n,v_n)}(\Db(E^n))\longrightarrow
\Hom_\ZZ(\Lambda_n,\CC)\times \RR,
\qquad
\sigma=(Z,\cP)\longmapsto \bigl(Z,\phi_\sigma(\cO_p)\bigr)
\end{equation*}
is injective, where $p \in E^n$ is any fixed closed point. 
\end{enumerate}
\end{theorem}

\begin{proof}
Let $\pi\colon E^n_{\overline{\KK}}\to E^n$ be the base change morphism. For any closed point $p\in E^n$, its residue field $\kappa(p)$ is a finite extension of $\KK$. Therefore, $\pi^*(\cO_p)=\cO_{\pi^{-1}(p)}$, where $\pi^{-1}(p)$ is a zero-dimensional closed subscheme of $E^n_{\overline{\KK}}$ of length $[\kappa(p):\KK]$. Thus, \ref{enum:Eninj1} follows from \cite[Theorem~1.1, Proposition~2.9]{FLZ:Albanese} and Theorem \ref{thm:base-change-stab}. Now, \ref{enum:Eninj2} can be deduced from Lemma \ref{lem:base-change-stab-manifold} and \cite[Theorem~1.1]{curves} by base change to $\overline{\KK}$ as well.
\end{proof}

For every $(a,b)\in \RR_{>0}\times \RR$, define a group homomorphism
\[
Z^{a,b}\colon \Lambda_n\longrightarrow \CC
\]
by
\begin{equation}\label{eq:zab}
\begin{split}
Z^{a,b}(v_n(F))
&= -\int_{E^n} e^{-(b+ia)H}\,\ch(F) \\
&= -\ch_n(F)
+(b+ia)H\ch_{n-1}(F)
-\dots
+(-1)^{n+1}\frac{(b+ia)^n}{n!}H^n\rk(F).
\end{split}    
\end{equation}

The following result is an extension of~\cite[Theorem~5.9]{Liu:StabProd} and~\cite[Theorem~4.5]{curves}.

\begin{theorem}[Yucheng Liu]\label{thm:Yucheng}
For any $(a,b)\in \QQ_{>0}\times \QQ$, there exists a stability condition
\[
\sigma^{a,b}=(Z^{a,b},\cP^{a,b})\in\Stab_{(\Lambda_n,v_n)}(\Db(E^n))
\]
with central charge $Z^{a,b}$ given by~\eqref{eq:zab}, and satisfying
$\phi_{\sigma^{a,b}}(\cO_p)=1$ for any closed point $p\in E^n$. 
\end{theorem}

\begin{proof}
When $\KK=\overline{\KK}$, the result follows from~\cite{Liu:StabProd,curves}. 
For general $\KK$, we can argue by base change to $\overline{\KK}$. 
Indeed, it suffices to carry out the construction of~\cite[Theorem~5.9]{Liu:StabProd}, since we have Theorem~\ref{thm:Eninj} and the remaining argument of~\cite[Theorem~4.5]{curves} works without any change.

We give here a quick sketch of the argument, by following the presentation in~\cite{Liu:StabProd}; we will see in Section~\ref{subsec:StableLineBundles} a more explicit version of this.
Let $X, S$ be smooth projective varieties over $\KK$ and $\sigma=(Z,\cA)$ be a numerical stability condition on $\Db(X)$ such that $Z$ is rational and $\dim(S)=1$. Then $\cA$ and $\cA_{\overline{\KK}}$ are Noetherian. As in \cite[Theorem~3.3]{Liu:StabProd}, we consider the heart $\cA_S$ constructed in \cite[Theorem 3.3.6(ii)]{P-t-structures} and a central charge $Z_S\colon \rK_0(X\times_\KK S)\to \CC$. From Theorem \ref{thm:base-change-stab} and the construction, it is clear that $(\cA_S)_{\overline{\KK}}=(\cA_{\overline{\KK}})_{S_{\overline{\KK}}}$ and the pullback functor 
\[
\Db(X\times_\KK S)\to \Db(X_{\overline{\KK}}\times_{\overline{\KK}} S_{\overline{\KK}})
\]
is t-exact. Moreover, by \citestacks{0AA7}, we have 
\[
(Z_{\overline{\KK}})_{S_{\overline{\KK}}}(E_{\overline{\KK}})=Z_S(E),
\]
for any $E\in \Db(X\times_\KK S)$, so $Z_S$ is a weak stability function on $\cA_S$ as $(Z_{\overline{\KK}})_{S_{\overline{\KK}}}$ is a weak stability function on $(\cA_{\overline{\KK}})_{S_{\overline{\KK}}}$ by \cite[Theorem~3.3]{Liu:StabProd}. Note that $\cA_S$ is Noetherian by \cite[Theorem 3.3.6(i)]{P-t-structures} and $Z_S$ is rational. Therefore, any object in $\cA_S$ admits an HN filtration with respect to $Z_S$-semistability. Then as in \cite[Section~4]{Liu:StabProd}, for any $(t,s)\in \QQ_{>0}\times \QQ_{>0}$, we have the tilted heart $\cA^t_S$ with a central charge $Z^{s,t}_S$.

It is clear that \cite[Lemma~4.1]{Liu:StabProd} holds without any assumption on $\KK$. Moreover, the formulations of \cite[Lemma~4.2 and Lemma 4.3]{Liu:StabProd} are compatible with base change to $\overline{\KK}$, so they are also true over $\KK$. Therefore, the same argument in~\cite[Proposition 4.6]{Liu:StabProd} works and we see that $Z^{s,t}_S$ is a stability function on $\cA^t_S$.

By changing the base to $\overline{\KK}$, we know that \cite[Lemma~5.5]{Liu:StabProd} is valid over $\KK$. 
Therefore, using the formal arguments of \cite[Lemma~5.6 and Lemma 5.7]{Liu:StabProd}, we see that the pair $(Z^{s,t}_S,\cA^t_S)$ satisfies the support property. 
Thus, we can conclude that $(Z^{s,t}_S, \cA^t_S)$ is a stability condition on $\Db(X\times_\KK S)$ as in \cite[Theorem 4.7]{Liu:StabProd}.
\end{proof}

We now show that these stability conditions with rational coefficients extend naturally to a continuous family: in particular, they all belong to the same connected component of $\Stab_{(\Lambda_n,v_n)}(\Db(E^n))$. 
For a positive integer $m$, denote by
\[
\pi_m \colon E^n \longrightarrow E^n, \qquad z \longmapsto m z,
\]
the multiplication by~$m$ isogeny, which is finite and faithfully flat.

\begin{lemma}\label{lem:Multiplication}
Let $\sigma\in\Stab_{(\Lambda_n,v_n)}(\Db(E^n))$. 
Then for all $m\geq1$, the pullback and pushforward of $\sigma$ along $\pi_m$ are stability conditions on $E^n$, i.e., 
\[
\pi_m^\sharp\sigma\in\Stab_{(\Lambda_n,v_n)}(\Db(E^n)) \quad \text{and} \quad \pi_{m,\sharp}\sigma\in\Stab_{(\Lambda_n,v_n)}(\Db(E^n)).
\]
\end{lemma}

\begin{proof}
This is an extension of~\cite[Lemma 5.3]{chunyi-stability} to general fields $\KK$. 
More precisely, first of all, by Lemma \ref{lem:base-change}, we may assume that $\KK=\overline{\KK}$. 
According to~\cite[Lemma~4.20]{mukai:semihomogeneous}, $\pi_{m*}\cO_{E^n}$ is a homogeneous bundle, and therefore  admits a filtration by subbundles, with each factor belonging to $\Pic^0(E^n)$ by~\cite[Theorem 4.17]{mukai:semihomogeneous}. 
This means that $\pi_m$ has the cofiltration property, with respect to the invertible sheaves $\cL_i\in\Pic^0(E^n)$ such that $\cL_i^m\cong\cO_{E^n}$, all $l_i = 0$, and $N = \infty$.
Since $\sigma$ is invariant under the action of $\Pic^0(E^{n})$ (see~\cite[Theorem 3.5.1]{P-t-structures} and~\cite[Theorem 2.23]{curves}), we may apply Proposition~\ref{prop:PullbackStability} to conclude that $\pi_m^\sharp\sigma\in\Stab_{(\Lambda_n,v_n)}(\Db(E^n))$. 

For the pushforward, by~\cite[Example 2.68]{milne:alg-group}, we have a Cartesian diagram
\[
\begin{tikzcd}
	{E^n\times_{\KK} \ker(\pi_m)} & {E^n} \\
	{E^n} & {E^n,}
	\arrow[from=1-1, to=1-2]
	\arrow[from=1-1, to=2-1]
	\arrow["{\pi_m}", from=1-2, to=2-2]
	\arrow["{\pi_m}", from=2-1, to=2-2]
\end{tikzcd}
\]
where the two morphisms $E^n\times_{\KK} \ker(\pi_m)\to E^n$ are the projection to the first factor and the action map, respectively. 
Since $\ker(\pi_m)$ is zero-dimensional, for any $F\in \Db(E^n)$, $\pi_m^*\pi_{m*}F$ is an iterated extension of $t_a^*F$ for $a\in \ker(\pi_m)$, where $t_a$ is the corresponding translation morphism. 
This shows that $\pi_m$ also has the filtration property.
Since $\sigma$ is also invariant under translations (again by~\cite[Theorem 3.5.1]{P-t-structures} and~\cite[Theorem 2.23]{curves}), we may apply Proposition~\ref{prop:PushforwardStability} to conclude that $\pi_{m,\sharp}\sigma\in\Stab_{(\Lambda_n,v_n)}(\Db(E^n))$. 
\end{proof}

\begin{remark}
The above result holds more generally for any abelian variety.
\end{remark}

\begin{proposition}\label{prop:sigmaabRonEn}
There exists a family of stability conditions
\[
\sigma^{a,b}=(Z^{a,b},\cP^{a,b})\in\Stab_{(\Lambda_n,v_n)}(\Db(E^n)),
\]
parameterized by $(a,b)\in \RR_{>0}\times \RR$, with central charge $Z^{a,b}$
given by~\eqref{eq:zab} and satisfying $\phi_{\sigma^{a,b}}(\cO_p)=1$ for any closed point $p \in E^n$. In particular, this family gives an injective continuous map
\[
\RR_{>0}\times\RR\longhookrightarrow\Stab_{(\Lambda_n,v_n)}(\Db(E^n)),\qquad (a,b)\longmapsto\sigma^{a,b}.
\]
\end{proposition}

\begin{proof}
By Theorem~\ref{thm:bridgeland-deformation} together with Theorem~\ref{thm:Yucheng}, there exists
$\delta_0>0$ such that, for any $|a|,|b|\le \delta_0$, there exists a
stability condition $\sigma^{1+a,b}$ whose central charge is
$Z^{1+a,b}$ as in \eqref{eq:zab}.

For any $a\in \RR_{>0}$ and $b\in \RR$, the stability conditions from Lemma~\ref{lem:Multiplication} can be explicitly computed as follows: 
\begin{align}
\label{pim-pushforward-sigma} \pi_{m,\sharp}\sigma^{a,b} & = m^{2n}\,\sigma^{a/m^2,b/m^2}, \\
\label{pim-pullback-sigma} \pi_m^\sharp \sigma^{a,b} & = \sigma^{m^2a,m^2b},
\end{align}
where $m^{2n}$ denotes the $\CC$-action that rescales only the central charge.
In particular, for any $m\in \QQ^\times$, the existence of $\sigma^{a,b}$ implies the existence of $\sigma^{m^2a,m^2b}$.
Moreover, if $\cL$ is a line bundle whose numerical first Chern class
equals $H$, then
\begin{equation}
\label{sigmab-otimes-L}
\sigma^{a,b}\otimes \cL = \sigma^{a,b+1}.
\end{equation}

We claim that the operations $\pi_{m,\sharp}$, $\pi_m^\sharp$, the
$\CC$-action, and tensoring by $\cL$ allow one to transport the family
\begin{equation*}\label{eq:deltasquare}
\{\sigma^{1+a,b} \mid |a|,|b|\le \delta_0\}
\end{equation*}
to a stability condition $\sigma^{a,b}$ for any 
$(a,b)\in \RR_{>0}\times \RR$.
By Theorem~\ref{thm:Eninj} and the observations above, this reduces to the elementary Lemma~\ref{lem:deltasquare} below. 
\end{proof}

\begin{lemma}\label{lem:deltasquare}
Consider the following actions on the upper half-plane
$\RR_{>0}\times \RR$:
\begin{align*}
   \{R_m\}_{m\in\QQ^\times}\colon & \RR_{>0}\times \RR \longrightarrow \RR_{>0}\times \RR,
   \qquad (a,b)\longmapsto (m^2 a, m^2 b),\\
   L\colon & \RR_{>0}\times \RR \longrightarrow \RR_{>0}\times \RR,
   \qquad (a,b)\longmapsto (a, b+1).
\end{align*}
For any $\delta$-square
\[
S_\delta=\{(1+a,b)\mid |a|\le \delta,\, |b|\le \delta\} 
\]
where $0<\delta <1$, 
the actions of $R_m$ and $L$ translate $S_\delta$ so as to cover the entire
upper half-plane $\RR_{>0}\times \RR$.
\end{lemma}

\begin{proof}
First observe that
\[
R_{1/m}\circ L \circ R_m (a,b)
= (a, b+1/m^2).
\]
Hence, by combining the actions of $R_m$ and $L$, we obtain translations of the $b$-coordinate by arbitrary small values.

Next, for any $x\in \RR_{>0}$, there exists $s\in \QQ^\times$ such that
\[
x \in \bigl(s^2(1-\delta),\, s^2(1+\delta)\bigr).
\]
Therefore, by applying a suitable $R_s$ to $S_\delta$, we can reach any
desired $a$-coordinate (with an open neighborhood along the $b$-coordinate). Together with the translations in the $b$-direction, this shows that the $R_m$- and $L$-actions move $S_\delta$ to cover the whole upper half-plane.
\end{proof}

We conclude this subsection with the following key results on the Bayer property for stability conditions on $\Db(E^n)$.
These are immediate consequences of analogous results in~\cite{chunyi-stability}, but we will need a more precise reformulation.

\begin{lemma}\label{lem:BayerPropertyEllipticCurves1}
Let $\sigma=(Z,\cP)\in\Stab_{(\Lambda_n,v_n)}(\Db(E^n))$.
For every line bundle $\cL$ whose numerical first Chern class is of the form
\[
a_1 H_1 + \dots + a_nH_n,
\]
with $a_1,\dots,a_n\in\ZZ_{\geq0}$, we have
\begin{equation*}\label{eq:BayerPropertyEllipticCurves1}
    \cP\preceq\cP\otimes\cL.
\end{equation*}
\end{lemma}

\begin{proof}
This follows from \cite[Lemma 2.3]{chunyi-stability}. 
To be precise, the statement there was proven under the assumption $\KK=\CC$. 
However, the same proof works as soon as the field is algebraically closed, so that the curves have enough $\KK$-rational points, and 
we can reduce to this case by Lemma~\ref{lem:base-change-preserve-order-stab}.
\end{proof}

\begin{lemma}\label{lem:BayerPropertyEllipticCurves2}
Let us fix $m_0\in\ZZ_{>0}$.
Then there exists $a_0=a_0(m_0)>0$ such that for any $(a,b)\in\RR_{>0}\times\RR$ with $a>a_0$, we have
\begin{equation*}\label{eq:BayerPropertyEllipticCurves2}
    \cP^{a,b}\prec\cP^{a,b}\otimes\cL[1],
\end{equation*}
where $\cL$ is any line bundle whose numerical first Chern class is $-m_0H$.
\end{lemma}

\begin{proof}
This is a slightly more precise version of~\cite[Proposition 5.4]{chunyi-stability}.
By Theorem~\ref{thm:bridgeland-deformation} and Proposition~\ref{prop:sigmaabRonEn}, there exists $\delta_0>0$ such that
\[
\mathrm{dist}(\sigma^{a,b},\sigma^{1,0})<\frac 12,
\]
for all $(a,b)\in S_{2\delta_0}$, where $S_\delta$ denotes the $\delta$-square centered at $(1,0)$, as in Lemma~\ref{lem:deltasquare}.
In particular, $\mathrm{dist}(\cP^{a,b},\cP^{1,0})<\frac 12$.

By~\eqref{pim-pushforward-sigma}, \eqref{pim-pullback-sigma}, and Lemmas~\ref{lemma-pullback-slicing-compatibility}\ref{pullback-slicing-compatibility-dist} and~\ref{lemma-pushforward-slicing-compatibility}\ref{pushforward-slicing-compatibility-dist}, for every rational number $m=\frac{p}{q}\in\QQ_{>0}$ with $p,q\in\ZZ_{>0}$ and all $(a,b),(a',b')\in\RR_{>0}\times\RR$, we have 
\[
\mathrm{dist}(\cP^{m^2a,m^2b},\cP^{m^2a',m^2b'})=\mathrm{dist}(\pi^\sharp_{p} \pi_{q,\sharp}\cP^{a,b},\pi^\sharp_{p} \pi_{q,\sharp}\cP^{a',b'})\leq\mathrm{dist}(\cP^{a,b},\cP^{a',b'}). 
\]
Let us take $m\in\ZZ_{>0}$ such that $m^2>\frac{1}{\delta_0}$. 
Then for all $(a,b)\in m^2S_{1/m^2}$, we have
\begin{equation}\label{eq:ProjectiveScheme1}
\mathrm{dist}(\cP^{a,b},\cP^{a,b+1})\leq\mathrm{dist}(\cP^{a,b},\cP^{m^2,0})+\mathrm{dist}(\cP^{m^2,0},\cP^{a,b+1})<1.    
\end{equation}
By using the two operations induced by tensoring with the line bundle $\cO_{E^n}(H)$ and $\pi^\sharp_{p} \pi_{q,\sharp}$, for $p,q\in\ZZ_{>0}$ (as in Lemma~\ref{lem:deltasquare}), we can then cover the whole half-plane with $a>m^2$.
More precisely, for $(a,b)\in\RR_{>m^2}\times\RR$, we have that~\eqref{eq:ProjectiveScheme1} holds; thus, by~\eqref{sigmab-otimes-L} and~\eqref{eq:Distance}, we deduce that
\[
\cP^{a,b}\otimes\cO_{E^n}(H)=\cP^{a,b+1}\prec\cP^{a,b}[1].
\]
Finally, by applying $\pi_{m_0}^\sharp$ and using Lemma~\ref{lem:BayerPropertyEllipticCurves1}, we get that there exists $a_0>0$ such that for $(a,b)\in\RR_{>a_0}\times\RR$, we have that
\[
\cP^{a,b}\otimes\cO_{E^n}(m_0H)\preceq\cP^{a,b}\otimes\cO_{E^n}(m^2_0H)\prec\cP^{a,b}[1],
\]
as we claimed.
\end{proof}

\subsection{Stable line bundles}\label{subsec:StableLineBundles}

Our next goal is to identify some stable objects in the heart of $\sigma^{a,b}$ for suitable parameters $(a,b)$. 
By \cite[Corollary~2.16]{FLZ:Albanese} (based on~\cite[Proposition~3.1.4]{Polishchuk:lagrangian}), all line bundles are $\sigma$-stable, for all $\sigma\in\Stab_{(\Lambda_n,v_n)}(\Db(E^n))$.
We are interested, for some special line bundles, in determining which shift of them is in the heart $\cA^{a,b}\coloneqq\cP^{a,b}((0,1])$.

\begin{notation}
We will need to work with several (but finitely many!) products $E^n$ at the same time.
For this purpose, we need a more precise notation than the previous subsection.
Let us fix an integer $n_0\geq1$.
For $n=1,\dots,n_0$,  we will use the subscript $n$ to distinguish objects associated to $E^n$ for different $n$.
For instance, we use $\sigma_n^{a,b}$ to denote the stability condition of Proposition~\ref{prop:sigmaabRonEn} on $\Db(E^n)$, $Z^{a,b}_n$ to denote the central charge of $\sigma_n^{a,b}$, and $\cA_n^{a,b}$ to denote its heart.

We denote by $\underline{n}$ the set $\{1,\dots,n\}$.
For $k=0,1,\dots,n$, we consider a subset $I_{k,n}\subset\underline{n}$ of cardinality $k$.
For $r\in\ZZ$, we use the notation $\cL_{I_{k,n}}^{(r)}$ to denote an invertible sheaf on $E^n$ whose numerical first Chern class is given by
\[
\ch_1(\cL_{I_{k,n}}^{(r)})=rH_{i_1}+\dots+rH_{i_k} \quad \text{where } I_{k,n}=\{i_1,\dots,i_k\}.
\]
For instance, $I_{0,n}=\emptyset$, and $\cL_{I_{0,n}}^{(r)}$ denotes an invertible sheaf in $\Pic^0(E^n)$, for all $r$. 
Similarly, $I_{n,n}=\underline{n}$, and $\cL_{I_{n,n}}^{(r)}$ denotes an invertible sheaf $\cO_{E^n}(rH)$ as in Section~\ref{subsec:ContinuousFamilyElliptic} (see~\eqref{eq:LineBundleH}).
\end{notation} 

\begin{proposition}\label{prop:LineBundleStability}
Let $n_0,r_0\geq1$ be positive integers, and fix a real number $b_0 \in \RR$ satisfying 
\begin{equation*}
r_0\frac{n_0-1}{n_0}<b_0<r_0. 
\end{equation*}
Then there exists a positive real number $a_0\in\RR_{>0}$, which depends only on $n_0$, $r_0$, and $b_0$, such that for all $0<a<a_0$ and all integers $n=1,\dots, n_0$, we have
\[
\cL_{I_{0,n}}^{(r_0)}[n-1],\cL_{I_{k,n}}^{(r_0)}[n-k]\in\cA_n^{a,b_0} 
\]
with
\[
0<\phi_{\sigma_n^{a,b_0}}(\cL_{I_{0,n}}^{(r_0)}[n-1])<\phi_{\sigma_n^{a,b_0}}(\cL_{I_{k,n}}^{(r_0)}[n-k])\leq 1,
\]
for all $k=1,\dots,n$ and $I_{k,n}\subset\underline{n}$.
\end{proposition}

The proof of Proposition~\ref{prop:LineBundleStability} will take the rest of this section.
The argument is by induction on $n$.
We start by recalling the explicit construction in~\cite{Liu:StabProd} to pass from $n$ to $n+1$.

First of all, by Theorem~\ref{thm:base-change-stab}\ref{enum:base-change-stab3}, we can assume that the base field $\KK$ is algebraically closed; hence, we can use the original presentation in~\cite{Liu:StabProd}.
We write $E^{n+1}=E^n\times E$ and consider the two projections $p_{\leq n}\colon E^{n+1}\to E^{n}$ and $p_{n+1}\colon E^{n+1}\to E$.
Fix $(a,b)\in\RR_{>0}\times\RR$ and consider the stability condition $\sigma_n^{a,b}$ on $\Db(E^n)$.
We consider the abelian category
\[
\cC^{a,b}_{n+1}\coloneqq\left(\cA_n^{a,b}\right)_E=\left\{A\in\Db(E^{n+1})\st p_{\leq n,*}(A\otimes p_{n+1}^*\cO_E(m))\in\cA^{a,b}_n\textup{ for }m\gg0\right\},
\]
where $\cO_E(1)$ is an ample line bundle on $E$ of degree~1 such that $p_{n+1}^*\cO_E(1)=\cO_{E^{n+1}}(H_{n+1})$. Note that the abelian category $\cC^{a,b}_{n+1}$ is precisely the abelian category $\cA_S$ appearing in the proof of Theorem~\ref{thm:Yucheng} for $S=E$.
We define the functions $\alpha^{a,b}_{n+1},\beta^{a,b}_{n+1},\gamma^{a,b}_{n+1},\delta^{a,b}_{n+1}$ on $\cC^{a,b}_{n+1}$ by the equations
\[
\begin{split}
&\Re Z^{a,b}_n\left(p_{\leq n,*}(A\otimes p_{n+1}^*\cO_E(m))\right)=\alpha^{a,b}_{n+1}(A)\,m+\beta^{a,b}_{n+1}(A)\\
&\Im Z^{a,b}_n\left(p_{\leq n,*}(A\otimes p_{n+1}^*\cO_E(m))\right)= \gamma^{a,b}_{n+1}(A)\, m+\delta^{a,b}_{n+1}(A),
\end{split}
\]
for $m\gg0$.

Notice that the objects of the form $p_{\leq n}^*F\otimes p_{n+1}^*\cO_E(r)$ are in $\cC^{a,b}_{n+1}$ for any $r\in\ZZ$ and any $F\in\cA_n^{a,b}$.
An easy computation shows the following:
\begin{equation}\label{eq:PositivityZParis2}
\begin{split}
	&\alpha^{a,b}_{n+1}\left(p_{\leq n}^*F\otimes p_{n+1}^*\cO_E(r)\right)=\Re Z^{a,b}_n(F)\\
	&\beta^{a,b}_{n+1}\left(p_{\leq n}^*F\otimes p_{n+1}^*\cO_E(r)\right)=r\, \Re Z^{a,b}_n(F)\\
	&\gamma^{a,b}_{n+1}\left(p_{\leq n}^*F\otimes p_{n+1}^*\cO_E(r)\right)=\Im Z^{a,b}_n(F)\\
	&\delta^{a,b}_{n+1}\left(p_{\leq n}^*F\otimes p_{n+1}^*\cO_E(r)\right)=r\,\Im Z^{a,b}_n(F).
\end{split}    
\end{equation}

By~\cite[Proposition 4.6]{Liu:StabProd} and~\cite[Section 4.1]{curves}, the weak stability function 
\begin{equation*}
'Z_{n+1}^{a,b}\coloneqq \left(a\alpha^{a,b}_{n+1}+b\gamma^{a,b}_{n+1}-\delta^{a,b}_{n+1}\right)+ia\gamma^{a,b}_{n+1}
\end{equation*}
defines the torsion pair 
\begin{align*}
&\cF_{n+1}^{a,b}\coloneqq\left\{F\in\cC^{a,b}_{n+1}~\left|~\begin{array}{l}\text{for every injection }0\neq G\hookrightarrow F\text{ in }\cC^{a,b}_{n+1},\\ \gamma^{a,b}_{n+1}(G)\neq 0\text{ and }\Re\, {'Z_{n+1}^{a,b}(G)}\geq0\end{array}\right. \right\} \\
&\cT_{n+1}^{a,b}\coloneqq\left\{F\in\cC^{a,b}_{n+1}~\left|~\begin{array}{l}\text{for every surjection }F\twoheadrightarrow G \text{ in }\cC^{a,b}_{n+1},\,\text{either}\\ \gamma^{a,b}_{n+1}(G)=0, \text{ or } \gamma^{a,b}_{n+1}(G)\neq 0\text{ and }\Re\, {'Z_{n+1}^{a,b}(G)}<0\end{array}\right. \right\} 
\end{align*}
on $\cC^{a,b}_{n+1}$, whose tilted heart is exactly $\cA_{n+1}^{a,b}$.

Moreover,
\begin{equation}\label{eq:PositivityZParis3}
Z^{a,b}_{n+1}=\left(-b\alpha^{a,b}_{n+1}+\beta^{a,b}_{n+1}+a\gamma^{a,b}_{n+1}\right)+i\left(-a\alpha^{a,b}_{n+1}-b\gamma^{a,b}_{n+1}+\delta^{a,b}_{n+1}\right).    
\end{equation}

This is the key point where we will use the induction assumption.
In fact, given an object $F\in\cA_n^{a,b}$ (in our case this will be an appropriate shift of one of the line bundles $\cL^{(r_0)}_{I_{k,n}}$ as in Proposition~\ref{prop:LineBundleStability}), we consider $p_{\leq n}^*F\otimes p_{n+1}^*\cO_E(r)\in\cC_{n+1}^{a,b}$.
A priori, there is a nontrivial condition for $p_{\leq n}^*F\otimes p_{n+1}^*\cO_E(r)$ to belong to either $\cF_{n+1}^{a,b}$ or $\cT_{n+1}^{a,b}$.
In the case of line bundles, we do know that a certain shift of them is in the tilted category, which is $\cA_{n+1}^{a,b}$.
Therefore, we deduce that, if $F$ is such a shift of a line bundle, $p_{\leq n}^*F\otimes p_{n+1}^*\cO_E(r)$ must belong to either $\cF_{n+1}^{a,b}$ or $\cT_{n+1}^{a,b}$, and which case holds is  uniquely determined by the sign of
\[
\Re\, {'Z_{n+1}^{a,b}}(p_{\leq n}^*F\otimes p_{n+1}^*\cO_E(r))=-\Im Z_{n+1}^{a,b}(p_{\leq n}^*F\otimes p_{n+1}^*\cO_E(r)).
\]
To summarize: if $F\in\cA_n^{a,b}$ is a shift of a line bundle, then
\begin{equation}\label{eq:PositivityZParis6}
p_{\leq n}^*F\otimes p_{n+1}^*\cO_E(r)\in\cA_{n+1}^{a,b}\quad\text{ if and only if }\quad \Im Z_{n+1}^{a,b}(p_{\leq n}^*F\otimes p_{n+1}^*\cO_E(r))>0,    
\end{equation}
and
\begin{equation}\label{eq:PositivityZParis7}
p_{\leq n}^*F\otimes p_{n+1}^*\cO_E(r)[1]\in\cA_{n+1}^{a,b}\quad\text{ if and only if }\quad \Im Z_{n+1}^{a,b}(p_{\leq n}^*F\otimes p_{n+1}^*\cO_E(r))\leq0.
\end{equation}

\begin{proof}[Proof of Proposition~\ref{prop:LineBundleStability}.]
By Theorem~\ref{thm:Eninj}\ref{enum:Eninj2}, since $Z^{a,b}_{n}$ is $\fS_{n}$-invariant by definition, we get that $\sigma^{a,b}_{n}$ is also $\fS_{n}$-invariant.
Thus, it is enough to consider subsets of the form
\[
I_{k,n}=\underline{k}\subset\underline{n},
\]
for $k=0,\dots,n$. 
We therefore simplify notation and write $\cL_{k,n}$ for a line bundle $\cL_{I_{k,n}}^{(r_0)}$ on $E^n$, associated to $\underline{k}$. 
Since $\sigma^{a,b}_{n}$ is also invariant under the action of $\Pic^0(E^{n})$ (see~\cite[Theorem 3.5.1]{P-t-structures} and~\cite[Theorem 2.23]{curves}), we can further assume
\begin{equation}\label{eq:PositivityZParis5}
\begin{split}
&\cL_{0,n}=\cO_{E^{n}},\\
&\cL_{k,n}=p_{\leq n-1}^*\cL_{k,n-1},\quad \text{ for } k=1,\dots,n-1,\\
&\cL_{n,n}=p_{\leq n-1}^*\cL_{n-1,n-1}\otimes p_{n}^*\cO_E(r_0).
\end{split}
\end{equation}

We can now use~\eqref{eq:PositivityZParis2} and~\eqref{eq:PositivityZParis3} iteratively, to compute:
\begin{equation}\label{eq:PositivityZParis4}
Z_{n}^{a,b}(\cL_{k,n})=(-1)^{n-k+1}(b+ia)^{n-k}(r_0-b-ia)^{k}.    
\end{equation}
We fix $b=b_0$ and expand the imaginary part of~\eqref{eq:PositivityZParis4} in $a$ when $k=1,\dots,n$:
\[
\Im\, Z_{n}^{a,b_0}(\cL_{k,n})=(-1)^{n-k+1}a(r_0-b_0)^{k-1}b_0^{n-k-1}((n-k)r_0-nb_0)+O(a^2).
\]
By our assumptions on $b_0$, we get that the leading term in $a$ is exactly $(-1)^{n-k}$ multiplied by a strictly positive number, if $k=1,\dots,n$.
For $k=0$, the leading term in $a$ is instead $(-1)^{n-1}nab^{n-1}_0$.
Since $n_0$, $r_0$, and $b_0$ are fixed and everything depends only on finitely many choices in $k$ and $n$, we can find $a_0'>0$ such that, for all $a$ in the open interval $(0,a_0')$, the sign of $\Im\, Z_{n}^{a,b_0}(\cL_{k,n})$ is exactly $(-1)^{n-k}$, when $k=1,\dots,n$, and $(-1)^{n-1}$ for $k=0$.

We can now prove the first part of the statement by induction on $n=1,\dots,n_0$.
For $n=1$, $\cL_{0,1},\cL_{1,1}\in\cA^{a,b}_1=\Coh(E)$, for all $a\in\RR_{>0}$ and $b\in\RR$.
If the statement is true for $n-1\geq1$ (with our choice of $a$ and $b_0$), we get
\[
\cL_{0,n-1}[n-2],\cL_{k,n-1}[n-k-1]\in\cA_{n-1}^{a,b_0}.
\]
We now use~\eqref{eq:PositivityZParis5}, together with~\eqref{eq:PositivityZParis6} and~\eqref{eq:PositivityZParis7}.
Since the sign of $\Im\, Z_{n}^{a,b_0}(\cL_{0,n})$ is $(-1)^{n-1}$, we get that $\cL_{0,n}[n-1]\in\cA_n^{a,b_0}$.
Similarly, for $k=1,\dots,n$, since the sign of $\Im\, Z_{n}^{a,b_0}(\cL_{k,n})$ is $(-1)^{n-k}$, we get that $\cL_{k,n}[n-k]\in\cA_n^{a,b_0}$, which is what we wanted to prove.

To prove the second part of the statement, we only need to check the condition on the phases. 
To do this, we expand the real part of~\eqref{eq:PositivityZParis4} in $a$ and obtain:
\[
\Re\, Z_{n}^{a,b_0}(\cL_{k,n})=(-1)^{n-k+1}b_0^{n-k}(r_0-b_0)^k+O(a).
\]
We can choose $0<a_0\leq a_0'$ such that, for all $a$ in the open interval $(0,a_0)$, the sign of $\Re\, Z_{n}^{a,b_0}(\cL_{k,n})$ is exactly $(-1)^{n-k+1}$.
In particular, we get
\[
0<\phi_{\sigma_n^{a,b_0}}(\cL_{0,n}[n-1])<\frac 12<\phi_{\sigma_n^{a,b_0}}(\cL_{k,n}[n-k])\leq1,
\]
for $k=1,\dots,n$, which is what we wanted.
\end{proof}


\section{Stability conditions on projective space}\label{sec:ProjectiveSpace}

In this section, we use the results on powers of elliptic curves to identify a special connected component of the space of numerical stability conditions on $\PP^n$.
We follow~\cite{chunyi-stability}: we first induce stability conditions on $\Db((\PP^1)^n)$ and then move to $\Db(\PP^n)$. The main technique is pushforward of stability conditions, studied in detail in Section~\ref{sec:PullbackPushforward}.
The new result is Theorem~\ref{thm:AlgebraicP1n}, which shows that on $\Db((\PP^1)^n)$ the stability conditions we construct are \emph{algebraic}, namely associated to a full strong exceptional collection. This will be used in Part~\ref{part:Relative} to prove the existence of stability conditions over any base, endowed with proper moduli spaces.

We continue to work over a fixed base field $\KK$.
Since we are interested in pushing forward stability conditions along faithfully flat morphisms between smooth projective varieties, we can use Lemma~\ref{lem:base-change}\ref{enum:BaseChange2} and assume that $\KK$ is algebraically closed, when needed.

\subsection{Powers of the projective line}\label{subsec:P1n} 

Let $E$ be an elliptic curve over $\KK$ and $n\geq1$ an integer.
Consider the action of the multiplication by $-1$ on $E$ and the induced morphism
\[
f\colon E^n\longrightarrow (\PP^1)^n.
\]

\begin{proposition}\label{prop:Elliptic2P1n}
The pushforward map gives a continuous map
\[
f_\sharp\colon\Stab_{(\Lambda_n,v_n)}(\Db(E^n))\longrightarrow\Stab_{\mathrm{num}}(\Db((\PP^1)^n)).
\]
In particular, we get an injective continuous map
\[
\RR_{>0}\times\RR\longhookrightarrow\Stab_{\mathrm{num}}(\Db((\PP^1)^n)),\qquad (a,b)\longmapsto f_\sharp\sigma^{a,b}.
\]
Moreover, for all $\sigma=(Z,\cP)\in\Stab_{(\Lambda_n,v_n)}(\Db(E^n))$, we have
\begin{equation}\label{eq:BayerPropertyP1n}
f_\sharp\cP\preceq f_\sharp\cP\otimes\cO_{(\PP^1)^n}(a_1,\dots,a_n),    
\end{equation}
for all $a_1,\dots,a_n\in\ZZ_{\geq0}$.
\end{proposition}

\begin{proof}
We apply Proposition~\ref{prop:PushforwardStability} inductively to the morphism
\[
f_r\colon\underbrace{\PP^1\times\dots\times\PP^1}_{r\text{-times}}\times\underbrace{E\times\dots\times E}_{(n-r)\text{-times}}\longrightarrow\underbrace{\PP^1\times\dots\times\PP^1}_{(r+1)\text{-times}}\times\underbrace{E\times\dots\times E}_{(n-r-1)\text{-times}},
\]
for $r=0,\dots,n-1$.

Indeed, the morphism $f_r$ is finite, faithfully flat, and the condition on the dualizing complex is satisfied, since we are in the smooth setting.
Moreover, as already remarked in Example~\ref{ex:PushForward}\eqref{enum:ProductElliptic}, $f_r$ has the filtration property.
We need to check properties~\ref{enum:BayerP} and~\ref{enum:Push} of Proposition~\ref{prop:PushforwardStability}.
Property~\ref{enum:BayerP} follows from Lemma~\ref{lem:BayerPropertyEllipticCurves1}.
To show~\ref{enum:Push}, we simply observe that by Theorem~\ref{thm:Eninj}\ref{enum:Eninj2}, since the multiplication by $-1$ preserves the class of skyscraper sheaves, we only need to check that the central charge $Z$ is invariant under multiplication by $-1$.
This follows immediately, since the abelian group $\Lambda_n$ is invariant.

In view of Lemma~\ref{lemma-pushforward-slicing-compatibility}\ref{pushforward-slicing-compatibility-tensorL}, the last claim follows from Lemma~\ref{lem:BayerPropertyEllipticCurves1}.
\end{proof}

For $(a,b)\in\RR_{>0}\times\RR$, we define
\begin{equation}\label{eq:StabonP1n}
\sigma^{a,b}_{(\PP^1)^n}\coloneqq 2^{-n} f_\sharp\sigma^{2a,2b}.    
\end{equation}
This is motivated by the following observation.
We consider the line bundle $\cO_{(\PP^1)^n}(1,\dots,1)$ on $(\PP^1)^n$, and denote by $H_{(\PP^1)^n}$ its numerical first Chern class.
Then $f^*H_{(\PP^1)^n}=2H$ (where $H$ was defined in~\eqref{equation-H}).
Hence, it follows from~\eqref{eq:zab} that the central charge $Z^{a,b}_{(\PP^1)^n}$ associated to $\sigma^{a,b}_{(\PP^1)^n}$ is exactly
\begin{equation*}\label{eq:DefOfSigmaabP1n}
Z^{a,b}_{(\PP^1)^n}=-\int_{(\PP^1)^n} e^{-(b+ia)H_{(\PP^1)^n}}\ch.    
\end{equation*}
We denote by $\cA^{a,b}_{(\PP^1)^n}$ the heart of the stability condition $\sigma^{a,b}_{(\PP^1)^n}$. 

We apply Proposition~\ref{prop:Elliptic2P1n} to obtain a version of Proposition~\ref{prop:LineBundleStability} for $(\PP^1)^n$.
For a cardinality $k$ subset $I_{k}=\{i_1,\dots,i_k\}\subset \underline{n} = \{1,\dots,n\}$ (we drop here the subscript $n$ on $I_{k}$ which appeared in the notation in Section~\ref{subsec:StableLineBundles}, since we do not need it anymore), we consider the line bundle on $(\PP^1)^n$ given by
\[
\cO_{(\PP^1)^n}(I_{k})\coloneqq p_{i_1}^*\cO_{\PP^1}(1)\otimes\dots\otimes p_{i_k}^*\cO_{\PP^1}(1),
\]
where, as usual, $p_i\colon(\PP^1)^n\to\PP^1$ denotes the projection onto the $i$-th factor.
For instance, $\cO_{(\PP^1)^n}(I_{0})=\cO_{(\PP^1)^n}$ and $\cO_{(\PP^1)^n}(I_{n})=\cO_{(\PP^1)^n}(1,\dots,1)$.

\begin{proposition}\label{prop:LineBundleStabilityP1n}
Let $n\geq1$ be a positive integer, and fix a real number $b_0 \in \RR$ satisfying 
\begin{equation*}
    \frac{n-1}{n}<b_0<1. 
\end{equation*}
Then there exists a positive real number $a_0\coloneqq a(n,b_0)\in\RR_{>0}$, which depends only on $n$ and $b_0$, such that for all $0<a<a_0$, we have
\[
\cO_{(\PP^1)^n}[n-1],\cO_{(\PP^1)^n}(I_{k})[n-k]\in\cA_{(\PP^1)^n}^{a,b_0} 
\]
with
\[
0<\phi_{\sigma_{(\PP^1)^n}^{a,b_0}}(\cO_{(\PP^1)^n}[n-1])<\phi_{\sigma_{(\PP^1)^n}^{a,b_0}}(\cO_{(\PP^1)^n}(I_{k})[n-k])\leq 1,
\]
for all $k=1,\dots,n$ and $I_{k}\subset\{1,\dots,n\}$.
\end{proposition}

\begin{proof}
Note that, in the notation of Section~\ref{subsec:StableLineBundles},  we have $f^*\cO_{(\PP^1)^n}(I_{k})\cong\cL_{I_{k,n}}^{(2)}$. 
Hence, the proposition follows immediately from Proposition~\ref{prop:LineBundleStability} and the definition~\eqref{eq:StabonP1n} of $\sigma^{a,b}_{(\PP^1)^n}$ as a pushforward stability condition.
\end{proof}

We notice that the line bundles $\cO_{(\PP^1)^n}(I_{k})$ for $I_k \subset \underline{n}$ form a \emph{full strong exceptional collection}. More precisely, they have the following properties.
\begin{itemize}
\item For all $I_{k}\subset\underline{n}$, we have
\[
\Hom\left(\cO_{(\PP^1)^n}(I_{k}),\cO_{(\PP^1)^n}(I_{k})[p]\right)\cong\begin{cases}\KK&\text{if }p=0\\0&\text{otherwise.}\end{cases}
\]
\item For all $k>k'$,  $I_{k},I_{k'}\subset\underline{n}$ with $|I_{k}| = k$ and $|I_{k'}| = k'$, and all $p\in\ZZ$, we have
\[
\Hom\left(\cO_{(\PP^1)^n}(I_{k}),\cO_{(\PP^1)^n}(I_{k'})[p]\right)=0.
\]
\item For all $I_{k}\neq I_{k}'\subset\underline{n}$ and all $0 \neq p\in\ZZ$, we have
\[
\Hom\left(\cO_{(\PP^1)^n}(I_{k}),\cO_{(\PP^1)^n}(I_{k}')[p]\right)=0.
\]
\item $\Db((\PP^1)^n)$ is generated as a triangulated category by the objects $\cO_{(\PP^1)^n}(I_{k})$ for $I_k \subset \underline{n}$.
\end{itemize}

As a consequence, the category generated by extensions
\[
\cA_\mathrm{alg}\coloneqq\langle\cO_{(\PP^1)^n}(I_{0})[n],\dots,\cO_{(\PP^1)^n}(I_{k})[n-k],\dots,\cO_{(\PP^1)^n}(I_{n})\rangle
\]
is the heart of a bounded t-structure on $\Db((\PP^1)^n)$, which is equivalent to the abelian category mod-$A$ of (right) finitely generated modules over a finite-dimensional $\KK$-algebra $A$ (see, e.g.,~\cite[Theorem 6.2]{Bondal}).

For each $I_k \subset \underline{n}$, we choose an element $z_{I_k} \in \mathbb{H}\cup\RR_{<0}$ in the union of the complex upper half plane and the negative real line.
We define a numerical central charge $Z_{\mathrm{alg}} \colon \Knum((\PP^1)^n) \to \CC$ by setting
\[
Z_\mathrm{alg}\left(\cO_{(\PP^1)^n}(I_{k})[n-k]\right)\coloneqq z_{I_k}.
\]
Then the pair $(Z_\mathrm{alg},\cA_\mathrm{alg})$ is a numerical stability condition on $\Db((\PP^1)^n)$, which is called \emph{algebraic} with respect to $\{\cO_{(\PP^1)^n}(I_{k})\}_{I_k\subset\underline{n}}$ (see, e.g., \cite[Section 3.3]{macricurves}).

\begin{remark}\label{rmk:AlgebraicSatbilityMod}
Algebraic stability conditions have the property that (semi)stable objects in $\cA_\mathrm{alg}$ are $\theta$-(semi)stable in the sense of King (\cite{King}; see, e.g.,~\cite[Remark 3.6]{macricurves}).
In particular, proper moduli spaces of semistable objects exist for these stability conditions.
\end{remark}

\begin{theorem}\label{thm:AlgebraicP1n}
Let $\frac{n-1}{n}<b_0<1$ and let $0<a<a(n,b_0)$, with $a(n,b_0)$ as in Proposition~\ref{prop:LineBundleStabilityP1n}.
Then there exists $G\in\widetilde{\GL}_2^+(\RR)$ such that $G\cdot\sigma^{a,b_0}_{(\PP^1)^n}$ is algebraic with respect to $\{\cO_{(\PP^1)^n}(I_{k})\}_{I_k\subset\underline{n}}$.
\end{theorem}

\begin{proof}
We use Proposition~\ref{prop:LineBundleStabilityP1n}.
Let us choose a phase $\phi_0\in\RR$ such that
\[
0<\phi_{\sigma_{(\PP^1)^n}^{a,b_0}}(\cO_{(\PP^1)^n}[n-1])<\phi_0<\phi_{\sigma_{(\PP^1)^n}^{a,b_0}}(\cO_{(\PP^1)^n}(I_{k})[n-k])\leq 1,
\]
for all $k=1,\dots,n$ and $I_{k}\subset\{1,\dots,n\}$, and consider the heart $\cP^{a,b_0}_{(\PP^1)^n}((\phi_0,\phi_0+1])$.
Then
\[
\cO_{(\PP^1)^n}(I_{0})[n],\dots,\cO_{(\PP^1)^n}(I_{k})[n-k],\dots,\cO_{(\PP^1)^n}(I_{n})\in\cP^{a,b_0}_{(\PP^1)^n}((\phi_0,\phi_0+1]).
\]
By~\cite[Lemma 2.3]{macricurves}, we get that 
\[
\cP^{a,b_0}_{(\PP^1)^n}((\phi_0,\phi_0+1])=\cA_\mathrm{alg},
\]
and thus, by acting with the rotation by $\phi_0$, we get that $\sigma_{(\PP^1)^n}^{a,b_0}$ is algebraic, as we wanted.
\end{proof}

\begin{remark}\label{rmK:SupportField}
As a last comment, which will be useful later, we notice that the support property for the stability conditions $\sigma_{(\PP^1)^n}^{a,b}$ does not depend on the field $\KK$.
To construct stability conditions over an arbitrary base, we will first consider algebraic stability conditions $\sigma_{(\PP^1)^n}^{a,b_0}$ (the positive real number $a(n,b_0)$ does not depend on $\KK$ either) and then deform back to arbitrary $(a,b)\in\RR_{>0}\times\RR$ by using the support property.
\end{remark}

\subsection{Projective space}\label{subsec:Pn}

In this subsection, we complete the induction procedure from $E^n$ to $\PP^n$.
We need to restrict the abelian group $\Lambda_n$ first.
We recall the notation $H$ for the ample numerical divisor class on $E^n$ defined in~\eqref{equation-H}.
We define the abelian group $\Lambda_{H,n}\cong\ZZ^{n+1}$ as a quotient of $\Lambda_n$ together with the numerical homomorphism $v_{H,n}\colon\rK_0(E^n)\to\Lambda_{H,n}$ given by
\[
v_{H,n}(F)\coloneqq\bigl(H^s\ch_{n-s}(F)\bigr)_{0\leq s\leq n}.
\]
We consider the space of stability conditions
\[
\Stab_{(\Lambda_{H,n},v_{H,n})}(\Db(E^n))
\]
with respect to $(\Lambda_{H,n},v_{H,n})$.
Note that the stability conditions $\sigma^{a,b}$ on $\Db(E^n)$ from Theorem~\ref{thm:Yucheng} give rise to stability conditions with respect to $(\Lambda_{H,n},v_{H,n})$, since the central charge $Z^{a,b}$ given by~\eqref{eq:zab} factors via $\Lambda_{H,n}$; in particular, $\Stab_{(\Lambda_{H,n},v_{H,n})}(\Db(E^n))$ is nonempty.

The key property we are going to use is that stability conditions in  $\Stab_{(\Lambda_{H,n},v_{H,n})}(\Db(E^n))$ are invariant with respect to the Weyl group $W(B_n)\coloneqq(\ZZ/2)^n\rtimes\fS_n$.
This group acts naturally on $E^n$, where the $i$-th copy of $\ZZ/2$ acts by $-1$ on the $i$-th factor of $E^n$ and $\fS_n$ acts by permutation of the factors. 
The quotient map 
\begin{equation*}
g\colon E^n\to E^n/W(B_n)\cong\PP^n 
\end{equation*}
factors through the map $f\colon E^n\to(\PP^1)^n$ of Section~\ref{subsec:P1n}.

\begin{theorem}[{\cite{chunyi-stability}}]\label{theorem-chunyi-Pn}
The pushforward map gives a continuous map
\[
g_\sharp\colon\Stab_{(\Lambda_{H,n},v_{H,n})}(\Db(E^n))\longrightarrow\Stab_\mathrm{num}(\Db(\PP^n)).
\]
In particular, we get an injective continuous map
\[
\RR_{>0}\times\RR\longhookrightarrow\Stab_{\mathrm{num}}(\Db(\PP^n)),\qquad (a,b)\longmapsto g_\sharp\sigma^{a,b}.
\]
Moreover, for all $\sigma=(Z,\cP)\in\Stab_{(\Lambda_{H,n},v_{H,n})}(\Db(E^n))$, we have
\begin{equation}\label{eq:BayerPropertyPn}
g_\sharp\cP\preceq g_\sharp\cP\otimes\cO_{\PP^n}(1).    
\end{equation}
\end{theorem}

For all $(a,b)\in\RR_{>0}\times\RR$, we define
\begin{equation}\label{eq:stabonPn}
\sigma^{a,b}_{\PP^n}\coloneqq \frac{1}{2^n n!} g_\sharp\sigma^{2a,2b}.    
\end{equation}
As in the case of $(\PP^1)^n$, this is motivated by the following observation: if we denote by $H_{\PP^n}$ the numerical first Chern class of $\cO_{\PP^n}(1)$, then $g^*H_{\PP^n}=2H$ and $\deg(g)=2^nn!$; thus, the central charge $Z^{a,b}_{\PP^n}$ of $\sigma^{a,b}_{\PP^n}$ is given by 
\begin{equation}\label{eq:ZstabonPn}
Z^{a,b}_{\PP^n}=-\int_{\PP^n} e^{-(b+ia)H_{\PP^n}}\ch.
\end{equation}

\begin{remark}
\label{sigmaab-Pn-properties} 
By Theorem~\ref{thm:Eninj}\ref{enum:Eninj1} and Remark~\ref{rmk:GeometricPushPull}, for all $(a,b)\in\RR_{>0}\times\RR$ the stability conditions $\sigma_{\PP^n}^{a,b}$ are geometric, and skyscraper sheaves of closed points have phase~$1$. 
We also note that~\eqref{eq:BayerPropertyPn} implies 
\begin{equation}
\label{eq:BayerPropertyPn-ab}
    \sigma^{a,b}_{\PP^n} \preceq \sigma^{a,b}_{\PP^n} \otimes \cO_{\PP^n}(1). 
\end{equation}
\end{remark}

\begin{proof}
This is~\cite[Theorem 6.3]{chunyi-stability}.
By Proposition~\ref{prop:Elliptic2P1n}, we only have to analyze the morphism $h\colon(\PP^1)^n\to\PP^n$ given by the quotient by $\fS_n$, and apply Proposition~\ref{prop:PushforwardStability} to $f_{\sharp} \sigma$, where $\sigma\in \Stab_{(\Lambda_{H,n},v_{H,n})}(\Db(E^n))$.
By Example~\ref{ex:PushForward}\eqref{enum:ProductP1}, $h$ has the filtration property with respect to automorphisms given by the group action of $\fS_n$ and effective invertible sheaves.
The Bayer property then follows from~\eqref{eq:BayerPropertyP1n}.
The invariance property can be proved by reducing to the case of $\Db(E^n)$, where it amounts to the fact remarked above that stability conditions in $\Stab_{(\Lambda_{H,n},v_{H,n})}(\Db(E^n))$ are 
invariant by the action of $W(B_n)$.

Finally, in view of Lemma~\ref{lemma-pushforward-slicing-compatibility}\ref{pushforward-slicing-compatibility-tensorL}, the last claim follows from~\eqref{eq:BayerPropertyP1n}.
\end{proof}

We will need the following observation about the Bayer property.
For $(a,b) \in \RR_{>0} \times \RR$, we write $\cP_{\PP^n}^{a,b}$ for the slicing underlying the stability condition $\sigma_{\PP^n}^{a,b}$. 

\begin{lemma}\label{lem:BayerPropertyPn}
Given $m_0\in\ZZ_{>0}$, there exists $a_0=a_0(m_0)>0$ such that for any $(a,b)\in\RR_{>0}\times\RR$ with $a>a_0$, we have 
\[
\cP^{a,b}_{\PP^n}\prec\cP^{a,b}_{\PP^n}\otimes\cO_{\PP^n}(-m_0)[1].
\]
\end{lemma}

\begin{proof}
We first work on $E^n$.
By Lemma~\ref{lem:BayerPropertyEllipticCurves2}, given $m_0$, there exists $a_0>0$ such that for any $(a,b)\in\RR_{>0}\times\RR$ with $a>a_0$, we have:
\[
\cP^{a,b}\prec\cP^{a,b} \otimes \cO_{E^n}(-2m_0H)[1].
\]
By Lemma~\ref{lemma-pushforward-slicing-compatibility}\ref{pushforward-slicing-compatibility-tensorL} and the definition of the induced stability condition $\sigma^{a,b}_{\PP^n}$, the result follows.
\end{proof}

We use Theorem~\ref{theorem-chunyi-Pn} to identify a distinguished connected component of $\Stab_\mathrm{num}(\Db(\PP^n))$.

\begin{definition}
The \emph{distinguished component} $\Stab^\dagger(\Db(\PP^n))\subset\Stab_\mathrm{num}(\Db(\PP^n))$ is the connected component containing the stability conditions $\sigma_{\PP^n}^{a,b}$.
\end{definition}

Note that, as remarked before, $\Stab^\dagger(\Db(\PP^n))$ contains a nonempty open subset consisting of geometric stability conditions.


\section{Stability conditions on projective schemes}\label{sec:ProjectiveSchemes}

In this section, we finish reviewing the construction in \cite{chunyi-stability} by inducing stability conditions from $\PP^n$ to any projective scheme $X$ (Theorem~\ref{thm:ProjectiveScheme}). 
Using this, we define a distinguished connected component of the space of stability conditions, whose properties are studied in Section~\ref{subsec:ConnectedComponent}.

We continue to work over a base field $\KK$. 
When needed, by Lemma~\ref{lem:base-change} we can assume that $\KK$ is algebraically closed. 

\subsection{Projective schemes}\label{subsec:ProjectiveSchemes}
Let $X$ be a projective scheme over $\KK$ of dimension $d\geq0$.
To construct stability conditions on $\Db(X)$, we first fix a closed embedding $\iota\colon X\hookrightarrow\PP^n$ over $\KK$ and let $\cO_X(1)\coloneqq\iota^*\cO_{\PP^n}(1)$. 
We define $\Lambda_{\iota}$ as the image of the composition 
\begin{equation*}
    \rK_0(X) \xlongrightarrow{\iota_*} \rK_0(\PP^n) \xlongrightarrow{v} \Knum(\PP^n) \cong \ZZ^{n+1},
\end{equation*}
where $v$ is the quotient map (an isomorphism in this case). 
By taking an alteration (see, e.g.,~\cite[Lemma~12.7]{stability-families}) and using Grothendieck--Riemann--Roch, one finds that $\Lambda_{\iota} \cong \ZZ^{d+1}$ (see also Example~\ref{ex:SmoothAbelianGroup} below). 
We define
\[
v_{\iota}\colon\rK_0(X)\longrightarrow\Lambda_{\iota}, \quad v_{\iota}(F)\coloneqq v(\iota_*F). 
\]

\begin{example}\label{ex:SmoothAbelianGroup}
If $X$ is smooth geometrically integral of dimension~$d$, then $v_{\iota}$ is, up to an automorphism coming from the Hirzebruch--Riemann--Roch Theorem, the same as the more familiar morphism $(H_X^s\ch^\mathrm{td}_{d-s})_{0\leq s\leq d}$, where $H_X=c_1(\cO_X(1))$ and $\ch^\mathrm{td}\coloneqq\ch.\td_X$; in particular, in such a case, stability conditions with respect to $(\Lambda_{\iota},v_{\iota})$ coincide if we choose either morphism.    
\end{example}

\begin{theorem}[{\cite{chunyi-stability}}]\label{thm:ProjectiveScheme}
Let $X$ be a projective scheme over $\KK$ with a closed embedding $\iota \colon X \hookrightarrow \PP^n$. 
Then 
\[
\Stab_{(\Lambda_{\iota},v_{\iota})}(\Db(X))\neq\emptyset.
\]
More precisely, there exists $a_0>0$, which only depends on the Hilbert polynomial of $X$ in $\PP^n$, such that for all $(a,b) \in \RR_{> 0} \times \RR$ with $a > a_0$, the pullback 
\begin{equation}\label{eq:Paris20260702}
\widetilde{\sigma}_X^{a,b}=(\widetilde{Z}_X^{a,b},\widetilde{\cP}_X^{a,b})\coloneqq\iota^\sharp\sigma_{\PP^n}^{a,b}\in\Stab_{(\Lambda_{\iota},v_{\iota})}(\Db(X))    
\end{equation}
is a geometric stability condition for which skyscraper sheaves of closed points have phase $1$ and 
\begin{equation}\label{eq:BayerPropertyX1}
\widetilde{\cP}^{a,b}_{X}\preceq\widetilde{\cP}^{a,b}_{X}\otimes\cO_{X}(1) .
\end{equation}
Moreover, given $m\in\ZZ_{>0}$, there exists $a_1=a_1(m)\geq a_0$ such that, for any $(a,b)\in\RR_{>0}\times\RR$ with $a>a_1$, we have 
\begin{equation}\label{eq:BayerPropertyX2}
\widetilde{\cP}^{a,b}_{X}\prec\widetilde{\cP}^{a,b}_{X}\otimes\cO_{X}(-m)[1].
\end{equation}
\end{theorem}

\begin{proof}
As we already pointed out, we can use Lemma~\ref{lem:base-change}\ref{enum:BaseChange1} to reduce to the case $\KK=\overline{\KK}$.
In this case, we argue as in~\cite[Theorem 6.5]{chunyi-stability}, using Proposition~\ref{prop:PullbackStability}. 
As in Example~\ref{ex:Pullback}\eqref{enum:minimal-resolution}, if we consider the minimal resolution  
\[
0\longrightarrow \bigoplus_{j \geq n} \cO_{\PP^n}(-j)^{\oplus \beta_{n,j}}\longrightarrow\dots\longrightarrow\bigoplus_{j\geq 1} \cO_{\PP^n}(-j)^{\oplus \beta_{1,j}}\longrightarrow\cO_{\PP^n}\longrightarrow\iota_*\cO_X\longrightarrow 0, 
\]
then $\iota$ has the cofiltration property with respect to $\cL_{i,j} = \cO_{\PP^n}(-j)$ for $\beta_{i,j} \neq 0$, $l_{i,j} = i$, and $N = \infty$.
In particular, hypothesis~\ref{enum:Pull} in Proposition~\ref{prop:PullbackStability} is automatically satisfied. 

Now we show that hypothesis~\ref{enum:LiP} on the Bayer property holds for suitably large parameters $a$.  
Choose a positive integer $m_0$ such that
\[
m_0\geq\max_{1\leq i \leq n}\left\{\frac{j}{i}~\Bigg|~\beta_{i,j}\neq 0\right\}.
\]
By Lemma~\ref{lem:BayerPropertyPn}, there exists $a_0=a_0(m_0)>0$ such that for
all $(a,b)\in\RR_{>a_0}\times\RR$, we have
\[
\sigma_{\PP^n}^{a,b}\prec\sigma_{\PP^n}^{a,b}\otimes\cO_{\PP^n}(-m_0)[1].
\]
Iterating this, for such $(a,b)$ and any $i \geq 1$ we find 
\begin{equation*}
    \sigma_{\PP^n}^{a,b} \prec \sigma_{\PP^n}^{a,b}\otimes\cO_{\PP^n}(-im_0)[i]. 
\end{equation*}
If $\beta_{i,j} \neq 0$, then by our choice of $m_0$ we have $im_0 \geq j$, so by~\eqref{eq:BayerPropertyPn-ab} we obtain  
\begin{equation*}
\sigma_{\PP^n}^{a,b}\otimes\cO_{\PP^n}(-im_0)[i] \preceq 
\sigma_{\PP^n}^{a,b}\otimes\cO_{\PP^n}(-j)[i]. 
\end{equation*} 
Altogether, this shows
\begin{equation*}
\sigma_{\PP^n}^{a,b} \preceq 
\sigma_{\PP^n}^{a,b}\otimes\cO_{\PP^n}(-j)[i], 
\end{equation*}
which verifies hypothesis~\ref{enum:LiP} in Proposition~\ref{prop:PullbackStability}. 

Fixing the Hilbert polynomial of $X$ only gives finitely many possibilities for the Betti numbers $\beta_{i,j}$ in the minimal resolution of $\iota_*\cO_X$; hence, for a sufficiently large $m_0$ depending only on the Hilbert polynomial of $X$, we can ensure that the above statements hold. 

We conclude that $\widetilde{\sigma}_X^{a,b} \coloneqq\iota^\sharp\sigma^{a,b}_{\PP^n}$ is a stability condition for all $(a,b) \in \RR_{>a_0} \times \RR$.
It follows from Remarks~\ref{sigmaab-Pn-properties} and~\ref{rmk:GeometricPushPull} that $\widetilde{\sigma}_X^{a,b}$ is a geometric stability condition for which skyscraper sheaves have phase $1$.  
In view of Lemma~\ref{lemma-pullback-slicing-compatibility}\ref{pullback-slicing-compatibility-tensorL}, the Bayer properties~\eqref{eq:BayerPropertyX1} and~\eqref{eq:BayerPropertyX2} follow again from \eqref{eq:BayerPropertyPn-ab} and Lemma~\ref{lem:BayerPropertyPn}.
\end{proof}

\begin{remark}\label{rmk:CompatibilityEn}
A word of caution about notation.
Due to the presence of the Todd class, the central charge $\widetilde{Z}_{X}^{a,b}$ is in general not of the form
\[
Z_X^{a,b}=-\int_X e^{-(b+ia)H_X}\ch,
\]
where $H_X$ denotes the numerical first Chern class of $\cO_X(1)$.
In particular, when $X=E^n$, it is not a priori clear that the stability conditions $\sigma^{a,b}$ are in the connected component which contains $\widetilde{\sigma}_{X}^{a,b}$.
Similarly, for $X=(\PP^1)^n$.
For smooth surfaces and threefolds over $\CC$, we will see results in this direction in Proposition~\ref{prop:MassHomBoundSurfaces} and Proposition~\ref{prop:BMTmassbound}.
\end{remark}

\subsection{The distinguished connected component}\label{subsec:ConnectedComponent}

The construction of Section~\ref{subsec:ProjectiveSchemes} depends a priori on the choice of the closed embedding $\iota \colon X \hookrightarrow \PP^n$.
The goal of this section is to explain that the construction in fact only depends on the choice of an ample numerical class. 
To this end, we make the following definition.

\begin{definition}\label{def:StabdaggerAbsolute}
Let $X$ be a projective scheme over $\KK$ and $H_X\in N^1(X)$ be an ample numerical class.
We choose a closed embedding $\iota\colon X\hookrightarrow\PP^n$ such that $c_1(\iota^*\cO_{\PP^n}(1))=mH_X$, for some $m\in\ZZ_{>0}$.
We define
\[
(\Lambda_{H_X},v_{H_X})\coloneqq(\Lambda_{\iota},v_{\iota}),
\]
and
\[
\Stabdagger_{H_X}(\Db(X))\subset\Stab_{(\Lambda_{H_X},v_{H_X})}(\Db(X))
\]
as the connected component containing the stability conditions $\widetilde{\sigma}_{X}^{a,b}=\iota^\sharp\sigma_{\PP^n}^{a,b}$, for $a\gg0$.
\end{definition}

We call $\Stabdagger_{H_X}(\Db(X))$ the \emph{distinguished component associated to $H_X$}.
Note that by Theorem~\ref{thm:ProjectiveScheme}, $\Stabdagger_{H_X}(\Db(X))$ contains a nonempty open subset consisting of geometric stability conditions. 

Our goal is to show that Definition~\ref{def:StabdaggerAbsolute} does not depend on the choice of the embedding $\iota$. 
We start with the numerical data $(\Lambda_{H_X},v_{H_X})$; this is the content of the following elementary result.

\begin{lemma}\label{lem:NumericalData Equal}
Let $\iota_1\colon X\hookrightarrow\PP^{n_1}$ and $\iota_2\colon X\hookrightarrow\PP^{n_2}$ be two closed embeddings so that
\[
c_1(\iota^*_{1}\cO_{\PP^{n_1}}(1))=m_1 H_X,\qquad c_1(\iota^*_{2}\cO_{\PP^{n_2}}(1))=m_2 H_X,
\]
for an ample divisor class $H_X$ and some $m_1,m_2\in \ZZ_{>0}$. 
Then there exists an isomorphism $u\colon\Lambda_{\iota_1}\xrightarrow{\cong}\Lambda_{\iota_2}$ such that $v_{\iota_2}=u\circ v_{\iota_1}$.
\end{lemma}

\begin{proof}
First of all, we observe that two line bundles have the same numerical first Chern class if and only if their classes in $\Knum(\Dperf(X))$ are the same.
Also, if we fix a positive integer $r\in\ZZ_{>0}$, a class $\beta\in\rK_0(\PP^n)$ is trivial if and only if $\chi(\cO_{\PP^n}(rk),\beta)=0$, for all $k\in\ZZ$. 

We use these two observations to show that $\ker(\iota_{1,*})=\ker(\iota_{2,*})\subset\rK_0(X)$.
Indeed, let $\alpha\in\ker(\iota_{1,*})$.
For all $k\in\ZZ$, we have
\begin{equation*}
\begin{split}
0=\chi(\cO_{\PP^{n_1}}(m_2k),\iota_{1,*}\alpha)&=\chi(\iota_1^*\cO_{\PP^{n_1}}(m_2k),\alpha)\\
&=\chi(\iota_2^*\cO_{\PP^{n_2}}(m_1k),\alpha)\\
&=\chi(\cO_{\PP^{n_2}}(m_1k),\iota_{2,*}\alpha)=0,
\end{split}
\end{equation*}
and so $\alpha\in\ker(\iota_{2,*})$, as we wanted.

Since $\Lambda_{\iota_1}$ and $\Lambda_{\iota_2}$ are free abelian groups of finite rank, the conclusion follows.
\end{proof}

We now pass to the distinguished connected component.
Before stating the result, we need a few preliminary results.
We start with the following description of the support property for $\sigma^{a,0}_{\PP^n}$. 
Denote by $H_{\PP^n}$ the hyperplane class of $\PP^n$.

\begin{lemma}\label{lem:formula-support}
For each $n\geq 1$, there exists a quadratic form $q_n$ on $\Knum(\PP^n)_{\RR}$ such that for any $a>0$, 
\[
Q_a(-)\coloneqq q_n(v_{aH_{\PP^n}}(-))
\]
is a quadratic form on $\Knum(\PP^n)_{\RR}$ and gives the support property for $\sigma_{\PP^n}^{a,0}$.
\end{lemma}

\begin{proof}
By Lemma~\ref{lem:base-change-stab-manifold} and Lemma~\ref{lem:base-change}, we may assume in the proof that $\KK=\overline{\KK}$. 
We may also assume that $a\in \QQ_{>0}$, as the extension to all $a>0$ follows from the standard deformation argument. 
Moreover, from the construction of $\sigma_{\PP^n}^{a,0}$, it is easy to see that we only need to prove a corresponding statement on $E^n$, as described in the following.

Let
\[
E^n=E_1\times\dots\times E_n
\]
be the product of $n$ copies of an elliptic curve $E$ and let
$H_i$ be the pullback of a degree-one divisor class on $E_i$.
Following the notation of \cite[Section~4.2]{curves}, we write
\[
\Lambda_n\cong
\bigoplus_{I\subset\{1,\dots,n\}}\ZZ e_I
\]
and
\[
v_n(F)=
\left(H_{i_1}\cdots H_{i_k}\ch_{n-k}(F)\right)_
{I=\{i_1,\dots,i_k\}\subset\underline{n}}.
\]
For $a>0$ and $(x_I)_{I\subset \underline{n}}\in (\Lambda_n)_{\RR}$, we define a linear transformation
\[
M_{n,a}\cdot (x_I)_{I\subset \underline{n}}\coloneqq (a^{|I|}x_I)_{I\subset \underline{n}}.
\]
Then the central charge
\[
Z_n^{a,0}(-)=-\int_{E^n}e^{-ia(H_1+\dots+H_n)}\ch(-)
\]
can be written as $Z_n^{a,0}=W_n\circ M_{n,a}$ for a fixed real-linear map $W_n\colon(\Lambda_n)_\RR\to\CC$, independent of $a$, defined by
\[
W_n((x_I)_{I\subset \underline{n}})=-\sum_{I\subset\underline n}(-i)^{|I|}x_I.
\]

We claim that there is a quadratic form $q^E_n$ on $(\Lambda_n)_{\RR}$,
independent of $a$, such that
\[
Q^E_{n,a}(-)\coloneqq q^E_n(M_{n,a}\cdot -)
\]
gives the support property for $\sigma_n^{a,0}$.
We prove this by induction on $n$. 
For $n=1$, we may take
\[
q^E_1(x_0,x_1)=x_0^2+x_1^2.
\]
Assume the claim is known for $n\geq1$. 
We prove it for $n+1$. 
Write
\[
E^{n+1}=E^n\times E
\]
with projections $p_{\leq n}$ and $p_{n+1}$ as in Section~\ref{subsec:StableLineBundles}. 
For $x=(x_I)_{I\subset\underline{n+1}}\in(\Lambda_{n+1})_{\RR}$, write
\[
x=(x',x'')
\]
where
\[
x'_I=x_{I\cup\{n+1\}},
\qquad
x''_I=x_I,
\qquad I\subset\underline{n}.
\]
Then
\[
M_{n+1,a}\cdot x=(aM_{n,a}\cdot x',\,M_{n,a}\cdot x'').
\]

Fix an object $F\in \Db(E^{n+1})$ and take $x= v_{n+1}(F)$. 
By the definition of the functions
$\alpha^{a,0}_{n+1},\beta^{a,0}_{n+1},\gamma^{a,0}_{n+1},\delta^{a,0}_{n+1}$ in Section~\ref{subsec:StableLineBundles}, we have
\[
\alpha^{a,0}_{n+1}(F)+i\gamma^{a,0}_{n+1}(F)=W_n(M_{n,a}\cdot x')
\]
and
\[
\beta^{a,0}_{n+1}(F)+i\delta^{a,0}_{n+1}(F)=W_n(M_{n,a}\cdot x'').
\]
Equivalently, for $y=(y',y'')\in(\Lambda_{n+1})_{\RR}$, define fixed
real-linear functions
\[
\widehat\alpha_{n+1}(y)+i\widehat\gamma_{n+1}(y)\coloneqq W_n(y'),
\]
and
\[
\widehat\beta_{n+1}(y)+i\widehat\delta_{n+1}(y)
\coloneqq W_n(y'').
\]
If $y=M_{n+1,a}\cdot x$, then
\[\widehat\alpha_{n+1}(y)=a\alpha^{a,0}_{n+1}(F),
\qquad
\widehat\gamma_{n+1}(y)=a\gamma^{a,0}_{n+1}(F),\]
and
\[\widehat\beta_{n+1}(y)=\beta^{a,0}_{n+1}(F),
\qquad
\widehat\delta_{n+1}(y)=\delta^{a,0}_{n+1}(F).\]
Using \eqref{eq:PositivityZParis3} with $b=0$, we obtain
\[
Z_{n+1}^{a,0}(F)=\beta^{a,0}_{n+1}(F)+a\gamma^{a,0}_{n+1}(F)
+i\left(-a\alpha^{a,0}_{n+1}(F)+\delta^{a,0}_{n+1}(F)\right).
\]
Thus
\[
W_{n+1}(y)=\widehat\beta_{n+1}(y)+\widehat\gamma_{n+1}(y)
+i\left(\widehat\delta_{n+1}(y)-\widehat\alpha_{n+1}(y)\right).
\]

In the notation above, the quadratic
form in \cite[Lemma 5.7]{Liu:StabProd} has the shape
\[
\beta^{a,0}_{n+1}(F)\gamma^{a,0}_{n+1}(F)-\alpha^{a,0}_{n+1}(F)\delta^{a,0}_{n+1}(F)+\eta Q^E_{n,a}(x')
\]
for 
\[
0< \eta\leq \frac{a}{C^2}
\]
with a constant $C>0$ independent of $a$, which gives the support property for $\sigma^{a,0}_{n+1}$ with respect to a quotient lattice of $\Lambda_{n+1}$; see \cite[Remark 5.8]{Liu:StabProd}. Fix a real number
$0<\lambda<1/C^2$ and set $\eta=\lambda a.$
Then, writing $y=M_{n+1,a}\cdot x$, we have
\[
\beta^{a,0}_{n+1}(F)\gamma^{a,0}_{n+1}(F)-\alpha^{a,0}_{n+1}(F)\delta^{a,0}_{n+1}(F)=\frac{1}{a}
\left(
\widehat\beta_{n+1}\widehat\gamma_{n+1}-\widehat\alpha_{n+1}\widehat\delta_{n+1}
\right)(y),
\]
and, by the induction hypothesis,
\[
Q^E_{n,a}(x')=q^E_n(M_{n,a}x')=q^E_n(y'/a)=\frac{1}{a^2}q^E_n(y').
\]
We define a quadratic form
\[
q^{(n+1)}_{n+1}(y)
\coloneqq
\widehat\beta_{n+1}(y)\widehat\gamma_{n+1}(y)-\widehat\alpha_{n+1}(y)\widehat\delta_{n+1}(y)+\lambda q^E_n(y')
\]
which is independent of $a$. Repeating the same construction after permuting the factors, we obtain
quadratic forms $q^{(r)}_{n+1}$ for each $r\in \underline{n+1}$. Define a quadratic form on $(\Lambda_{n+1})_{\RR}$ by
\[
q^E_{n+1}\coloneqq\sum_{r=1}^{n+1}q^{(r)}_{n+1},
\]
which is independent of $a$. Then, as in the proof of \cite[Theorem 4.5]{curves}, the quadratic form
\[
Q^E_{n+1,a}(-)\coloneqq q^E_{n+1}(M_{n+1,a}\cdot -)
\]
on $(\Lambda_{n+1})_{\RR}$ gives the support property for $\sigma_{n+1}^{a,0}$ by \cite[Lemma 7.5]{stab-E3} and the induction hypothesis, as claimed.
\end{proof}

By Proposition~\ref{prop:lvlimpliesbayer}, Lemma \ref{lem:unique-general}, and Proposition \ref{prop:BayerGeometric}, we deduce the following uniqueness result; see also \cite[Corollary 6.13]{Li:sb}.

\begin{lemma}\label{lem:unique-Bayer}
Let $X$ be a smooth projective scheme and $\cL$ be a very ample line bundle. 
Let $\sigma_1=(Z_1, \cP_1)$ and $\sigma_2=(Z_2, \cP_2)$ be two numerical stability conditions on $\Db(X)$ satisfying $Z_{1}=Z_{2}$. 
If 
\[
\cP_i\prec\cP_i\otimes\cL^{\vee}[1] \quad \text{for } i=1,2,
\]
and $\phi_{\sigma_1}(\cO_x)=\phi_{\sigma_2}(\cO_x)$ for any closed point $x\in X$, then $\sigma_1=\sigma_2$.
\end{lemma}

We can now prove the main result for this section.

\begin{theorem}\label{thm:canonical-component}
Let $X$ be a projective scheme over $\KK$ equipped with an ample numerical class $H_X \in N^1(X)$. 
Let $\iota_1\colon X\hookrightarrow \PP^{n_1}$ and $\iota_2\colon X\hookrightarrow \PP^{n_2}$ be two closed embeddings such that 
\[
c_1(\iota^*_{1}\cO_{\PP^{n_1}}(1))=m_1 H_X,\qquad c_1(\iota^*_{2}\cO_{\PP^{n_2}}(1))=m_2 H_X,
\]
for some $m_1,m_2\in \ZZ_{>0}$. 
Then for any $(a_1,b_1),(a_2,b_2)\in \RR_{>0}\times \RR$ with $a_1,a_2\gg 0$, the stability conditions $\iota_1^{\sharp}\sigma_{\PP^{n_1}}^{a_1,b_1}$ and $\iota_2^{\sharp}\sigma_{\PP^{n_2}}^{a_2,b_2}$ lie in the same connected component of $\Stab_{(\Lambda_{H_X},v_{H_X})}(\Db(X))$.
\end{theorem}

\begin{proof}
By Lemma~\ref{lem:base-change-stab-manifold} and Lemma~\ref{lem:base-change}, we may assume in the proof that $\KK=\overline{\KK}$.

According to Theorem~\ref{thm:ProjectiveScheme}, to prove the statement, it suffices to show that for some pairs $(a_1,b_1),(a_2,b_2)\in \RR_{>0}\times\RR$ with $a_i\gg 0$, the stability conditions $\iota_1^{\sharp}\sigma_{\PP^{n_1}}^{a_1,b_1}$ and $\iota_2^{\sharp}\sigma_{\PP^{n_2}}^{a_2,b_2}$ are connected by a path in $\Stab_{(\Lambda_{H_X},v_{H_X})}(\Db(X))$.

We first treat the case where $X=\PP^r$ and $\iota_1$ is an isomorphism, namely we are in the following situation: $\iota\colon \PP^r\hookrightarrow\PP^s$ is a closed embedding such that $\iota^*\cO_{\PP^s}(1)\cong\cO_{\PP^r}(k)$. 
The statement becomes then the following claim: for $a\gg 0$, $\sigma_{\PP^r}^{a,0}$ and $\iota^{\sharp}\sigma_{\PP^s}^{\frac{a}{k},b_{\iota}}$ can be connected by a path in $\Stab_{\num}(\Db(\PP^r))$, where
\begin{equation*}\label{eq:Canonicitybalpha}
	b_{\iota}\coloneqq\frac{r+1-(s+1)k}{2k}.
\end{equation*}

Set $T(2x):=x/\sinh(x)$, then Grothendieck--Riemann--Roch
and the choice of $b_\iota$ give
\begin{equation*}\label{eq:CanonicityThetaCharge}
	\widetilde Z_{\PP^r}^{a/k,b_\iota}(E)
	\coloneqq Z_{\PP^s}^{a/k,b_\iota}(\iota_*E)
	=-\int_{\PP^r}e^{-iaH_{\PP^r}}\ch(E)\frac{T(H_{\PP^r})^{r+1}}{T(k H_{\PP^r})^{s+1}}.
\end{equation*}
as $\td(\cO_{\PP^r}(m))=e^{mH_{\PP^r}/2}T(mH_{\PP^r})$. We therefore consider the path
\begin{equation*}\label{eq:CanonicityCentralChargePath}
Z_{t,a}(E)
\coloneqq
-\int_{\PP^r}e^{-iaH_{\PP^r}}\ch(E)\left(1+t\frac{T(H_{\PP^r})^{r+1}}{T(k H_{\PP^r})^{s+1}}-t\right).
\end{equation*}
for $t\in[0,1]$. Its endpoints are $Z_{\PP^r}^{a,0}$ and
$\widetilde Z_{\PP^r}^{a/k,b_\iota}$.

Put $M_a\coloneqq\operatorname{diag}(1,a,\ldots,a^r)$. Since $T(x)=1+O(x^2)$, after fixing a norm on the space
$V\coloneqq\Knum(\PP^r)_{\RR}$, there is a constant $C>0$, independent
of $a$ and $t$, such that
\begin{equation}\label{eq:diff-Zt}
	\left|(Z_{t,a}-Z_{0,a})(v)\right|
	\leq\frac{C}{a^2}\|M_av\|
\end{equation}
for all $v\in V$ and $t\in[0,1]$.

By Lemma~\ref{lem:formula-support}, there is a quadratic form $q_r$,
independent of $a$, such that
\[
Q_a(v)=q_r(M_av)
\]
gives the support property for $\sigma_{\PP^r}^{a,0}$. Moreover, $q_r$
is negative definite on $\ker W_0$, where
\[
W_0\coloneqq Z_{0,a}\circ M_a^{-1}
\]
is independent of $a$. Hence $\{u\in V\mid q_r(u)\geq0\}\cap\ker W_0=\{0\}$.

By homogeneity and compactness, there exists $c_0>0$ such that
\[
\|u\|\leq c_0|W_0(u)|
\]
whenever $q_r(u)\geq0$. Combining this with \eqref{eq:diff-Zt}, we obtain
\begin{equation}\label{eq:diff-Zt-2}
	\left|(Z_{t,a}-Z_{0,a})(v)\right|
	\leq\frac{c_0C}{a^2}|Z_{0,a}(v)|
\end{equation}
whenever $Q_a(v)\geq0$.

If $c_0C/a^2<1$, then $Q_a$ is negative definite on
$\ker Z_{t,a}$ for every $t\in[0,1]$. Indeed, if
$Z_{t,a}(v)=0$ and $Q_a(v)\geq0$, then \eqref{eq:diff-Zt-2} gives
\[
|Z_{0,a}(v)|\leq\frac{c_0C}{a^2}|Z_{0,a}(v)|,
\]
and hence $Z_{0,a}(v)=0$. Thus
$M_av\in\ker W_0$ and $q_r(M_av)\geq0$, which forces $v=0$.

Now, we consider the Bayer property along this path. Let $
m_a=\lfloor\sqrt{2a}\rfloor,\alpha_a=2a / m_a^2$ and set
\[
d_{a,l}\coloneqq
\mathrm{dist}\bigl(
\cP_{\PP^r}^{a,0},
\cP_{\PP^r}^{a,0}\otimes\cO_{\PP^r}(-l)
\bigr)
\]
for $l\in\ZZ_{>0}$. Using~\eqref{eq:stabonPn}, \eqref{pim-pullback-sigma}, the tensor-shift
and $\Pic^0(E^r)$-invariance from Proposition~\ref{prop:sigmaabRonEn},
and the pushforward and pullback slicing-distance compatibilities, we
obtain
\[
\begin{aligned}
	d_{a,l}
	&\leq
	\mathrm{dist}\bigl(
	\cP_{E^r}^{2a,0},\cP_{E^r}^{2a,-2l}
	\bigr)\\
	&\leq
	\mathrm{dist}\left(
	\cP_{E^r}^{\alpha_a,0},
	\cP_{E^r}^{\alpha_a,-2l/m_a^2}
	\right)
	\longrightarrow0,
\end{aligned}
\]
where the limit follows from
Proposition~\ref{prop:sigmaabRonEn}, since $\alpha_a\to1$ and
$m_a\to\infty$ when $a\to \infty$.

Choose $a\gg0$ such that
\[
d_{a,k}<\frac12\quad\textup{and}\quad
\frac{c_0C}{a^2}<\sin\left(\frac{\pi}{8}\right).
\]
By~\cite[Theorem~1.2]{bayer:short-proof}, the path $Z_{t,a}$ lifts from
$\sigma_{\PP^r}^{a,0}$ to a path
$\sigma_t=(Z_{t,a},\cP_t)$. Equation~\eqref{eq:diff-Zt-2} and the standard evaluation argument using~\cite[Lemma~2.9]{bayer:short-proof} give
\[
\mathrm{dist}(\cP_t,\cP_0)<\frac18
\]
for all $t\in[0,1]$. Consequently, one has
\[
\mathrm{dist}\bigl(
\cP_t,\cP_t\otimes\cO_{\PP^r}(-k)
\bigr)
\leq2\mathrm{dist}(\cP_t,\cP_0)+d_{a,k}
<\frac34<1
\]
for any $t\in[0,1]$. Thus~\eqref{eq:Distance} yields $\cP_t\prec\cP_t\otimes\cO_{\PP^r}(-k)[1]$.

The stability condition
$\iota^\sharp\sigma_{\PP^s}^{a/k,b_\iota}$ has central charge
$Z_{1,a}$ and, by Theorem~\ref{thm:ProjectiveScheme}, satisfies the
same Bayer property. Both endpoint stability conditions assign phase $1$ to skyscraper sheaves: for $\sigma_1$, this follows from Proposition~\ref{prop:BayerGeometric}, the equality
$Z_{1,a}(\cO_x)=-1$, and the above distance bound. In particular,
Lemma~\ref{lem:unique-Bayer} yields
$\sigma_1=\iota^{\sharp}\sigma_{\PP^s}^{\frac{a}{k},b_{\iota}}$. This proves the claim and thus the theorem when $X=\PP^r$ and $\iota_1$ is an isomorphism.

It is not difficult now to deduce the general case from this. 
Indeed, since we can take a larger embedding of $X$ that factors through both $\iota_1$ and $\iota_2$ up to a $\Pic^0$-action which is harmless by \cite[Theorem~2.23(1)]{curves}, it suffices to show that if $f\colon X\hookrightarrow \PP^r$ and $\iota\colon \PP^r\hookrightarrow \PP^s$ are closed embeddings, then $(\iota\circ f)^{\sharp}\sigma_{\PP^s}^{a/k,b_{\iota}}$ and $f^{\sharp}\sigma_{\PP^r}^{a,0}$ lie in the same connected component. 

To this end, let $\{\sigma_t=(Z_t, \cP_t)\}_{t\in [0,1]}$ be the path in $\Stab_{\num}(\Db(\PP^r))$ constructed above. By taking $l$ and $a$ sufficiently large, the same argument as in Theorem \ref{thm:ProjectiveScheme} shows that $f^{\sharp}\sigma_t$ is a stability condition on $\Db(X)$ for any $t\in [0,1]$. Therefore, $\{f^{\sharp}\sigma_t\}_{t\in [0,1]}$ gives a continuous path in $\Stab_{(\Lambda_{H_X}, v_{H_X})}(\Db(X))$ by Lemma~\ref{lemma-pullback-slicing-compatibility}\ref{pullback-slicing-compatibility-dist}. This completes the proof.
\end{proof}


\section{Complements}\label{sec:Complements}

In this section, we discuss some complements to the above results. 
First, in Section~\ref{subsec:MassHomBounds} we recall the notion of a mass-Hom bound from~\cite{DHL-robotis}, and show that it is satisfied by stability conditions in the distinguished component $\Stab^\dagger_{H_X}(\Db(X))$ of a polarized scheme $(X, H_X)$. 
Next, in Section~\ref{subsec:Threefolds}, we compare the connected component $\Stab^\dagger_{H_X}(\Db(X))$ with known constructions of stability conditions in the case of complex surfaces and threefolds.
Finally, in Section~\ref{subsec:LocalFano}, we use the stability conditions on $\Db(X)$ from Theorem~\ref{thm:ProjectiveScheme} to induce stability conditions on the supported derived category of the total space of certain vector bundles, including all local Calabi--Yau varieties.

We continue to work over a base field $\KK$; in Section~\ref{subsec:Threefolds}, we restrict to the case $\KK=\CC$.

\subsection{Mass-Hom bounds}\label{subsec:MassHomBounds} 

In this section, we discuss the notion of mass-Hom bound, which was introduced in \cite[Definition 2.7]{DHL-robotis}.
For simplicity, as in the rest of the paper, we work with derived categories of schemes, although the basic definitions make sense for more general triangulated categories.
This will be related (at least when $X$ is smooth and projective, and $\KK$ has characteristic~$0$) to the existence of proper moduli spaces for stability conditions in the connected component $\Stabdagger_{H_X}(\Db(X))$, which will be the subject of Part~\ref{part:Relative} of this paper (see Theorem~\ref{theorem-DHL-robotis}).

\begin{definition}\label{definition-mass-hom}
Let $X$ be a proper scheme over a field $\KK$. 
A pre-stability condition $\sigma=(Z,\cP)$ on $\Db(X)$ has a \emph{mass-Hom bound} if for every $A\in \Dperf(X)$, there exists a constant $C_A > 0$, which depends on $A$ and $\sigma$, such that for all $F\in\Db(X)$ we have 
\begin{equation}\label{mass-hom-inequality}
\dim_\KK\Hom(A,F)\leq C_A m_{\sigma}(F). 
\end{equation}
\end{definition}

In the above definition, recall that the mass was defined in~\eqref{eq:DefMass}.


\begin{remark}\label{remark-mass-hom-criterion}
According to \cite[Lemma 2.8]{DHL-robotis}, the condition of having a mass-Hom bound is equivalent to the following a priori weaker condition. 
Let $\cG$ be a set of objects which classically generate $\Dperf(X)$, or in other words, such that $\Dperf(X)$ is the smallest idempotent complete triangulated subcategory which contains $\cG$. 
Let $\sigma$ be a stability condition on $\Db(X)$ satisfying the following property: for every $G \in \cG$, there exists a constant $C_G>0$ such that the mass-Hom inequality~\eqref{mass-hom-inequality} holds for all $\sigma$-stable objects $F\in\Db(X)$.  
Then $\sigma$ has a mass-Hom bound.
\end{remark}

\begin{remark}\label{remark-mass-hom-connected}
The existence of a mass-Hom bound only depends on the connected component of $\Stab_{(\Lambda, v)}(\Db(X))$ in which a stability condition lies.
In fact, let $\sigma,\sigma'\in\Stab_{(\Lambda, v)}(\Db(X))$ be two stability conditions in the same connected component. 
Then by~\cite[Corollary~2.14]{DHL-robotis}, $\sigma$ has a mass-Hom bound if and only if $\sigma'$ does: this follows easily from the fact that $\sigma,\sigma'$ being in the same connected component, there exists a constant $C_{\sigma,\sigma'}>0$ such that $m_{\sigma}(F)\leq C_{\sigma,\sigma'} m_{\sigma'}(F)$, for all $F\in\Db(X)$.
\end{remark}

The existence of a mass-Hom bound is preserved under pullback and pushforward of stability conditions:

\begin{lemma}\label{lemma-pullback-mass-hom-bound}
Let $f \colon X \to Y$ be a morphism between proper noetherian schemes over $\KK$.
\begin{enumerate}[{\rm (1)}]
\item\label{enum:MassPullback} Assume that $f$ is finite. Let $\sigma\in\Stab_{(\Lambda,v)}(\Db(Y))$ be a stability condition with a mass-Hom bound. If $f^\sharp\sigma$ is a stability condition, then it has a mass-Hom bound. 

\item\label{enum:MassPushforward} Assume that $f$ is faithfully flat and that the relative dualizing complex $\omega_f^\bullet$ belongs to $\Dperf(X)$. Let $\sigma\in\Stab_{(\Lambda,v)}(\Db(X))$ be a stability condition with a mass-Hom bound. If $f_\sharp\sigma$ is a stability condition, then it has a mass-Hom bound.
\end{enumerate}
\end{lemma}

\begin{proof}
Let us start with~\ref{enum:MassPullback}.
Recall that by the well-known Neeman--Ravenel criterion~\citestacks{09SR}, the condition that a perfect complex is a classical generator for the category of perfect complexes is equivalent to it being a generator of the category of quasi-coherent complexes.
Let $G$ be a classical generator of $\Dperf(Y)$, which exists for instance by~\cite[Proposition 2.5]{neeman-Grothendieck-duality}. 
Then its pullback $f^*G \in \Dperf(X)$ is a classical generator. 
Indeed, to see this, we must show that if $F \in \Dqc(X)$ is an object such that $\RHom(f^*G, F) = 0$, then $F = 0$. 
By adjunction, we find $\RHom(G, f_*F) = 0$, which implies that $f_*F = 0$ by the same criterion mentioned above applied on $Y$; but then $F = 0$ because $f$ is finite, and thus the functor $f_* \colon \Dqc(X) \to \Dqc(Y)$ is conservative.

Now, for a $f^\sharp\sigma$-stable object $F\in\Db(X)$, we have the inequalities
\[
\dim_\KK \Hom(f^*G, F) = \dim_\KK \Hom(G, f_*F) \leq C_G m_{\sigma}(f_*F) = C_G m_{f^\sharp\sigma}(F).
\]
Here the first equality is by adjunction, the inequality comes from the fact that $\sigma$ has a mass-Hom bound, and the last equality follows since $f_*F$ is $\sigma$-semistable with $Z(v(f_*F))=Z((f^\sharp v)(F))$.

The proof of~\ref{enum:MassPushforward} is analogous, by replacing $f_*$ with $f^*$: the only assumptions we need are~\ref{enum:Polishchuk1}--\ref{enum:Polishchuk3} in Theorem~\ref{thm:Polishchuk}, which are satisfied in our case, as remarked in Example~\ref{ex:PullPushShriek}\ref{enum:PullPushShriek2}.
\end{proof}

The main result is then the following result, which is a restatement of Theorem~\ref{thm:MassHomBoundMain} from the introduction. 

\begin{theorem}\label{thm:MassHomBoundProjective}
Let $X$ be a projective scheme over a field $\KK$ equipped with an ample numerical class $H_X$.  
Then any stability condition $\sigma \in \Stabdagger_{H_X}(\Db(X))$ admits a mass-Hom bound.
\end{theorem}

\begin{proof}
By Remark~\ref{remark-mass-hom-connected}, it is enough to produce a stability condition $\sigma\in\Stabdagger_{H_X}(\Db(X))$ with a mass-Hom bound. 
This can be achieved on $(\PP^1)^n$ by Theorem~\ref{thm:AlgebraicP1n} together with~\cite[Example 2.9]{DHL-robotis}. 
By Theorem~\ref{theorem-chunyi-Pn} and Lemma~\ref{lemma-pullback-mass-hom-bound}\ref{enum:MassPushforward}, all stability conditions in $\Stabdagger(\Db(\PP^n))$ have a mass-Hom bound.
Finally, we conclude by Theorem~\ref{thm:ProjectiveScheme} and Lemma~\ref{lemma-pullback-mass-hom-bound}\ref{enum:MassPullback}.
\end{proof}

\subsection{Tilt-stability}
\label{subsec:Threefolds}

In this section we specialize our base field $\KK=\CC$ to be the complex numbers and focus on surfaces and threefolds. 
We can apply Theorem~\ref{thm:MassHomBoundProjective} to these examples to show that all stability conditions constructed by tilting (see~\cite{BMT:BG,BMS:StabCY3s}) are in the distinguished connected component, and thus have a mass-Hom bound.

First of all, we deal with the easy case of surfaces.
Let $X$ be a smooth projective surface over $\CC$.
For a geometric stability condition $\sigma$ in $\Stabdagger_{H_X}(\Db(X))$, we can use~\cite[Theorem 5.10]{hannah:stabonabquot}\footnote{In fact, to be precise, this follows from the proof in~\emph{loc.~cit.}, applied to geometric stability conditions with respect to $(\Lambda_{H_X},v_{H_X})$.} to deduce that, up to the action of $\widetilde{\GL}_2^+(\RR)$, $\sigma$ is as in~\cite[Proposition 5.15]{hannah:stabonabquot}, and it satisfies the full support property with respect to $\Knum(X)$ (see~\cite[Proposition 5.35]{hannah:stabonabquot}).
Since geometric stability conditions with respect to $\Knum(X)$ are connected (see~\cite[Theorem 5.36]{hannah:stabonabquot}), we can then define $\Stab^\dagger_{\mathrm{num}}(\Db(X))$ as the connected component containing all geometric stability conditions with respect to $\Knum(X)$.
By Theorem~\ref{thm:MassHomBoundProjective} (and Remark~\ref{remark-mass-hom-connected}), we immediately deduce the following result:

\begin{proposition}\label{prop:MassHomBoundSurfaces}
Let $X$ be a smooth projective surface over $\CC$.
Then the stability conditions in $\Stab^\dagger_{\mathrm{num}}(\Db(X))$ have a mass-Hom bound; in particular, this applies to all geometric stability conditions on $\Db(X)$ with respect to $\Knum(X)$.
\end{proposition}

We now pass to the threefold case.
We keep the notation as in Sections~\ref{subsec:ProjectiveSchemes} and~\ref{subsec:ConnectedComponent}.
Let $X$ be a smooth projective threefold over $\CC$, and let $H_X$ denote the numerical class of an ample divisor.
We fix a closed embedding $\iota\colon X\hookrightarrow\PP^n$ such that $c_1(\iota^*\cO_{\PP^n}(1))=mH_X$, for some $m\in\ZZ_{>0}$.
By Theorem~\ref{thm:canonical-component} and Remark~\ref{remark-mass-hom-connected}, we can assume $m=1$.
We assume the following:
\begin{enumerate}[{\rm (i)}]
\item\label{enum:MassHomBoundProjective1} the pair $(\Lambda_{H_X},v_{H_X})$ is isomorphic to $(\Lambda_{H_X}',v_{H_X}')$ as quotients of $\Knum(X)$, where
\[
v_{H_X}'=\left(H_X^3\rk,H_X^2\ch_1,H_X\ch_2,\ch_3\right)
\]
and $\Lambda_{H_X}\coloneqq \mathrm{im}(v_{H_X}')\subset \QQ^4$;

\item\label{enum:MassHomBoundProjective2} $(X,H_X)$ satisfies~\cite[Conjecture 4.1]{BMS:StabCY3s}.
\end{enumerate}

In view of Example~\ref{ex:SmoothAbelianGroup}, Assumption~\ref{enum:MassHomBoundProjective1} means that the Todd classes $\td_{X,1}$, resp.~$\td_{X,2}$, are multiples of $H_X$, resp.~$H^2_X$.
In particular, this holds if $X$ has Picard rank~1.
Since the morphisms $v_{H_X}$ and $v_{H_X}'$ differ by multiplication by the Todd classes of $X$ and $\PP^n$, the space of stability conditions $\Stab_{(\Lambda_{H_X},v_{H_X})}(\Db(X))$ and $\Stab_{(\Lambda_{H_X},v_{H_X}')}(\Db(X))$ are the same.

By~\cite[Theorem 8.2]{BMS:StabCY3s}, there exists a connected open subset of geometric stability conditions
\[
\widetilde{\mathfrak P}\subset\Stab_{(\Lambda_{H_X},v_{H_X})}(\Db(X)),
\]
whose central charge is explicitly described in~\cite[Definition 8.1]{BMS:StabCY3s}.

\begin{proposition}\label{prop:BMTmassbound}
In the above notation, assume that~\ref{enum:MassHomBoundProjective1} and~\ref{enum:MassHomBoundProjective2} hold. 
Then
\[
\widetilde{\mathfrak P}\subset\Stabdagger_{H_X}(\Db(X)).
\]
In particular, any stability condition in $\widetilde{\mathfrak P}$ has a mass-Hom bound.
\end{proposition}

\begin{proof}
For the first part of the statement, since $\widetilde{\mathfrak P}$ is a connected open subset, it is enough to show that for $a\gg 0$, the pullback stability condition $\widetilde{\sigma}_X^{a,0}=\iota^\sharp\sigma^{a,0}_{\PP^n}$ belongs to $\widetilde{\mathfrak P}$.
We first prove that the central charge $\widetilde{Z}_X^{a,0}$ of $\widetilde{\sigma}_X^{a,0}$ is the same as the central charge of a stability condition in $\widetilde{\mathfrak P}$.

Recall, by~\eqref{eq:ZstabonPn}, that
\[
Z^{a,0}_{\PP^n}(F)=-\ch_n(F)+ia\ch_{n-1}(F)+\dots+(-1)^{n+1}\frac{(ia)^n}{n!}\rk(F),
\]
where as usual we abuse notation and identify $\ch_j$ with $H_{\PP^n}^{n-j}.\ch_j$.
Consider the polynomials
\begin{align*}
&\Re \widetilde{Z}^{a,0}_X(\cO_X(t))=\Re Z^{a,0}_{\PP^n}(\iota_*\cO_X(t))=-\ch_n(\iota_*\cO_X(t))+\frac{a^2}{2}\ch_{n-2}(\iota_*\cO_X(t)),\\
&\Im \widetilde{Z}^{a,0}_X(\cO_X(t))=\Im Z^{a,0}_{\PP^n}(\iota_*\cO_X(t))=a\Bigl(\ch_{n-1}(\iota_*\cO_X(t))-\frac{a^2}{6}H^3\Bigr),
\end{align*}
with variable $t\in\mathbb R$.
For $a\gg 0$, the equation
$\Re\widetilde{Z}^{a,0}_X(\cO_X(t))=0$ has three real roots, asymptotic to
$\pm\sqrt{3}\,a+O(1)$ and $O(1)$, respectively. 
The equation
$\Im \widetilde{Z}^{a,0}_X(\cO_X(t))=0$ has two real roots, asymptotic to
$\pm\sqrt{1/3}\,a+O(1)$.
In particular, the roots of the two polynomials interlace. 
We deduce, by~\cite[Section 7]{Li:sb}, that for $a\gg0$, the central charge $\widetilde{Z}_X^{a,0}$ is the central charge of a stability condition $\sigma^a\in\widetilde{\mathfrak P}$, as we wanted.

Since all stability conditions in $\widetilde{\mathfrak P}$ are geometric, we can use~\cite[Lemma 8.3]{BMS:StabCY3s} and choose $\sigma^a$ in a unique way such that $\phi_{\sigma^a}(\cO_p)=1$, for all closed points $p\in X$.
We claim that $\sigma^a=\widetilde{\sigma}_X^{a,0}$, for $a\gg0$.
Indeed, one can see
\[ \sigma^a\prec\sigma^a\otimes\cO_X(-1)[1]. \]
for $a\gg0$ from \cite[Proposition 8.5(4)]{Li:sb}. On the other hand, \eqref{eq:BayerPropertyX2} with $m=1$ gives \[ \widetilde{\sigma}_X^{a,0} \prec \widetilde{\sigma}_X^{a,0}\otimes\cO_X(-1)[1]. \]
so Lemma~\ref{lem:unique-Bayer} implies $\sigma^a=\widetilde{\sigma}_X^{a,0}$. 

The second part of the statement follows from Theorem~\ref{thm:MassHomBoundProjective}.
\end{proof}

\subsection{Supported derived categories of total spaces of vector bundles}
\label{subsec:LocalFano}

As an application of the construction in Section~\ref{subsec:ProjectiveSchemes}, we show that stability conditions also exist in the supported case.
Recall first that if $Y$ is a noetherian scheme and $X\subset Y$ is a closed subset, then $\mathrm{D}_X^{\mathrm{b}}(Y)$ is the full triangulated subcategory of $\Db(Y)$ consisting of those objects with (set-theoretic) support on $X$.
For us, $Y$ will be the total space $\mathrm{Tot}(\cE)$ of a vector bundle $\cE$ on a projective scheme $X$ over $\KK$, and the closed embedding
\[
\xymatrix{
X\ar@{^{(}->}[r]^-{s}\ar[dr]_-{\mathrm{id}}&\mathrm{Tot}(\cE)\ar[d]^-{p}\\
&X
}
\]
is given by the $0$-section.
In particular, we have the natural identification:
\begin{equation*}
s_*\colon \rK_0(X)\xlongrightarrow{\cong}\rK_0\left(\mathrm{D}_X^{\mathrm{b}}(\mathrm{Tot}(\cE))\right).
\end{equation*}

Let $H_X$ be the numerical first Chern class of a very ample line bundle $\cO_X(1)$ on $X$.
We consider $(\Lambda_{H_X},v_{H_X})$ as in Section~\ref{subsec:ProjectiveSchemes}.
We then have the following result.

\begin{theorem}\label{thm:stabonlocal}
Let $X$ be a projective scheme over $\KK$ and $\cE$ a vector bundle on $X$. 
If $\cE^{\vee}$ is globally generated, then
\[
\Stab_{(\Lambda_{H_X},v_{H_X}\circ s_*^{-1})}\left(\mathrm{D}_X^{\mathrm{b}}(\mathrm{Tot}(\cE))\right)\neq\emptyset.
\]
\end{theorem}

\begin{proof}
We want to apply Proposition~\ref{prop:extendstabcond} to the functor $s_*\colon\Db(X)\to \mathrm{D}_X^{\mathrm{b}}(\operatorname{Tot}(\cE))$.
Since condition~\ref{enum:extendslicing1} of Proposition~\ref{prop:extendslicing} is automatic in our setting, we only need to check~\ref{enum:extendslicing2}.

We consider the stability conditions $\widetilde{\sigma}^{a,b}_X$ in Theorem~\ref{thm:ProjectiveScheme} for $a\gg 0$.
We claim that, for any globally generated locally free sheaf $\cF$ on $X$, we have
\begin{equation}\label{eq:stabonlocal1}
\widetilde{\cP}^{a,b}_X(\phi)\otimes\cF\subset\widetilde{\cP}^{a,b}_X(\geq\phi).    
\end{equation}
Indeed, we may assume that $\phi\in(0,1]$.
Let us consider the embedding $\iota\colon X\hookrightarrow\PP^n$ associated to the complete linear system of $\cO_X(1)$.
As remarked in the proof of~\cite[Theorem 6.5]{chunyi-stability}, since $\sigma^{a,b}_{\PP^n}$ is geometric, \cite[Lemma 10.1]{Bridgeland:K3} implies that
\[
\Coh(\PP^n)\subset\cP^{a,b}_{\PP^n}((1-n, 1])\quad\text{ and }\quad \cP^{a,b}_{\PP^n}((0,1])\subset\langle\Coh(\PP^n),\dots,\Coh(\PP^n)[n-1]\rangle,
\]
where, as before, $\langle-\rangle$ denotes the extension-closure.
Hence, since $\widetilde{\sigma}^{a,b}_X=\iota^\sharp\sigma^{a,b}_{\PP^n}$ and $\iota_*$ is conservative, we deduce that
\[
\Coh(X)\subset\widetilde{\cP}^{a,b}_X((1-n, 1])\quad\text{ and }\quad \widetilde{\cP}^{a,b}_{X}((0,1])\subset\langle\Coh(X),\dots,\Coh(X)[n-1]\rangle.
\]
In particular, for all locally free sheaves $\cG$ on $X$, we have
\[
\widetilde{\cP}^{a,b}_X((0,1])\otimes\cG[n]\subset\widetilde{\cP}^{a,b}_X(>1).
\]
Since $\cF$ is globally generated, we can consider a truncated resolution of the form
\[
0\longrightarrow \cG\longrightarrow\cO_X(-b_n)^{\oplus a_n}\longrightarrow\dots\longrightarrow\cO_X(-b_2)^{\oplus a_2}\longrightarrow\cO_X^{\oplus a_1}\longrightarrow\cF\longrightarrow 0,
\]
with $a_1,\dots,a_n,b_2,\dots,b_n\geq1$, and $\cG$ locally free on $X$.
By arguing as in the proof of Theorem~\ref{thm:ProjectiveScheme}, if we choose $a\gg 0$, for all $\widetilde{\sigma}^{a,b}_X$-semistable object $F\in\widetilde{\cP}^{a,b}_X((0,1])$, we get
\[
\phi_{\widetilde{\sigma}^{a,b}_X}^-(F\otimes\cF)\geq\min\left\{\underbrace{\phi^{\phantom{-}}_{\widetilde{\sigma}^{a,b}_X}(F)}_{\leq 1},\underbrace{\phi^-_{\widetilde{\sigma}^{a,b}_X}(F\otimes\cG[n])}_{>1}\right\}\geq\phi^{\phantom{-}}_{\widetilde{\sigma}^{a,b}_X}(F),
\]
and thus~\eqref{eq:stabonlocal1} follows.

Now, we notice that 
\[
s^*s_*F=s^*(p^*F\otimes s_*\cO_X)=F\otimes s^*s_*\cO_X=F\otimes\left(\bigoplus_{i=0}^{\rk(\cE)} \bigwedge^i\cE^\vee[i]\right).
\]
Since each $\bigwedge^i\cE^\vee$ is globally generated (and there are finitely many of them), we can find $a\gg0$ which works for all of them, and the above claim gives 
\[
\Hom(s_*F,s_*F')=\Hom(s^*s_*F, F')=\Hom(F,F')
\]
for all $F\in\widetilde{\cP}_X^{a,b}(\phi)$, $F'\in \widetilde{\cP}_X^{a,b}\left(<\phi+1\right)$, and $\phi\in\RR$. 
This completes the proof.
\end{proof}

Analogously, we have the following variant, which includes all local Calabi--Yau varieties.

\begin{theorem}\label{thm:local-cy}
Let $X$ be a projective scheme over $\KK$ and $\cL$ be a line bundle on $X$. 
If $\cL^{\vee}$ is ample, then
\[
\Stab_{(\Lambda_{H_X},v_{H_X}\circ s_*^{-1})}\left(\mathrm{D}_X^{\mathrm{b}}(\mathrm{Tot}(\cL))\right)\neq\emptyset
\]
where $H_X\coloneqq c_1(\cL^{\vee})$.
\end{theorem}

\begin{proof}
As in Theorem~\ref{thm:stabonlocal}, it suffices to find a stability condition $\sigma\in \Stab_{(\Lambda_{H_X},v_{H_X})}(\Db(X))$ such that $\sigma \preceq \sigma\otimes \cL[1]$. To this end, we take $m\gg 0$ such that $(\cL^{\vee})^{\otimes m}$ induces a closed embedding $\iota\colon X\hookrightarrow \PP^n$. Consider $\widetilde{\sigma}_X^{a,b}$ constructed in \eqref{eq:Paris20260702}. Then by Lemma \ref{lem:unique-general}, we have $\widetilde{\sigma}_X^{a,b}\otimes \cL=\widetilde{\sigma}_X^{a,b-\frac{1}{m}}$. As in the proof of Lemma \ref{lem:BayerPropertyEllipticCurves2} (see also Proposition \ref{prop:sigmaabRonEn}), we can further take $a\gg 0$ such that
\[
\mathrm{dist}(\widetilde{\sigma}_X^{a,b},\widetilde{\sigma}_X^{a,b}\otimes \cL)<1,
\]
and the result follows from Remark \ref{eq:Distance}.
\end{proof}

\begin{remark}
\label{remark-inducing-TotL}
An interesting case where the nonemptiness of the space of stability conditions has been studied is when $X$ is a Fano manifold and we consider the total space $\mathrm{Tot}(\omega_X)$; for instance $\mathrm{Tot}(\omega_{\PP^2})$, the \emph{local $\PP^2$} (see~\cite{BMLocalP2}).
Actually, in the situation when $\cE=\cL$ is an invertible sheaf on $X$, as a byproduct of the proof of Theorem~\ref{thm:stabonlocal}, we can write the condition we need in terms of the Bayer property.
Namely, given a slicing $\cP$ of $\Db(X)$ such that $\cP\preceq\cP\otimes\cL[1]$, we have an induced slicing on $\mathrm{D}^\mathrm{b}_X(\mathrm{Tot}(\cL))$; similarly, for stability conditions. We are informed that a similar result for stability conditions on local Fano manifolds has been obtained via global dimension arguments in \cite{WuZhang:LocalFano}.
\end{remark} 


\newpage 

\part{Relative case}\label{part:Relative}


The goal for the second part of the paper is to handle the general situation, where we consider projective families over a base scheme, as well as moduli spaces of semistable objects. 

Section~\ref{sec:StabilityFamily} is broadly devoted to setting the stage to work with families of stability conditions and moduli space of semistable objects. 
In Section~\ref{sec:PullPushRelativeStability}, we discuss the key generalization to the relative setting of the results from Section~\ref{sec:PullbackPushforward} on pullback and pushforward of stability conditions.
In Section~\ref{sec:ProjectiveFamilies}, we prove our main theorem, Theorem~\ref{main-theorem}.
Finally, in Section~\ref{sec:BondalOrlov}, we study in more detail the stability of objects numerically equivalent to skyscraper sheaves and their moduli, and as an application give a new proof of the Bondal--Orlov reconstruction theorem.


\section{Stability conditions in families}\label{sec:StabilityFamily}

The goal of this section is to present the definition of stability condition over a base, following~\cite{stability-families}, with a few technical improvements which we will need in the following sections. 
When considering relative stability conditions, we will work in the following setup:  

\begin{setup}
\label{assum-stab}
    Let $\pi \colon X\to S$ be a flat projective morphism, where $S$ is a noetherian Nagata scheme of finite Krull dimension which is quasi-projective over a noetherian affine scheme. 
\end{setup}

In Section~\ref{subsec:PerfectStacks} we introduce universally gluable, relatively perfect objects, and use them to define the moduli functor, and in Section~\ref{subsec:RelativeKnum} we recall the technical notion of the relative numerical Grothendieck group. 
Sections~\ref{subsec:HNstructure} and~\ref{subsec:StabilityBase} are the key parts, where we give the definition of HN structure and use this to define stability conditions over an arbitrary base.
Finally, in Section~\ref{subsection-moduli-spaces}, we explain that for a smooth projective scheme $X$ over a field $\KK$, a stability condition on $\Db(X)$ with a mass-Hom bound is necessarily a stability condition over~$\KK$. 
In view of our earlier results, this yields a proof of Theorem~\ref{main-theorem} in the case where $X$ is smooth and projective over a field; the general case will be treated with a different argument in Section~\ref{sec:ProjectiveFamilies}. 

\subsection{Relative perfect objects and moduli stacks}\label{subsec:PerfectStacks}

In this section we briefly review the moduli functor for objects in the bounded derived category of coherent sheaves.
Relative perfect complexes are the first ingredient in the definition of the moduli functor, since they generalize the notion of a flat family of coherent sheaves over a base to bounded complexes.
We recall here the definition (see~\citestacks{0DI0}) and a few of their basic properties.

\begin{definition}\label{def:RelativePerfect}
Let $g\colon X\to S$ be a morphism which is flat and locally of finite presentation. An object $F\in\Dqc(X)$ is \emph{perfect relative to $S$}, or \emph{$S$-perfect}, if $F$ is pseudo-coherent and locally of finite Tor-dimension over $g^{-1}\cO_S$.
\end{definition}

We need the following flattening stratification of relative perfect objects.

\begin{lemma}\label{lem:flatten}
Let $X\to S$ be a flat projective morphism between noetherian schemes and $F\in \Dqc(X)$ be an $S$-perfect object. 
\begin{enumerate}[{\rm (1)}]
\item\label{enum:flatten1} There exists an interval $[a,b]\subset\RR$ so that $\cH^i(F)=\cH^i(F_s)=0$ for any $i\notin [a,b]$ and any $s\in S$.
\item\label{enum:flatten2} We can find a finite set of locally closed subschemes $\{S_j\}_{j\in J}$ of $S$ so that
\begin{itemize}
\item $\cH^i(F_{S_j})=0$ for $i\notin [a,b]$,
\item $\cH^i(F_{S_j})$ is flat over $S_j$ for $i\in [a,b]$, and
\item $S=\bigcup_{j\in J}S_j$ as a set.
\end{itemize}
\end{enumerate}
\end{lemma}

\begin{proof}
By covering $S$ by finitely many affine open subsets, we may assume that $S$ is affine. 
Then by~\cite[Corollary 2.1.7]{lieblich:moduli-of-complex}, we have $F\cong P^{\bullet}$, where $P^{\bullet}$ is a bounded complex of coherent sheaves on $X$ such that each $P^i$ is flat over $S$ and $P^i=0$ for $i\notin [a,b]$, and part~\ref{enum:flatten1} follows.

For part~\ref{enum:flatten2}, after base change to its reduction, we may assume that $S$ is reduced. 
Applying~\citestacks{052B} to $\bigoplus_{i\in [a,b]} \cH^i(F)$, we can find a nonempty open subset $S_1$ of $S$ so that $(\cH^i(F))_{S_1}=\cH^i(F_{S_1})$ is flat over $S_1$ for each $i\in [a,b]$. 
Note that since $F_{S\setminus S_1}\cong P^{\bullet}_{S\setminus S_1}$, we also have $\cH^i(F_{S\setminus S_1})=0$ for $i\notin [a,b]$. 
Using~\citestacks{052B} again, we find an open subset $S_2$ of $S\setminus S_1$ so that $\cH^i(F_{S_2})$ is flat over $S_2$ for each $i\in [a,b]$. 
We continue this process: for each $k\geq 1$, we get an open subset $S_k\subset S\setminus \bigcup_{t=1}^{k-1} S_t$ so that
\[
\cH^i(F_{S_k})=0 \text{ for } i\notin [a,b] \quad \text{ and } \quad \cH^i(F_{S_k}) \text{ is flat over } S_k \text{ for } i\in [a,b].
\]
Then we obtain a sequence of closed subschemes
\[
\dots \subset S\setminus \bigcup_{t=1}^{k-1} S_t\subset S\setminus \bigcup_{t=1}^{k-2} S_t\subset \dots \subset S\setminus S_1\subset S.
\]
Note that by~\citestacks{052B}, if $S\setminus\bigcup_{t=1}^{k-1} S_t$ is nonempty, then $S_k$ is also nonempty. 
So by the noetherian property of $S$, the above sequence stabilizes at the empty set, i.e., $S=\bigcup_{j\in J} S_j$ for a finite index set $J$. 
This proves~\ref{enum:flatten2}.
\end{proof}

The next step is universally gluable objects.

\begin{definition}
Let $X\to S$ be a flat, proper, finitely presented morphism of schemes. 
An $S$-perfect object $F\in\Dqc(X)$ is \emph{universally gluable} if, for every $s\in S$, we have $\Ext^i(F_s,F_s)=0$ for all $i<0$.
\end{definition}

We denote by $\Dpug(X/S)\subset \Dqc(X)$ the full subcategory of universally gluable $S$-perfect objects.
We can now define the moduli functor.
Let $\mathrm{Sch}/S$ denote the category of all $S$-schemes and $\mathrm{Gpds}$ the (2-)category of groupoids.

\begin{definition}\label{def:Mpug}
We let
\[
\cM_{\mathrm{pug}}(X/S)\colon (\mathrm{Sch}/S)^{\mathrm{op}}\to \mathrm{Gpds}
\]
be the functor whose value on $T\in(\mathrm{Sch}/S)$ consists of all $E\in \Dpug(X_T/T)$.
\end{definition}

According to \cite{lieblich:moduli-of-complex}, $\cM_{\mathrm{pug}}(X/S)$ is an algebraic stack locally of finite presentation over $S$.
For a stability condition over $S$, the subfunctor parameterizing semistable objects will be open and bounded. To conclude this section, we recall the notion of boundedness.

\begin{definition}
A subfunctor $\cM\subset\cM_{\mathrm{pug}}(X/S)$ is \emph{bounded} if there is a pair $(B,\cF)$ where $B$ is a scheme of finite type over $S$ and $\cF\in \cM(B)$ is an object such that for every geometric point $\bar{s}$ over $S$ and $F\in \cM(\kappa(\bar{s}))$, there exists a $\kappa(\bar{s})$-rational point $b$ of $B\times_S \Spec(\kappa(\bar{s}))$ such that $\cF_b\cong F$. 
\end{definition}

We will say that $(B,\cF)$ \emph{witnesses the boundedness of $\cM$.}
If $\cM\subset \cM_{\mathrm{pug}}(X/S)$ is an open substack, then it is easy to see that $\cM$ is bounded if and only if there is a surjective morphism $B\to \cM$, where $B$ is a scheme of finite type over $S$ (see, e.g.,~\cite[Lemma 9.5]{stability-families}).

\subsection{The relative numerical Grothendieck group}\label{subsec:RelativeKnum}

In this section, we define the numerical K-group in the relative setting (see~\cite[Definition 21.1]{stability-families}). 
This is slightly technical, but it is useful for comparing different fibers and will be employed throughout Part~\ref{part:Relative} of the paper. 
Recall from~\eqref{Knumell} that for any projective scheme over $\KK$ and a field extension $\KK \subset \ell$, we have an overlattice $\Knum(X) \subset \Knum(X)_{\ell}$; below, we use this in the case $\ell = \overline{\KK}$. 

\begin{definition}\label{def:RelativeGrothendieckGroup}
Let $\pi \colon X\to S$ be a morphism as in Setup~\ref{assum-stab}. 
We define the \emph{relative numerical K-group} $\Knum(X/S)$ as the quotient of $\bigoplus_{s\in S} \Knum(X_s)_{\overline{s}}$ by the saturation of the subgroup generated by the elements of the form
\begin{equation}\label{eq:RelativeGrothendieckGroup}
\eta^{\vee}_{t_1/u(t_1)}(F_{t_1})-\eta^{\vee}_{t_2/u(t_2)}(F_{t_2})
\end{equation}
for all tuples $(u,F,t_1,t_2)$ where $u\colon T\to S$ is a morphism from a connected scheme $T$, \mbox{$F\in\Dqc(X_T)$} is a $T$-perfect object, and $t_1,t_2\in T$.
\end{definition}

As observed in~\cite[Remark 21.2]{stability-families}, in Definition~\ref{def:RelativeGrothendieckGroup} it is enough to consider morphisms $u\colon T\to S$ of finite type, from a connected affine scheme $T$.
Also, the following homomorphism of groups will play an important role for us:
\begin{equation}\label{eq:KnumClassIndependent}
[-]_t\colon\rK_0(X_t)\longrightarrow\Knum(X_t)\longrightarrow\Knum(X_{u(t)})_t\longhookrightarrow\Knum(X_{u(t)})_{\overline{u(t)}}\longrightarrow\Knum(X/S),    
\end{equation}
for all $t\in T$.

\begin{remark}\label{rmk:KnumClassIndependent}
For $F\in\Dqc(X_T)$ as in the definition, the image of $[F_t]\in\rK_0(X_t)$ under the morphism~\eqref{eq:KnumClassIndependent} is independent of $t\in T$. 
In this case, we denote this image by $[F]\in\Knum(X/S)$.
\end{remark}

If $\Lambda$ denotes an abelian group, and $v\colon\Knum(X/S)\to\Lambda$ is a group homomorphism, we denote by $v_t$ the composition
\begin{equation}\label{eq:DefOfVt}
v_t\colon\rK_0(X_t)\xlongrightarrow{[-]_t}\Knum(X/S)\xlongrightarrow{v}\Lambda.
\end{equation}

\begin{example}\label{ex:UniformlyNumericalGrothendieck}
In this paper, we will sometimes pick $\Lambda$ to be a quotient of the~\emph{uniformly numerical Grothendieck group} $\cN(X/S)$ defined in~\cite[Proposition and Definition 21.5]{stability-families}.
More precisely, $\cN(X/S)$ is the quotient $\Knum(X/S)/\ker(\chi)$, where
\[
\chi\colon\rK_0(\Dperf(X))\times\Knum(X/S)\longrightarrow\ZZ,\qquad\chi([A],[B])\coloneqq\chi(A_t,B),
\]
for some point $t$ over $S$ and objects $B\in\Db(X_t)$ and $A\in\Dperf(X)$.
\end{example}

The functoriality of the relative numerical Grothendieck group is slightly delicate.
The following result addresses it in a special case.

\begin{proposition}\label{prop:FunctorialityKnum}
Let $X\to S$ and $Y\to S$ be two morphisms satisfying the hypotheses in Setup~\ref{assum-stab}.
Let $f\colon X\to Y$ be a proper morphism over $S$. 
\begin{enumerate}[{\rm (1)}]
\item\label{enum:FunctorialityKnum1} The pushforward $f_*$ induces a morphism
\[
f_*\colon\Knum(X/S)\longrightarrow\Knum(Y/S).
\]
\item\label{enum:FunctorialityKnum2} 
If $f$ has finite Tor dimension and the relative dualizing complex $\omega_f^\bullet$ is contained in $\Dperf(X)$, then the pullback $f^*$ induces a morphism
\[
f^*\colon\Knum(Y/S)\longrightarrow\Knum(X/S).
\]
\end{enumerate}
\end{proposition}

\begin{proof}
We prove~\ref{enum:FunctorialityKnum1}.
Let us first fix a point $s\in S$ and let $\KK$ denote its residue field.
Let $\ell$ be a field extension of $\KK$.
An explicit computation using the definition of $\eta^\vee_{\ell/\KK}$ in \cite[(12.2)]{stability-families} shows that $f_{s,*}$ induces a morphism
\[
f_{s,*}\colon\Knum(X_s)_\ell\longrightarrow\Knum(Y_s)_\ell.
\]
Indeed, under our assumption, $f_{s,*}$ (and its base change to $\ell$) preserves bounded coherent complexes, while $f^*_s$ (and its base change to $\ell$) preserves perfect complexes.
Hence, we have a morphism
\[
f_*\colon\bigoplus_{s\in S}\Knum(X_s)_{\overline{s}}\longrightarrow\bigoplus_{s\in S}\Knum(Y_s)_{\overline{s}}\longrightarrow\Knum(Y/S).
\]

We are left to check that $f_*$ factors through $\Knum(X/S)$; namely,  that any element as in \eqref{eq:RelativeGrothendieckGroup} is sent to $0$ by $f_*$.
Hence, let $u\colon T\to S$ be a morphism of finite type from a connected affine scheme $T$, $F\in\Dqc(X_T)$ be a $T$-perfect object, and $t_1,t_2\in T$.
First of all, $f_{T,*}F\in\Dqc(Y_T)$ is also $T$-perfect, by Lemma~\ref{lem:PushforwardTperfect} below.
Now, by the compatibility of $\eta^\vee$ with pushforward we have just observed, we have
\[
f_*\left(\eta^{\vee}_{t_1/u(t_1),X}(F_{t_1})-\eta^{\vee}_{t_2/u(t_2),X}(F_{t_2})\right)=\eta^{\vee}_{t_1/u(t_1),Y}((f_{T,*}F)_{t_1})-\eta^{\vee}_{t_2/u(t_2),Y}((f_{T,*}F)_{t_2})=0.
\]
This is what we need.

The argument for~\ref{enum:FunctorialityKnum2} is analogous by observing that, under our assumptions, $f_{s,!}$ (and its base change to $\ell$) preserves relatively perfect complexes (see Example~\ref{ex:PullPushShriek}\ref{enum:PullPushShriek2}), and the dualizing complex of the base change to $\ell$ is the pullback of the dualizing complex to $\ell$ (see, e.g.,~\cite[Theorem 2.7]{LN}).
\end{proof}

\begin{lemma}\label{lem:PushforwardTperfect}
Let $X\to S$ and $Y\to S$ be flat morphisms of finite presentation between noetherian schemes, and let $f\colon X\to Y$ be an $S$-morphism.
Let $F\in\Dqc(X)$.
Then, for all $s\in S$, we have
\[
(f_*F)_s=f_{s,*}(F_s).
\]
In particular, if $f$ is proper and $F$ is $S$-perfect, then $f_*F$ is also $S$-perfect.
\end{lemma}

\begin{proof}
We can assume that $S$ is affine.
By~\citestacks{0CTA}, for any $s \in S$, $Y_s$ and $X$ are Tor-independent over $Y$. 
Thus $(f_*F)_s=f_{s,*}(F_s)$, by~\citestacks{08IB}.
For the second part of the statement, if $F$ is $S$-perfect, then $F\in\Db(X)$.
Thus, $f_*F\in\Db(Y)$, since $f$ is proper. 
The conclusion then follows from the first part of the statement, and~\citestacks{0GEH}.
\end{proof}

\subsection{Harder--Narasimhan structures over curves}\label{subsec:HNstructure}

We fix a morphism $X\to S$ as in Setup~\ref{assum-stab}.
In this section, we specialize to the case when $S=C$ is a \emph{Dedekind scheme}, i.e., $C$ is furthermore an integral, regular, Noetherian scheme of dimension one. 
We write $p\in C$ for a closed point, $c\in C$ for an arbitrary point, and $K$ for the fraction field of $C$.

We have the following notions of torsion and torsion-free objects relative to $C$.

\begin{definition}
An object in $\Db(X)$ is called \emph{$C$-torsion} if it is the pushforward of an object in $\Db(X_W)$, for some proper closed subscheme $W\subset C$. 
\end{definition}

We denote by $\Db(X)_{\Ctor}$ the subcategory of $C$-torsion objects in $\Db(X)$.
By~\cite[Lemma 6.4]{stability-families}, an object $F\in\Db(X)$ is $C$-torsion if and only if $F_K=0$. 
Moreover, we have an exact triangle of triangulated categories:
\[
\Db(X)_{\Ctor}\longrightarrow\Db(X)\longrightarrow\Db(X_{K}).
\]

The theory of Harder--Narasimhan structures over curves was introduced in \cite[Section 13]{stability-families} in order to do semistable reduction and verify the valuative criterion of properness of moduli spaces.
We review the basic theory here in the setting of numerical stability conditions.
First of all, we consider the relative numerical Grothendieck group $\Knum(X/C)$ (introduced in Definition~\ref{def:RelativeGrothendieckGroup}).
Notice that under our assumption, all complexes in $\Db(X)$ are $C$-perfect.
Hence, for any $F\in\Db(X)$, we can consider its class $[F]\in\Knum(X/C)$ (as in Remark~\ref{rmk:KnumClassIndependent}).
Moreover, by~\eqref{eq:KnumClassIndependent}, we obtain the following two natural morphisms
\[
[-]_K\colon\rK_0(X_K)\longrightarrow\Knum(X/C),\qquad[-]_{\Ctor}\colon\rK_0(\Db(X)_{\Ctor})\longrightarrow\Knum(X/C).
\]
Remark~\ref{rmk:KnumClassIndependent} can then be rephrased as the following compatibility: for any $F\in\Db(X)$, we have
\[
[F]=[F_K]_K=\frac{1}{\mathrm{length}(W)}[\iota_{W*}\iota_{W}^*F]_{\Ctor},
\]
for all proper closed subschemes $W\subset C$ and $\iota_W\colon X_W\hookrightarrow X$.

\begin{definition}\label{def:HNstructure}
Let $\pi \colon X\to C$ be a morphism as in Setup~\ref{assum-stab}, where $C$ is a Dedekind scheme. 
Let $v \colon \Knum(X/C) \to \Lambda$ be a homomorphism to an abelian group $\Lambda$. 
A \emph{Harder--Narasimhan (HN) structure with respect to $(\Lambda,v)$} on $\Db(X)$ over $C$ is a pair $\sigma_C=(Z,\cP)$ where
\begin{itemize}
\item $Z\colon\Lambda\to\CC$ is a homomorphism, and
\item $\cP$ is a $C$-local slicing\footnote{See Definition~\ref{def:SlocalSlicing} in Appendix~\ref{app:Inducing}.} of $\Db(X)$
\end{itemize}
with the following property: for all $\phi\in\RR$ and $0\neq F\in\cP(\phi)$, we have
\begin{itemize}
\item either $F_K\neq0$ and $Z(v([F_K]_K))\in\RR_{>0}\cdot e^{i\pi\phi}$,
\item or $F\in\Db(X)_{\Ctor}$ and $Z(v([F]_{\Ctor}))\in\RR_{>0}\cdot e^{i\pi\phi}$.
\end{itemize}
\end{definition}

\begin{warning}
    The above terminology is slightly different from that in~\cite[Definition 13.3]{stability-families}. 
    More precisely, there the notion of a HN structure on $\Db(X)$ over $C$ is defined without reference to a pair $(\Lambda, v)$. 
    However, a HN structure with respect to a pair $(\Lambda, v)$ in our sense gives in particular a HN structure on $\Db(X)$ over $C$ in the sense of~\cite[Definition 13.3]{stability-families}. 
    All HN structures on $\Db(X)$ over $C$ that arise in practice are with respect to a pair $(\Lambda, v)$ (for instance, $(\Knum(X/C), \id)$), so this is just convenient language.
    Finally, a word on notation: in~\cite{stability-families}, the composition $Z\circ v\circ [-]_K$ was denoted by $Z_K$ and $Z\circ v\circ [-]_{\Ctor}$ by $Z_{\Ctor}$.
\end{warning}

Given a HN structure as above, we set $\cA_C\coloneqq\cP((0,1])$; it is the heart of a bounded $C$-local t-structure\footnote{See Definition~\ref{def:SlocalTstructure} in Appendix~\ref{app:Inducing}.} on $\Db(X)$.
Recall that, for the heart $\cA_C\subset \Db(X)$ of a $C$-local t-structure, we have an induced heart $\cA_c\subset \Db(X_c)$, for any $c\in C$ (see, e.g.,~\cite[Theorem 5.7]{stability-families}). 
The following result is an extension to pre-stability conditions; see~\cite[Lemma 13.11]{stability-families} and recall the notation $v_c$ from~\eqref{eq:DefOfVt}.

\begin{lemma}\label{lem-HN-induce}
A HN structure $\sigma_C$ with respect to $(\Lambda,v)$ on $\Db(X)$ over $C$ induces a pre-stability condition $\sigma_c=(Z,\cA_c)$ with respect to $(\Lambda,v_c)$ on $\Db(X_c)$, for every $c\in C$.
\end{lemma}

We say that $\sigma_c$ as above is \emph{induced by a HN structure} on $\Db(X)$, for all $c\in C$.

\subsection{Stability conditions over a base}\label{subsec:StabilityBase}

We use the previous notion of HN structure to define stability conditions over a base.
Let $\pi \colon X\to S$ be a morphism as in Setup~\ref{assum-stab}.
We fix a free abelian group of finite rank $\Lambda$ and a morphism $v\colon\Knum(X/S)\to\Lambda$.
By~\eqref{eq:DefOfVt}, we have induced morphisms $v_t\colon\rK_0(X_t)\to\Lambda$, for all points $t$ over $S$.
We will consider the following subfunctors of $\cM_{\mathrm{pug}}(X/S)$ (see Definition~\ref{def:Mpug}).

\begin{definition}
Let ${\sigma}=(\sigma_s=(Z,\cP_s))_{s\in S}$ be a collection of numerical stability conditions on $\Db(X_s)$ with respect to $(\Lambda,v_s)$, for each $s\in S$. 
Fix a vector $\bv\in\Lambda$ and $\phi\in\RR$ such that $Z(\bv)\in\RR_{>0}\cdot e^{i\pi\phi}$. 
\begin{enumerate}[{\rm(1)}]
\item We denote by
\[
\cM^{\mathrm{st}}_{{\sigma}}(\bv,\phi)\colon (\mathrm{Sch}/S)^{\mathrm{op}}\longrightarrow \mathrm{Gpds}
\]
the subfunctor of $\cM_{\mathrm{pug}}(X/S)$ whose value on $T\in (\mathrm{Sch}/S)$ consists of all $T$-perfect objects $F\in \Dpug(X_T/T)$ such that $F_t$ is geometrically $\sigma_t$-stable of phase $\phi$ and $v_t(F_t)=\bv$ in $\Lambda$ for all $t\in T$. 
Here, $\sigma_t$ denotes the base change of $\sigma_s$ given by Theorem~\ref{thm:base-change-stab}, where $s \in S$ is the image of $t$.
\item We denote by
\[
\cM_{{\sigma}}(\bv,\phi)\colon(\mathrm{Sch}/S)^{\mathrm{op}}\longrightarrow \mathrm{Gpds}
\]
the subfunctor of $\cM_{\mathrm{pug}}(X/S)$ whose value on $T\in (\mathrm{Sch}/S)$ consists of all $T$-perfect objects $F\in \Dpug(X_T/T)$ such that $F_t$ is $\sigma_t$-semistable of phase $\phi$ and $v_t(F_t)=\bv$ in $\Lambda$ for all $t\in T$.
\item  For an interval $I\subset\RR$, we denote by
\[
\cP_{{\sigma}}(I;\bv)\colon(\mathrm{Sch}/S)^{\mathrm{op}}\longrightarrow \mathrm{Gpds}
\]
the subfunctor of $\cM_{\mathrm{pug}}(X/S)$ whose value on $T\in(\mathrm{Sch}/S)$ consists of all $T$-perfect objects $F\in \Dpug(X_T/T)$ such that $F_t\in \cP_t(I)$ and $v_t(F_t)=\bv$ in $\Lambda$ for all $t\in T$.
\end{enumerate}
\end{definition}

This is the key definition in Part~\ref{part:Relative},~\cite[Definition 21.15]{stability-families}:

\begin{definition}\label{def:FamilyStabilityConditions}
Let $X\to S$ be a morphism as in Setup~\ref{assum-stab}. 
Let $v\colon\Knum(X/S)\to\Lambda$ be a homomorphism to a finite rank free abelian group. 
Let ${\sigma}=(\sigma_s=(Z,\cP_s))_{s\in S}$ be a collection of numerical stability conditions on $\Db(X_s)$ with respect to $(\Lambda,v_s)$, for each $s\in S$. 
We say that ${\sigma}$ is a \emph{stability condition on $\Db(X)$ over $S$} with respect to $(\Lambda,v)$ if the following conditions hold:
\begin{enumerate}[{\rm (i)}]
\item\label{c1} There exists a quadratic form $Q$ on $\Lambda_\RR$ for which $\sigma_s$ is a stability condition with respect to $Q$, for all $s\in S$, i.e.,
\begin{itemize}
    \item $Q$ is negative definite on $\ker(Z_\RR)$;
    \item $Q(v_s(F_s))\geq0$, for all $\sigma_s$-semistable objects $F_s\in\Db(X_s)$ and all $s\in S$.
\end{itemize}
\item\label{c2} For every morphism $T\to S$ and every $T$-perfect object $F\in \Dqc(X_T)$, the set 
\[
\{t\in T \st F_t \text{ is geometrically }\sigma_t\text{-stable}\}
\]
is open in $T$.
\item\label{c3} For any morphism essentially of finite type\footnote{For the definition, see \cite[Definition 11.16]{stability-families}.}  $C\to S$ from a Dedekind scheme $C$, the stability conditions $\sigma_c$ for $c \in C$ are  induced by an HN structure $\sigma_C$ with respect to $(\Lambda,v_C)$ on $\Db(X_C)$ over $C$, where $v_C \colon\Knum(X_C/C)\to\Knum(X/S)\xrightarrow{v}\Lambda$ is the homomorphism induced by $v$. 
\item\label{c4} $\cM^{\mathrm{st}}_{{\sigma}}(\bv,\phi)$ is bounded, for every $\bv\in \Lambda$ and $\phi\in \RR$ such that $Z(\bv)\in\RR_{>0}\cdot e^{i\pi\phi}$.
\end{enumerate}
\end{definition}

We record some facts about this definition. 
First of all, as a consequence of Remark~\ref{rmk:KnumClassIndependent}, for every morphism $T\to S$ and every $T$-perfect object $F\in\Dqc(X_T)$, the function $T\to\CC$ given by $t\mapsto Z(v_t(F_t))$ is locally constant.
Moreover, by~\cite[Lemma 20.3]{stability-families}, to check~\ref{c2}, we only need to consider finite type morphisms $T\to S$ from an affine scheme $T$.

By~\cite[Proposition 20.8]{stability-families},~\ref{c2} implies the following semistable version:
\begin{enumerate}[label={\rm (ii')}]
    \item\label{sc2} For every morphism $T\to S$ and every $T$-perfect object $F\in \Dqc(X_T)$, the set 
    \[
    \{t\in T \st F_t \text{ is }\sigma_t\text{-semistable}\}
    \]
    is open in $T$.
\end{enumerate}
Similarly, by~\cite[Lemma~21.20]{stability-families},~\ref{c4} implies the following semistable version: 
\begin{enumerate}[label={\rm(iv')}]
    \item \label{sc4} 
    $\cM_{{\sigma}}(\bv,\phi)$ is bounded, for every $\bv\in \Lambda$ and $\phi\in \RR$ such that $Z(\bv)\in\RR_{>0}\cdot e^{i\pi\phi}$.
\end{enumerate}

\begin{remark}\label{rmk:FamilyStabilityConditions}
The following properties hold for the moduli stack $\cM_{{\sigma}}(\bv,\phi)$ of semistable objects.
\begin{enumerate}
\item\label{enum:FamilyStabilityConditions1} It is clear that~\ref{c2} implies that $\cM^{\mathrm{st}}_{{\sigma}}(\bv,\phi)$ is an open substack of $\cM_{\mathrm{pug}}(X/S)$. Similarly,~\ref{sc2} gives that $\cM_{{\sigma}}(\bv,\phi)$ is an open substack of $\cM_{\mathrm{pug}}(X/S)$. 
Further, by~\ref{c4} and~\ref{sc4}, both of these substacks are bounded.
\item\label{enum:FamilyStabilityConditions3} By~\cite[Lemma 21.22]{stability-families},~\ref{c3} shows that the morphism $\cM_{{\sigma}}(\bv,\phi)\to S$ satisfies the strong existence part of the valuative criterion, for any DVR essentially of finite type over $S$.
Hence, according to~\cite[Theorem 21.24]{stability-families}, if $S$ is of characteristic zero and ${\sigma}$ is a stability condition on $X$ over $S$, then the moduli stack $\cM_{{\sigma}}(\bv,\phi)$ admits a proper good moduli space $M_{{\sigma}}(\bv,\phi)$ over $S$, for any $\bv\in\Lambda$ and $\phi\in \RR$.   
\end{enumerate}
\end{remark}

The reason to include openness and boundedness in the definition of stability condition over a base is that Bridgeland's deformation theorem (Theorem~\ref{thm:bridgeland-deformation}) holds in this setting as well.
More precisely, we write $\Stab_{(\Lambda,v)}(\Db(X)/S)$ for the set of all stability conditions on $\Db(X)$ over $S$ with respect to $(\Lambda,v)$.
We endow it with the coarsest topology so that the maps 
\[
\Stab_{(\Lambda,v)}(\Db(X)/S)\longrightarrow\Stab_{(\Lambda,v_s)}(\Db(X_s)), \qquad {\sigma}\longmapsto\sigma_s
\]
are continuous for all $s\in S$.
Then~\cite[Theorem 22.2]{stability-families} gives:

\begin{theorem}\label{thm:BridgelandDeformationFamilyVersion}
The forgetful map
\[
\cZ\colon\Stab_{(\Lambda,v)}(\Db(X)/S)\longrightarrow \Hom_{\ZZ}(\Lambda,\CC),\quad {\sigma}\longmapsto Z
\]
is a local homeomorphism.
\end{theorem}

\begin{remark}\label{rmk:BridgelandDeformationFamilyVersion}
For later use, we will need the precise version of Theorem~\ref{thm:BridgelandDeformationFamilyVersion}.
In fact, if we consider the uniform quadratic form $Q$ in~\ref{c1}, and we write $P_Z\subset\Hom_{\ZZ}(\Lambda,\CC)$ for the connected component containing $Z$ of the set of central charges whose kernel is negative definite with respect to $Q$, then there is an open neighborhood ${\sigma}\in U\subset\Stab_{(\Lambda,v)}(\Db(X)/S)$ such that $\cZ|_U\colon U\to P_Z$ is a covering.    
\end{remark}

We conclude the section with two further results.
The first one is useful to check several openness properties.

\begin{lemma}\label{lem:openness-flat}
Let $X\to S$ be a morphism as in Setup~\ref{assum-stab}. 
Let $v\colon\Knum(X/S)\to\Lambda$ be a homomorphism to a finite rank free abelian group. 
Let ${\sigma}=(\sigma_s=(Z,\cP_s))_{s\in S}$ be a collection of numerical stability conditions on $\Db(X_s)$ with respect to $(\Lambda,v_s)$, for each $s\in S$. 
Assume that ${\sigma}$ satisfies~\ref{sc2} and~\ref{c3}.
Then, for any morphism $T\to S$ of finite type and any $T$-perfect object $F$ on $X_T$, the functions
\[
\phi^+_F\colon T\to \RR\cup\{-\infty\},\quad t\mapsto \phi^+_{t}(F_t)
\]
and
\[
\phi^-_F\colon T\to \RR\cup\{+\infty\},\quad t\mapsto \phi^-_{t}(F_t)
\]
are, respectively, upper and lower semicontinuous constructible functions on $T$. 
Here we set $\phi^{\pm}$ of the zero object to be $\mp\infty$.
\end{lemma}

\begin{proof}
The statement can be found in~\cite[Lemma 20.9]{stability-families} under additional assumption~\ref{c2}.
However, in the proof there, \ref{c2} is only used to apply \cite[Lemma 20.4]{stability-families}, whose conclusion is implied by \ref{sc2}. 
Therefore, the same proof of~\cite[Lemma 20.9]{stability-families} works in our setting.
\end{proof}

Finally, we will also need a criterion for stability conditions over a base by replacing~\ref{c2} and~\ref{c4} with the corresponding conditions for moduli stacks of semistable objects.

\begin{lemma}\label{lem:open-geo-stab}
Let $X\to S$ be a morphism as in Setup~\ref{assum-stab}. 
Let $v\colon\Knum(X/S)\to\Lambda$ be a homomorphism to a finite rank free abelian group. 
Let ${\sigma}=(\sigma_s=(Z,\cP_s))_{s\in S}$ be a collection of numerical stability conditions on $\Db(X_s)$ with respect to $(\Lambda,v_s)$, for each $s\in S$. 
Assume that ${\sigma}$ satisfies~\ref{c1},~\ref{sc2},~\ref{c3}, and~\ref{sc4}.
Then ${\sigma}$ is a stability condition on $\Db(X)$ over $S$ with respect to $(\Lambda,v)$.
\end{lemma}

\begin{proof}
Note that we only need to verify~\ref{c2}, since~\ref{c4} follows from~\ref{c2} and the boundedness of $\cM_{{\sigma}}(\bv,\phi)$.

First, we show that $\cP_{{\sigma}}(I;\bw)$ is bounded for each $\bw\in \Lambda$ and interval $I\subset \RR$ of length less than 1. By Lemma \ref{lem:openness-flat}, we know that $\cP_{{\sigma}}(I;\bw)$ is an open substack of $\cM_{\mathrm{pug}}(X/S)$. Let $\phi_0$ and $\phi_1$ be the end points of the closure of $I$. Consider the set
\[
S_{\bw}\coloneqq \{v(A)\in \Lambda\colon A\text{ is an HN factor of } F_t\in \cP_t(I) \text{ with respect to } \sigma_t \text{ and } v(F_t)=\bw\}.
\]
Then it is clear that $Z(S_{\bw})\subset \CC$ is contained in the parallelogram with angles $\pi\phi_0$ and $\pi\phi_1$ and with $0$ and $Z(\bw)$ as opposite vertices. Using \ref{c1}, we know that $S_{\bw}$ is a finite set. 
Since $\cM_{{\sigma}}(\bv,\phi)$ is bounded for every $\bv\in \Lambda$ and $\phi\in \RR$, $\cP_{{\sigma}}(I;\bw)$ is bounded and is an algebraic stack of finite type over $S$ by~\cite[Lemma 9.5 and Lemma 9.6]{stability-families}.

Now, let $T\to S$ be a morphism of finite type from an affine scheme $T$ and $F\in \Dqc(X_T)$ be a $T$-perfect object such that $F_{t_0}$ is geometrically $\sigma_{t_0}$-stable of class $\bv$ and phase $\phi$ for a point $t_0\in T$. 
By Lemma \ref{lem:openness-flat}, after replacing $T$ with an open neighborhood of $t_0$ and shift $F$ if necessary, we may assume that $\phi\in (0,1]$, $F_t\in \cA_{t}$ for each $t\in T$, and $T$ is connected.

Similar to~\cite[Lemma 21.21]{stability-families}, we define a functor
\[
\Quot_T^{\leq \phi}(F)\colon (\mathrm{Sch}/T)^{\mathrm{op}}\to \mathrm{Sets}
\]
whose value on $B\in (\mathrm{Sch}/T)$ is the set of all morphisms $F_B\to Q$ in $\Dqc(X_B)$ where $Q$ is $B$-perfect and $F_b\to Q_b$ is a surjection in $\cA_b$ with $\phi_{\sigma_b}(Q_b)\leq \phi$ for each $b\in B$. 
By Lemma~\ref{lem:openness-flat}, \cite[Proposition 11.6]{stability-families}, and \ref{c1}, we know that $$\Quot_T^{\leq \phi}(F)\to T$$ is an algebraic space locally of finite type. 
Since $\cP_{{\sigma}}(I;\bw)\to S$ is of finite type for each $\bw\in \Lambda$ and interval $I$ of length less than 1, the argument of \cite[Lemma 21.21]{stability-families} implies that $\Quot_T^{\leq \phi}(F)\to T$ is of finite type. Moreover, by Lemma \ref{lem:openness-flat}, \ref{c3}, and \cite[Theorem 17.1]{stability-families}, ${\sigma}$ meets the assumption of \cite[Proposition 11.11]{stability-families}. 
Thus, we can conclude that $\Quot_T^{\leq \phi}(F)\to T$ is universally closed. 
Since the non-geometrically stable locus is the image of the union of those connected components of $\Quot_T^{\leq \phi}(F)$ parameterizing objects of classes not equal to $v(F)$, this image is closed by universal closedness, hence the geometrically stable locus is open. 
This verifies the condition \ref{c2} and the result follows.
\end{proof}

\subsection{Mass-Hom bounds}\label{subsection-moduli-spaces}
The main importance of a mass-Hom bound for a stability condition is that it implies the existence of proper moduli spaces of semistable objects, at least in the smooth case.
More precisely, we have the following version of~\cite[Theorem 2.35]{DHL-robotis}; we also expect that there is a suitable extension of this result to the singular case, but we have not pursued this.

\begin{theorem}[Halpern-Leistner, Robotis]\label{theorem-DHL-robotis} 
Let $X$ be a smooth projective scheme over a field $\KK$ and $\sigma$ be a numerical stability condition on $\Db(X)$ with respect to $(\Lambda,v)$.
Assume that $\sigma$ has a mass-Hom bound. 
Then $\sigma$ is a stability condition on $\Db(X)$ over $\KK$.
In particular, if $\mathrm{char}(\KK)=0$, then for any $\bv\in\Lambda$ and $\phi\in\RR$ such that $Z(\bv)\in\RR_{>0}\cdot e^{i\pi\phi}$, the stack $\cM_{\sigma}(\bv,\phi)$ admits a proper good moduli space $M_{\sigma}(\bv,\phi)$ over $\KK$. 
\end{theorem}

By Theorem~\ref{thm:MassHomBoundProjective}, when $(X, H_X)$ is a smooth polarized scheme over $\KK$, Theorem~\ref{theorem-DHL-robotis} applies to all stability conditions in the connected component $\Stabdagger_{H_X}(\Db(X))$. 
In Section~\ref{sec:ProjectiveFamilies}, we will give a different argument that applies to possibly singular schemes (as well as over a general base scheme). 

\begin{proof}
Since $X$ is smooth projective over $\KK$, we know that the category $\Db(X)$ is smooth and proper over $\KK$. 
Therefore, the theory of \cite{DHL-robotis} is applicable by \cite[Example 2.29]{DHL-robotis}. 

Fix a classical generator $G$ of $\Db(X)$. 
We claim that there exists a constant $C_G>0$, such that for any field $\ell$ with a finite type morphism $\Spec(\ell)\to\Spec(\KK)$ and any object $F\in\Db(X_{\ell})$, we have
\begin{equation}\label{eq:mass-hom-extension}
\dim_{\ell} \Hom(G_{\ell},F)\leq C_Gm_{\sigma_{\ell}}(F).
\end{equation}
Indeed, by Zariski's lemma, we know that $\ell/\KK$ is a finite extension. 
Denote by $\pi\colon X_{\ell}\to X$ the base change morphism. 
Then by adjunction, we obtain
\[
[\ell:k]\cdot \dim_{\ell} \Hom(G_{\ell},F)=\dim_\KK \Hom(G,\pi_*F).
\]
From Step 1 of~\cite[Theorem 12.17]{stability-families}, we know that $\pi_*$ preserves HN filtrations. 
Thus, by the mass-Hom bound of $\sigma$, there exists a constant $C_G>0$ so that
\[
\dim_\KK \Hom(G,\pi_*F)\leq C_Gm_{\sigma}(\pi_*F)=C_G\cdot [\ell:\KK]\cdot m_{\sigma_{\ell}}(F).
\]
This verifies~\eqref{eq:mass-hom-extension}. 

In particular, $\sigma$ satisfies~\cite[Theorem 2.35(1)]{DHL-robotis}. 
Now, from the same proof of the implication (1) to (2) in \cite[Theorem 2.35]{DHL-robotis}, $\sigma$ also satisfies \cite[Theorem 2.35(2)]{DHL-robotis}. 
Therefore, the condition \ref{sc2} holds for $\sigma$ by \cite[Proposition 2.38]{DHL-robotis}. Furthermore, the combination of \cite[Theorem 18.7, 5.7]{stability-families} with \cite[Proposition 2.38]{DHL-robotis} verifies \ref{c3} for $\sigma$. 
Finally, as explained in the second paragraph of the proof of \cite[Proposition 2.41]{DHL-robotis}, $\cM_{\sigma}(\bv,\phi)$ is bounded for any $\bv\in \Lambda$ and $\phi\in \RR$. 
Thus, we can conclude from Lemma~\ref{lem:open-geo-stab} that $\sigma$ is a stability condition on $\Db(X)$ over $\KK$. 
When $\mathrm{char}(\KK)=0$, the existence of a proper good moduli space directly follows from~\eqref{eq:mass-hom-extension} and \cite[Theorem 2.35]{DHL-robotis}.
\end{proof}


\section{Pullback and pushforward of stability conditions in families}\label{sec:PullPushRelativeStability}

The goal of this section is to extend Section~\ref{sec:PullbackPushforward} to the context of stability conditions over a base (introduced in Section~\ref{subsec:StabilityBase}); in particular, to include moduli spaces.

We fix a free abelian group $\Lambda$ of finite rank.
The fundamental result is the following: once the pushforward (or pullback) of a numerical stability condition exists on each point of the base, then it exists over the base. 
The proof will take the whole section.

\begin{theorem}\label{thm:main-criterion}
Let $X\to S$ and $Y\to S$ be two morphisms as in Setup~\ref{assum-stab}.
Suppose that $f\colon X\to Y$ is a finite morphism over $S$.
\begin{enumerate}[{\rm (1)}]
\item\label{enum:main-criterion1} Assume that $f$ is a closed embedding of finite Tor dimension. Let ${\sigma}=(\sigma_s=(Z, \cP_s))_{s\in S}$ be a stability condition on $\Db(Y)$ over $S$ with respect to $(\Lambda,v)$ such that
\begin{equation}\label{eq:push}
    f_{s*} \cO_{X_s} \otimes \cP_{s}(\phi)\subset\cP_{s}(\geq \phi),
\end{equation}
for each $s\in S$ and $\phi\in\RR$. 
Then $f^{\sharp}{\sigma}\coloneqq(f_s^{\sharp}\sigma_s)_{s\in S}$ is a stability condition on $\Db(X)$ over $S$ with respect to $(\Lambda,f^{\sharp}v\coloneqq v\circ f_*)$.
\item\label{enum:main-criterion2} 
Assume that $f$ is faithfully flat and the relative dualizing complex $\omega_{f}^{\bullet}$ is contained in $\Dperf(X)$. 
Let ${\sigma}=(\sigma_s=(Z,\cP_s))_{s\in S}$ be a stability condition on $\Db(X)$ over $S$ with respect to $(\Lambda,v)$ such that 
\begin{equation}\label{eq:pullback}
    f^*_s f_{s*}(\cP_{s}(\phi))\subset\cP_{s}(\leq \phi),
\end{equation}
for each $s\in S$ and $\phi\in\RR$.
Then $f_{\sharp}{\sigma}\coloneqq(f_{s,\sharp}\sigma_s)_{s\in S}$ is a stability condition on $\Db(Y)$ over $S$ with respect to $(\Lambda,f_{\sharp}v\coloneqq v\circ f^*)$.
\end{enumerate}
\end{theorem}

Note that in the statement of Theorem~\ref{thm:main-criterion}, $f^*$ and $f_*$ are well-defined at the level of relative numerical Grothendieck groups, in view of Proposition~\ref{prop:FunctorialityKnum}.

\begin{proof}
We prove~\ref{enum:main-criterion1}. 
First of all, by Proposition~\ref{proposition-pullback-stability}, $f^{\sharp}{\sigma}=(f_s^{\sharp}\sigma_s)_{s\in S}$ is a collection of stability conditions on $\Db(X_s)$ with respect to $(\Lambda, v_s\circ f_{s,*})$.
Moreover, condition~\ref{c1} in Definition~\ref{def:FamilyStabilityConditions} holds.

We want to show that~\ref{sc2} holds for $f^\sharp{\sigma}$.
Let $T\to S$ be a finite type morphism, with $T$ affine.
If $F\in\Dqc(X_T)$ is a $T$-perfect complex, then $f_{T,*}F\in\Dqc(Y_T)$ is $T$-perfect as well by Lemma~\ref{lem:PushforwardTperfect}.
Since~\ref{sc2} holds for ${\sigma}$, we can conclude by using Lemma~\ref{lem:base-change}\ref{enum:BaseChange1}.
Moreover, since Lemma \ref{lem:openness-flat} applies to ${\sigma}$, the conclusion of Lemma \ref{lem:openness-flat} also holds for $f^{\sharp}{\sigma}$ by Lemma \ref{lemma-pullback-slicing-compatibility}\ref{pullback-slicing-compatibility-phipm}.

Next, we verify~\ref{c3}. 
Let $C\to S$ be a morphism which is essentially of finite type from a Dedekind scheme $C$ and $\sigma_C=(Z,\cP)$ be the HN structure with respect to $(\Lambda,v_C)$ on $\Db(Y_C)$ over $C$ corresponding to ${\sigma}$.
We claim that
\begin{equation}\label{eq:PushHN}
f_{C,*}f_C^*(F)\in\cP(\geq\phi),
\end{equation}
for each $\phi\in\RR$ and $F\in\cP(\phi)$. 
Indeed, first of all we observe that, by \cite[Corollary 20.10 and Proposition 17.6]{stability-families}, $\sigma_C$ has a $C$-torsion theory, in the sense of~\cite[Definition 17.4]{stability-families}.
Hence, to show \eqref{eq:PushHN}, we can assume that $F$ is either in $\cP(\phi)_{\Ctor}\coloneqq\cP(\phi)\cap\Db(Y_C)_{\Ctor}$, or in $\cP(\phi)_{\Ctf}\coloneqq\cP(\phi)_{\Ctor}^\perp$.
In the latter case, by~\cite[Lemma 13.12]{stability-families}, we have 
\[
(f_{C,*}f_C^*(F))_c=f_{c,*}f_c^*(F_c)\in \cP_c(\geq \phi),
\]
for any point $c\in C$ by \eqref{eq:push}. 
So \eqref{eq:PushHN} can be deduced from \cite[Lemma 15.6 and Lemma 13.12]{stability-families}.
When $F$ is $C$-torsion, by \cite[Lemma 6.11]{stability-families}, we may assume that $F=i_{p*}F'$ for a closed point $p\in C$ and $F'\in\cP_p(\phi)$, where $i_p \colon Y_p \hookrightarrow Y_C$ is the inclusion. 
Then by \eqref{eq:push} we have
\[
f_{C,*}f_C^*(F)=i_{p,*}(f_{p,*}f_p^*(F'))\in i_{p,*}(\cP_p(\geq \phi)), 
\]
and the claim again follows from~\cite[Lemma 15.6]{stability-families}.

In view of~\eqref{eq:PushHN}, we conclude by Proposition~\ref{proposition-pullback-stability} that the pullback $f^{\sharp}_{C} \cP$ is a $C$-local slicing of $\Db(X_C)$. 
Combining with the property \ref{sc2} and the conclusion of Lemma \ref{lem:openness-flat} for $f^{\sharp}{\sigma}$, 
we see that the pair $(Z, f_C^{\sharp}\cP)$ meets the assumptions in~\cite[Theorem 18.7]{stability-families}, which gives the desired HN structure on $\Db(X_C)$ over $C$. This verifies~\ref{c3}.

Finally, we will prove~\ref{sc4}. 
By \ref{sc2} for $f^{\sharp}{\sigma}$, the moduli stack $\cM_{f^{\sharp}{\sigma}}(\bv,\phi)$ is an open substack of $\cM_{\mathrm{pug}}(X/S)$ for any $\bv\in \Lambda$ and $Z(\bv)\in \RR_{>0}\cdot e^{i\pi\phi}$. 
We claim that $\cM_{f^{\sharp}{\sigma}}(\bv,\phi)$ is bounded. 
By Remark~\ref{rmk:FamilyStabilityConditions}\eqref{enum:FamilyStabilityConditions1}, we can take a pair $(T, \cF)$ witnessing the boundedness of $\cM_{{\sigma}}(\bv,\phi)$. 
Since $T$ is of finite type over $S$, by replacing $T$ with the disjoint union of the strata in Lemma \ref{lem:flatten}, we may assume that each $\cH^i(\cF)$ is flat over $T$ and $\cH^i(\cF)=0$ for $i\notin [a,b]$, where $[a,b]\subset \RR$ is an interval. 
Therefore $f^*_T\cH^{i}(\cF)$ is $T$-perfect by~\citestacks{0GEH} as $f_t$ is of finite Tor dimension for each $t\in T$. 
Applying Lemma \ref{lem:flatten} to
\[
\bigoplus_{i\in [a,b]}f^*_T\cH^{i}(\cF),
\]
we may also assume that $\cH^j(f^*_T\cH^{i}(\cF))$ is flat over $T$ and $\cH^j(f^*_T\cH^{i}(\cF))=0$ for each $i\in [a,b]$ and $j\notin [x,y]$, where $[x,y]\subset \RR_{\leq 0}$ is an interval. 
In particular, $\cF^i\coloneqq \cH^0(f_T^*\cH^i(\cF))$ is flat over $T$ and we have $\cF^i_t=\cH^0(f_t^*\cH^i(\cF_t))=\cH^0((f_T^*\cH^i(\cF))_t)$ for any $i$ and $t\in T$. 

If we define a subfunctor
\[
\cM_i(\bv,\phi)\colon (\mathrm{Sch}/S)^{\mathrm{op}}\to \mathrm{Gpds}
\]
of $\cM_{\mathrm{pug}}(X/S)$ whose value on $B\in (\mathrm{Sch}/S)$ consists of all $F^i\in\Dpug(X_B/B)$ such that $F_b^i\cong\cH^0(f_b^*\cH^i(F))[-i]$ for some $\sigma_b$-semistable object $F$ of class $\bv$ and phase $\phi$ for each $b\in B$, then the pair $(T,\cF^i[-i])$ witnesses the boundedness of $\cM_i(\bv,\phi)$, for each $i\in [a,b]$. 
Note that if $\cF_t=f_{t,*}F_t'$ for $F_t'\in \Db(X_t)$, then $\cF^i_t=\cH^i(F'_t)$. 
Therefore, the object corresponding to each geometric point of $\cM_{f^{\sharp}{\sigma}}(\bv,\phi)$ can be written as an extension of objects associated with geometric points in $\cM_i(\bv,\phi)$, so $\cM_{f^{\sharp}{\sigma}}(\bv,\phi)$ is bounded by~\cite[Lemma 9.6]{stability-families}. 
This verifies~\ref{sc4} for $f^{\sharp}\sigma$. 
Thus, we can conclude that $f^{\sharp}{\sigma}$ is a stability condition on $\Db(X)$ over $S$ by Lemma~\ref{lem:open-geo-stab}.

We now prove \ref{enum:main-criterion2}. 
Note that for $s \in S$, the dualizing complex $\omega^{\bullet}_{f_s}$ is perfect, since it is the pullback of $\omega^{\bullet}_f$. 
Thus, by Proposition~\ref{proposition-pushforward-stability}, $f_\sharp{\sigma}=(f_{s,\sharp}\sigma_s)_{s\in S}$ is a collection of stability conditions with respect to $(\Lambda,v_s\circ f_s^*)$ on $\Db(Y_s)$.
Moreover, condition \ref{c1} in Definition~\ref{def:FamilyStabilityConditions} holds for $f_{\sharp} \sigma$, and \ref{sc2} follows from the corresponding property for ${\sigma}$.  
Therefore, for any $\bv\in\Lambda$ and $Z(\bv)\in\RR_{>0}\cdot e^{i\pi\phi}$, the moduli stack $\cM_{f_{\sharp}{\sigma}}(\bv,\phi)$ is an open substack of $\cM_{\mathrm{pug}}(Y/S)$. 

The condition~\ref{c3} for $f_{\sharp}\sigma$ follows similarly to~\ref{enum:main-criterion1}. 
More precisely, let $C \to S$ be a morphism which is essentially of finite type from a Dedekind scheme $C$, and let $\sigma_C = (Z, \cP)$ be the HN structure with respect to $(\Lambda, v_C)$ on $\Db(X_C)$ over $C$ corresponding to $\sigma$. 
Then arguing as in~\ref{enum:main-criterion1}, we find that 
\begin{equation*}
    f_{C}^*f_{C,*}(\cP(\phi)) \subset \cP(\leq \phi). 
\end{equation*}
By Proposition~\ref{proposition-pushforward-stability}, we conclude that $f_{C, \sharp}\cP$ is a $C$-local slicing of $\Db(Y_C)$, and then as above that $(Z, f_{C, \sharp}\cP)$ is the desired HN structure on $\Db(Y_C)$ over $C$. 

By Lemma \ref{lem:open-geo-stab}, it remains to show that $f_{\sharp}\sigma$ satisfies~\ref{sc4}, i.e. that $\cM_{f_{\sharp}{\sigma}}(\bv,\phi)$ is bounded.
Set $Z\coloneqq X\times_Y X$ and $W\coloneqq X\times_Y X\times_Y X$. 
Let $p,q\colon Z\to X$ be the projections to the first and second factors, respectively. 
Similarly, we have morphisms $p_{12}, p_{13}, p_{23}\colon W\to Z$. 
We define a moduli functor of descent datum
\[
\mathcal{D}_{{\sigma}}(\bv,\phi)\colon (\mathrm{Sch}/S)^{\mathrm{op}}\to \mathrm{Gpds}
\]
whose value on $T\in (\mathrm{Sch}/S)$ consists of all pairs $(F,\alpha)$, where $F\in\Dpug(X_T/T)$ with $F_t$ $\sigma_t$-semistable of class $\bv$ and phase $\phi$, and $\alpha\colon p_{T}^*F\xrightarrow{\cong}q_{T}^*F$ is an isomorphism satisfying the cocycle and unit conditions:
\[
p_{23,T}^*\alpha\circ p_{12,T}^*\alpha= p_{13,T}^*\alpha
\]
on $W_T$ and
\[
\Delta_{X_T/Y_T}^*(\alpha)=\mathrm{id}_F
\]
on $X_T$.

Then we have a natural morphism
\[
\psi\colon\cM_{f_{\sharp}{\sigma}}(\bv,\phi)\longrightarrow \mathcal{D}_{{\sigma}}(\bv,\phi), \qquad \psi(T)\colon F'\longmapsto (f^*_TF', \alpha_{F'}),
\]
where $(f^*_TF',\alpha_{F'})$ is the canonical descent data of $F'\in\Dpug(Y_T/T)$. 
By the definition of $f_{\sharp}{\sigma}$, Lemma~\ref{lem:base-change}, and~\citestacks{0DLH}, $\psi$ is an equivalence and we have a commutative diagram
\[
\begin{tikzcd}
	{\cM_{f_{\sharp}{\sigma}}(\bv,\phi)} & {\cM_{{\sigma}}(\bv,\phi)} \\
	{\mathcal{D}_{{\sigma}}(\bv,\phi)}
	\arrow[from=1-1, to=1-2]
	\arrow["\psi"', from=1-1, to=2-1]
	\arrow["\xi"', from=2-1, to=1-2]
\end{tikzcd}
\]
where $\xi$ is the forgetful functor. 

By Remark~\ref{rmk:FamilyStabilityConditions}\eqref{enum:FamilyStabilityConditions1}, we can take a pair $(T,\cF)$ witnessing the boundedness of $\cM_{{\sigma}}(\bv,\phi)$. 
We define a functor 
\[
\underline{\mathrm{Isom}}_T(p_T^*\cF, q_T^*\cF)\colon (\mathrm{Sch}/T)^{\mathrm{op}}\longrightarrow \mathrm{Sets} 
\]
which maps $B\in (\mathrm{Sch}/T)$ to the set of isomorphisms between $p_B^*\cF_B$ and $q_B^*\cF_B$. 
Note that for any $t\in T$ and $i\in \ZZ$, we have
\[
\Hom(p_t^*\cF_t, q_t^*\cF_t[i])=\Hom(\cF_t, p_{t*}q_t^*\cF_t[i])=\Hom(\cF_t, f_t^*f_{t*}\cF_t[i]),
\]
so it vanishes for any $i<0$ by~\eqref{eq:pullback}. 
Therefore, by~\citestacks{0DLC} and the proof of~\citestacks{0DPW}, $\underline{\mathrm{Isom}}_T(p_T^*\cF, q_T^*\cF)$ is a scheme which is affine and of finite presentation over $T$. 
Let $u$ be the universal isomorphism on $Z_T\times_T \underline{\mathrm{Isom}}_T(p_T^*\cF, q_T^*\cF)$. 
Then $\Delta_{X_T/Y_T}^*(u)$ induces a section $s_1$ of 
\[
\underline{\mathrm{Isom}}_T(\cF, \cF)\times_T\underline{\mathrm{Isom}}_T(p_T^*\cF, q_T^*\cF)\longrightarrow\underline{\mathrm{Isom}}_T(p_T^*\cF, q_T^*\cF).
\]
Moreover, we also have a trivial section $s_2$ induced by $\mathrm{id}_{\cF}$. 
Therefore, from the definition, an isomorphism from $p_B^*\cF_B$ to $q_B^*\cF_B$ satisfies the unit condition if and only if the corresponding morphism $B\to \underline{\mathrm{Isom}}_T(p_T^*\cF, q_T^*\cF)$ factors through the locus $H_1$ where $s_1$ and $s_2$ coincide. 
Since $\cF\in\Dpug(X_T/T)$, we know that $\underline{\mathrm{Isom}}_T(\cF,\cF)$ is a scheme which is affine and of finite presentation over $T$ by~\citestacks{0DLC}. 
Therefore, $H_1$ is a closed subscheme of $\underline{\mathrm{Isom}}_T(p_T^*\cF, q_T^*\cF)$. 
Similarly, the cocycle condition corresponds to the equality locus of two sections of
\[
\underline{\mathrm{Isom}}_T(q_1^*\cF, q_3^*\cF)\times_T\underline{\mathrm{Isom}}_T(p_T^*\cF, q_T^*\cF)\longrightarrow\underline{\mathrm{Isom}}_T(p_T^*\cF, q_T^*\cF),
\]
where $q_i\colon W_T\to X_T$ is the projection to the $i$-th factor. 
Since $\underline{\mathrm{Isom}}_T(q_1^*\cF, q_3^*\cF)$ is also a scheme which is affine and of finite presentation over $T$ by~\citestacks{0DLC} and~\eqref{eq:pullback}, we get a closed subscheme $H_2\subset \underline{\mathrm{Isom}}_T(p_T^*\cF, q_T^*\cF)$.

By definition, we have a morphism $\mathcal{D}_{{\sigma}}(\bv,\phi)\times_{\cM_{{\sigma}}(\bv,\phi)} T \to \underline{\mathrm{Isom}}_T(p_T^*\cF, q_T^*\cF)$ given by 
\[
\left(\mathcal{D}_{{\sigma}}(\bv,\phi)\times_{\cM_{{\sigma}}(\bv,\phi)} T\right)(B)\longrightarrow\underline{\mathrm{Isom}}_T(p_T^*\cF, q_T^*\cF)(B),\qquad (\cF_B, \alpha)\longmapsto \alpha
\]
for any $B\in (\mathrm{Sch}/T)$. 
Moreover, its image is exactly those $\alpha$ satisfying the cocycle and unit condition. 
Therefore, we get a factorization 
\[
\mathcal{D}_{{\sigma}}(\bv,\phi)\times_{\cM_{{\sigma}}(\bv,\phi)} T\longtwoheadrightarrow H_1\cap H_2\longhookrightarrow \underline{\mathrm{Isom}}_T(p_T^*\cF, q_T^*\cF).
\]
By descending $\cF_{H_1\cap H_2}$ along $f_{H_1\cap H_2}$, we then get an object $\cF'\in \Dqc(Y_{H_1\cap H_2})$, which induces a surjection $H_1\cap H_2\to \cM_{f_{\sharp}{\sigma}}(\bv,\phi)$. 
Since $\underline{\mathrm{Isom}}_T(p_T^*\cF, q_T^*\cF)\to T$ is of finite type, so is $H_1\cap H_2$. 
Therefore, the pair $(H_1\cap H_2, \cF')$ witnesses the boundedness of $\cM_{f_{\sharp}{\sigma}}(\bv,\phi)$.
This completes the proof of \ref{enum:main-criterion2}.
\end{proof}


\section{Stability conditions on projective families}\label{sec:ProjectiveFamilies}

In this section we complete the proof of the main result of the paper, Theorem~\ref{main-theorem}.
In Section~\ref{subsec:FamiliesStabilityP1n}, we consider the case of the product of $n$ copies of $\PP^1_S$. 
Then in Section~\ref{subsec:ProofMainTheorem} we handle the general case, by first passing to $\PP^n_S$ and then to all projective schemes over~$S$.
An interesting consequence is the existence of proper moduli spaces of semistable objects for stability conditions in $\Stabdagger_{H_X}(\Db(X))$, when $X$ is a projective scheme over a field $\KK$ of characteristic zero (see Corollary~\ref{cor:Stabdagger=Stabdagger}).

\subsection{Families of stability conditions on $(\PP^1_S)^n$}\label{subsec:FamiliesStabilityP1n}

Let $S$ be a noetherian Nagata scheme of finite Krull dimension which is quasi-projective over a noetherian affine scheme. 
The goal of this section is to construct stability conditions on
\[
\pi \colon (\PP^1_S)^n=\underbrace{\PP^1_S\times_S\dots\times_S\PP^1_S}_{n\text{-times}}\longrightarrow S.
\]
Note that the assumptions on $S$ imply that $\pi$ is as in Setup~\ref{assum-stab}.

Let $\Lambda\coloneqq\cN((\PP^1_S)^n/S)$ be as in Example~\ref{ex:UniformlyNumericalGrothendieck} and $v\colon\Knum((\PP^1_S)^n/S)\twoheadrightarrow\Lambda$ be the quotient morphism. 
In this context, $\Lambda$ is a free abelian group of rank~$2^n$ such that, for each point $t$ over $S$, the composition $v_t$ of~\eqref{eq:DefOfVt} can be identified with the quotient map
\[
\rK_0((\PP^1_t)^n)\longtwoheadrightarrow\Knum((\PP^1_t)^n)=\Lambda, 
\]
which is in fact an isomorphism in this case. 

\begin{proposition}\label{prop:AlgebraicP1nFamily}
In the above notation, for any $(a,b)\in\RR_{>0}\times\RR$, the collection 
\[
{\sigma}^{a,b}_{(\PP^1_S)^n}\coloneqq\left(\sigma_{(\PP^1_s)^n}^{a,b}=\left(Z^{a,b}_{(\PP^1_s)^n},\cP_{(\PP^1_s)^n}^{a,b}\right)\right)_{s\in S}
\]
of numerical stability conditions on $\Db((\PP^1_s)^n)$ for $s\in S$ given by \eqref{eq:StabonP1n}  is a stability condition on $\Db((\PP^1_S)^n)$ over $S$ with respect to $(\Lambda,v)$.
This gives an injective continuous map
\[
\RR_{>0}\times\RR\longhookrightarrow\Stab_{(\Lambda,v)}\left(\Db((\PP^1_S)^n)/S\right).
\]
\end{proposition}

\begin{proof}
We need to check the conditions in Definition~\ref{def:FamilyStabilityConditions}.
By Remark~\ref{rmK:SupportField},~\ref{c1} follows since the quadratic form $Q^{a,b}$ giving the support property is independent of $s\in S$.

We actually have a more precise statement.
Let $(a_0,b_0)\in\RR_{>0}\times\RR$, and assume that we know that ${\sigma}^{a_0,b_0}_{(\PP^1_S)^n}$ is a stability condition on $\Db((\PP^1_S)^n)$ over $S$ with respect to $(\Lambda,v)$.
By Theorem~\ref{thm:BridgelandDeformationFamilyVersion} and Remark~\ref{rmk:BridgelandDeformationFamilyVersion}, we can deform ${\sigma}^{a_0,b_0}_{(\PP^1_S)^n}$ to ${\sigma}^{a,b}_{(\PP^1_S)^n}$ as long as the quadratic form giving the support property is negative definite on $\ker(Z^{a,b}_{(\PP^1_s)^n})_\RR$.
This provides a connected open subset $U_0\subset\RR_{>0}\times\RR$ containing $(a_0,b_0)$ such that for any $(a,b)\in U_0$, ${\sigma}^{a,b}_{(\PP^1_S)^n}$ is a stability condition on $\Db((\PP^1_S)^n)$ over $S$ with respect to $(\Lambda,v)$.
We claim that $U_0$ is also closed and thus equal to $\RR_{>0}\times\RR$.
Indeed, if $(a,b)$ is in the closure of $U_0$, then there exists $(a',b')\in U_0$ such that $Q^{a,b}$ is negative definite on $\ker(Z^{a',b'}_{(\PP^1_s)^n})_\RR$ and it provides the support property for $\sigma^{a',b'}_{(\PP^1_s)^n}$, for all $s\in S$.
Hence, by Theorem~\ref{thm:BridgelandDeformationFamilyVersion}, $(a,b)\in U_0$, as observed before.

Hence, we only need to show that there exists $(a,b)$ such that ${\sigma}^{a,b}_{(\PP^1_S)^n}$ is a stability condition on $\Db((\PP^1_S)^n)$ over $S$ with respect to $(\Lambda,v)$.
Let us choose $(a,b_0)\in\RR_{>0}\times\RR$ as in Theorem~\ref{thm:AlgebraicP1n}.
Then there exists $G\in\widetilde{\GL}_2^+(\RR)$, independent of $s\in S$, such that $G\cdot\sigma^{a,b_0}_{(\PP^1_s)^n}$ is algebraic.
In view of Remark~\ref{rmk:AlgebraicSatbilityMod}, conditions \ref{c2} and \ref{c4} are then a consequence of the corresponding statements for finite-dimensional algebras.
Finally, to show~\ref{c3}, given a morphism $C\to S$ essentially of finite type from a Dedekind scheme $C$, 
we obtain a $C$-local bounded t-structure on $\Db((\PP^1_C)^n)$ associated to the relative full strong exceptional collection $\cO_{(\PP^1_C)^n}(I_{k,n})$. 
Then by \cite[Theorem~18.7]{stability-families}, we obtain the required HN structure on $\Db((\PP^1_C)^n)$. 
\end{proof}

\subsection{Proof of the main theorem}\label{subsec:ProofMainTheorem}

We let $S$ be as in the previous section, i.e., $S$ is a noetherian Nagata scheme of finite Krull dimension which is quasi-projective over a noetherian affine scheme. 
We proceed as in the absolute case, by first pushing forward the stability conditions in Section~\ref{subsec:FamiliesStabilityP1n} to $\PP^n_S$ and then restricting to an arbitrary projective scheme over $S$.

Let us start with $\PP^n_S$.
Again, the morphism $\pi \colon\PP^n_S\rightarrow S$ satisfies the hypotheses of Setup~\ref{assum-stab}.
Let $\Lambda\coloneqq\cN(\PP^n_S/S)\cong\ZZ^{n+1}$ and let $v$ be the quotient morphism.
Again, for each point $t$ over $S$, the morphism $v_t$ can be identified with the quotient map onto the numerical Grothendieck group $\Knum(\PP^n_t)$.

\begin{proposition}\label{prop:AlgebraicPnFamily}
In the above notation, for any $(a,b)\in\RR_{>0}\times\RR$, the collection 
\[
{\sigma}^{a,b}_{\PP^n_S}\coloneqq\left(\sigma_{\PP^n_s}^{a,b}=\left(Z^{a,b}_{\PP^n_s},\cP_{\PP^n_s}^{a,b}\right)\right)_{s\in S}
\]
of numerical stability conditions on $\Db(\PP^n_s)$ in \eqref{eq:stabonPn}, for each $s\in S$, is a stability condition on $\Db(\PP^n_S)$ over $S$ with respect to $(\Lambda,v)$.
This gives an injective continuous map
\[
\RR_{>0}\times\RR\longhookrightarrow\Stab_{(\Lambda,v)}\left(\Db(\PP^n_S)/S\right).
\]
\end{proposition}

\begin{proof}
This follows immediately from Theorem~\ref{theorem-chunyi-Pn}, Proposition~\ref{prop:AlgebraicP1nFamily}, and Theorem~\ref{thm:main-criterion}\ref{enum:main-criterion2}.
\end{proof}

Let us now consider a morphism $\pi \colon X\to S$ as in Setup~\ref{assum-stab}.
Since $S$ is quasi-projective over an affine scheme, there exist $n>0$ and a closed embedding over $S$
\[
\iota\colon X\longhookrightarrow\PP^n_S
\]
(see~\citestacks{087S}).
We let $H_{X/S}$ be the relative numerical first Chern class of the relative ample line bundle $\cO_{X/S}(1)=\iota^*\cO_{\PP^n_S/S}(1)$.
We consider the abelian group $\Lambda_{H_{X/S}}$, which is the image of the composite
\[
\Knum(X/S)\xlongrightarrow{\iota_*}\Knum(\PP^n_S/S)\longrightarrow\cN(\PP^n_S/S),
\]
where the first morphism is induced by $\iota_*$ in view of Proposition~\ref{prop:FunctorialityKnum}, and 
\begin{equation}
\label{definition-vHXS}
v_{H_{X/S}}\colon\Knum(X/S)\longrightarrow\Lambda_{H_{X/S}}
\end{equation}
is the quotient map.
Notice that $v_{H_{X/S}}$ factors through $\cN(X/S)$.

The precise form of (the nonemptiness part of) Theorem~\ref{main-theorem} is then the following.

\begin{theorem}\label{thm:MainSecondVersion}
In the above notation, we have
\[
\Stab_{\left(\Lambda_{H_{X/S}},v_{H_{X/S}}\right)}(\Db(X)/S)\neq\emptyset.
\]
More precisely, there exists $a_0>0$ such that, for any $(a,b)\in\RR_{>a_0}\times\RR$, the collection 
\[
{\widetilde{\sigma}}^{a,b}_{X/S}\coloneqq\iota^\sharp{\sigma}^{a,b}_{\PP^n_S}=\left(\widetilde{\sigma}_{X_s}^{a,b}=\left(\widetilde{Z}^{a,b}_{X_s},\widetilde{\cP}_{X_s}^{a,b}\right)\right)_{s\in S}
\]
of numerical stability conditions on $\Db(X_s)$ in \eqref{eq:Paris20260702}, for each $s\in S$, is a stability condition on $\Db(X)$ over $S$ with respect to $(\Lambda_{H_{X/S}},v_{H_{X/S}})$.
\end{theorem}

\begin{proof}
Since there are only finitely many choices for the Hilbert polynomial of $X_s$, we deduce the statement from Theorem~\ref{thm:ProjectiveScheme}, together with Proposition~\ref{prop:AlgebraicPnFamily} and Theorem~\ref{thm:main-criterion}\ref{enum:main-criterion1}. 
Note that since $X \to S$ is flat and finitely presented and $\PP^n_S \to S$ is smooth, the closed immersion $\iota \colon X \hookrightarrow \PP^n_S$ is of finite Tor-dimension by \cite[\href{https://stacks.math.columbia.edu/tag/068D}{Tag 068D}]{stacks-project}, so Theorem~\ref{thm:main-criterion}\ref{enum:main-criterion1} does indeed apply. 
\end{proof}

\begin{remark}\label{rmk:LineBundle}
Since $v_{H_{X/S}}$ factors through $\cN(X/S)$, by~\cite[Theorem 21.25]{stability-families}, there is a relative real numerical Cartier divisor class $\ell_{{\sigma}}\in N^1(\cM_{{\sigma}}(\bv,\phi)/S)$ that
is relatively nef.
If either $S$ is of characteristic~$0$ or $\cM_{{\sigma}}(\bv,\phi)=\cM_{{\sigma}}^{\mathrm{st}}(\bv,\phi)$, this class descends to a relative numerical Cartier divisor class $l_{{\sigma}}\in N^1(M_{{\sigma}}(\bv,\phi)/S)$ that is relatively strictly nef, where 
$M_{{\sigma}}(\bv,\phi)$ denotes the associated good moduli space, which is an algebraic space proper over $S$ (see Remark~\ref{rmk:FamilyStabilityConditions}\eqref{enum:FamilyStabilityConditions3}).
\end{remark}

As in the absolute case, if we denote by $H_{X/S}$ the relative numerical class of a relatively ample line bundle, and we choose a relatively very ample line bundle $\cL)$ on $X$ over $S$ such that its numerical first Chern class is $mH_{X/S}$, we can define the connected component
\[
\Stabdagger_{H_{X/S}}(\Db(X)/S)\subset\Stab_{\left(\Lambda_{H_{X/S}},v_{H_{X/S}}\right)}(\Db(X)/S)
\]
containing the stability conditions $\widetilde{{\sigma}}^{a,b}_{X/S}\coloneqq\iota^\sharp{\sigma}^{a,b}_{\PP^n_S}$, for $a\gg0$. Note that Theorem~\ref{thm:canonical-component} implies that $\Stabdagger_{H_{X/S}}(\Db(X)/S)$ does not depend on the choice of $\cO_{X/S}(1)$, nor on the integer $m$.

\begin{corollary}\label{cor:Stabdagger=Stabdagger}
Let $X$ be a projective scheme over a field $\KK$ equipped with an ample numerical class $H_X$.  
Then
\[
\Stabdagger_{H_X}(\Db(X))=\Stabdagger_{H_{X}}(\Db(X)/\Spec(\KK)).
\]
\end{corollary}

In particular, all stability conditions in $\Stabdagger_{H_X}(\Db(X))$ have proper moduli spaces of semistable objects, if $\mathrm{char}(\KK)=0$.

\begin{proof}
We can apply the above results to $S=\Spec(\KK)$.
Then, we have
\[
\Big(\Lambda_{H_{X/S}},v_{H_{X/S}}\Big)=\Big(\Lambda_{H_{X}},v_{H_{X}}\Big),
\]
both connected components contain
\[
\widetilde{{\sigma}}^{a,b}_{X/S}=\widetilde{\sigma}^{a,b}_{X}
\]
for $a\gg0$, and the support property is defined in both cases by the same quadratic form.
\end{proof}

\begin{remark}\label{rmk:ModuliSurfacesThreefolds}
By the results in Section~\ref{subsec:Threefolds}, we then immediately deduce the following consequences (either of Corollary~\ref{cor:Stabdagger=Stabdagger}, or Theorem~\ref{theorem-DHL-robotis}):
\begin{enumerate}[{\rm (1)}]
\item\label{enum:ModuliSurfacesThreefolds1} If $X$ is a smooth projective surface over $\CC$, then all geometric stability conditions have proper moduli spaces.
\item\label{enum:ModuliSurfacesThreefolds2} If $X$ is a smooth projective threefold over $\CC$ satisfying the assumptions of Proposition~\ref{prop:BMTmassbound}, then all stability conditions in $\widetilde{\mathfrak{P}}$ have proper moduli spaces.
\end{enumerate}
This recovers all the existing results for moduli spaces on surfaces and threefolds (e.g.,~\cite{TodaK3,PT:Moduli}).
\end{remark}


\section{Stability of skyscraper sheaves and the Bondal--Orlov theorem}
\label{sec:BondalOrlov}

In this section we characterize skyscraper sheaves of points, once we have a stability condition satisfying a Bayer property.
The main result is Proposition~\ref{prop:Mpt=X}. 
As an application, in Theorem~\ref{thm:BOreconstruction} we give a new proof of the Bondal--Orlov reconstruction theorem.

Throughout, we work over an algebraically closed field $\KK$.

\subsection{Skyscraper sheaves}
We start with an easy observation.

\begin{lemma}\label{lem:AuxiliaryBayerSkyscraper1}
Let $X$ be a projective scheme over $\KK=\overline{\KK}$, and let $\cL$ be an ample invertible sheaf on $X$.
Let $F\in\Db(X)$ be such that $F\cong F\otimes\cL$.
Then $\dim\operatorname{Supp}(F)=0$.
Moreover:
\begin{enumerate}[{\rm (1)}]
    \item\label{enum:AuxiliaryBayerSkyscraper1i} If $F$ is indecomposable, then its set-theoretic support must be a single closed point $\{x\}$ in $X$.
    \item\label{enum:AuxiliaryBayerSkyscraper1ii} If $F$ is indecomposable and $\Hom^{<0}(F,F)=0$, then $F\cong T_x[m]$, for some $m\in\ZZ$, where $T_x$ is a torsion sheaf on $X$ set-theoretically supported at a closed point $x$.
    \item\label{enum:AuxiliaryBayerSkyscraper1iii} If $\Hom^{<0}(F,F)=0$ and $\Hom(F,F)=\KK$, then $F\cong\cO_x[m]$, for some $m\in\ZZ$ and a closed point $x\in X$.
\end{enumerate}
\end{lemma}

\begin{proof}
The first part of the statement follows immediately from the fact that $\cL$ is ample, by looking at the cohomology sheaves of $F$.
Similarly,~\ref{enum:AuxiliaryBayerSkyscraper1i} is immediate.
To see~\ref{enum:AuxiliaryBayerSkyscraper1ii}, up to shift we can assume that $F$ has nonzero cohomology sheaves in degrees between $0$ and $l\geq0$.
Suppose, for a contradiction, that $l>0$.
We write
\[
F'\longrightarrow F\longrightarrow \cH^l(F)[-l],
\]
where $\cH^l(F)\cong T_x'$, for some torsion sheaf $T_x'$ topologically supported at $x$, and $F'$ has cohomology sheaves in degrees between $0$ and $l-1$.
Then, we have an exact sequence
\[
\Hom(F'[1],F[-l])\longrightarrow\Hom(T_x'[-l],F[-l])\longrightarrow\Hom(F,F[-l]). 
\]
But $\Hom(F'[1],F[-l])=0$, and $\Hom(F,F[-l])=0$ by assumption.
Hence, $\Hom(T_x',F)=0$.
But $\cH^0(F)\cong T_x$, for some other torsion sheaf $T_x$ topologically supported at $x$. In particular, we have a nonzero morphism $T_x'\twoheadrightarrow\cO_x\hookrightarrow T_x$, which then induces a nonzero morphism in $\Hom(T_x',F)$, a contradiction.

Finally, to show~\ref{enum:AuxiliaryBayerSkyscraper1iii}, we observe that since $\Hom(F,F)=\KK$, $F$ is indecomposable. Hence, by~\ref{enum:AuxiliaryBayerSkyscraper1ii}, we have that $F\cong T_x[m]$, for some $m$. If $T_x\not\cong\cO_x$, we have $\dim_\KK\Hom(T_x,T_x)>1$, a contradiction. 
\end{proof}

As a consequence, we have:

\begin{lemma}\label{lem:AuxiliaryBayerSkyscraper2}
Let $X$ be a projective scheme over $\KK=\overline{\KK}$, and let $\cL$ be a very ample invertible sheaf on $X$.
Let $\sigma=(Z,\cP)$ be a stability condition on $\Db(X)$ with respect to $(\Lambda,v)$.
Let $F\in\cP(\phi)$ be a stable object, for some $\phi\in\RR$.
If $F\otimes\cL^\vee\in\cP(\phi)$ and $v(F)=v(F\otimes\cL^\vee)$, then $F\cong\cO_x[m]$, for some closed point $x\in X$ and $m\in\ZZ$.
\end{lemma}

\begin{proof}
Since $\cL$ is very ample, there exists a nonzero morphism
\[
f\colon F\otimes\cL^\vee\longrightarrow F.
\]
Since $F$ is a simple object in $\cP(\phi)$ and $F\otimes\cL^\vee\in\cP(\phi)$, $f$ must be surjective.
Since $v(F)=v(F\otimes\cL^\vee)$, $f$ must then be an isomorphism.
Since $F$ is $\sigma$-stable, we must have $\Hom^{<0}(F,F)=0$ and $\Hom(F,F)=\KK$.
The conclusion then follows from Lemma~\ref{lem:AuxiliaryBayerSkyscraper1}\ref{enum:AuxiliaryBayerSkyscraper1iii}.
\end{proof}

We apply the above lemma with the Bayer property, as follows.

\begin{lemma}\label{lem:BayerCharacterizeStabilitySkyscraper}
Let $X$ be a projective scheme over $\KK=\overline{\KK}$, and let $\cL$ be a very ample invertible sheaf on $X$.
Let $\sigma=(Z,\cP)$ be a numerical stability condition on $\Db(X)$ with respect to $(\Lambda,v)$ which satisfies 
$\cP\prec\cP\otimes\cL^\vee[1]$. Let $F\in\Db(X)$ be a $\sigma$-stable object such that $v(F)=v(F\otimes\cL^\vee)$.
Then $F\cong\cO_x[m]$, for some closed point $x\in X$ and $m\in\ZZ$.
\end{lemma}

\begin{proof}
Let $\phi\in\RR$ be such that $F\in\cP(\phi)$.
By Proposition~\ref{prop:lvlimpliesbayer}, we have
\[
F\otimes\cL^\vee\in\cP((\phi-1,\phi]).
\]
By assumption, $v(F\otimes\cL^\vee)=v(F)$; hence, we deduce that $F\otimes\cL^\vee\in\cP(\phi)$.
We can then conclude by Lemma~\ref{lem:AuxiliaryBayerSkyscraper2}.
\end{proof}

We can now characterize skyscraper sheaves:

\begin{proposition}\label{prop:BayerCharacterizeStabilitySkyscraper}
Let $X$ be a smooth connected projective scheme over $\KK=\overline{\KK}$, and let $\cL$ be a very ample invertible sheaf on $X$.
Let $\sigma=(Z,\cP)$ be a numerical stability condition on $\Db(X)$ with respect to $(\Lambda,v)$ which satisfies $\cP\prec\cP\otimes\cL^\vee[1]$.
We assume that the autoequivalence $(-)\otimes\cL$ of $\Db(X)$ preserves $(\Lambda,v)$.
Let $\bv\coloneqq kv(\cO_{x_0})\in\Lambda$, where $x_0\in X$ is a closed point and $k\in\ZZ$.
Then an object $F\in\Db(X)$ such that $v(F)=\bv$ is $\sigma$-stable if and only if $k=\pm1$ and $F\cong\cO_x[m]$, for some closed point $x\in X$ and $m\in\ZZ$.
In such a case, $m$ is even if and only if $k=1$.
\end{proposition}

\begin{proof}
Let $F$ be a $\sigma$-stable object with $v(F)=\bv$.
Since $(-)\otimes\cL$ preserves $(\Lambda,v)$, we deduce that $v(F\otimes\cL^\vee)=\bv=v(F)$.
Hence, by Lemma~\ref{lem:BayerCharacterizeStabilitySkyscraper}, since $X$ is connected and $v$ numerical, we get immediately that $k=\pm1$ and $F\cong\cO_x[m]$, for some closed point $x\in X$ and $m\in\ZZ$.
If $v(F)=v(\cO_x)$, the shift must be even.

Conversely, by Proposition~\ref{prop:BayerGeometric}, all skyscraper sheaves are $\sigma$-semistable; we need to show they are $\sigma$-stable.
Let $x\in X$ and let $\phi\in\RR$ such that $\cO_x\in\cP(\phi)$.
Assume that there exists a short exact sequence:
\begin{equation}\label{eq:414}
0\longrightarrow A\longrightarrow\mathcal{O}_x\longrightarrow B\longrightarrow 0,
\end{equation}
with $A,B\in\cP(\phi)$, $A$ $\sigma$-stable.
Again, by Proposition~\ref{prop:lvlimpliesbayer}, we have
\[
A\otimes\cL^\vee,B\otimes\cL^\vee\in\cP((\phi-1,\phi]).
\]
Thus, tensoring~\eqref{eq:414} with $\cL^\vee$ yields an exact sequence in $\cP((\phi-1,\phi])$
\[
0\longrightarrow A\otimes\cL^\vee\longrightarrow\mathcal{O}_x\longrightarrow B\otimes\cL^\vee\longrightarrow 0. 
\]
Since $\mathcal{O}_x\in \cP_\sigma(\phi)$, we deduce that
\[
A\otimes\cL^\vee,B\otimes\cL^\vee\in\cP(\phi).
\]

Iterating, we obtain short exact sequences in $\cP(\phi)$
\[
0\longrightarrow A\otimes(\cL^\vee)^{\otimes l}\longrightarrow\mathcal{O}_x\longrightarrow B\otimes(\cL^\vee)^{\otimes l}\longrightarrow 0,
\]
for all $l\geq0$.
Since $\cP(\phi)$ is of finite length, there are only finitely many possible  classes in $\Lambda$ of subobjects of a given semistable object. 
Hence, there exist $l_1,l_2\geq0$, $l_1\neq l_2$, such that
\[
v(A\otimes(\cL^\vee)^{\otimes l_1})=v(A\otimes(\cL^\vee)^{\otimes l_2}).
\]
Since $(-)\otimes\cL$ preserves $(\Lambda,v)$, we deduce that
\[
v(A)=v(A\otimes(\cL^\vee)^{\otimes l}),
\]
for some $l\geq1$.
By applying Lemma~\ref{lem:AuxiliaryBayerSkyscraper2} to $\cL^{\otimes l}$, we deduce that $A\cong\cO_y[m]$, for some closed point $y\in X$.
Since there is a nonzero morphism $A\to\cO_x$, we must have $y=x$, and so $\cO_x$ is $\sigma$-stable, as we wanted. 
\end{proof}

The previous proposition applies in particular when $\Lambda=\Knum(X)$ or $\Lambda=\Lambda_{H}$ of Section~\ref{subsec:ProjectiveSchemes}, where $H$ denotes the numerical first Chern class of $\cL$: in this case, $(-)\otimes\cL$ preserves indeed $(\Lambda,v)$.

A small issue with Proposition~\ref{prop:BayerCharacterizeStabilitySkyscraper} is that in this generality we have a priori no control on the phases of skyscraper sheaves; for instance, if $X$ is not connected.
In the smooth connected case, when $v$ is numerical, if all skyscraper sheaves are $\sigma$-stable, then they have the same phase, by~\cite[Proposition 2.9]{FLZ:Albanese}\footnote{In \emph{loc.~cit.}, the result is proved in the assumption that $\KK=\CC$, but the proof does not use this.}.
More precisely, we can rephrase this as follows:

\begin{lemma}\label{lem:MptContX}
Let $X$ be a smooth connected projective scheme over $\KK=\overline{\KK}$.
Let $\sigma=(Z,\cP)\in\Stab_{(\Lambda,v)}(\Db(X)/\KK)$ be a numerical stability condition on $\Db(X)$ over $\Spec(\KK)$ with respect to $(\Lambda,v)$.
Assume that all skyscraper sheaves are $\sigma$-stable.
Let $\bv=v(\cO_{x_0})\in\Lambda$, for some closed point $x_0\in X$, and $\phi\in\RR$ such that $Z(\bv)\in\RR_{>0}\cdot e^{i\pi\phi}$.
Then the moduli space $M^{\mathrm{st}}_{\sigma}(\bv,\phi)$ has a connected component isomorphic to $X$.
\end{lemma}

\begin{proof}
By~\cite[Proposition 2.9]{FLZ:Albanese}, under our assumptions all skyscraper sheaves have the same phase.
Moreover, their moduli space is smooth and projective and isomorphic to $X$.
On the other hand, the moduli stack $\cM^{\mathrm{st}}_{\sigma}(\bv,\phi)$ represents the same functor at a skyscraper sheaf; thus $X$ identifies with a connected component of the associated moduli space.
\end{proof}

As a consequence of Proposition~\ref{prop:BayerCharacterizeStabilitySkyscraper}, we can say more if the Bayer property is satisfied.

\begin{proposition}\label{prop:Mpt=X}
Let $X$ be a smooth connected projective scheme over $\KK=\overline{\KK}$, and let $\cL$ be a very ample invertible sheaf on $X$.
Let $\sigma=(Z,\cP)\in\Stab_{(\Lambda,v)}(\Db(X)/\KK)$ be a numerical stability condition on $\Db(X)$ over $\Spec(\KK)$ with respect to $(\Lambda,v)$.
We assume that $\sigma$ satisfies $\cP\prec\cP\otimes\cL^\vee[1]$, and that the autoequivalence $(-)\otimes\cL$ of $\Db(X)$ preserves $(\Lambda,v)$.
Let $\bv=v(\cO_{x_0})\in\Lambda$, for some closed point $x_0\in X$, and $\phi\in\RR$ such that $Z(\bv)\in\RR_{>0}\cdot e^{i\pi\phi}$.
Then
\[
M^{\mathrm{st}}_{\sigma}(\bv,\phi)\cong X.
\]    
\end{proposition}

\begin{proof}
Let $F\in\cP(\phi)$ be a $\sigma$-stable object with $v(F)=\bv$.
By Proposition~\ref{prop:BayerCharacterizeStabilitySkyscraper}, $F\cong\cO_x[2m]$, for some closed point $x\in X$ and $m\in\ZZ$.
Conversely, all skyscraper sheaves are $\sigma$-stable.
The conclusion follows from Lemma~\ref{lem:MptContX}.
\end{proof}

\subsection{The Bondal--Orlov theorem} 
Proposition~\ref{prop:Mpt=X} is closely related to the Bondal--Orlov reconstruction theorem~\cite{BO}.
For instance, we can immediately get back a weaker form of it, by using stability conditions: if $X$ is smooth and projective over $\KK=\overline{\KK}$ with canonical bundle which is either ample or antiample, then $X$ can be uniquely recovered from $\Db(X)$.
Indeed, the canonical bundle $\omega_X$ can be recovered from the Serre functor of $X$ and then we can apply Proposition~\ref{prop:Mpt=X} with $\cL$ a suitable power of $\omega_X$ or of $\omega_X^\vee$, where $\sigma$ is a stability condition from Theorem~\ref{thm:ProjectiveScheme} (here we are using Corollary~\ref{cor:Stabdagger=Stabdagger}).

Actually, the whole Bondal--Orlov Theorem can be deduced from Lemma~\ref{lem:MptContX}, together with standard arguments; the only new input here is that we can use the theory of moduli spaces. 

\begin{theorem}[Bondal--Orlov]\label{thm:BOreconstruction}
Let $X$ be a smooth connected projective scheme over an algebraically closed field $\KK=\overline{\KK}$.
Assume that either the canonical or the anticanonical line bundle is ample.
Let $X'$ be a projective scheme over $\KK$ such that we have a $\KK$-linear exact equivalence $\Db(X')\cong\Db(X)$.
Then $X'$ is isomorphic to $X$ over $\KK$.
\end{theorem}

\begin{proof}
Since $X$ is smooth over $\KK$, it is a general fact (see, e.g.,~\cite[Lemma 4.9(6),(7)]{NCHPD}) that $X'$ is also smooth. 
Thus, if $\Phi\colon\Db(X)\xrightarrow{\cong}\Db(X')$ is a triangulated equivalence, then by~\cite{Orlov}, it is of Fourier--Mukai type.
The connectedness of $X'$ follows from that of $X$, e.g. by considering Hochschild cohomology.
Since the Serre functor is invariant with respect to $\Phi$, i.e., we have $\Phi\circ\mathsf{S}_X=\mathsf{S}_{X'}\circ\Phi$, we also deduce that $\dim(X)=\dim(X')$.
Moreover, if $\cL$ denotes either $\omega_X^{\otimes m}$ or $(\omega_X^\vee)^{\otimes m}$, for $m\gg1$, so that $\cL$ is very ample on $X$, since $\dim(X)=\dim(X')$, we then have
\[
\Phi\circ((-)\otimes\cL)\circ\Phi^{-1}=(-)\otimes\cL',
\]
where $\cL'$ is an invertible sheaf on $X'$.
For any closed point $x'\in X'$, we consider the object $F\coloneqq\Phi^{-1}(\cO_{x'})\in\Db(X)$.
Since $\cO_{x'}\otimes\cL'\cong\cO_{x'}$, we deduce that $F\cong F\otimes\cL$.
Moreover, from the corresponding property for $\cO_{x'}$, we have that $\Hom^{<0}(F,F)=0$ and $\Hom(F,F)=\KK$.
By Lemma~\ref{lem:AuxiliaryBayerSkyscraper1}\ref{enum:AuxiliaryBayerSkyscraper1iii}, we deduce that $F\cong\cO_x[m]$, for some $m\in\ZZ$.

Let $H_X$ be the numerical first Chern class of $\cL$. 
By Theorem~\ref{thm:ProjectiveScheme}, there exists a geometric stability condition $\sigma\in\Stabdagger_{H_X}(\Db(X))$ such that $\sigma$ satisfies the Bayer property with respect to $\cL^\vee$ and $l=1$.
Moreover, as observed before, the autoequivalence $(-)\otimes\cL$ preserves $(\Lambda_{H_X},v_{H_X})$. 
Let us consider $\sigma' \coloneqq \Phi(\sigma)$.
By Corollary~\ref{cor:Stabdagger=Stabdagger} and since exact equivalences preserve moduli spaces, $\sigma'$ is a numerical stability condition on $\Db(X')$ over $\KK$ with respect to $(\Lambda_{H_X},v_{H_X}\circ\Phi^{-1}_*)$.
By the previous argument, skyscraper sheaves of closed points in $X'$ are $\sigma'$-stable.
By Lemma~\ref{lem:MptContX}, if we let $\bv'=v_{H_X} \circ\Phi^{-1}_*(\cO_{x'})$ and $\phi'$ be the phase of skyscraper sheaves of closed points on $X'$ with respect to $\sigma'$, then the moduli space $M_{\sigma'}(\bv',\phi')$ has a connected component isomorphic to $X'$. 
On the other hand, by Proposition~\ref{prop:Mpt=X}, this moduli space is isomorphic to $X$. 
This proves the theorem. 
\end{proof}


\newpage 

\part{Appendices}\label{part:Appendices}

\appendix 


\section{Inducing t-structures and slicings}\label{app:Inducing}

In this appendix, we review some results on inducing t-structures and slicings. 
We mostly follow~\cite[Section 2]{P-t-structures}, but with some clarifications and improvements to cover the situations of interest in this paper. 

\subsection{Preliminaries on t-structures} 
\label{app-t-structures-preliminaries}
We begin by briefly reviewing some basic definitions. 

\begin{definition}
    Let $\cD$ be a triangulated category. 
    A \emph{t-structure} on $\cD$ is a pair $\tau = (\cD^{\leqslant 0}, \cD^{\geqslant 0})$ of strictly full subcategories satisfying: 
    \begin{enumerate}
        \item There are inclusions $\cD^{\leqslant 0}[1] \subset \cD^{\leqslant 0}$ and 
        $\cD^{\geqslant 0}[-1] \subset \cD^{\geqslant 0}$. 
        \item \label{t-structure-Hom-vanishing} If $F \in \cD^{\leqslant 0}$ and $F' \in \cD^{\geqslant 0}[-1]$, then $\Hom(F, F') = 0$. 
        \item For every object $F \in \cD$, there exists a distinguished triangle 
        \begin{equation}
        \label{truncation-triangle} 
            F^{\leqslant 0} \to F \to F^{\geqslant 1} 
        \end{equation}
    such that $F^{\leqslant 0}\in \cD^{\leqslant 0}$ and $F^{\geqslant 1}\in \cD^{\geqslant 1}$.
    \end{enumerate}
\end{definition}

It follows from condition~\eqref{t-structure-Hom-vanishing} that the distinguished triangle~\eqref{truncation-triangle} is unique and functorial, and thus that there exist \emph{truncation functors} $\tau^{\leqslant 0} , \tau^{\geqslant 1} \colon \cD \to \cD$ given objectwise by 
\begin{equation*}
\tau^{\leqslant 0}(F) = F^{\leqslant 0} \quad \text{and} \quad 
\tau^{\geqslant 1}(F) = F^{\geqslant 1}. 
\end{equation*}
Further, for any $n \in \bZ$ we define truncation functors $\tau^{\leqslant n}, \tau^{\geqslant n} \colon \cD \to \cD$ by 
\begin{equation*}
    \tau^{\leqslant n} \coloneqq [-n] \circ \tau^{\leqslant 0} \circ [n] \quad \text{and} 
    \quad 
    \tau^{\geqslant n} \coloneqq [-(n-1)] \circ \tau^{\geqslant 1} \circ [n-1]. 
\end{equation*}
These functors take values in the subcategories 
\begin{equation*}
    \cD^{\leqslant n} = \cD^{\leqslant 0}[-n] \quad \text{and} \quad 
    \cD^{\geqslant n} = \cD^{\geqslant 0}[-n]. 
\end{equation*}
The \emph{heart} of $\tau$ is the intersection 
\begin{equation*}
    \cD^{\heartsuit} = \cD^{\leqslant 0} \cap \cD^{\geqslant 0}, 
\end{equation*}
which is an abelian category \cite{BBDG}. 
For $n \in \bZ$ we have cohomology functors 
\begin{equation*}
    {^\tau}\cH^{n} \coloneqq \tau^{\leqslant 0} \circ \tau^{\geqslant 0} \circ [n] \colon \cD \to \cD^{\heartsuit}, 
\end{equation*}
such that distinguished triangles in $\cD$ are transformed into long exact sequences on cohomology objects in $\cD^{\heartsuit}$. 

The following finiteness conditions play an important role. 
\begin{definition}
    Let $\tau = (\cD^{\leqslant 0}, \cD^{\geqslant 0})$ be a t-structure on a triangulated category $\cD$. 
    \begin{enumerate}
    \item $\tau$ is \emph{nondegenerate} if $\bigcap_{n\in\bZ}\cD^{\geqslant n}= \bigcap_{n\in\bZ}\cD^{\leqslant n}=\{0\}$. 
    \item $\tau$ is \emph{bounded} if $\cD = \bigcup_{m,n \in \bZ} \cD^{\geqslant m} \cap \cD^{\leqslant n}$. 
    \end{enumerate}
\end{definition}

\begin{remark}
\label{remark-bounded-t-structures}
    It follows easily from the definitions that $\tau$ is bounded if and only if it is nondegenerate and for any $F \in \cD$, the cohomology objects ${^\tau}\cH^n(F)$ vanish for all but finitely many $n \in \bZ$. 
\end{remark}

\begin{definition}
    Let $\cD$ and $\cD'$ be triangulated categories with t-structures $\tau = (\cD^{\leqslant 0}, \cD^{\geqslant 0})$ and $\tau'  = (\cD'^{\leqslant 0}, \cD'^{\geqslant 0})$. 
    Let $\Phi \colon \cD \to \cD'$ be a triangulated functor. 
    \begin{enumerate}
        \item $\Phi$ is \emph{left t-exact} with respect to $\tau$ and $\tau'$ if $\Phi(\cD^{ \geqslant 0}) \subset \cD'^{\geqslant 0}$. 
        \item $\Phi$ is \emph{right t-exact} with respect to $\tau$ and $\tau'$ if $\Phi(\cD^{ \leqslant 0}) \subset \cD'^{\leqslant 0}$.  
        \item $\Phi$ is \emph{t-exact} with respect to $\tau$ and $\tau'$ if it is both left and right t-exact. 
    \end{enumerate}
\end{definition}

Given a family of schemes over a base, it is natural to consider t-structures on the derived category of the total space which are suitably compatible with the structure morphism to the base. 
This is formalized by the notion of local t-structure, first introduced in~\cite[Section 2.3]{P-t-structures}.

\begin{definition}\label{def:SlocalTstructure}
Let $\pi \colon X\to S$ be a morphism of schemes.
A t-structure $\tau$ on $\Db(X)$ is \emph{$S$-local} if for every quasi-compact open subset $U\subset S$, there exists a t-structure $\tau_U$ on $\Db(X_U)$ such that the restriction functor $\Db(X)\to\Db(X_U)$ is t-exact with respect to $\tau$ and $\tau_U$.  
\end{definition}

\begin{remark}
\label{remark-local-t-structures}
Let us collect some observations about the above notion: 
\begin{enumerate}
\item When it exists, the t-structure $\tau_U$ on $\Db(X_U)$ above is uniquely determined. 
If $X$ is noetherian, then this follows from the essential surjectivity of the restriction functor $\Db(X) \to \Db(X_U)$. 
In general, this can be shown using the compactly generated t-structures $\widehat{\tau}$ on $\Dqc(X)$ and $\widehat{\tau_U}$ on $\Dqc(X_U)$ given by Lemma~\ref{lem:Neeman} below, together with the essential surjectivity of the restriction functor $\Dqc(X) \to \Dqc(X_U)$ \cite[\href{https://stacks.math.columbia.edu/tag/08ED}{Tag 08ED}]{stacks-project}. 

\item  \label{remark-local-over-affine}  There is a simple criterion for $S$-locality of a bounded t-structure when $S$ is quasi-projective over an affine scheme. 
Namely, in this case, by~\cite[Theorem 2.3.2]{P-t-structures} a bounded t-structure $\tau$ on $\Db(X)$ is $S$-local if and only if there exists an ample line bundle $\cL$ on $S$ such that $- \otimes \pi^*\cL \colon \Db(X) \to \Db(X)$ is left t-exact with respect to $\tau$; moreover, in this case $-\otimes\pi^* \cL$ is in fact t-exact for every ample line bundle $\cL$ on $S$. 
In particular, if $S$ is affine, then any bounded t-structure on $\Db(X)$ is $S$-local. 

\item    There are more general versions of the notion of a local t-structure, which apply to triangulated categories $\cD$ that are not necessarily of the form $\Db(X)$ for a morphism $X \to S$. For instance, the case where $\cD$ is a suitable semiorthogonal component of $\Db(X)$ is treated in \cite[Section~4.2]{stability-families}. 
\end{enumerate} 
\end{remark}

Below, we will discuss a strategy for constructing t-structures on small categories by restriction from t-structures on a large category; here, following the terminology of \cite[Section 3]{KLP}, ``restriction'' is meant in the following precise sense. 

\begin{definition}
Let $\cD$ be a triangulated category with a t-structure $\tau = (\cD^{\leqslant 0}, \cD^{\geqslant 0})$. 
Let $\cC \subset \cD$ be a full triangulated subcategory. 
We say $\tau$ \emph{restricts to a t-structure} on $\cC$ if the pair $\tau \vert_{\cC} = (\cC^{\leqslant 0}, \cC^{\geqslant 0})$ with 
\begin{equation*}
    \cC^{\leqslant 0} = \cD^{\leqslant 0}\cap \cC \quad \text{and} \quad 
    \cC^{\geqslant 0} = \cD^{\geqslant0} \cap \cC
\end{equation*}
defines a t-structure on $\cC$. 
\end{definition}

\begin{remark}
It is easy to see that $\tau$ restricts to a t-structure on $\cC$ if and only if either of the truncation functors $\tau^{\leqslant 0}$ or $\tau^{\geqslant0}$ preserves $\cC$, in which case all of the truncation functors preserve $\cC$. 
\end{remark}

\subsection{Inducing t-structures} 

Our next goal is to describe a powerful method for constructing t-structures; this relies on the notion of a t-structure generated by a set of compact objects.

\begin{definition}
    Let $\cD$ be a triangulated category which admits small coproducts. 
    An object $F \in \cD$ is called \emph{compact} if for any collection $\set{ G_i }$ of objects in $\cD$, the natural map 
\begin{equation*}
    \bigoplus_{i \in I} \Hom(F, G_i) \to 
    \Hom \left(F, \bigoplus_{i \in I} G_i \right) 
\end{equation*}
is an isomorphism. 
\end{definition}

\begin{example}
    By~\cite{BVDB}, if $X$ is a quasi-compact quasi-separated scheme, then $\Dperf(X)$ is the subcategory of compact objects in $\Dqc(X)$, and it generates $\Dqc(X)$.
\end{example}

\begin{definition}
\label{definition-CG-t-structure}
    Let $\cD$ be a triangulated category which admits small coproducts. 
    A t-structure $\tau = (\cD^{\leqslant 0}, \cD^{\geqslant 0})$ on $\cD$ is called \emph{compactly generated} if there exists a set $\cG$ of compact objects in $\cD$ such that 
    \begin{equation*}
    \cD^{\leqslant 0} = \Coprod(\cG), 
    \end{equation*} 
    where $\Coprod(\cG)$ denotes the smallest full subcategory of $\cD$ containing $\cG$ and closed under the formation of coproducts and extensions. 
\end{definition}

\begin{remark}
\label{remark-CG-t-structures}
    Let us record a few observations about compactly generated t-structures: 
    \begin{enumerate}
    \item In the above definition, 
    we may always assume that $\cG$ satisfies $\cG[1] \subset \cG$, by adding in shifts of objects in $\cG$ if necessary. 
    \item The subcategory of coconnective objects is given explicitly by 
    \begin{equation}
    \label{Dgeq-generated-G}
            \cD^{\geqslant 0} = \set{ F \in \cD \st \Hom(G, F) = 0 \textup{ for all } G \in \cG[1]}, 
    \end{equation}
    because the property $\Hom(G, F) = 0$ is stable under coproducts and extensions in the variable~$G$. 
    \item \label{Dgeq-coproduct-closed} The subcategory $\cD^{\geqslant 0}$ is closed under coproducts, due to the formula~\eqref{Dgeq-generated-G} and the fact that $\cG$ consists of compact objects. 
    \end{enumerate}
\end{remark}

The importance of this notion is due to the following fundamental existence result. 

\begin{theorem}[{\cite[Theorem A.1]{AJT:tstructures} and~\cite[Theorem 2.3.3]{CNS:Passage}}]
\label{theorem-generating-t-structures}
Let $\cD$ be a triangulated category which admits small coproducts. 
Let $\cG$ be a set of compact objects in $\cD$ which satisfies $\cG[1] \subset \cG$. 
Then there exists a t-structure $\tau = (\cD^{\leqslant 0}, \cD^{\geqslant 0})$ on $\cD$ which is compactly generated by $\cG$, i.e., $\cD^{\leqslant 0} = \Coprod(\cG)$.
\end{theorem}

In the geometric setting where $X$ is a noetherian scheme, this leads to the following strategy for constructing t-structures on $\Db(X)$: first, for a wisely chosen $\cG$ use the above theorem to obtain a compactly generated t-structure on $\Dqc(X)$; then, prove that this t-structure restricts to one on $\Db(X)$. 
Before employing this strategy to prove an inducing result for t-structures, we explain conversely how to pass from bounded t-structures on $\Db(X)$ to compactly generated ones on $\Dqc(X)$. 

\begin{lemma}[Neeman]\label{lem:Neeman}
Let $X$ be a noetherian scheme. Let $\tau=(\cD^{\leqslant0},\cD^{\geqslant0})$ be a bounded t-structure on $\Db(X)$.
Then there exists a t-structure $\widehat{\tau}=(\widehat{\cD}^{\leqslant0},\widehat{\cD}^{\geqslant0})$ on $\Dqc(X)$ such that
\begin{enumerate}[{\rm (i)}]
    \item\label{enum:Neeman1} $\widehat{\cD}^{\leqslant0}=\Coprod(\cD^{\leqslant0})$;
    \item\label{enum:Neeman2} $\widehat{\tau}$ is compactly generated; and 
    \item\label{enum:Neeman3} $\widehat{\tau}$ restricts to the t-structure $\tau$ on $\cD$, i.e.,  $\cD^{\leqslant0}=\widehat{\cD}^{\leqslant0}\cap\Db(X)$ and $\cD^{\geqslant1}=\widehat{\cD}^{\geqslant1}\cap\Db(X)$.
\end{enumerate}
\end{lemma}

We will refer to $\widehat{\tau}$ as the \emph{Ind-extended t-structure} associated to $\tau$.

\begin{proof}
Properties~\ref{enum:Neeman1} and~\ref{enum:Neeman2} follow directly from~\cite[Lemma 4.1 and Remark 4.2]{Neeman:BoundedTstructPerfect} (which in particular rely on Theorem~\ref{theorem-generating-t-structures}).

To prove~\ref{enum:Neeman3}, we first observe that
$\Hom(F,F') = 0$ for all $F \in \widehat{\cD}^{\leqslant 0}$ and $F' \in \cD^{\geqslant 1}$. Indeed, the claim holds for $F \in \cD^{\leqslant 0}$ because $\tau$ is a t-structure, and hence for all $F \in \widehat{\cD}^{\leqslant 0}$ because the property is stable under coproducts and extensions. 

Now we prove $\cD^{\leqslant0}=\widehat{\cD}^{\leqslant0}\cap\Db(X)$. The forward inclusion is clear by definition. 
On the other hand, by the above observation, $\widehat{\cD}^{\leqslant0}\cap\Db(X)$ is contained in the left $\Hom$-orthogonal to $\cD^{\geqslant 1} \subset \Db(X)$, which is equal to $\cD^{\leqslant 0}$ because $\tau$ is a t-structure. 

The equality $\cD^{\geqslant1}=\widehat{\cD}^{\geqslant1}\cap\Db(X)$ follows similarly. 
Namely, by the above orthogonality observation, $\cD^{\geqslant1}$ is contained in the right $\Hom$-orthogonal to $\widehat{\cD}^{\leqslant 0} \subset \widehat{\cD}$, which is equal to $\widehat{\cD}^{\geqslant1}$ because $\widehat{\tau}$ is a t-structure. 
On the other hand, $\widehat{\cD}^{\geqslant1}\cap\Db(X)$ is contained in the right $\Hom$-orthogonal to $\cD^{\leqslant 0} \subset \Db(X)$ (because $\cD^{\leqslant 0} \subset \widehat{\cD}^{\leqslant 0}$), which is equal to $\cD^{\geqslant 1}$ because $\tau$ is a t-structure. 
\end{proof}

\begin{example}
\label{example-ind-completion-tauX}
    In the situation of Lemma~\ref{lem:Neeman}, if $\tau = \tau_X$ is the standard t-structure on $\Db(X)$, then $\widehat{\tau}$ is the standard t-structure on $\Dqc(X)$. 
\end{example}

The formation of Ind-extended t-structures has good functorial behavior. 
\begin{lemma}
\label{lemma-ind-extension-exact-functor}
    Let $X$ and $Y$ be noetherian schemes. 
    Let $\tau_{X} = (\cD_{X}^{\leq 0}, \cD_{X}^{\geq 0})$ be a bounded t-structure on $\Db(X)$ and \mbox{$\tau_Y = (\cD_{Y}^{\leq 0}, \cD_{Y}^{\geq 0})$} be bounded t-structure  on $\Db(Y)$, and let $\widehat{\tau}_X$ and $\widehat{\tau}_Y$ be their Ind-extensions to t-structures on $\Dqc(X)$ and $\Dqc(Y)$. 
    Let $\Phi \colon \Dqc(X) \to \Dqc(Y)$ be a triangulated functor satisfying the following assumptions: 
    \begin{enumerate}[{\rm (i)}]
        \item $\Phi$ commutes with coproducts. 
        \item $\Phi$ restricts to a triangulated functor $\Phi^{\mathrm{b}} \colon \Db(X) \to \Db(Y)$ on bounded derived categories of coherent sheaves. 
    \end{enumerate}  
    Then 
    the following hold: 
    \begin{enumerate}[{\rm (1)}]
        \item \label{Phib-right-t-exact} If $\Phi^b$ is right t-exact with respect to $\tau_X$ and $\tau_Y$, then $\Phi$ is right t-exact with respect to  $\widehat{\tau}_X$ and $\widehat{\tau}_Y$.
        \item \label{Phib-left-t-exact} If $\Phi$ admits a left adjoint $\Phi^L \colon \Dqc(Y) \to \Dqc(X)$ and $\Phi^b$ is left t-exact with respect to $\tau_X$ and $\tau_Y$, then $\Phi$ is left t-exact with respect to  $\widehat{\tau}_X$ and~$\widehat{\tau}_Y$.
    \end{enumerate}
\end{lemma}

\begin{proof}
    Part~\ref{Phib-right-t-exact} follows immediately from Lemma~\ref{lem:Neeman}\ref{enum:Neeman1} and the assumption that $\Phi$ commutes with coproducts. 

    For~\ref{Phib-left-t-exact}, 
    let $\cG_Y \subset \cD_Y^{\leq 0}$ be a set of compact objects in $\Dqc(Y)$ which generate $\widehat{\tau}_Y$. 
    The functor $\Phi^L$ preserves compact objects, being the left adjoint of a coproduct preserving functor; in particular, $\Phi^L(\cG_Y) \subset \Dperf(X) \subset \Db(X)$. 
    Moreover, $\Phi^L(\cG_Y) \subset \cD_X^{\leq 0}$; indeed, by adjunction this is equivalent to the vanishing $\Hom(G, \Phi(F)) = 0$ for all $G \in \cG_Y$ and $F \in \cD_X^{ \geq 1}$, which holds because $\cG_Y \subset \cD_Y^{\leq 0}$ and $\Phi(F) \in \cD_Y^{\geq 1}$ by left t-exactness of $\Phi^b$. 
    
    Now the left t-exactness of $\Phi$ follows by a similar argument. Namely, by~\eqref{Dgeq-generated-G} and adjunction, the claim is equivalent to the vanishing  $\Hom(\Phi^L(G), F) = 0$ for all $G \in \cG_Y[1]$ and $F \in \widehat{\cD}_X^{\geq 0}$, which holds because by the previous paragraph $\Phi^L(G) \in \cD_X^{\leq -1}$.     
\end{proof}

Now we can state the key result about inducing t-structures on the source of a suitable functor from one on the target. 

\begin{theorem}[Polishchuk]\label{thm:Polishchuk}
Let $S$ be a quasi-projective scheme over a noetherian affine scheme.
Let $\pi_X \colon X \to S$ and $\pi_Y \colon Y \to S$ be noetherian schemes over $S$. 
Let $\tau_Y=(\cD_Y^{\leqslant0},\cD_Y^{\geqslant0})$ be a bounded t-structure on $\Db(Y)$ which is $S$-local, and let  $\widehat{\tau}_Y$ be its Ind-extension to $\Dqc(Y)$.
Let $\Phi\colon\Dqc(X)\to\Dqc(Y)$ be a triangulated functor satisfying the following assumptions: 
\begin{enumerate}[{\rm (i)}]
    \item \label{enum:Polishchuk0} $\Phi$ is $S$-linear in the sense that for every $F\in\Dqc(X)$ and $C\in\Dperf(S)$ there exists an isomorphism $\Phi(F\otimes \pi_X^*C)\cong\Phi(F)\otimes \pi_Y^*C$;\footnote{In the definition of $S$-linearity, one typically requires that the isomorphism is suitably functorial. This stronger condition always holds in practice, but we do not need it for the result; in fact, as the proof shows, we only need the existence of such an isomorphism for $C$ a fixed ample line bundle on $S$.}
    \item\label{enum:Polishchuk1} $\Phi$ commutes with coproducts and admits a left adjoint $\Phi^L\colon\Dqc(Y)\to\Dqc(X)$;
    \item\label{enum:Polishchuk2} $\Phi$ is conservative on $\Db(X)$, namely for $F\in\Db(X)$, $\Phi(F)=0$ implies $F=0$;
    \item\label{enum:Polishchuk3} for $F\in\Dqc(X)$, we have $\Phi(F)\in\Db(Y)$ if and only if $F\in\Db(X)$; and 
    \item\label{enum:Polishchuk4} $\Phi\circ\Phi^L$ is right t-exact with respect to $\widehat{\tau}_Y$. 
\end{enumerate}
Then there exists an $S$-local bounded t-structure $\tau_X=(\cD_X^{\leqslant0},\cD_X^{\geqslant0})$ on $\Db(X)$ such that
\begin{align}
\label{eq:Polishchuk-leq} \cD_X^{\leqslant0} & =\left\{F\in\Db(X)\st \Phi(F)\in\cD_Y^{\leqslant0}\right\} ,\\ 
\label{eq:Polishchuk-geq} \cD_X^{\geqslant0} & =\left\{F\in\Db(X)\st \Phi(F)\in\cD_Y^{\geqslant0}\right\}.
\end{align} 
In particular, $\Phi$ restricts to a t-exact functor $\Db(X) \to \Db(Y)$ with respect to $\tau_X$ and $\tau_Y$.
\end{theorem}

\begin{proof}
The proof follows along the lines of~\cite[Theorem 2.1.2]{P-t-structures}. 
We break the argument into several steps. 

\begin{step}{1}
\label{polishchuk-step-1} 
There is a compactly generated t-structure $\widehat{\tau}_X = (\widehat{\cD}_X^{\leqslant 0}, \widehat{\cD}_X^{\geqslant 0})$ on $\Dqc(X)$ such that 
\begin{equation}
\label{widehatDX-geq0}
    \widehat{\cD}_X^{\geqslant0}=\left\{F\in\Dqc(X)\st \Phi(F)\in\widehat{\cD}_Y^{\geqslant0}\right\},   
\end{equation}
and $\Phi$ is  t-exact with respect to $\widehat{\tau}_X$ and $\widehat{\tau}_Y$. 
\end{step}

By Lemma~\ref{lem:Neeman}, the t-structure $\widehat{\tau}_Y$ is compactly generated by a set $\cG_Y$ of objects in $\Dperf(Y)$, which we may assume satisfies $\cG_Y[1] \subset \cG_Y$. 
As $\Phi^L$ is the left adjoint of a coproduct preserving functor, it must send compact objects to compact objects. 
Thus $\cG_X = \Phi^L(\cG_Y)$ is a set of compact objects in $\Dqc(X)$ satisfying $\cG_X[1] \subset \cG_X$, so by  Theorem~\ref{theorem-generating-t-structures} it compactly generates  a t-structure  $\widehat{\tau}_X = (\widehat{\cD}_X^{\leqslant 0}, \widehat{\cD}_X^{\geqslant 0})$ on $\Dqc(X)$. 

The formula~\eqref{widehatDX-geq0} follows from adjunction and the formula~\eqref{Dgeq-generated-G} for the coconnective part of a compactly generated t-structure. 

On the other hand, since $\Phi$ commutes with coproducts, we have 
\begin{equation*}
    \Phi(\widehat{\cD}_X^{\leqslant 0}) 
    \subset 
    \Coprod(\Phi(\cG_X)). 
\end{equation*}
By our assumption~\ref{enum:Polishchuk4} that $\Phi \circ \Phi^L$ is right t-exact, we have $\Phi(\cG_X) \subset \widehat{\cD}_Y^{\leqslant 0}$. Since $\widehat{\cD}_Y^{\leqslant 0}$ is closed under coproducts and extensions, 
we conclude that 
\begin{equation*}
    \Phi(\widehat{\cD}_X^{\leqslant 0}) \subset \widehat{\cD}_Y^{\leqslant 0}. 
\end{equation*}
Together with~\eqref{widehatDX-geq0}, this shows that $\Phi$ is t-exact with respect to $\widehat{\tau}_X$ and $\widehat{\tau}_Y$. 

\begin{step}{2}
The t-structure $\widehat{\tau}_{X}$ restricts to a t-structure on $\Db(X)$, i.e., there is a t-structure $\tau_X=(\cD_X^{\leqslant0},\cD_X^{\geqslant0})$ on $\Db(X)$ given by 
\begin{equation*}
\cD_X^{\leqslant0}=\widehat{\cD}_X^{\leqslant0}\cap\Db(X) 
\quad \textup{and} \quad \cD_X^{\geqslant0}=\widehat{\cD}_X^{\geqslant0}\cap\Db(X).
\end{equation*}
\end{step} 

It suffices to prove that the truncation functor $\widehat{\tau}_X^{\leqslant 0}$ preserves the subcategory $\Db(X)$. 
Thus, by assumption~\ref{enum:Polishchuk3}, for $F \in \Db(X)$ it suffices to prove that $\Phi(\widehat{\tau}_X^{\leqslant 0}(F))$ is contained in $\Db(Y)$. 

By the t-exactness of $\Phi$ with respect to $\widehat{\tau}_X$ and $\widehat{\tau}_Y$, we have 
$$\Phi(\widehat{\tau}_X^{\leqslant 0}(F)) \cong \widehat{\tau}_Y^{\leqslant 0}(\Phi(F)).$$
By assumption~\ref{enum:Polishchuk3} again, the object $\Phi(F)$ is contained in $\Db(Y)$. 
By Lemma~\ref{lem:Neeman}\ref{enum:Neeman3}, the truncation functor $\widehat{\tau}_Y^{\leqslant 0}$ coincides with $\tau_Y^{\leqslant 0}$ on the subcategory $\Db(Y)$. 
Altogether, this proves that $\Phi(\widehat{\tau}_X^{\leqslant 0}(F))$ is contained in $\Db(Y)$, as needed. 

\begin{step}{3}
The formulas~\eqref{eq:Polishchuk-leq} and~\eqref{eq:Polishchuk-geq} for $\cD_X^{\leqslant 0}$ and $\cD_X^{\geqslant 0}$ hold. 
\end{step} 

Note that by assumption~\ref{enum:Polishchuk3}, $\Phi$ restricts to a functor $\Db(X) \to \Db(Y)$. 
Moreover, this restricted functor is t-exact with respect to $\tau_X$ and $\tau_Y$; indeed, since $\tau_X$ and $\tau_Y$ are the restrictions of $\widehat{\tau}_X$ and $\widehat{\tau}_Y$, this follows from the t-exactness of $\Phi$ with respect to $\widehat{\tau}_X$ and $\widehat{\tau}_Y$ proved in Step~\ref{polishchuk-step-1}.    

Now we prove~\eqref{eq:Polishchuk-leq}. 
By construction and the t-exactness of the restriction $\Db(X) \to \Db(Y)$ of $\Phi$, we have an inclusion 
\begin{equation*}
\cD_X^{\leqslant0} \subset \left\{F\in\Db(X)\st \Phi(F)\in\cD_Y^{\leqslant0}\right\}. 
\end{equation*} 
For the reverse inclusion, assume that $F$ is an object contained in the right-hand side. 
Then again by t-exactness, we have 
\begin{equation*}
    \Phi(\tau_X^{\geqslant 1}(F)) \cong \tau_Y^{\geqslant 1}(\Phi(F)) = 0,
\end{equation*}
which by~\ref{enum:Polishchuk2} implies $\tau_X^{\geqslant 1}(F) = 0$, i.e., $F \in \cD_X^{\leqslant 0}$. 

The formula~\eqref{eq:Polishchuk-geq} follows similarly.

\begin{step}{4}
The t-structure $\tau_X$ is $S$-local and bounded. 
\end{step} 

First, we observe that $\tau_X$ is nondegenerate. 
Indeed, since by assumption~\ref{enum:Polishchuk2} the functor $\Phi$ is conservative on $\Db(X)$, this follows from the formulas~\eqref{eq:Polishchuk-leq} and~\eqref{eq:Polishchuk-geq} and the nondegeneracy of $\tau_Y$. 
Similarly, we deduce that any object $F \in \Db(X)$ has finitely many nonzero cohomology objects with respect to $\tau_X$, so that $\tau_X$ is bounded by Remark~\ref{remark-bounded-t-structures}. 
Finally, by our assumption~\ref{enum:Polishchuk0} that $\Phi$ is $S$-linear and the criterion for $S$-locality of a t-structure from Remark~\ref{remark-local-t-structures}\eqref{remark-local-over-affine}, the $S$-locality of $\tau_X$ follows from that of $\tau_Y$ and the formula~\eqref{eq:Polishchuk-leq}.
\end{proof}

\begin{example}\label{ex:PullPushShriek}
Let $S$, $\pi_X \colon X \to S$, $\pi_Y \colon Y \to S$, and $\tau_Y$ be as in Theorem~\ref{thm:Polishchuk}. 
\begin{enumerate}[{\rm (1)}]
\item\label{enum:PullPushShriek1} Let  $f\colon X\to Y$ be a finite $S$-morphism. Then the pushforward functor $\Phi = f_*$ satisfies assumptions~\ref{enum:Polishchuk0}--\ref{enum:Polishchuk3} of Theorem~\ref{thm:Polishchuk}. 
Indeed,~\ref{enum:Polishchuk0} holds for any $S$-morphism $f$; \ref{enum:Polishchuk1} holds for any quasi-separated quasi-compact morphism $f$ \cite[\href{https://stacks.math.columbia.edu/tag/08DZ}{Tag 08DZ}]{stacks-project};~\ref{enum:Polishchuk2} holds for any affine morphism $f$; and~\ref{enum:Polishchuk3} is a standard property of pushforward along a finite morphism. 
Moreover, by the projection formula, the assumption~\ref{enum:Polishchuk4} amounts to the condition that $- \otimes f_*\cO_X \colon \Dqc(Y) \to \Dqc(Y)$ is right t-exact with respect to $\widehat{\tau}_Y$. 
\item\label{enum:PullPushShriek2} Let $f\colon X\to Y$ be a proper faithfully flat $S$-morphism. 
Then the pullback functor $\Phi = f^*$ satisfies assumptions~\ref{enum:Polishchuk0}--\ref{enum:Polishchuk3} of Theorem~\ref{thm:Polishchuk}. 
Indeed,~\ref{enum:Polishchuk0} again holds for any $S$-morphism, while \ref{enum:Polishchuk2} and~\ref{enum:Polishchuk3} are standard properties of pullback along a faithfully flat morphism. 
It remains to verify~\ref{enum:Polishchuk1}. 
Since $f^*$ is a left adjoint, it commutes with coproducts. 
Since $f$ is proper and flat, the pushforward functor $f_* \colon \Dqc(X) \to \Dqc(Y)$ takes compact objects (i.e., perfect complexes) to compact objects \cite[\href{https://stacks.math.columbia.edu/tag/0B6G}{Tag 0B6G}]{stacks-project}. 
Thus by \cite[Lemma~4.2 and Remark 4.3]{neeman-derived-cats-grothendieck-duality}, we find that $f^*$ admits a left adjoint $f_! \colon \Dqc(X) \to \Dqc(Y)$. 
In fact, by \cite[Theorem~3.3]{balmer}, the left adjoint $f_!$ is given explicitly by the formula 
\begin{equation*}
    f_{!}=f_*\circ(-\otimes\omega_f^{\bullet}),  
\end{equation*}
where $\omega_f^{\bullet}$ is the relative dualizing complex. 
In particular, we note that when $\omega_f^\bullet$ is a perfect complex, then $f_!$ restricts to a functor $\Db(X) \to \Db(Y)$ between bounded derived categories of coherent sheaves. 
\end{enumerate}
\end{example}

\subsection{Inducing slicings} 
Our next goal is to deduce a version of Theorem~\ref{thm:Polishchuk} for slicings. 
The formulation of the result uses the notion of a local slicing, which is an analog of Definition~\ref{def:SlocalTstructure}. 

\begin{definition}\label{def:SlocalSlicing}
Let $\pi \colon X\to S$ be a morphism of schemes.   
A slicing $\cP$ of $\Db(X)$ is \emph{$S$-local} if for every quasi-compact open subset $U\subset S$, there exists a slicing $\cP_U$ of $\Db(X_U)$ such that the restriction functor $\Db(X) \to \Db(X_U)$ sends $\cP(\phi)$ to $\cP_U(\phi)$ for every $\phi\in \mathbb{R}$.
\end{definition}

\begin{remark}
The analog of Remark~\ref{remark-local-t-structures} holds for local slicings. 
In particular, if $S$ is quasi-projective over an affine scheme, then 
$\cP$ is $S$-local if and only if for an ample line bundle $\cL$ on $S$ we have 
\begin{equation*}
\cP(\phi)\otimes \pi^*\cL\subset\cP(\leq\phi)
\end{equation*}
for all $\phi\in\RR$, and in this case we in fact have equality 
\begin{equation*}
    \cP(\phi)\otimes \pi^*\cL=\cP(\phi)
\end{equation*}
for all ample line bundles $\cL$ on $S$ and $\phi \in \RR$. 
As a special case, we see that if $S$ is affine, then any slicing of $\Db(X)$ is automatically $S$-local.
\end{remark}

\begin{remark}
As noted in~\eqref{tau-phi-in-text}, for any slicing $\cP$ of $\Db(X)$ and $\phi \in \RR$ there is an associated bounded t-structure on $\Db(X)$ given by 
\begin{equation*}
\tau_{\phi} \coloneqq  (\cP(>\phi),\cP(\leq \phi+1)). 
\end{equation*}
Swapping which inequalities are strict, there is also a variant bounded t-structure on $\Db(X)$ given by  
\begin{equation}\label{tau-prime}
\tau'_{\phi}\coloneqq\left(\cP(\geq\phi),\cP(<\phi+1)\right),
\end{equation}
which is sometimes convenient. 
It is easy to see that if $\cP$ is $S$-local, then so are $\tau_{\phi}$ and $\tau'_{\phi}$ for all $\phi \in \RR$. 
\end{remark}

\begin{definition}\label{definition-Ind-slicing}
    Let $X$ be a noetherian scheme. 
    Let $\cP$ be a slicing of $\Db(X)$. 
    Then for any $\phi \in \RR$, we denote by 
    \begin{equation*}
        \widehat{\tau}_{\phi} = (\widehat{\cP}(> \phi), \widehat{\cP}(\leq \phi+1)) \quad \text{and} \quad 
        \widehat{\tau'_{\phi}} = (\widehat{\cP}( \geq \phi), \widehat{\cP}(< \phi +1)) 
    \end{equation*}
    the Ind-extensions of $\tau_{\phi}$ and $\tau'_{\phi}$, as given by Lemma~\ref{lem:Neeman}. 
    Explicitly, we have 
    \begin{align*}
        \widehat{\cP}(> \phi) & = \Coprod(\cP(> \phi)) , \\ 
        \widehat{\cP}(\leq \phi+1) & = \set{F \in \Dqc(X) \st \Hom(G, F) = 0 \text{ for all } G \in \cP(>\phi +1)}, \\ 
        \widehat{\cP}(\geq \phi) & = \Coprod( \cP(\geq \phi)), \\ 
        \widehat{\cP}(< \phi + 1) & = \set{ F \in \Dqc(X) \st \Hom(G, F) = 0 \text{ for all } G \in \cP(\geq \phi + 1)}, 
    \end{align*}
where $\Coprod$ is the construction introduced in Definition~\ref{definition-CG-t-structure}. 
\end{definition}

Now we can state the slicing analog of Theorem~\ref{thm:Polishchuk}. 

\begin{corollary}\label{cor:Polishchuk}
Let $S$ be a quasi-projective scheme over a noetherian affine scheme. 
Let $\pi_X \colon X \to S$ and $\pi_Y \colon Y \to S$ be noetherian schemes over $S$.
Let $\cP_Y$ be an $S$-local slicing of $\Db(Y)$.
Let $\Phi \colon \Dqc(X) \to \Dqc(Y)$ be a triangulated functor satisfying the assumptions~\ref{enum:Polishchuk0}--\ref{enum:Polishchuk3} of Theorem~\ref{thm:Polishchuk}, as well as the following assumption: 
\begin{enumerate}[{\rm (v')}]
    \item \label{enum:PolishchukSlicing} For all $\phi \in \RR$, we have $\Phi\Phi^L(\cP_Y(\phi))\subset\widehat{\cP}_{Y}(\geq\phi)$. 
\end{enumerate}
Then the collection
\begin{equation}\label{eq:PolishchukSlicing2}
\cP_X(\phi)\coloneqq\left\{F\in\Db(X)\st\Phi(F)\in\cP_Y(\phi)\right\}    
\end{equation}
for $\phi\in\RR$ defines an $S$-local slicing $\cP_X$ of $\Db(X)$.
\end{corollary}

\begin{proof}
We may apply Theorem~\ref{thm:Polishchuk} to the family of t-structures $\tau'_{Y,\phi}$ associated to $\cP_Y$ as in~\eqref{tau-prime}.
Indeed, assumption \ref{enum:PolishchukSlicing} precisely corresponds to assumption~\ref{enum:Polishchuk4} of Theorem~\ref{thm:Polishchuk}.  
We obtain a family of $S$-local bounded t-structures $\tau'_{X,\phi} = (\cD_{X,\phi}^{\leq 0}, \cD_{X,\phi}^{\geq 1})$ on $\Db(X)$ for all $\phi\in\RR$, given by
\begin{equation}\label{eq:PolishchukSlicing3}
\begin{split}
&\cD_{X,\phi}^{\leqslant0}=\left\{F\in\Db(X)\st \Phi(F)\in\cP_{Y}(\geq\phi)\right\}\\
&\cD_{X,\phi}^{\geqslant1}=\left\{F\in\Db(X)\st \Phi(F)\in\cP_{Y}(<\phi)\right\}.  
\end{split}
\end{equation}

We want to check the conditions in Definition~\ref{def:slicing} for $\cP_X$ defined in~\eqref{eq:PolishchukSlicing2}.
First of all,~\ref{enum:slicing1} is immediate.
To prove~\ref{enum:slicing2}, let us take $\phi_1>\phi_2$ and $F_1\in\cP_X(\phi_1)$, $F_2\in\cP_X(\phi_2)$.
Since, by~\eqref{eq:PolishchukSlicing3}, $\cP_X(\phi_1)\subset\cD_{X,\phi_1}^{\leqslant0}$ and $\cP_X(\phi_2)\subset\cD_{X,\phi_1}^{\geqslant1}$, we get $\Hom(F_1,F_2)=0$, as needed. 

Finally, to show~\ref{enum:slicing3}, let $F\in\Db(X)$ be an object.
Since $\cP_Y$ is a slicing of $\Db(Y)$, the object $\Phi(F) \in \Db(Y)$ has a HN filtration with factors $A_1,\dots,A_m\in\Db(Y)$, where $A_i\in\cP_Y(\phi_i)$ for some $\phi_1>\dots>\phi_m$. 
In the bounded t-structure $\tau'_{X,\phi_1}$, we consider the truncation triangle 
\[
F_{\geq\phi_1}\longrightarrow F\longrightarrow F_{<\phi_1},
\]
with $F_{\geq\phi_1}\in\cD_{X,\phi_1}^{\leqslant0}$ and $F_{<\phi_1}\in\cD_{X,\phi_1}^{\geqslant1}$.
Applying $\Phi$ gives a distinguished triangle 
\[
\Phi(F_{\geq\phi_1})\longrightarrow\Phi(F)\longrightarrow\Phi(F_{<\phi_1}),
\]
with $\Phi(F_{\geq\phi_1})\in\cP_Y(\geq\phi_1)$ and $\Phi(F_{<\phi_1})\in\cP_Y(<\phi_1)$. 
But since $\Phi(F)\in\cP_Y([\phi_m,\phi_1])$, this implies that in fact $\Phi(F_{\geq \phi_1}) \in \cP_Y(\phi_1)$.
Iterating the argument, in a finite number of steps we get the result.
\end{proof}

We finish the section by proving a variant of Corollary~\ref{cor:Polishchuk}.
In the main part of the paper it has been applied to supported categories (see Section~\ref{subsec:LocalFano}). 
It is the version for slicings of~\cite[Lemma 2.11]{MMS:Inducing}.

\begin{proposition}\label{prop:extendslicing}
Let $\Phi\colon\cD\to\cD'$ be an exact functor between triangulated categories, and let $\cP$ be a slicing of $\cD$. 
We assume:
\begin{enumerate}[{\rm (i)}]
    \item\label{enum:extendslicing1} $\cD'$ is generated by extensions by $\Phi(\cD)$, i.e., $\langle\Phi(F)\st F\in\cD\rangle=\cD'$;
    \item\label{enum:extendslicing2} the induced map
    \[
    \Phi\colon\Hom(F_1,F_2)\longrightarrow\Hom\left(\Phi(F_1),\Phi(F_2)\right)
    \]
    is an isomorphism, for any $\phi\in\RR$, $F_1\in\cP(\phi)$, and
    $F_2\in \cP(<\phi+1)$.
\end{enumerate}
Then the collection
\begin{equation}\label{eq:cP'}
\cP'(\phi)\coloneqq\langle\Phi(F)\st F\in\cP(\phi)\rangle,
\end{equation}
for $\phi\in\RR$, defines a slicing $\cP'$ of $\cD'$. 
\end{proposition}

\begin{proof}
Condition~\ref{enum:slicing1} in Definition~\ref{def:slicing} is clear.
We notice that our assumption~\ref{enum:extendslicing2} in the statement implies that
\[
\Hom(\Phi(C_1),\Phi(C_2))=0
\]
for any $C_1\in\cP(\phi_1)$ and $C_2\in\cP(\phi_2)$, with $\phi_1>\phi_2$. 
Hence, condition~\ref{enum:slicing2} in Definition~\ref{def:slicing} follows as well.
It suffices to show that every nonzero object $F'\in\cD'$ admits a HN filtration with respect to $\cP'$.

By~\ref{enum:extendslicing1} in the statement, any nonzero $F'\in\cD'$ admits a filtration
\begin{equation}\label{eq:5.2}
0=F'_0\xlongrightarrow{f'_1}F'_1\xlongrightarrow{f'_2}\dots\xlongrightarrow{f'_m}F'_m=F'
\end{equation}
such that $A'_i\coloneqq\cone(f_i')=\Phi(A_i)$, for some $A_i\in\cD$.
Each $A_i$ admits a HN filtration with respect to $\cP$. 
Let $k_i\in\ZZ_{>0}$ denote the number of HN factors of $A_i$. 
We associate to the filtration $F_\bullet'$ the tuple
\[
N(F'_\bullet)\coloneqq (m,k_m,\dots,k_1).
\]
We let
\[
N(F')\coloneqq\min\left\{N(F'_\bullet)\st F'_\bullet \text{ is a filtration of $F$ of the form \eqref{eq:5.2}}\right\}.
\]
Here we are considering the minimum with respect to the lexicographic order, and this exists since all tuples consist of positive numbers.

To show the existence of the HN filtration for $F'$, we proceed by induction on $N(F')$.
Suppose first that $N(F')=(1,k_1)$.
Then $F'=\Phi(F)$, for some $F\in\cD$. 
Since $F$ admits a HN filtration with respect to $\cP$, applying $\Phi$ yields a HN filtration of $F'$ with respect to $\cP'$.

Assume now that $N(F')=(m,k_m,\dots,k_1)$, with $m\ge 2$.
Let $F'_\bullet$ be a filtration of the form~\eqref{eq:5.2} satisfying $N(F')=N(F'_\bullet)$.

\begin{claim*}
For every $i=2,\dots,m$, we have
\begin{equation}\label{eq:5.3}
\phi^+_\cP(A_{i-1})\geq\phi^+_\cP(A_i).
\end{equation}
\end{claim*}

\begin{proof}[Proof of the claim.]
Let $A_i^+$ denote the first HN factor of $A_i$, so that we have a distinguished triangle
\[
A_i^+\xlongrightarrow{g_i} A_i\longrightarrow \widetilde{A}_i.
\]
Applying $\Phi$, we obtain a distinguished triangle
\[
\Phi(A_i^+)\xlongrightarrow{g'_i=\Phi(g_i)} A'_i=\Phi(A_i)\longrightarrow \Phi(\widetilde{A}_i).
\]
In the filtration~\eqref{eq:5.2}, let
\[
F_{i-1}'\xlongrightarrow{f_i'}F_i'\longrightarrow A_i'\xlongrightarrow{h'_i} F'_{i-1}[1]
\]
be the associated exact triangle. 
We set
\[
B'_{i-1}\coloneqq\cone(h'_i\circ g'_i)[-1].
\]

Applying the octahedral axiom to the composition of $h_i'$ and $g_i'$, we get the exact triangle
\begin{equation}\label{eq5.5}
B_{i-1}'\longrightarrow F'_i\longrightarrow\Phi(\widetilde{A}_i).
\end{equation}
Applying it to the composition
\[
l'_i\colon\Phi(A_i^+)[-1]\xlongrightarrow{h'_i\circ g'_i}F'_{i-1}\longrightarrow A'_{i-1}=\Phi(A_{i-1}),
\]
we obtain the exact triangle
\begin{equation}\label{eq5.6}
F'_{i-2}\longrightarrow F'_{i-1}\longrightarrow \cone(l_i').
\end{equation}

Suppose that~\eqref{eq:5.3} fails, namely
\[
\phi^+_\cP(A_{i-1})<\phi_\cP(A^+_i)=\phi_\cP(A^+_i[-1])+1.
\]
Assumption~\ref{enum:extendslicing2} in the statement implies that $l_i'=\Phi(l_i)$ for some $l_i\in\Hom(A^+_i[-1],A_{i-1})$. 
In particular, we have
\[
\cone(l_i')\cong\Phi(\cone(l_i)).
\]

Replacing $F'_{i-1}$ by $B_{i-1}'$, equations \eqref{eq5.5} and \eqref{eq5.6} produce a new filtration $B'_\bullet$ of the form \eqref{eq:5.2}.
Moreover, either the total length of the filtration decreases (if $\Phi(\widetilde{A}_i)=0$), or the $i$-th factor has fewer HN factors.
In either case, we have $N(B_\bullet')<N(F'_\bullet)$, contradicting the minimality of $N(F'_\bullet)$. 
Therefore \eqref{eq:5.3} holds.
\end{proof}

Returning to the proof of the proposition, for each $i=2,\dots,m$, we consider the composition
\[
\Phi(A_1^+)\xlongrightarrow{g'_1} \Phi(A_1)=A'_1=F'_1\longrightarrow F'_i
\]
and we let $C_i'$ denote its cone. 
We set $C'=C_m'$.
We then have the following diagram:
\begin{equation}\label{eq:BigD1}
\begin{tikzcd}[column sep=large,row sep=large]
\Phi(A_1^+) \arrow[d]
    & \Phi(A_1^+) \arrow[l,equal] \arrow[d]
    & \dots \arrow[l,equal]
    & \Phi(A_1^+) \arrow[l,equal] \arrow[d]
    & \Phi(A_1^+) \arrow[l,equal] \arrow[d]
\\
F'_1 \arrow[r]\arrow[d]
    & F'_2 \arrow[r]\arrow[d]
    & \dots \arrow[r]
    & F'_{m-1} \arrow[r]\arrow[d]
    & F'\arrow[d]
\\
C_1' \arrow[r]
    & C_2' \arrow[r]
    & \dots \arrow[r]
    & C_{m-1}' \arrow[r]
    & C'.
\end{tikzcd}
\end{equation}
By the octahedral axiom, the objects $C_i'$ define a filtration of $C'$ whose factors satisfy
\begin{equation}\label{eq:5.11}
\begin{split}
&C_1'\cong\Phi(\widetilde{A}_1),\\
&\cone(C_{i-1}'\to C_i')\cong  A_i',\qquad i=2,\dots,m.
\end{split}
\end{equation}
In particular,
\[
N(C')\leq\begin{cases}
(m-1,k_m,\dots,k_2),\quad\text{ if } k_1=1,\\
(m,k_m,\dots,k_1-1),\quad\text{ otherwise.}
\end{cases}
\]
In either case, we have $N(C')<N(F')$.
By induction, the object $C'$ admits a HN filtration
\[
0=W_0'\xlongrightarrow{u_1'}W_1'\xlongrightarrow{u_2'}\dots\xlongrightarrow{u_r'} W_r'=C'
\]
such that $V_i'\coloneqq\cone(u_i')\in\cP'(\phi_i)$, for $\phi_1>\dots>\phi_r.$
By the claim above,~\eqref{eq:5.11}, and assumption~\ref{enum:extendslicing2} in the statement, we have
\[
\Hom(\Phi(D),C')=0,\qquad \text{for all } D\in\cP\left(>\phi_\cP(A_1^+)\right).
\]
Consequently, we have
\begin{equation}\label{eq:Phases1}
\phi_1\leq\phi_\cP(A_1^+).
\end{equation}
Finally, let $K'_i[1]$ denote the cone of the composition of morphisms
\[
W_i'\xlongrightarrow{u_r'\circ\dots\circ u_{i+1}'} C'\longrightarrow \Phi(A_1^+)[1],
\]
where the rightmost morphism is the one induced by the rightmost vertical exact triangle in~\eqref{eq:BigD1}.
Notice that $K'_r=F'$.
We have the following diagram:
\[
\begin{tikzcd}[column sep=large,row sep=large]
K_1' \arrow[r]\arrow[d]
    & K'_2 \arrow[r]\arrow[d]
    & \dots \arrow[r]
    & K'_{r-1} \arrow[r]\arrow[d]
    & F'\arrow[d]
\\
W_1' \arrow[r]\arrow[d]
    & W_2' \arrow[r]\arrow[d]
    & \dots \arrow[r]
    & W_{r-1}' \arrow[r]\arrow[d]
    & C'\arrow[d]
\\
\Phi(A_1^+)[1]
    & \Phi(A_1^+)[1] \arrow[l,equal]
    & \dots \arrow[l,equal]
    & \Phi(A_1^+)[1] \arrow[l,equal]
    & \Phi(A_1^+)[1]. \arrow[l,equal]
\end{tikzcd}
\]
By applying the octahedral axiom again, we have
\[
\cone(K'_{i-1}\to K'_i)\cong V_i'
\]
for all $i=2,\dots,r$.

We look at the leftmost vertical exact triangle
\[
\Phi(A_1^+)\longrightarrow K_1'\longrightarrow W_1'=V_1',
\]
where $\Phi(A_1^+)\in\cP'(\phi_\cP(A_1^+))$ and $V_1'\in\cP'(\phi_1)$.
By using~\eqref{eq:Phases1}, we have that either $\phi_\cP(A_1^+)=\phi_1$, in which case 
$K_1'\in\cP'(\phi_1)$, and so $K_\bullet'$ is a HN filtration with respect to $\cP'$ of $F'$, or $\phi_\cP(A_1^+)>\phi_1$, in which case the HN filtration of $F'$ with respect to $\cP'$ is
\[
0\longrightarrow\Phi(A_1^+)\longrightarrow K_1'\longrightarrow\dots\longrightarrow K_{r}'=F'.
\]
This completes the proof.
\end{proof}

It is immediate then to obtain a version of Proposition~\ref{prop:extendslicing} for (pre-)stability conditions.
In fact, let us fix $(\Lambda,v)$, where $v\colon\rK_0(\cD)\to\Lambda$.
We consider an exact functor $\Phi\colon\cD\to\cD'$ and the induced morphism
$\Phi_*\colon\rK_0(\cD)\to\rK_0(\cD')$.
Let us consider
\[
\Lambda'\coloneqq \Lambda/(v(\ker(\Phi_*)))_{\mathrm{sat}},\qquad v'\colon\rK_0(\cD')\longrightarrow\Lambda',
\]
where $(-)_{\mathrm{sat}}$ denotes the saturation, and $v'$ is the induced morphism.
For a group homomorphism $Z\colon\Lambda\to\CC$ such that $Z|_{v(\ker(\Phi_*))}=0$, we then have an induced group homomorphism
\begin{equation}\label{eq:DefOfZ'}
Z'\colon\Lambda'\longrightarrow\CC.
\end{equation}

\begin{proposition}\label{prop:extendstabcond}
In the above setting, let $\sigma=(Z,\cP)$ be a (pre-)stability condition on $\cD$ with respect to $(\Lambda,v)$ such that
\begin{itemize}
    \item $\Phi$ and $\cP$ satisfy assumptions~\ref{enum:extendslicing1} and~\ref{enum:extendslicing2} of Proposition~\ref{prop:extendslicing},
    \item $Z|_{v(\ker(\Phi_*))}=0$.
\end{itemize}
Then $\sigma'=(Z',\cP')$ is a (pre-)stability condition on $\cD'$ with respect to $(\Lambda',v')$, where $Z'$ and $\cP'$ are defined in~\eqref{eq:DefOfZ'} and~\eqref{eq:cP'}, respectively. 
\end{proposition}


\section{Almost disconnected morphisms}\label{app:Filtration}

In this appendix, we show that the morphism $h\colon(\PP^1)^n\to\PP^n$ given by the quotient by the symmetric group has the filtration property (see Example~\ref{ex:PushForward}\eqref{enum:ProductP1}). 
This provides an alternative argument to~\cite[Section 4]{chunyi-stability}.

We begin with the following property of a morphism.

\begin{definition}\label{def:AlmostDisconnected}
A morphism of schemes $p\colon X\to Y$ is \emph{almost disconnected} if there exists a filtration
\[
0=F_0 \subset F_1 \subset\dots\subset F_{m-1}\subset F_m=\cO_X,
\]
such that, for each $i=1,\dots,m$,
\[
F_i/F_{i-1}\cong \iota_{i,*}\cL_i^{-1},
\]
where $\iota_i\colon X_i\hookrightarrow X$ is a closed subscheme such that $p|_{X_i}\colon X_i\xrightarrow{\cong}Y$ is an isomorphism, and $\cL_i$ is an invertible sheaf on $X_i$.
Furthermore, if $Y$ is proper over a field $\KK$, we say that $p$ is \emph{nef almost disconnected} if it is almost disconnected and the invertible sheaves $\cL_i$ are nef.
\end{definition}

We have the following basic properties.

\begin{lemma}\label{lemma-filt-basic}
The property of being almost disconnected is preserved under flat base change and composition.
\end{lemma}

\begin{proof}
The first statement is clear as flat pullback preserves exactness and isomorphisms. 

For the second statement, let $X\to Y$ and $Y\to Z$ be two almost disconnected morphisms. 
Then we have a filtration $F_\bullet$ of $\cO_X$ such that $F_i/F_{i-1}$ is the dual of an invertible sheaf $\cL_i$ on $X_i\cong Y$.
By tensoring the filtration $G_\bullet$ of $\cO_Y$ with $\cL_i^{-1}$, this can be lifted to a refinement of $F_{i-1}\subset F_{i}$. 
It is easy to see that this refinement satisfies the required properties.
\end{proof}

\begin{remark}\label{rmk:FiltBasicNefDisconnected}
Lemma~\ref{lemma-filt-basic} is true for nef almost disconnected morphisms, if the schemes involved are proper over $\KK$.
Indeed, nefness is preserved under pullback and tensor products.
\end{remark}

In these terms, the filtration property for a morphism, introduced in Definition~\ref{def:FiltrationProperty}, admits the following description. 
\begin{lemma}\label{lem:FiltrationVSAlmost} 
Let $f\colon X\to Y$ be a separated morphism.    
Then $f$ has the filtration property if and only if the projection $p_1\colon X\times_Y X\to X$ is almost disconnected and the corresponding subschemes $X_i$ have the property that $p_2|_{X_i}$ is also an isomorphism. 

More precisely, the automorphisms $g_i\in\Aut_Y(X)$ in Definition~\ref{def:FiltrationProperty} correspond to the compositions $p_2\circ(p_1)|_{X_i}^{-1}$ in Definition~\ref{def:AlmostDisconnected}, and the invertible sheaves in Definition~\ref{def:FiltrationProperty} correspond to $g_i^*\cL_i$ for $\cL_i$ as in Definition~\ref{def:AlmostDisconnected}. 
\end{lemma}

\begin{proof}
By definition, if $f$ has the filtration property, then both projections are almost disconnected with the condition in the statement.
Conversely, if $p_1$ is almost disconnected and $p_2|_{X_i}$ is an isomorphism, then by defining $g_i\coloneqq p_2\circ(p_1)|_{X_i}^{-1}$ as above and considering $g_i^*\cL_i$, it is immediate to see that the filtration that makes $p_1$ an almost disconnected morphism provides the filtration that makes $f$ a morphism with the filtration property.
\end{proof}

As a consequence, we deduce the following result.

\begin{lemma}\label{lem:FiltrationBaseChange}
Let $f \colon X \to Y$ be a separated morphism. 
Then the property that $f$ has the filtration property is preserved under flat base change. 
\end{lemma}

\begin{proof}
Let $Y'\to Y$ be a flat morphism. 
Let $f'\colon X'\to Y'$ be the base change of $f$.
Then the projection $p_1'\colon X'\times_{Y'}X'\to X'$ is the base change of $p_1\colon X\times_Y X\to X$ by the flat morphism $X'\to X$; hence, $p_1'$ is almost disconnected, by Lemma~\ref{lemma-filt-basic} and Lemma~\ref{lem:FiltrationVSAlmost}.
Then, for all $i$, the supporting closed subscheme $X_i'$ in the base change filtration $F_\bullet'$ of $F_\bullet$ is just the base change $X_i\times_Y Y'$, and thus the fact that $p_2'|_{X_i'}$ is an isomorphism follows immediately.
We can then use Lemma~\ref{lem:FiltrationVSAlmost} again to conclude.
\end{proof}

We are now ready to show our result.

\begin{proposition}\label{prop:P1ntoPnHasFiltrationProperty}
Over an arbitrary field $\KK$, 
let $h\colon(\PP^1)^n\to\PP^n$ be the morphism given by the quotient by the action of the symmetric group $\fS_n$ which permutes the factors. 
Then $h$ 
has the filtration property with respect to the automorphisms $f_1,\dots,f_m$, $m = n!$, given by the elements of $\fS_n$ and invertible sheaves $\cL_1, \dots, \cL_m$ on $(\PP^1)^n$ which are effective. 
\end{proposition}

Of course, for an invertible sheaf on $(\PP^1)^n$, being effective is equivalent to being nef, which is also equivalent to being of the form $\cO_{(\PP^1)^n}(a_1,\dots,a_n)$, with $a_1,\dots,a_n\in\ZZ_{\geq0}$.

\begin{proof}
We proceed by induction on $n$.
The case $n=1$ is clear.
We therefore assume $n\geq2$, and factor $h$ as
\[
(\PP^1)^n\xlongrightarrow{u}\PP^1\times\PP^{n-1}\xlongrightarrow{v}\PP^n,
\]
where $u$ is the quotient by $\fS_{n-1}$ on the last $n-1$ factors. 
Explicitly, thinking of $\PP^{n-1}$ as parameterizing effective degree $n-1$ divisors on $\PP^1$ and similarly for $\PP^n$, we have 
\begin{equation*}
u(x_1, \dots, x_n) = (x_1, x_2 + \cdots + x_n) \quad \text{and} \quad 
v(y, D) = y + D. 
\end{equation*}

Let 
\[
M = (\PP^1)^n\times_{\PP^n} (\PP^1)^n, 
\]
and consider the diagram of fiber products
\[
\xymatrix{
M\ar[d]_{p_1}\ar[r]^-{u'}\ar[d]& N \ar[r]^-{v'}\ar[d] & (\PP^1)^n \ar[d]\ar[d]^{h}\\
(\PP^1)^n \ar[r]^-{u} & \PP^1\times \PP^{n-1} \ar[r]^-{v}                     & \PP^n.
}
\]
Note that $u$ and $v$ are finite surjective morphisms between regular schemes, hence flat, so all morphisms in the diagram are flat.

We claim that $v'\colon N \to (\PP^1)^n$ is nef almost disconnected. 
First, from its definition as the fiber product, $N$ corresponds to the locus of points $(x, D, (y_1, \dots, y_n))$ in $\PP^1 \times \PP^{n-1} \times (\PP^1)^n$ such that 
\begin{equation*}
    x + D = y_1 + \dots + y_n. 
\end{equation*}
It follows that $N$ identifies with the incidence divisor in $\PP^1 \times (\PP^1)^n$ consisting of points $(x, (y_1, \dots, y_n))$ for which $x = y_i$ for some $i$. Thus, we have an equality of effective Cartier divisors 
\begin{equation*}
    N = N_1 + \dots + N_n
\end{equation*}
in $\PP^1 \times (\PP^1)^n$, where $N_i \subset \PP^1 \times (\PP^1)^n$ is the locus where $x = y_i$. 
In particular, the projection $v' \colon N \to (\PP^1)^n$ restricts to an isomorphism on each $N_i$. 
For $0 \leq i \leq n$, we define 
\begin{equation*}
    F_i \coloneqq \frac{\cO_{\PP^1 \times (\PP^1)^n}(-N_{i+1} - \cdots - N_n)}{\cO_{\PP^1 \times (\PP^1)^n}(-N)}. 
\end{equation*}
These form a filtration 
\begin{equation}
\label{Fi-ON}
    0 = F_0 \subset F_1 \subset \cdots \subset F_n = \cO_{N}
\end{equation}
such that 
\begin{equation*}
    F_i/F_{i-1}\cong \iota_{i,*} A_i^{-1}, 
\end{equation*}
where $\iota_{i} \colon N_i \to N$ is the inclusion and $A_i = \cO_{\PP^1\times (\PP^1)^n}(N_{i+1} +\dots +N_{n})|_{N_i}$. 
Since the $A_i$ are nef, this proves our  claim. 

Let $M_i = M \times_{N} N_i$ be the pullback of $N_i$ along $u'$. 
It is easy to see that $M_i \subset M$ corresponds to the locus of points $((x_1, \dots, x_n), (y_1, \dots, y_n))$ in $(\PP^1)^n \times_{\PP^n} (\PP^1)^n$ such that $x_1 = y_i$. 
Therefore, there is an isomorphism 
\begin{equation*}
M_i \cong \PP^1 \times ((\PP^1)^{n-1} \times_{\PP^{n-1}} (\PP^1)^{n-1})
\end{equation*} 
given by 
\begin{equation*}
((x_1, \dots, x_n), (y_1, \dots, y_n)) \mapsto (x_1 = y_i, (x_2, \dots, x_n), (y_1, \dots, \widehat{y_i}, \dots, y_n)). 
\end{equation*} 
By the previous paragraph, we also have an isomorphism $N_i \cong \PP^1 \times (\PP^1)^{n-1}$ given by the restriction of $\id_{\PP^1} \times q_i \colon \PP^1 \times (\PP^1)^n \to \PP^1 \times (\PP^1)^{n-1}$, where $q_i \colon (\PP^1)^n \to (\PP^1)^{n-1}$ forgets the $i$th coordinate. 
Under these isomorphisms, the map 
\begin{equation*}
    u'_i \colon M_i \to N_i 
\end{equation*}
is identified with the map 
\begin{equation*}
    \id_{\PP^1} \times p_2 \colon \PP^1 \times  ((\PP^1)^{n-1} \times_{\PP^{n-1}} (\PP^1)^{n-1}) \to \PP^1 \times (\PP^1)^{n-1}. 
\end{equation*}
By the inductive hypothesis and Lemma~\ref{lem:FiltrationVSAlmost}\footnote{Lemma~\ref{lem:FiltrationVSAlmost} is stated in terms of the first projection $p_1$, but by symmetry it implies a version for $p_2$ as well, which is what we use here.}, 
$p_2 \colon (\PP^1)^{n-1} \times_{\PP^{n-1}} (\PP^1)^{n-1} \to (\PP^1)^{n-1}$ is nef almost disconnected. 
Thus, since by the above $u'_i$ is a flat base change of this map, by Lemma~\ref{lemma-filt-basic} and Remark~\ref{rmk:FiltBasicNefDisconnected} it is also nef almost disconnected. By construction, the supporting subschemes in the resulting filtration of $\cO_{M_i}$ are given by $\PP^1 \times \Gamma(\tau)$, where $\tau \in \fS_{n-1}$. 

Since $u' \colon M \to N$ is flat, pulling back the filtration $F_{\bullet}$ from~\eqref{Fi-ON} provides a filtration 
\begin{equation*}
0 = F'_0 \subset F'_1 \subset \cdots \subset F'_n = \cO_M 
\end{equation*} 
such that 
\begin{equation*}
    F'_{i}/F'_{i-1} \cong \iota'_{i,*} (A'_i)^{-1},
\end{equation*}
where $\iota'_i \colon M_i \to M$ is the inclusion and $A'_i = (u'_i)^*A_i$ is nef. 
Tensoring the filtration of $\cO_{M_i}$ provided by the previous paragraph and pushing forward along $\iota'_i$, we can then refine $F'_{\bullet}$ to a filtration 
\begin{equation*}
    0 = G_0 \subset G_1 \subset \cdots \subset G_{m} = \cO_{M}, 
\end{equation*}
of length $m = n!$. 
It is easy to see from the construction that the graded pieces are given by the pushforward of the inverse of nef line bundles from the graphs $\Gamma(\sigma)$ of permutations $\sigma \in \fS_{n}$. 
This proves that $h$ has the filtration property of the desired form.
\end{proof}


\addtocontents{toc}{\vspace{\normalbaselineskip}}



\begin{thebibliography}{99}

\bibitem{AJT:tstructures}
Alonso Tarr\'io, L., Jerem\'ias L\'opez, A., Souto Salorio, M., Construction of t-structures and equivalences of derived categories, {\it Trans. Amer. Math. Soc.} {\bf 355} (2003), 2523--2543.

\bibitem{AHLH}
Alper, J., Halpern-Leistner, D., Heinloth, J., Existence of moduli spaces for algebraic stacks, {\it Invent. Math.} {\bf 234} (2023), 949--1038.

\bibitem{ABM:Stability}
Anno, R., Bezrukavnikov, R., Mirkovi\'c, I., Stability conditions for Slodowy slices and real variations of stability, {\it Mosc. Math. J.} {\bf 15} (2015), 187--203, 403.

\bibitem{ArcaraBertram}
Arcara, D., Bertram, A., Bridgeland-stable moduli spaces for K-trivial surfaces (with an appendix by M. Lieblich), {\it J. Eur. Math. Soc. (JEMS)} {\bf 15} (2013), 1--38.

\bibitem{balmer}
Balmer, P., Dell'Ambrogio, I., Sanders, B., Grothendieck--{N}eeman duality and the {W}irthm\"uller isomorphism, {\it Compos. Math.} {\bf 152} (2016), 1740--1776.

\bibitem{bayer:short-proof}
Bayer, A., A short proof of the deformation property of Bridgeland stability conditions, {\it Math. Ann.} {\bf 375} (2019), 1597--1613.

\bibitem{stability-families}
Bayer, A., Lahoz, M., Macr\`i, E., Nuer, H., Perry, A., Stellari, P., Stability conditions in families, {\it Publ. Math. Inst. Hautes \'Etudes Sci.} {\bf 133} (2021), 157--325.

\bibitem{BLMS}
Bayer, A., Lahoz, M., Macr\`i, E., Stellari, P., Stability conditions on Kuznetsov components, {\it Ann.\ Sci.\ \'Ec.\ Norm.\ Sup\'er.} {\bf 56} (2023), 517--570.

\bibitem{BMLocalP2}
Bayer, A., Macr\`i, E., The space of stability conditions on the local projective plane, {\it Duke Math. J.} {\bf 160} (2011), 263--322.

\bibitem{BM-nef-divisor}
\bysame, Projectivity and birational geometry of Bridgeland moduli spaces, {\it J. Amer. Math. Soc.} {\bf 27} (2014), 707--752.

\bibitem{BMS:StabCY3s}
Bayer, A., Macr\`i, E., Stellari, P., The space of stability conditions on abelian threefolds, and on some {C}alabi-{Y}au threefolds, {\it Invent. Math.} {\bf 206} (2016), 1--65.

\bibitem{BMT:BG}
Bayer, A., Macr\`i, E., Toda, Y., Bridgeland stability conditions on threefolds I: Bogomolov-Gieseker type inequalities, {\it J. Algebraic Geom.} {\bf 23} (2014), 117--163.


\bibitem{BBDG}
Beilinson, A., Bernstein, J., Deligne, P., Gabber, O., Faisceaux pervers, {\it Ast\'erisque} {\bf 100} (2018), vi+180 pp.


\bibitem{Bondal}
Bondal, A., Representations of associative algebras and coherent sheaves,
{\it Izv.\ Akad.\ Nauk SSSR Ser. Mat.} {\bf 53} (1989), 25--44.

\bibitem{BO}
Bondal, A., Orlov, D., Reconstruction of a variety from the derived category
and groups of autoequivalences, {\it Compos.\ Math.} {\bf 125} (2001), 327--344.

\bibitem{BVDB}
Bondal, A., van den Bergh, M., Generators and representability of functors in commutative and noncommutative geometry, {\it Mosc. Math. J.} {\bf 3} (2003), 1--36, 258.

\bibitem{Bou:Takahashi}
Bousseau, P., A proof of N.~Takahashi's conjecture for $(\PP^2,E)$ and a refined sheaves/Gromov--Witten correspondence, {\it Duke Math. J.} {\bf 172} (2023), 2895--2955.

\bibitem{BDLP:Dentroscopy}
Bousseau, P., Descombes, P., Le Floch, B., Pioline, B., BPS dendroscopy on local $\PP^2$, {\it Comm. Math. Phys.} {\bf 405} (2024), Paper No. 108, 98 pp.

\bibitem{Bridgeland:NonCompactCY}
Bridgeland, T., Stability conditions on a non-compact Calabi--Yau threefold, {\it Comm. Math. Phys.} {\bf 266} (2006), 715--733.

\bibitem{Bridgeland:Stab}
\bysame, Stability conditions on triangulated categories, {\it Ann. of Math. (2)} {\bf 166} (2007), 317--345.

\bibitem{Bridgeland:K3}
\bysame, Stability conditions on K3 surfaces, {\it Duke Math. J.} {\bf 141} (2008), 241--291.

\bibitem{BridgelandSmith:IHES}
Bridgeland, T., Smith, I., Quadratic differentials as stability conditions,
{\it Publ. Math. Inst. Hautes \'Etudes Sci.} {\bf 121} (2015), 155--278.

\bibitem{CNS:Passage}
Canonaco, A., Neeman, A., Stellari, P., The passage among the subcategories of weakly approximable triangulated categories (with an appendix by C. Haesemeyer), eprint {\tt arXiv:2402.04605}.


\bibitem{FeyYiran:Threefolds}
Cheng, Y., Feyzbakhsh, S., Stability conditions on threefolds, eprint {\tt arXiv:2607.04788}.

\bibitem{hannah:stabonabquot}
Dell, H., Stability conditions on free abelian quotients, {\it \'Epijournal G\'eom. Alg\'ebrique} {\bf 9} (2025), Art.~16, 39 pp.

\bibitem{FKLR:CY}
Feyzbakhsh, S., Koseki, N., Liu, Z., Rekuski, N., Stability conditions on Calabi--Yau threefolds via Brill--Noether theory of curves, eprint {\tt arXiv:2509.24990}.

\bibitem{FLZ:Albanese}
Fu, L., Li, C., Zhao, X., Stability manifolds of varieties with finite {A}lbanese morphisms, {\it Trans. Amer. Math. Soc.} {\bf 375} (2022), 5669--5690.

\bibitem{HKK}
Haiden, F., Katzarkov, L., Kontsevich, M., Flat surfaces and stability structures, {\it Publ. Math. Inst. Hautes \'Etudes Sci.} {\bf 126} (2017), 247--318.

\bibitem{stab-E3}
Haiden, F., Sung, B., The stability manifold of $E\times E \times E$, eprint {\tt arXiv:2410.08028}.

\bibitem{DHL-robotis}
Halpern-Leistner, D., Robotis, A., The space of augmented stability conditions, eprint {\tt arXiv:2501.00710}.

\bibitem{HP:PeriodIndex}
Hotchkiss, J., Perry, A., The period-index conjecture for abelian threefolds and Donaldson--Thomas theory, eprint {\tt arXiv:2405.03315}.

\bibitem{HS:Twisted}
Huybrechts, D., Stellari, P., Proof of C\u{a}ld\u{a}raru's conjecture, in {\it Moduli spaces and arithmetic geometry}, 31--42, Adv. Stud. Pure Math. {\bf 45}, Math. Soc. Japan, Tokyo, 2006.

\bibitem{King}
King, A., Moduli of representations of finite-dimensional algebras, {\it Quart. J. Math. Oxford Ser. (2)} {\bf 45} (1994), 515--530.

\bibitem{KLP}
Kuznetsov, A., Liu, S., Perry, A., Inducing t-structures on semiorthogonal components, eprint {\tt arXiv:2606.26193}.

\bibitem{Li:FanoThreefolds}
Li, C., Stability conditions on Fano threefolds of Picard number 1, {\it J. Eur. Math. Soc. (JEMS)} {\bf 21} (2019), 709--726.

\bibitem{Li:CY}
\bysame, On stability conditions for the quintic threefold, {\it Invent. Math.} {\bf 218} (2019), 301--340.

\bibitem{Li:sb}
\bysame, A real reduction of the manifold of Bridgeland stability conditions, eprint {\tt arXiv:2506.21995}.

\bibitem{chunyi-stability}
\bysame, A remark on stability conditions on smooth projective varieties, eprint {\tt arXiv:2601.22994}.

\bibitem{curves}
Li, C., Macr\`i, E., Perry, A., Stellari, P., Zhao, X., Stability conditions on products of curves and Hilbert schemes of surfaces, eprint {\tt arXiv:2512.14207}.


\bibitem{lieblich:moduli-of-complex}
Lieblich, M., Moduli of complexes on a proper morphism, {\it J.\ Algebraic Geom.} \textbf{15} (2006), 175--206.

\bibitem{Lieblich:Twisted}
\bysame, Moduli of twisted sheaves, {\it Duke Math.\ J.} {\bf 138} (2007), 23--118.

\bibitem{LN}
Lipman, J., Neeman, A., Quasi-perfect scheme-maps and boundedness of the twisted inverse image functor, {\it Illinois J.\ Math.} {\bf 51} (2007), 209--236.

\bibitem{Peize-cubic-fivefolds}
Liu, P., Stability conditions and moduli spaces on the Kuznetsov component of cubic fivefolds, eprint {\tt arXiv:2509.21454}.

\bibitem{Liu:StabProd}
Liu, Y., Stability conditions on product varieties, {\it J.\ Reine Angew.\ Math.} \textbf{770} (2021), 135--157.

\bibitem{LiuMao:Tilt}
Liu, Z., Mao, T., Tilt-stability on singular schemes and Bogomolov--Gieseker-type inequalities, eprint {\tt arXiv:2605.13808}.

\bibitem{macricurves}
Macr\`i, E., Stability conditions on curves, {\it Math.\ Res.\ Lett.} \textbf{14} (2007), 657--672.

\bibitem{MMS:Inducing}
Macr\`i, E., Mehrotra, S., Stellari, P., Inducing stability conditions, {\it J.\ Algebraic Geom.} {\bf 18} (2009), 605--649.

\bibitem{milne:alg-group}
Milne, J.~S., {\it Algebraic groups}, Cambridge Stud. Adv. Math. {\bf 170}, Cambridge University Press, Cambridge, 2017.

\bibitem{mukai:semihomogeneous}
Mukai, S., Semi-homogeneous vector bundles on an Abelian variety, {\it J.\ Math.\ Kyoto Univ.} \textbf{18} (1978), 239--272.

\bibitem{neeman-Grothendieck-duality}
Neeman, A., The Grothendieck duality theorem via Bousfield's techniques and Brown representability, {\it J.\ Amer.\ Math.\ Soc.} {\bf 9} (1996), 205--236.

\bibitem{neeman-derived-cats-grothendieck-duality}
\bysame, Derived categories and Grothendieck duality, in {\it Triangulated categories}, 290--350, London Math. Soc. Lecture Note Ser. {\bf 375}, Cambridge University Press, Cambridge, 2010.

\bibitem{Neeman:BoundedTstructPerfect}
\bysame, Bounded $t$-structures on the category of perfect complexes,
{\it Acta Math.} {\bf 233} (2024), 239--284.

\bibitem{Orlov}
Orlov, D., Equivalences of derived categories and K3 surfaces, {\it J. Math. Sci. (New York)} {\bf 84} (1997), 1361--1381.

\bibitem{NCHPD}
Perry, A., Noncommutative homological projective duality, {\it Adv.\ Math.} {\bf 350} (2019), 877--972.

\bibitem{PPZ:GM}
Perry, A., Pertusi, L., Zhao, X., Stability conditions and moduli spaces for Kuznetsov components of Gushel--Mukai varieties, {\it Geom.\ Topol.} {\bf 26} (2022), 3055--3121.

\bibitem{PT:Moduli}
Piyaratne, D., Toda, Y., Moduli of Bridgeland semistable objects on 3-folds and Donaldson--Thomas invariants, {\it J. Reine Angew. Math.} {\bf 747} (2019), 175--219.

\bibitem{P-t-structures}
Polishchuk, A., Constant families of {$t$}-structures on derived categories of coherent sheaves, {\it Mosc.\ Math.\ J.} \textbf{7} (2007), 109--134.

\bibitem{Polishchuk:lagrangian}
\bysame, Phases of Lagrangian-invariant objects in the derived category
  of an abelian variety, {\it Kyoto J.\ Math.} \textbf{54} (2014), 427--482.

\bibitem{stacks-project}
The Stacks Project Authors, {\it Stacks Project}, \url{https://stacks.math.columbia.edu} (2026).

\bibitem{TodaK3}
Toda, Y., Moduli stacks and invariants of semistable objects on K3 surfaces, {\it Adv. Math.} {\bf 217} (2008), 2736--2781.

\bibitem{woolf}
Woolf, J., Some metric properties of spaces of stability conditions, {\it Bull.\ Lond.\ Math.\ Soc.} \textbf{44} (2012), 1274--1284.

\bibitem{WuZhang:LocalFano}
Wu, D., Zhang, N., Inducing stability conditions on canonical bundles over Fano varieties, in preparation.

\bibitem{Yos:Twisted}
Yoshioka, K., Moduli spaces of twisted sheaves on a projective variety, in {\it Moduli spaces and arithmetic geometry}, 1--30, Adv. Stud.\ Pure Math.\ {\bf 45}, Math.\ Soc.\ Japan, Tokyo, 2006.

\end{thebibliography}
\end{document}